\documentclass[11pt, draft]{amsart}
\usepackage{amssymb, amstext, amscd, amsmath, amssymb}
\usepackage{mathtools, xypic, paralist, color, dsfont, rotating}
\usepackage{verbatim}
\usepackage{enumerate,enumitem}
\usepackage{float}
\usepackage{setspace}
\usepackage{pifont}
\usepackage{tikz-cd}

\usepackage{marginnote}
\usepackage[a4paper]{geometry}
\floatstyle{boxed} 
\restylefloat{figure}

\numberwithin{equation}{section}
\let\OLDthebibliography\thebibliography
\renewcommand\thebibliography[1]{
  \OLDthebibliography{#1}
  \setlength{\parskip}{0pt}
  \setlength{\itemsep}{2pt plus 0.5ex}
}

\makeatletter
\def\@cite#1#2{{\m@th\upshape\bfseries%
[{#1\if@tempswa{\m@th\upshape\mdseries, #2}\fi}]}}
\makeatother

\theoremstyle{plain}
\newtheorem{theorem}{Theorem}[subsection]
\newtheorem{corollary}[theorem]{Corollary}
\newtheorem{proposition}[theorem]{Proposition}
\newtheorem{lemma}[theorem]{Lemma}

\theoremstyle{definition}
\newtheorem{definition}[theorem]{Definition}
\newtheorem{example}[theorem]{Example}

\newtheorem{remark}[theorem]{Remark}

\theoremstyle{remark}

\mathtoolsset{centercolon}

\newcommand{\B}{{\mathcal{B}}}

\newcommand{\F}{{\mathcal{F}}}

\newcommand{\I}{{\mathcal{I}}}
\newcommand{\J}{{\mathcal{J}}}
\newcommand{\K}{{\mathcal{K}}}
\renewcommand{\L}{{\mathcal{L}}}
\newcommand{\M}{{\mathcal{M}}}
\newcommand{\N}{{\mathcal{N}}}
\renewcommand{\O}{{\mathcal{O}}}
\renewcommand{\P}{{\mathcal{P}}}

\renewcommand{\S}{{\mathcal{S}}}
\newcommand{\T}{{\mathcal{T}}}

\def\al{\alpha}
\def\be{\beta}
\def\ga{\gamma}
\def\Ga{\Gamma}
\def\de{\delta}
\def\ze{\zeta}

\def\la{\lambda}
\def\La{\Lambda}

\def\si{\sigma}

\newcommand{\bC}{\mathbb{C}}

\newcommand{\bI}{\mathbb{I}}

\newcommand{\bN}{\mathbb{N}}
\newcommand{\bT}{\mathbb{T}}
\newcommand{\bZ}{\mathbb{Z}}

\newcommand{\fJ}{{\mathfrak{J}}}

\newcommand{\foral}{\text{ for all }}
\newcommand{\qand}{\quad\text{and}\quad}

\newcommand{\ca}{\mathrm{C}^*}

\newcommand{\ol}{\overline}

\newcommand{\Aut}{\operatorname{Aut}}

\newcommand{\End}{\operatorname{End}}

\newcommand{\id}{{\operatorname{id}}}

\newcommand{\inv}{{\operatorname{inv}}}
\newcommand{\Inv}{{\operatorname{Inv}}}

\newcommand{\mt}{\emptyset}

\newcommand{\PO}{\operatorname{PO}}

\newcommand{\spn}{\operatorname{span}}

\newcommand{\supp}{\operatorname{supp}}

\newcommand{\sca}[1]{\left\langle#1\right\rangle} 
\newcommand{\nor}[1]{\left\Vert #1\right\Vert} 
\newcommand{\un}[1]{{\underline{#1}}} 

\addtocontents{toc}{\protect\setcounter{tocdepth}{1}}

\begin{document}


\title[Equivariant Nica-Pimsner Quotients]{Equivariant Nica-Pimsner Quotients associated with Strong Compactly Aligned Product Systems}


\author[J.A. Dessi]{Joseph A. Dessi}
\address{School of Mathematics, Statistics and Physics\\ Newcastle University\\ Newcastle upon Tyne\\ NE1 7RU\\ UK}
\email{j.dessi@newcastle.ac.uk}

\author[E.T.A. Kakariadis]{Evgenios T.A. Kakariadis}
\address{School of Mathematics, Statistics and Physics\\ Newcastle University\\ Newcastle upon Tyne\\ NE1 7RU\\ UK}
\email{evgenios.kakariadis@newcastle.ac.uk}



\begin{abstract}
We parametrise the gauge-invariant ideals of the Toeplitz-Nica-Pimsner algebra of a strong compactly aligned product system over $\mathbb{Z}_+^d$ by using $2^d$-tuples of ideals of the coefficient algebra that are invariant, partially ordered, and maximal.
We give an algebraic characterisation of maximality that allows the iteration of a $2^d$-tuple to the maximal one inducing the same gauge-invariant ideal.
The parametrisation respects inclusions and intersections, while we characterise the join operation on the $2^d$-tuples that renders the parametrisation a lattice isomorphism.

The problem of the parametrisation of the gauge-invariant ideals is equivalent to the study of relative Cuntz-Nica-Pimsner algebras, for which we provide a generalised Gauge-Invariant Uniqueness Theorem.
We focus further on equivariant quotients of the Cuntz-Nica-Pimsner algebra and provide applications to regular product systems, C*-dynamical systems, strong finitely aligned higher-rank graphs, and product systems on finite frames.
In particular, we provide a description of the parametrisation for (possibly non-automorphic) C*-dynamical systems and row-finite higher-rank graphs, which squares with known results when restricting to crossed products and to locally convex row-finite higher-rank graphs.
\end{abstract}

\maketitle

\tableofcontents

\section{Introduction}\label{S:intro}

\subsection{Product systems}

Product systems offer a uniform language to encode a large number of C*-constructions associated with a unital subsemigroup $P$ of a discrete group $G$, e.g., C*-dynamical systems, higher-rank graphs, isometric semigroup representations, non-invertible dynamics, and subshifts.
They form the semigroup analogue of group Fell bundles, with the ability to encode transformations that may not be reversible.
Their associated C*-algebras provide a vast source for producing examples and counterexamples, and \emph{the current programme seeks to relate their C*-properties to properties of the geometric structure}.
Much progress has been made in this endeavour for $P = \bZ_+$, however many questions remain open even for $P = \bZ_+^d$.

When $P = \bZ_+$, the product system comes from a single C*-correspondence $X$.
The quantisation is employed by a Fock space construction where the elements of $X$ act as left creation operators, giving rise to the Toeplitz-Pimsner algebra $\T_X$.
There is a distinguished equivariant quotient of $\T_X$ that plays the role of the boundary or quasidiagonal universe for $X$, namely the Cuntz-Pimsner algebra $\O_X$, which is the minimal for equivariant representations that are injective on $X$.
The Cuntz-Pimsner algebra arises naturally in the C*-theory, as it encapsulates the Cuntz-Krieger algebra of a graph and the C*-crossed product of a single $*$-automorphism.

Due to their ample applications, C*-algebras of single C*-correspondences have been under thorough study.
Major developments in this direction include the parametrisation of equivariant ideals \cite{Kat07}, the study of ideal structure and simplicity \cite{CKO19}, K-theory computation \cite{Kat04} and classification \cite{BTZ18}, necessary and sufficient conditions for nuclearity and exactness \cite{Kat04}, and the parametrisation of the KMS-simplex \cite{LN04}.
The provided descriptions relate to structural properties of the C*-correspondence and its coefficient algebra, with direct translations into properties of the inducing geometric object.
For example, the gauge-invariant ideals of the Cuntz-Krieger algebra of a row-finite graph are in bijection with the hereditary saturated vertex sets of the ambient graph \cite{BPRS00}.

Several of these questions have been treated with much success beyond the $\bZ_+$-case.
One of the pivotal steps was introduced by Fowler \cite{Fow02} for compactly aligned product systems over quasi-lattices, where the associated C*-algebras admit a Wick ordering due to the Nica-Pimsner relations of the Fock representation.
The KMS-simplex of the Fock C*-algebra, and in particular KMS-states of finite type, was studied by Afsar, Neshveyev and Larsen \cite{ALN20}, encompassing the work of many hands. 
Extending further, Kwa\'{s}niewski and Larsen \cite{KL19a, KL19b} provided a detailed study of the C*-algebras associated with right LCM semigroups.
In this case the boundary quotient and its link to the C*-envelope was identified by Dor-On, the second-named author, Katsoulis, Laca and Li \cite{DKKLL20}, while nuclearity and exactness was answered by the second-named author, Katsoulis, Laca and Li \cite{KKLL21b}. 
In her breakthrough results, Sehnem \cite{Seh18, Seh21} identified in full generality the boundary quotient of the Fock representation through the strong covariance relations linking it to the C*-envelope of the nonselfadjoint tensor algebra.
In their recent work, Brix, Carlsen and Sims \cite{BCS23} explore the ideal structure for C*-algebras related to commuting local homeomorphisms, pushing the theory well beyond simplicity.

The aforementioned results fit in a wider programme for bringing the C*-algebras of product systems within Elliott's Classification Programme, which has been answered for $P=\bZ_+$ by Brown, Tikuisis and Zelenberg \cite{BTZ18}.
It asserts that if the coefficient algebra $A$ of a C*-correspondence $X$ is classifiable, the left action is by compacts, and $X$ admits a Rokhlin property, then $\O_X$ is classifiable.
One of the key features in the examination of $\O_X$ is that the strong covariance is encoded by a single ideal of $A$ coined by Katsura \cite{Kat04}.
However, the picture for general semigroups can be much more complicated, since the strong covariance relations make extensive use of families of submodules induced by the right ideal space of the semigroup \cite{Seh18}.
Even their description for compactly aligned product systems over quasi-lattices, as coined by Carlsen, Larsen, Sims and Vitadello \cite{CLSV11}, is based on checking relations on families of compact operators on all possible finite subsets of the semigroup. 

\subsection{Motivation}

When moving towards a classification in terms of \cite{BTZ18}, it is reasonable to ask whether the picture is simplified when the left actions are by compacts.
This has been achieved by Dor-On and the second-named author \cite{DK18} in the context of strong compactly aligned product systems over $P = \bZ_+^d$, thus opening a direction of lifting results from the $P = \bZ_+$ case.
Strong compact alignment induces an elimination argument on projections, while still accounting for major examples, e.g., C*-dynamical systems, row-finite higher-rank graphs, product systems on finite frames, and regular product systems.
The strong covariance relations here are reduced to checking covariance on a family of $2^d$ ideals of $A$, in direct connection to Katsura's ideal \cite{Kat04}.
This leads to our first point of motivation for the current work.
The parametrisation of the gauge-invariant ideals for $P = \bZ_+$ established by Katsura \cite{Kat07} is implemented by pairs of ideals of $A$ satisfying a compatibility relation.
Due to the analogy established in \cite{DK18} for the strong covariance relations, it is natural to ask whether a similar parametrisation lifts to the $P=\bZ_+^d$ case, as has been inquired in \cite[Question 9.2]{DK18}.

We draw further motivation from examples of strong compactly aligned product systems that have been examined in this respect.
The classical result appears in the case of automorphic dynamical systems over $\bZ^d$, where there is a lattice isomorphism between the gauge-invariant ideals of the crossed product $A \rtimes_\al \bZ^d$ and the $\al$-invariant ideals of $A$.
In the context of graph C*-algebras of locally convex row-finite higher-rank graphs, the parametrisation is obtained by hereditary saturated vertex sets through the work of Raeburn, Sims and Yeend \cite{RSY03, RSY04}.
Sims \cite{Sim06} further provided a parametrisation for finitely aligned higher-rank graphs, encoded by hereditary vertex sets together with satiated path sets.
In the current work we wish to encompass C*-dynamics and higher-rank graphs in a uniform way, and provide an alternative of \cite{Sim06} in the row-finite case depending only on properties of vertex sets.

Crossed products and graph C*-algebras are incarnations of the Cuntz-Nica-Pimsner algebra $\N\O_X$ of the associated product system $X$.
However, here we are interested in covering \emph{all} possible equivariant quotients of the Toeplitz-Nica-Pimsner algebra $\N\T_X$. 
As a base case, consider $\T \otimes \T$ (for the Toeplitz algebra $\T$) and use the vertices of the square (as $d=2$) for labelling ideals of $\T \otimes \T$ so that 
\begin{center}
\begin{tikzcd}
(0,1) \arrow[dash]{r} \arrow[dash]{d} & (1,1) \arrow[dash]{d} \\
(0,0) \arrow[dash]{r} & (1,0)
\end{tikzcd}
$\longrightarrow$
\begin{tikzcd}
\T \otimes \K \arrow[dash]{r} \arrow[dash]{d} & \K \otimes \K \arrow[dash]{d} \\
\{0\}\arrow[dash]{r} & \K \otimes \T
\end{tikzcd}
\end{center}
for the compact operators $\K \subseteq \T$.
Then the gauge-invariant ideals of $\T \otimes \T$ can be read off an inclusion-perserving association mapping vertex sets of the square to sums of the corresponding ideals.
In other words, we obtain
\begin{center}
\begin{tikzcd}
\{(0,1), (1,1)\} \arrow[dash]{r} \arrow[dash]{d} & \{(1,1)\} \arrow[dash]{d} \\
\{(0,0), (1,0), (0,1), (1,1)\} \arrow[dash]{r} & \{(1,0), (1,1)\}
\end{tikzcd}
$\longrightarrow$
\begin{tikzcd}
\T \otimes \K \arrow[dash]{r} \arrow[dash]{d} & \K \otimes \K \arrow[dash]{d} \\
\T\otimes \K + \K \otimes \T \arrow[dash]{r} & \K \otimes \T.
\end{tikzcd}
\end{center}
A similar decomposition for the boundary ideal $\ker\{ \N\T_X \to \N\O_X\}$ has been established in \cite{DK18} and is akin to the one of Deaconu \cite{Dea07}, further exploited by Fletcher \cite{Fle17}.
Such a decomposition has been used successfully for the computation of the K-theory of $\N\O_X$ in terms of $A$ for low-rank cases, such as finite $2$-rank graphs by Evans \cite{Eva08}, and for two commuting $*$-automorphisms by Barlak \cite{Bar16}.
The general case still remains unresolved, and our motivation is to shed more light on this construction.

At the same time, we are interested in equivariant quotients of $\N\T_X$ that may not be injective on $X$, as they arise naturally in the context of KMS-states.
There has been a great number of results on KMS-states for C*-algebras of finite graphs and dynamics, following the seminal work of Exel and Laca \cite{EL03}, and of Laca and Neshveyev \cite{LN04}.
Similar studies have been carried out for finite higher-rank graphs, with Christensen \cite{Chr20} providing the complete picture.
A parametrisation of the gauge-invariant KMS-states has been obtained by the second-named author in the presence of finite frames \cite{Kak20}.
One of the main tools in \cite{Chr20, Kak20} is the Wold decomposition of a KMS-state in parts that are $F$-finite and $F^c$-infinite, for $F \subseteq \{1, \dots, d\}$.
This corresponds to KMS-states annihilating the gauge-invariant ideal generated by the projections along $F^c$. 
The construction in \cite{Kak20} uses an ambient product system over $F$ with the coefficient algebra arising from the $F^c$-core.
Here we wish to close the circle with \cite{Kak20}, and provide a full characterisation for each $F$-quotient of the Toeplitz-Nica-Pimsner algebra as the Cuntz-Nica-Pimsner algebra of the $F$-induced product system.

There is an interplay between equivariant quotients and relative Cuntz-Nica-Pimsner algebras, for which we wish to establish a Gauge-Invariant Uniqueness Theorem.
This type of result was initiated by an Huef and Raeburn for Cuntz-Krieger algebras \cite{HR97}, and various generalisations were given by Doplicher, Pinzari and Zuccante \cite{DPZ98}, Fowler, Muhly and Raeburn \cite{FMR03}, and Fowler and Raeburn \cite{FR99}.
Katsura \cite{Kat07} proved the Gauge-Invariant Uniqueness Theorem in full generality for relative Cuntz-Pimsner algebras in the $P=\bZ_+$ case.
A further point of motivation has been to establish a similar result for $P = \bZ_+^d$.
This relies on analysing polynomial equations on the cores, as has been pointed out in \cite{DK18}.
Due to the Nica-Pimsner relations, the solutions need to adhere to invariance as well as a partial ordering; however simple examples show that different subsets of solutions may induce the same gauge-invariant ideal, thus posing the question of finding the appropriate compatibility conditions that characterise the maximal solution set.

\subsection{Summary of main results}

Before we provide the technical description of our results, let us present here the central points of this manuscript.
First of all we establish the parametrisation of the gauge-invariant ideals of the Toeplitz-Nica-Pimsner algebra $\N\T_X$ of a strong compactly aligned product system $X$ over $\bZ_+^d$, continuing the programme of \cite{DK18}.
This class contains product systems where the left actions are by compacts.
As a byproduct we obtain the parametrisation of the gauge-invariant ideals for any relative Cuntz-Nica-Pimsner algebra of $X$, including the boundary quotient $\N\O_X$.
In the process we produce a Gauge-Invariant Uniqueness Theorem for equivariant quotients in-between $\N\T_X$ and $\N\O_X$.
These results are in direct analogy to (and recover) the one-dimensional case \cite{Kat07}.

We apply our results to C*-dynamical systems over $\bZ_+^d$.
When the system is injective, the gauge-invariant ideals of the Cuntz-Nica-Pimsner algebra correspond to positively and negatively invariant ideals of $A$, hence recovering the classical C*-crossed product result for $*$-automorphisms.
Moreover, we recover the parametrisation of \cite{RSY04} for locally convex row-finite higher-rank graphs by hereditary saturated sets.
We provide the parametrisation in the case of (just) row-finite higher-rank graphs, showing that this too can be achieved through vertex sets alone.
Additionally, we interpret the parametrisation when the product system is regular; if in particular the coefficient algebra is simple, then $\N\O_X$ does not admit non-trivial gauge-invariant ideals.
In the presence of finite frames, we address the decomposition of \cite{Kak20}, and show that the quotient of $\N\T_X$ by the $F$-ideal can be realised as the boundary C*-algebra of a product system supported on $F$ with coefficients in the $F^c$-core (this is achieved in two ways).

\subsection{Description of main results}

Let us fix notation (see in conjunction with the general notation we adopt in Subsections \ref{Ss:notation} and \ref{Ss: scaps}).
We write $[d] := \{1, \dots, d\}$ for $d \in \bN$.
We write $\un{n}$ for the elements of $\bZ_+^d$ and will denote its generators by $\un{i}$ for $i \in [d]$.
We write $\un{n} \perp F$ for $F \subseteq [d]$ if $\supp \un{n} \cap F = \mt$.
Moreover we write $\un{1}_F:= \sum_{i \in F} \un{i}$ for $\mt \neq F \subseteq [d]$.
For a product system $X = \{X_{\un{n}}\}_{\un{n} \in \bZ_+^d}$ with coefficients in a C*-algebra $A$ and an ideal $I$ of $A$, we write
\[
X_{\un{n}}(I) := [\sca{X_{\un{n}}, I X_{\un{n}}}]
\qand
X_{\un{n}}^{-1}(I) := \{a \in A \mid \sca{X_{\un{n}}, a X_{\un{n}}} \subseteq I \}.
\]

A strong compactly aligned product system $X = \{X_{\un{n}}\}_{\un{n} \in \bZ_+^d}$ with coefficients in $A$ is a compactly aligned product system that in addition satisfies
\[
\K(X_{\un{n}}) \otimes \id_{X_\un{i}} \subseteq \K(X_{\un{n}} \otimes_A X_{\un{i}})
\textup{ whenever $\un{n}\perp\un{i}$, where $i\in[d],\un{n}\in\bZ_+^d\setminus\{\un{0}\}$}.
\]
A $2^d$-tuple $\L := \{\L_F\}_{F \subseteq [d]}$ of $X$ is a family of $2^d$ non-empty subsets of $A$.
We will say that $\L$ is \emph{invariant} if 
\[
[\sca{X_{\un{n}}, \L_F X_{\un{n}}}] \subseteq \L_F \text{ whenever } \un{n} \perp F,
\]
and we will say that $\L$ is \emph{partially ordered} if 
\[
\L_{F_1} \subseteq \L_{F_2} \text{ whenever } F_1 \subseteq F_2 \subseteq [d].
\]

If $(\pi,t)$ is a Nica-covariant representation of $X$ in some $\B(H)$, we write $\psi_{\un{n}}$ for the induced $*$-representation on $\K(X_{\un{n}})$.
For $i \in [d]$, we use an approximate unit $(k_{\un{i},\la})_{\la\in\La}$ of $\K(X_{\un{i}})$ to define the projection $p_{\un{i}}:=\textup{w*-}\lim_\la \psi_{\un{i}}(k_{\un{i},\la})$, and we set
\begin{equation*}
q_\mt:= I_H,
q_{\un{i}} := I_H - p_{\un{i}},
\text{ and }
q_F:=\prod_{i\in F}(I_H - p_{\un{i}}) \textup{ for $\mt\neq F\subseteq [d]$}.
\end{equation*}
The key relation is that if $a \in \bigcap \{ \phi_{\un{i}}^{-1}(\K(X_{\un{i}})) \mid i \in F \}$, then
\[
\pi(a) q_F = \pi(a) + \sum \{ (-1)^{|\un{n}|} \psi_{\un{n}}(\phi_{\un{n}}(a)) \mid \un{0} \neq \un{n} \leq \un{1}_F\},
\]
and thus $\pi(a) q_F \in \ca(\pi,t)$, although it may not be that $q_F \in \ca(\pi,t)$.
We reserve $(\ol{\pi}_X, \ol{t}_X)$ for the universal Nica-covariant representation of $X$.
Due to the aforementioned relation, if $\L$ is a $2^d$-tuple such that $\L_F \subseteq \bigcap \{ \phi_{\un{i}}^{-1}(\K(X_{\un{i}})) \mid i \in F \}$ for every $\mt \neq F \subseteq [d]$, then the ideal $\sca{\ol{\pi}_X(\L_F) \ol{q}_{X,F} \mid F \subseteq [d]}$ is a gauge-invariant ideal of $\N\T_X$.
We call such tuples \emph{relative}, and write $\N\O(\L,X)$ for the corresponding equivariant quotient.

The main result in \cite{DK18} is that $\N\O_X \cong \N\O(\I,X)$ for the family $\I := \{\I_F\}_{F \subseteq [d]}$, where
\[ 
\I_F := \bigcap\{X_\un{n}^{-1}(\J_F)\mid \un{n}\perp F\}
\quad \text{for} \quad
\J_F :=
(\bigcap_{i\in F}\ker\phi_{\un{i}})^\perp\cap(\bigcap_{i\in [d]}\phi_{\un{i}}^{-1}(\K(X_{\un{i}}))).
\]
We note that $\I_\mt = \J_\mt = \{0\}$.
Every $\I_F$ is $F^\perp$-invariant (in fact the largest $F^\perp$-invariant ideal of $\J_F$), and the family $\I$ is partially ordered.
In order to understand a general equivariant quotient, we pare down these properties.

First we consider the case where $\L$ is a $2^d$-tuple of $X$ satisfying $\L \subseteq \I$; we term such tuples as (E)-$2^d$-tuples (where ``E" stands for embedding).
By definition the quotient $\N\O(\L,X)$ lies in-between $\N\T_X$ and $\N\O_X$, and thus $X$ embeds into $\N\O(\L,X)$.
In Lemma \ref{L:e inv} and Lemma \ref{L:e po}, we show that we can then induce (E)-$2^d$-tuples $\Inv(\L)$ and $\PO(\L)$ defined by
\[
\Inv(\L)_F:= \ol{\spn} \{X_\un{n}(\L_F)\mid \un{n}\perp F\}
\qand
\PO(\L)_F:= \sum\{\sca{\L_D} \mid D\subseteq F\}
\]
for all $F \subseteq [d]$, such that $\L \subseteq \Inv(\L)$, $\L \subseteq \PO(\L)$ and
\[
\N\O(\L,X) = \N\O(\Inv(\L), X) = \N\O(\PO(\L), X).
\]
In particular $\Inv(\L)$ is invariant, $\PO(\L)$ is partially ordered, and $\PO(\L)$ is invariant if $\L$ is invariant.
Hence without loss of generality we may restrict to invariant, partially ordered (E)-$2^d$-tuples of ideals by replacing $\L$ with $\PO(\Inv(\L))$.
However, as we demonstrate in Example \ref{E:not en}, these properties alone do not provide injectivity of the association $\L \mapsto \N\O(\L,X)$.
It is not hard to see that there is a unique maximal (E)-$2^d$-tuple $\M$ such that $\PO(\Inv(\L)) \subseteq \M$ and 
\[
\N\O(\L, X) = \N\O(\PO(\Inv(\L)),X) = \N\O(\M,X).
\]
The prototypical example of maximal (E)-$2^d$-tuples is obtained from equivariant injective Nica-covariant representations $(\pi,t)$, by defining
\[
\L_F^{(\pi,t)} := \pi^{-1}( B_{(\un{0}, \un{1}_F]}^{(\pi,t)} ) \foral F \subseteq [d].
\]
Here $B_{(\un{0}, \un{1}_F]}^{(\pi,t)}$ denotes the $(\un{0}, \un{1}_F]$-core of $\ca(\pi,t)$, and $\L_\mt^{(\pi,t)} = \{0\}$ since $\pi$ is injective.
With this in hand, we make the following key observation towards a Gauge-Invariant Uniqueness Theorem (GIUT) for relative Cuntz-Nica-Pimsner algebras.

\medskip

\noindent
{\bf Theorem A.} \emph{(Theorem \ref{T:d GIUT M}) Let $X$ be a strong compactly aligned product system with coefficients in a C*-algebra $A$. 
Let $\L$ be a maximal (E)-$2^d$-tuple of $X$ and suppose that $(\pi,t)$ is a Nica-covariant representation of $X$. 
Then $\N\O(\L,X)\cong\ca(\pi,t)$ via a canonical $*$-isomorphism if and only if $(\pi,t)$ admits a gauge action and $\L^{(\pi,t)}=\L$.}

\medskip

Since maximality is a necessary ingredient for the GIUT, we give an algebraic characterisation without reference to a Nica-covariant representation.
To this end, for every $\mt \neq F \subsetneq [d]$ we define
\[
\L_{\inv, F} := \bigcap_{\un{m}\perp F}X_\un{m}^{-1}(\cap_{F\subsetneq D}\L_D)
\qand
\L_{\lim, F} := \{a\in A \mid \lim_{\un{m}\perp F}\|\phi_\un{m}(a)+\K(X_\un{m}\L_F)\|=0\}.
\]
Note that  the definitions of $\L_{\inv, F}$ and $\L_{\lim, F}$ do not require $\L$ to be an (E)-$2^d$-tuple.
When $\L$ is an (E)-$2^d$-tuple, we define the $2^d$-tuple $\L^{(1)}$ by
\[
\L_F^{(1)} 
:=
\begin{cases}
\{0\} & \text{ if } F = \mt, \\
\I_F \cap \L_{\inv, F} \cap \L_{\lim, F} & \text{ if } \mt \neq F \subsetneq [d], \\
\L_{[d]} & \text{ if } F = [d].
\end{cases}
\]
In Proposition \ref{P:L1 po} we show that $\L^{(1)}$ is an (E)-$2^d$-tuple of ideals that is invariant and partially ordered when $\L$ is so, satisfying
\[
\sca{\ol{\pi}_X(\L_F) \ol{q}_{X,F} \mid F \subseteq [d]}
=
\sca{\ol{\pi}_X(\L_F^{(1)}) \ol{q}_{X,F} \mid F \subseteq [d]}.
\]
Maximality is then described in terms of the first iteration.

\medskip

\noindent
{\bf Theorem B.} \emph{(Theorem \ref{T:m fam v2}) Let $X$ be a strong compactly aligned product system with coefficients in a C*-algebra $A$ and suppose that $\L$ is a $2^d$-tuple of $X$.
Then $\L$ is a maximal (E)-$2^d$-tuple of $X$ if and only if $\L$ satisfies the following four conditions:
\begin{enumerate}
\item $\L$ consists of ideals and $\L \subseteq \J$,
\item $\L$ is invariant,
\item $\L$ is partially ordered,
\item $\L^{(1)} \subseteq \L$.
\end{enumerate}
}

\medskip

For $k \in \bZ_+$ we write $\L^{(k+1)} := (\L^{(k)})^{(1)}$, and by convention we set $\L^{(0)}:=\L$.
By induction we have
\[
\L^{(k)} \subseteq \L^{(k+1)} \text{ and } \N\O(\L,X) = \N\O(\L^{(k)}, X), \text{ for all $k \in \bZ_+$}.
\]
Thus it is natural to ask if these iterations stabilise to the maximal (E)-$2^d$-tuple inducing $\N\O(\L,X)$.
In Theorem \ref{T:d-1 m fam} we show that 
\[
\L^{(d-1)} = \L^{(m)} \foral m \geq d-1,
\]
and thus $\L^{(d-1)}$ is the maximal (E)-$2^d$-tuple inducing $\N\O(\L,X)$.
Hence we obtain the maximal (E)-$2^d$-tuple inducing $\N\O(\L,X)$ by first enlarging $\L$ to an (E)-$2^d$-tuple that is invariant, partially ordered and consists of ideals, and then taking its $(d-1)$-iteration. 
Combining Theorem A with Theorem B then gives the full form of the Gauge-Invariant Uniqueness Theorem.

\medskip

\noindent
{\bf Theorem C.} \emph{(Theorem \ref{T:d GIUT E})
Let $X$ be a strong compactly aligned product system with coefficients in a C*-algebra $A$. 
Let $\L$ be an (E)-$2^d$-tuple of $X$ and $(\pi,t)$ be a Nica-covariant representation of $X$. 
Then $\N\O(\L,X)\cong\ca(\pi,t)$ via a canonical $*$-isomorphism if and only if $(\pi,t)$ admits a gauge action and 
\[
\L^{(\pi,t)}=\bigg(\PO\big(\Inv\left(\L\right) \big) \bigg)^{(d-1)}.
\]
}

\medskip

Next we pass to the parametrisation of the quotients that may not be injective on $X$.
We circumvent this by ``deleting the kernel", i.e., by utilising the quotient product system construction.
To this end, we write $[\hspace{1pt}\cdot\hspace{1pt}]_I$ for the quotient maps induced by a positively invariant ideal $I \subseteq A$, following the notation of \cite{Kat07}.
We say that $\L$ is an \emph{NT-$2^d$-tuple of $X$} if the family $[\L]_{\L_\mt} := \{ [\L_F]_{\L_\mt} \}_{F \subseteq [d]}$ is a maximal (E)-$2^d$-tuple of $[X]_{\L_\mt}$.
In Subsection \ref{Ss:NT tuple} we provide a detailed description of the structural properties that render a $2^d$-tuple $\L$ an NT-$2^d$-tuple.
In particular, in Proposition \ref{P:NTcom} we provide a characterisation of an NT-$2^d$-tuple with no reference to the quotient product system construction when $A$ acts on $X$ by compacts.
By performing this reduction and combining with the results for (E)-$2^d$-tuples, we obtain the parametrisation in its full generality.

\medskip

\noindent
{\bf Theorem D.} \emph{(Proposition \ref{P:JL conc}, Theorem \ref{T:NT param}). Let $X$ be a strong compactly aligned product system with coefficients in a C*-algebra $A$. 
Then there is a bijection between the set of NT-$2^d$-tuples of $X$ and the set of gauge-invariant ideals of $\N\T_X$ given by
\begin{align*}
\L & \mapsto \fJ^\L:= \ker \pi^\L \times t^\L ,\text{ for } \pi^\L \times t^\L \colon \N\T_X \to \N\O([\L]_{\L_\mt}, [X]_{\L_\mt}) \\
\fJ & \mapsto \L^\fJ:= \L^{(Q_\fJ \circ \ol{\pi}_X, Q_\fJ \circ \ol{t}_X)}, \text{ for } Q_\fJ \colon \N\T_X \to \N\T_X / \fJ.
\end{align*}
Moreover, if $\L$ is an NT-$2^d$-tuple of $X$, then we have
\begin{align*}
\mathfrak{J}^\L
=
\langle \ol{\pi}_X(a) +\sum_{\un{0}\neq\un{n}\leq\un{1}_F}(-1)^{|\un{n}|}\ol{\psi}_{X,\un{n}}(k_\un{n}) \mid 
& F\subseteq[d], a\in\L_F,  k_\un{n}\in\K(X_\un{n}), \\
& [\phi_\un{n}(a)]_{\L_\mt}=[k_\un{n}]_{\L_\mt} \; \textup{for all} \; \un{0}\neq\un{n}\leq\un{1}_F \rangle.
\end{align*}
}

\medskip

We note here that this description of $\fJ^\L$ in the $d=1$ case is not explicit in \cite{Kat07}, but it can be derived from \cite[Theorem 7.17]{Kwa13} with some additional work.
Moreover, the Nica-covariant representation $(\pi^\L, t^\L)$ is well-defined since the association $X \to [X]_{\L_\mt}$ lifts to a canonical $*$-epimorphism $\N\T_X \to \N\T_{[X]_{\L_\mt}}$.
It is known that this is not the case in general for $\N\O_X$ and $\N\O_{[X]_{\L_\mt}}$ (even for $d=1$, see Example \ref{E:surjcounter}).
Nevertheless, by using the NT-$2^d$-tuple machinery, we determine precisely when this holds as part of our applications in Subsection \ref{Ss:part}.
The key requirement is that $\L_\mt$ is both positively and negatively invariant, and $\I \subseteq \L$.

The bijection of Theorem C induces a lattice structure on the NT-$2^d$-tuples that renders it a lattice isomorphism, where we use the usual lattice operations for the gauge-invariant ideals.
It is then important to understand the join and meet operations in this setting.
The bijection is easily seen to preserve inclusions and intersections.
For the join operation we use the iteration process that describes the maximality property.

\medskip

\noindent
{\bf Theorem E.} \emph{(Proposition \ref{P:NTwedge}, Proposition \ref{P:NTvee}) Let $X$ be a strong compactly aligned product system with coefficients in a C*-algebra $A$.
We equip the set of NT-$2^d$-tuples of $X$ with the lattice structure determined by the operations
\begin{align*}
\L_1\vee\L_2  :=\L^{\fJ^{\L_1}+\fJ^{\L_2}} \qand
\L_1\wedge\L_2  :=\L^{\fJ^{\L_1}\cap\fJ^{\L_2}}.
\end{align*}
Then
\[
(\L_1\wedge\L_2)_F=\L_{1,F}\cap\L_{2,F}\foral F\subseteq[d],
\]
and
\begin{align*}
(\L_1\vee\L_2)_F
=
\begin{cases}
\ol{\pi}_X^{-1}(\fJ^{\L_1}+\fJ^{\L_2}) & \text{ if } F= \mt, \\
[\hspace{1pt}\cdot\hspace{1pt}]_{(\L_1\vee\L_2)_\mt}^{-1} \left[ \big( \left( \L_{1,F}+\L_{2,F}+(\L_1\vee\L_2)_\mt \right) / (\L_1\vee\L_2)_\mt \big)^{(d-1)} \right] & \text{ if } \mt\neq F\subseteq[d].
\end{cases}
\end{align*}
}

\medskip

The parametrisation on $\N\T_X$ descends naturally to $\N\O(\K,X)$ for any relative $2^d$-tuple $\K$.
The only difference is that the lattice isomorphism is by NT-$2^d$-tuples that contain $\K$.
By choosing $\K = \I$ we obtain the parametrisation on $\N\O_X$ by what we call \emph{NO-$2^d$-tuples}.

Next we pass to applying our results to specific classes of product systems.
First we deal with regular product systems, i.e., when each left action is injective and by compacts.
In Corollary \ref{C:regNObij} we show that the parametrisation for the gauge-invariant ideals of $\N\O_X$ is given by single ideals of $A$ that are positively and negatively invariant.
This relies on the fact that the positively and negatively invariant ideal will be the $\L_\mt$-member of an NO-$2^d$-tuple $\L$, while by regularity $\L_F = A$ for all $\mt \neq F \subseteq [d]$.
This implies that $\N\O_X$ does not admit non-trivial gauge-invariant ideals when $A$ is simple.
Passing to $\N\T_X$, in Corollary \ref{C:regNTbij} we show that the gauge-invariant ideals of $\N\T_X$ are in bijection with the families of incomparable subsets of $[d]$, when $A$ is non-zero and simple.
A direct application for the Toeplitz algebra $\T^{\otimes d}$ of $\bZ_+^d$ produces the parametrisation of its gauge-invariant ideals by vertex sets on the $d$-hypercube.

The second class of examples we consider arises from semigroup actions $\al \colon \bZ_+^d \to \End(A)$.
We write $X_\al$ for the associated product system.
We characterise the NT-$2^d$-tuples as follows.

\medskip

\noindent
{\bf Corollary F.} \emph{(Corollary \ref{C:dynsysjargNT}) Let $(A,\al,\bZ_+^d)$ be a C*-dynamical system. 
Let $\K$ and $\L$ be $2^d$-tuples of $X_\al$.
Then $\L$ is a $\K$-relative NO-$2^d$-tuple of $X_\al$ if and only if $\K\subseteq\L$ and the following hold:
\vspace{2pt}
\begin{enumerate} \setlength\itemsep{.3em}
\item $\L$ consists of ideals and $\L_F \cap (\bigcap_{i\in F}\al_\un{i}^{-1}(\L_\mt))\subseteq\L_\mt$ for all $\mt\neq F\subseteq[d]$,
\item $\L_F\subseteq\bigcap_{\un{n}\perp F}\al_\un{n}^{-1}(\L_F)\foral F\subseteq[d]$,
\item $\L$ is partially ordered,
\item $I_{1,F}\cap I_{2,F}\cap I_{3,F}\subseteq\L_F$ for all $\mt\neq F\subsetneq[d]$, where
\vspace{.3em}
\begin{itemize}  \setlength\itemsep{.3em}
\item $I_{1,F}:=\bigcap_{\un{n}\perp F}\al_\un{n}^{-1}(\{a\in A\mid a(\bigcap_{i\in F}\al_\un{i}^{-1}(\L_\mt))\subseteq\L_\mt\})$,
\item $I_{2,F}:=\bigcap_{\un{m}\perp F}\al_\un{m}^{-1}(\cap_{F\subsetneq D}\L_D)$,
\item $I_{3,F}:=\{a\in A\mid \lim_{\un{m}\perp F}\|\al_\un{m}(a)+[\al_\un{m}(A)\L_F\al_\un{m}(A)]\|=0\}$. 
\end{itemize}
\end{enumerate}
}

\medskip

If the dynamical system is in addition injective, then the associated product system is regular.
Thus we derive a bijection between the gauge-invariant ideals of $\N\O_{X_\al}$ and the single ideals $I$ of $A$ such that $\al_\un{n}(I) \subseteq I$ and $\al_{\un{n}}^{-1}(I) \subseteq I$ for all $\un{n} \in \bZ_+^d$.
When the system is automorphic we recover the well-known parametrisation of the gauge-invariant ideals of the crossed product $A \rtimes_\al \bZ^d$ by ideals $I \subseteq A$ such that $\al_{\un{n}}(I) = I$ for all $\un{n} \in \bZ_+^d$.

The third class of examples we consider relates to finitely aligned higher-rank graphs $(\La,d)$.
We write $X(\La)$ for the associated product system.
In keeping with the literature, in this case we write $k$ for the rank and reserve $d$ for the degree map.
We use $r$ and $s$ for the range and source maps.
One of the main subclasses here is that of row-finite higher-rank graphs.
The key property is that the strong covariance ideals are in bijection with $F$-tracing vertices, i.e., vertices that are not $F$-sources and the source of any $F^c$-path ending on them is not an $F$-source as well \cite{DK18}.
A family $H=\{H_F\}_{F\subseteq[k]}$ of subsets of vertices is called \emph{absorbent in $\La$} if the following holds for every $\mt\neq F\subsetneq[k]$: a vertex $v \in \La^{\un{0}}$ belongs to $H_F$ whenever it satisfies
\begin{enumerate}
\item $v$ is $F$-tracing,
\item $s(v\La^\un{m})\subseteq\cap_{F\subsetneq D}H_D\foral \un{m}\perp F$, and
\item there exists $\un{m}\perp F$ such that whenever $\un{n}\perp F$ and $\un{n}\geq\un{m}$, we have $s(v\La^\un{n})\subseteq H_F$ and $|v\La^\un{n}|<\infty$.
\end{enumerate}
By translating our characterisation of NT-$2^k$-tuples, we obtain the following corollary.

\medskip

\noindent
{\bf Corollary G.} \emph{(Proposition \ref{P:hrNT})
Let $(\La,d)$ be a strong finitely aligned $k$-graph.
Let $\L$ be a $2^k$-tuple of $X(\La)$ that consists of ideals and let $H_\L$ be the corresponding family of sets of vertices of $\La$.
Then $\L$ is an NT-$2^k$-tuple of $X(\La)$ if and only if  the following four conditions hold:
\begin{enumerate}
\item for each $\mt\neq F\subseteq[k]$, we have
\[
H_{\L, F} 
\subseteq 
H_{\L, \mt} 
\cup 
\{v \notin H_{\L, \mt} \mid |v\Ga(\La\setminus H_{\L, \mt})^\un{i}|<\infty \; \forall i\in[k]\; \textup{and} \; v \; \textup{is not an $F$-source in $\Ga(\La\setminus H_{\L, \mt})$}\},
\]
\item $H_\L$ is hereditary in $\La$,
\item $H_\L$ is partially ordered,
\item $H_\L \setminus H_{\L, \mt} := \{H_{\L, F} \setminus H_{\L, \mt} \}_{F\subseteq[k]}$ is absorbent in $\Ga(\La\setminus H_{\L, \mt})$.
\end{enumerate}
}

\medskip

Due to Proposition \ref{P:NTcom} we can provide a further translation for the case of row-finite graphs.
The following provides an alternative description of the parametrisation to the one given in \cite{Sim06} using just vertex sets.

\medskip

\noindent
{\bf Corollary H.} \emph{(Proposition \ref{P:rfNT})
Let $(\La,d)$ be a row-finite $k$-graph.
Let $\L$ be a $2^k$-tuple of $X(\La)$ that consists of ideals and let $H_\L$ be the corresponding family of sets of vertices of $\La$.
Then $\L$ is an NT-$2^k$-tuple of $X(\La)$ if and only if the following four conditions hold:
\begin{enumerate}
\item for each $\mt\neq F\subseteq[k]$, we have
\[
H_{\L, F} 
\subseteq 
H_{\L, \mt} 
\cup 
\{v \notin H_{\L, \mt} \mid v \; \textup{is not an $F$-source in $\Ga:=\Ga(\La\setminus H_{\L, \mt})$}\},
\]
\item $H_\L$ is hereditary in $\La$,
\item $H_\L$ is partially ordered,
\item $H_{1,F}\cap H_{2,F}\cap H_{3,F}\subseteq H_{\L,F}$ for all $\mt\neq F\subsetneq[k]$, where
\vspace{.3em}
\begin{itemize}
\item $H_{1,F}:=\bigcap_{\un{n}\perp F}\{v\in\La^\un{0}\mid s(v\La^\un{n})\subseteq H_{\L, \mt}\cup\{v \notin H_{\L, \mt} \mid v \; \textup{is not an $F$-source in $\Ga$}\}\}$,
\item $H_{2,F}:=\bigcap_{\un{m}\perp F}\{v\in\La^\un{0}\mid s(v\La^\un{m})\subseteq\cap_{F\subsetneq D}H_{\L,D}\}$,
\item $H_{3,F}$ is the set of all $v\in\La^\un{0}$ for which there exists $\un{m}\perp F$ such that whenever $\un{n}\perp F$ and $\un{n}\geq\un{m}$, we have $s(v\La^\un{n})\subseteq H_{\L,F}$.
\end{itemize}
\end{enumerate}
}

\medskip

If $\La$ is locally convex and row-finite, then positive (and negative) invariance of an ideal is equivalent to the related set of vertices being hereditary (and saturated).
In this case the NO-$2^k$-tuples $\L$ are determined just by $\L_\mt$, and thus our results recover the parametrisation of Raeburn, Sims and Yeend \cite{RSY04} (Corollary \ref{C:RSYparam1}).

Finally, we restrict our attention to product systems wherein each fibre admits a finite frame (except perhaps for $A$).
In connection with the parametrisation of the KMS-states, we exploit a decomposition of $X$ with respect to a given $\mt \neq F \subseteq [d]$.
One direction of this construction was implicit in \cite{Kak20}, and here we close this circle.
Namely, given $X$ and $F \subseteq [d]$, we define
\[
B^{F^\perp}:=\ca(\ol{\pi}_X(A),\ol{t}_{X,\un{i}}(X_\un{i})\mid i\in F^c)\subseteq\N\T_X,
\]
a collection $Z^{F^\perp}_X := \{X_\un{n}\}_{\un{n} \perp F}$, and a collection $Y^{F}_X := \{Y_{X,\un{n}}^F\}_{\un{n} \in \supp^{-1}(F)}$ by
\begin{align*}
Y_{X,\un{0}}^{F} & :=B^{F^\perp}
\qand
Y_{X,\un{n}}^{F} :=[\ol{t}_{X, \un{n}}(X_{\un{n}})B^{F^\perp}]\subseteq\N\T_X\foral \un{0} \neq \un{n} \in \supp^{-1}(F).
\end{align*}
Then the collections $Z^{F^\perp}_X$ and $Y^F_X$ become product systems with the structure inherited from $X$ and $\N\T_X$, with 
\[
\N\T_{Z^{F^\perp}_X}\cong B^{F^\perp}
\qand
\N\T_{Y_X^F} \cong \N\T_X.
\]
We are interested in describing the quotient of $\N\T_X$ by $\sca{\ol{\pi}_{X}(A) \ol{q}_{X, \un{i}} \mid i \in F}$ as a Cuntz-Nica-Pimsner algebra.
We have two approaches in this endeavour, differing only at which point we wish to delete the kernel.

\medskip

\noindent
{\bf Corollary I.} \emph{(Corollary \ref{C:somedir}) Let $X$ be a product system over $\bZ_+^d$ with coefficients in a C*-algebra $A$, wherein $X_\un{i}$ admits a finite frame for all $i\in[d]$, and fix $\mt \neq F \subseteq [d]$.
On the one hand, define the positively invariant ideal
\[
I_{Y_X^F} := \ker \{ Y^F_{X, \un{0}} \to \N\T_{Y^F_X} / \sca{\ol{\pi}_{Y^F_X}(Y^F_{X, \un{0}})\ol{q}_{Y^F_X,\un{i}}\mid i\in F} \}
\]
for the product system $Y^F_X$ related to $X$ and $F$.
On the other hand, define the positively invariant ideal
\[
I_X^F := \ker\{ A \to \N\T_X/ \sca{\ol{\pi}_X(A) \ol{q}_{X, \un{i}} \mid i \in F } \}
\]
for $X$, and consider the product system $Y^F_{[X]_{I_X^F}}$ related to $[X]_{I_X^F}$ and $F$.
Then there are canonical $*$-isomorphisms
\[
\N\O_{[Y^F_X]_{I_{Y_X^F}}} 
\cong
\N\T_X / \sca{\ol{\pi}_X(A) \ol{q}_{X, \un{i}} \mid i \in F }
\cong
\N\O_{Y^F_{[X]_{I_X^F}}}.
\]
If, in addition, $X_{\un{i}}$ is injective for all $i \in F$, then $Y_X^F$ is regular, $I_{Y_X^F} = \{0\}$ and $I_X^F = \{0\}$.
}

\medskip

This gives another avenue for obtaining the results of \cite{Kak20}.
The $F^c$-equivariant KMS-states of $\N\T_X$ that annihilate $\sca{\ol{\pi}_X(A) \ol{q}_{X, \un{i}} \mid i \in F }$ can be obtained from tracial states of $A$ annihilating $I_X^F$ by first inducing a KMS-state of finite type on the Toeplitz-Nica-Pimsner algebra of $Z^{F^\perp}_{[X]_{I_X^F}}$ (by using the Fock space construction) and then extending it to a KMS-state on the Cuntz-Nica-Pimsner algebra of $Y_{[X]_{I_X^F}}^F$ (by using a direct limit argument on the fixed point algebra).

We close with a note that several of the construction arguments we use apply to the general setting of product systems and we have opted to include them at that generality.
The problem of parametrising the gauge-invariant ideals in the general case requires a different approach, as there is no direct connection to ideals of just the coefficient algebra.
Indeed, this is the case even for the reduced strong covariant algebra \cite{DKKLL20, Seh18, Seh21}.

\subsection{Contents of sections}

In Section \ref{S:prod sys} we present the constructions that we will need.
We place extra attention on the quotient product system construction, which we present at full generality.
We then give a presentation of the main results of \cite{DK18}, and elucidate some of the key points that are used in the sequel.
Most importantly, we show that injectivity on the fixed point algebra reduces to checking just on the $[\un{0}, \un{1}_{[d]}]$-core, a trick that was used implicitly in \cite{DK18}.
Moreover, we study the $IXI$ construction and show an association between $\N\T_{IXI}$ and $\N\T_X$ (resp. $\N\O_{IXI}$ and $\N\O_X$) that is used in Section \ref{S:app}.

In Section \ref{S:relcnpalg} we focus on $2^d$-tuples and the ideals that they induce.
We give a step-by-step analysis of relativity, invariance, and partial ordering, and we demonstrate the maximality condition required for our parametrisation.
After presenting (E)-$2^d$-tuples and maximal (E)-$2^d$-tuples, we prove the Gauge-Invariant Uniqueness Theorem for equivariant quotients in-between $\N\T_X$ and $\N\O_X$.
We also capture maximal (E)-$2^d$-tuples algebraically and show how they can be constructed.

In Section \ref{S:g inv struc} we give the full parametrisation by NT-$2^d$-tuples, and show how this descends to the relative Cuntz-Nica-Pimsner algebras.
Moreover, we study the induced lattice isomorphism and give a description of the join and meet operations.

Finally, in Section \ref{S:app} we give applications and connections with the literature.
First we study a case where the NO-$2^d$-tuples depend only on the selection of a single ideal, which is then used in the specific cases of injective C*-dynamical systems and locally convex row-finite graphs.
Moving beyond that point, we interpret our parametrisation in terms of the structural data for C*-dynamical systems and row-finite higher-rank graphs in general.
As a special case we also study the impact on regular product systems.
We close this section by examining product systems on finite frames.

\medskip

\noindent {\bf Acknowledgements.}
Joseph Dessi acknowledges support from EPSRC as part of his PhD thesis on the programme ``The Structure of C*-Algebras of Product Systems'' (Ref. 2441268).
Evgenios Kakariadis acknowledges support from EPSRC as part of the programme ``Operator Algebras for Product Systems'' (EP/T02576X/1). 

The main results of the manuscript were pre-announced by Joseph Dessi at the ``Algebras of operators on Banach spaces and C*-algebras" workshop that took place at the Institute of Mathematics at the Czech Academy of Sciences on July 7-8, 2022. 
Joseph Dessi would like to thank the organisers for their hospitality.

After communicating a finalised version of the manuscript to Adam Dor-On, we were informed that Boris Bilich has been independently investigating the parametrisation of gauge-invariant ideals by using product system extensions when the left action is by compacts \cite{BDM23}.
The authors would like to thank Boris Bilich, Adam Dor-On and Ralf Meyer for bringing this line of research to their attention.

The authors would like to thank the reviewer for their comments and suggestions to improve the presentation of the content.

\medskip

\noindent {\bf Open access statement.}
For the purpose of open access, the authors have applied a Creative Commons Attribution (CC BY) license to any Author Accepted Manuscript (AAM) version arising.

\section{C*-correspondences and product systems}\label{S:prod sys}

\subsection{Notation} \label{Ss:notation}

By a lattice we will always mean a distributive lattice with operations $\vee$ and $\wedge$.
We write $\bZ_+$ for the nonnegative integers $\{0,1,\dots\}$ and $\bN$ for the positive integers $\{1,2,\dots\}$.
We denote the unit circle in the complex plane by $\bT$.
If $A,B$ and $C$ are sets and $f \colon A\times B\to C$ is a map, then we set
\[
f(A,B) := \{f(a,b)\mid a\in A,b\in B\};
\]
for example, if $H$ is a Hilbert space, then $\sca{H,H} := \{\sca{\xi,\eta}\mid\xi,\eta\in H\}$.
If $V$ is a normed vector space and $S\subseteq V$ is a subset, then $[S]$ denotes the norm-closed linear span of $S$ inside $V$.
If we only wish to take the linear span then this will always be clearly stated.
We recall the Hewitt-Cohen Factorisation Theorem, e.g., \cite[Proposition 2.33]{RW98}.

\begin{theorem}[Hewitt-Cohen Factorisation Theorem]\label{T:HCFT}
Let $A$ be a C*-algebra, $X$ be a Banach space and $\pi \colon A \to \B(X)$ be a bounded homomorphism.
Then $[\pi(A)X]=\pi(A) X$. 
\end{theorem}

All ideals of C*-algebras are taken to be two-sided and norm-closed.
If $A$ is a C*-algebra and $S\subseteq A$ is a subset, then $\sca{S}$ denotes the ideal of $A$ generated by $S$.
If $I\subseteq A$ is an ideal, then we set $I^\perp:=\{a\in A\mid aI=\{0\}\}$.

If $A = \ca(a_i \mid i \in \bI)$ and $B = \ca(b_i \mid i \in \bI)$ are C*-algebras, then a map $\Phi \colon A \to B$ will be called \emph{canonical} if it preserves generators of the same index, i.e., $\Phi(a_i) = b_i$ for all $i \in \bI$.

\subsection{C*-correspondences}\label{Ss:C*-cor}

We assume familiarity with the elementary theory of right Hilbert C*-modules.
The reader is addressed to \cite{Lan95, MT05} for an excellent introduction to the subject.
We will briefly outline the fundamentals of the theory of C*-correspondences.
We also recount Katsura's parametrisation in the C*-correspondences context \cite{Kat07}.

Let $A$ be a C*-algebra and $X$ be a right Hilbert $A$-module.
We write $\L(X)$ for the C*-algebra of adjointable operators on $X$, and $\K(X)$ for the ideal of (generalised) compact operators on $X$.
Recall that $\K(X)$ is densely spanned by the rank-one operators $\Theta_{\xi,\eta}^X \colon \zeta\mapsto\xi\sca{\eta,\zeta}$, for $\xi,\eta,\zeta\in X$.
When the right Hilbert C*-module $X$ is clear from the context, we will write $\Theta_{\xi,\eta}$ instead of $\Theta_{\xi, \eta}^X$.
A \emph{C*-correspondence} $X$ over a C*-algebra $A$ is a right Hilbert $A$-module equipped with a left action implemented by a $*$-homomorphism $\phi_X \colon A\to\L(X)$.
When the left action is clear from the context, we will abbreviate $\phi_X(a)\xi$ as $a\xi$, for $a\in A$ and $\xi\in X$.
The C*-correspondence $X$ is said to be \emph{non-degenerate} if $[\phi_X(A)X]=X$. 
If $\phi_X$ is injective, then we say that $X$ is \emph{injective}. 
If $X$ is injective and $\phi_X(A)\subseteq\K(X)$, then we say that $X$ is \emph{regular}.

Any C*-algebra $A$ can be viewed as a non-degenerate C*-correspondence over itself, with the right (resp. left) action given by right (resp. left) multiplication in $A$, and $A$-valued inner product given by $\sca{a,b}=a^*b$ for all $a,b\in A$.
Then $A\cong\K(A)$ by the left action $\phi_A$, and thus $A$ is non-degenerate by an application of an approximate unit.

Let $X$ and $Y$ be C*-correspondences over a C*-algebra $A$.
We call an $A$-bimodule linear map $u \colon X\to Y$ a \emph{unitary} if it is an inner-product-preserving surjection.
If such a $u$ exists, then it is adjointable, and we say that $X$ and $Y$ are \emph{unitarily equivalent} (symb. $X\cong Y$).

We write $X\otimes_A Y$ for the $A$-balanced tensor product.
Given $S\in\L(X)$, there exists an operator $S\otimes\text{id}_Y\in\L(X\otimes_A Y)$ determined on simple tensors by $\xi\otimes\eta\mapsto (S\xi)\otimes \eta$ for all $\xi\in X$ and $\eta\in Y$, e.g., \cite[p. 42]{Lan95}.
The assignment $S\mapsto S\otimes\text{id}_Y$ constitutes a unital $*$-homomorphism from $\L(X)$ to $\L(X\otimes_A Y)$.
In this way we can define a left action $\phi_{X\otimes_A Y}$ on $X\otimes_A Y$ by $\phi_{X\otimes_A Y}(a)= \phi_X(a) \otimes\id_Y$ for all $a\in A$, thereby endowing $X\otimes_A Y$ with the structure of a C*-correspondence over $A$.
The $A$-balanced tensor product is associative.
Moreover, the right action of $X$ yields a unitary $X \otimes_A A \to X$ determined by $\xi \otimes a\mapsto \xi a$ for all $a\in A$ and $\xi\in X$.
The left action of $X$ yields a unitary $A\otimes_A X\to[\phi_X(A)X]$ determined by $a\otimes\xi\mapsto\phi_X(a)\xi$ for all $a\in A$ and $\xi\in X$.

A \emph{(Toeplitz) representation} $(\pi,t)$ of the C*-correspondence $X$ on $\B(H)$ is a pair of a $*$-homomorphism $\pi \colon A\to \B(H)$ and a linear map $t \colon X\to \B(H)$ such that
\[
\pi(a) t(\xi) = t(\phi_X(a) \xi)
\text{ and }
t(\xi)^* t(\eta) = \pi(\sca{\xi,\eta})
\foral
a \in A, \xi, \eta \in X.
\]
Then $(\pi,t)$ automatically satisfies $t(\xi)\pi(a) = t(\xi a)$ for all $a \in A$ and $\xi \in X$.
Moreover, there exists a $*$-homomorphism $\psi \colon \K(X) \to \B(H)$ such that $\psi(\Theta_{\xi,\eta})=t(\xi)t(\eta)^*$ for all $\xi,\eta\in X$, e.g., \cite[Lemma 2.2]{KPW98}.
We say that $(\pi,t)$ is \emph{injective} if $\pi$ is injective; then both $t$ and $\psi$ are isometric.

Given a representation $(\pi,t)$ of $X$, we write $\ca(\pi,t)$ for the C*-algebra generated by $\pi(A)$ and $t(X)$.
A representation $(\pi,t)$ is said to \emph{admit a gauge action} $\ga \colon \bT\to\text{Aut}(\ca(\pi,t))$ if $\ga$ is a group homomorphism, $\{\ga_z\}_{z\in\bT}$ is point-norm continuous and
\[ \ga_z(\pi(a))=\pi(a) \foral a\in A \; \text{and} \; \ga_z(t(\xi))=zt(\xi) \foral \xi\in X, \]
for all $z\in\bT$.
When a gauge action $\ga$ exists, it is necessarily unique.
An ideal $\mathfrak{J}\subseteq\ca(\pi,t)$ is called \emph{gauge-invariant} or \emph{equivariant} if $\ga_z(\mathfrak{J})\subseteq\mathfrak{J}$ for all $z\in\bT$ (thus $\ga_z(\fJ) = \fJ$ for all $z \in \bT$).

The \emph{Toeplitz-Pimsner algebra} $\T_X$ is the universal C*-algebra with respect to the representations of $X$.
Let $J\subseteq A$ be such that $J\subseteq\phi_X^{-1}(\K(X))$. 
The \emph{$J$-relative Cuntz-Pimsner algebra} $\O(J,X)$ is the universal C*-algebra with respect to the \emph{$J$-covariant} representations of $X$; that is, the representations $(\pi,t)$ of $X$ satisfying $\pi(a) =\psi(\phi_X(a))$ for all $a\in J$.
Traditionally the relative Cuntz-Pimsner algebras are defined with respect to ideals of $A$ rather than just subsets, however the two versions are equivalent since $\O(J, X) = \O(\sca{J},X)$.
When $J=\{0\}$, we have $\O(J,X)=\T_X$.
For the ideal
\[
J_X:=(\ker\phi_X)^\perp\cap\phi_X^{-1}(\K(X)),
\]
we obtain that $\O(J_X,X)$ is the \emph{Cuntz-Pimsner algebra} $\O_X$  \cite{Kat04}.
Katsura's ideal $J_X$ is the largest ideal on which the restriction of $\phi_X$ is injective with image contained in $\K(X)$ \cite{Kat03}.

One of the main tools in the theory is the Gauge-Invariant Uniqueness Theorem, obtained in its full generality by Katsura \cite{Kat07}.
An alternative proof can be found in \cite{Kak16}, and Frei \cite{Fre21} extended this method to include all relative Cuntz-Pimsner algebras, in connection with \cite{Kat07}.

\begin{theorem}[$\bZ_+$-GIUT] \cite[Corollary 11.8]{Kat07} \label{T:giut}
Let $X$ be a C*-correspondence over a C*-algebra $A$, let $J\subseteq A$ be an ideal satisfying $J\subseteq J_X$ and let $(\pi,t)$ be a representation of $X$. 
Then $\O(J,X)\cong\ca(\pi,t)$ via a canonical $*$-isomorphism if and only if $(\pi,t)$ is injective, admits a gauge action and satisfies $J=\pi^{-1}(\psi(\K(X)))$.
\end{theorem}

Let $A$ be a C*-algebra, let $I\subseteq A$ be an ideal and let $X$ be a right Hilbert $A$-module.
Then the set $XI$ is a closed linear subspace of $X$ that is invariant under the right action of $A$, e.g, \cite[p. 576]{FMR03} or \cite[Corollary 1.4]{Kat07}.
In particular, we have $[XI]=XI$. 
Consequently, $XI$ is itself a right Hilbert $A$-module under the operations and $A$-valued inner product inherited from $X$. 
We may also view $XI$ as a right Hilbert $I$-module. 
Due to \cite[Lemma 2.6]{FMR03}, we will identify $\K(XI)$ as an ideal of $\K(X)$ in the following natural way:
\begin{align*}
\K(XI) & =\ol{\spn}\{\Theta_{\xi,\eta}^X \mid \xi,\eta\in XI\}\subseteq\K(X).
\end{align*}
When $X$ is in addition a C*-correspondence over $A$, we may equip $XI$ with a C*-correspondence structure via the left action
\[
\phi_{XI}\colon A\to\L(XI); \phi_{XI}(a)=\phi_X(a)|_{XI} \foral a \in A.
\]
By restricting $\phi_{XI}$ to $I$, we may also view $XI$ as a C*-correspondence over $I$.

Following \cite{Kat07}, and in order to ease notation, we will use the symbol $[ \hspace{1pt} \cdot \hspace{1pt} ]_I$ to denote the quotient maps associated with a right Hilbert $A$-module $X$ and an ideal $I\subseteq A$.
For example, we use it for both the quotient map $A\to A/I\equiv [A]_I$ and the quotient map $X\to X/XI\equiv[X]_I$.
We equip the complex vector space $[X]_I$ with the following right $[A]_I$-module multiplication:
\[ 
[\xi]_I[a]_I=[\xi a]_I \foral a\in A,\xi\in X, 
\]
as well as the following $[A]_I$-valued inner product:
\[ 
\sca{[\xi]_I,[\eta]_I}=[\sca{\xi,\eta}]_I \foral \xi,\eta\in X. 
\]
Consequently $[X]_I$ carries the structure of an inner-product right $[A]_I$-module. 
By \cite[Lemma 1.5]{Kat07}, the canonical norm on $[X]_I$ induced by the $[A]_I$-valued inner product coincides with the usual quotient norm. 
Thus $[X]_I$ is a right Hilbert $[A]_I$-module. 
We may define a $*$-homomorphism $[\hspace{1pt} \cdot \hspace{1pt}]_I\colon \L(X)\to\L([X]_I)$ by
\[ 
[S]_I[\xi]_I=[S\xi ]_I \foral S\in\L(X), \xi\in X. 
\]
We include \cite[Lemma 1.6]{Kat07} in its entirety, as we will be making frequent reference to it.

\begin{lemma}\label{L:Kat07}\cite[Lemma 1.6]{Kat07}
Let $X$ be a right Hilbert module over a C*-algebra $A$ and let $I\subseteq A$ be an ideal.
Then for all $\xi,\eta\in X$, we have $[\Theta_{\xi,\eta}^X]_I=\Theta_{[\xi]_I,[\eta]_I}^{[X]_I}$.
The restriction of the map $[\hspace{1pt} \cdot \hspace{1pt}]_I \colon \L(X)\to\L([X]_I)$ to $\K(X)$ is a surjection onto $\K([X]_I)$ with kernel $\K(XI)$.
\end{lemma}

Therefore, given an ideal $I \subseteq A$, we obtain the surjective maps
\begin{align*}
A \to A/I & \textup{ with kernel $I$}, \\
X \to X/ XI & \textup{ with kernel $XI$}, \\
\K(X) \to \K(X/XI) & \textup{ with kernel $\K(XI)$},
\end{align*}
as well as the map $\L(X) \to \L(X/XI)$ (which may \emph{not} be surjective), all of which will be denoted by the same symbol $[\hspace{1pt} \cdot \hspace{1pt}]_I$.
Lemma \ref{L:Kat07} implies that if $k\in\K(X)$, then
\begin{equation}\label{Eq: comp}
k\in\K(XI) \iff \sca{X, k X} \subseteq I.
\end{equation}
Lemma \ref{L:Kat07} provides a straightforward way of seeing $\K(XI)$ as an ideal in $\K(X)$, and in turn as an ideal in $\L(X)$.
Hence we may consider the quotient C*-algebra $\L(X)/\K(XI)$.

We recall the following result extracted from \cite[Lemma 4.6]{Lan95}, slightly rewritten to match our setting.

\begin{lemma}\label{L:lance} \cite[Lemma 4.6]{Lan95}
Let $X$ and $Y$ be C*-correspondences over a C*-algebra $A$.
For $x\in X$, the equation $\Theta_x(y)=x\otimes y$, for $y\in Y$, defines an element $\Theta_x\in\L(Y,X\otimes_AY)$ satisfying
\begin{align*}
\|\Theta_x\| & =\|\phi_Y(\sca{x,x}^{1/2})\|\leq\|x\| \qand
\Theta_x^*(x'\otimes y)  =\phi_Y(\sca{x,x'})y \foral x' \in X, y \in Y.
 \end{align*}
\end{lemma}

As a consequence, we obtain the following lemma and corollary.

\begin{lemma}\label{L:com id}
Let $X$ and $Y$ be C*-correspondences over a C*-algebra $A$.
Let $I\subseteq A$ be an ideal and suppose $a\in A$ satisfies $\phi_Y(a)\in\K(YI)$. 
Then $\Theta_{\xi a, \eta}^X \otimes \id_Y \in \K((X\otimes_A Y)I)$ for all $\xi,\eta\in X$.
\end{lemma}

\begin{proof}
One may directly verify that $\Theta_{\xi a,\eta}^X \otimes \text{id}_Y = \Theta_\xi\phi_Y(a)\Theta_\eta^* \in\K(X\otimes_A Y)$, where $\Theta_\xi,\Theta_\eta\in\L(Y,X\otimes_A Y)$ are defined as in Lemma \ref{L:lance}. 
Here we use that $\phi_Y(a)\in\K(YI)\subseteq\K(Y)$ together with \cite[p. 9, (1.6)]{Lan95}.
Hence, for $\zeta,\zeta'\in X\otimes_A Y$, we have
\[
\sca{\zeta, (\Theta_{\xi a,\eta}^X \otimes \text{id}_Y)( \zeta')}
=
\sca{\Theta_\xi^*\zeta,\phi_Y(a)\Theta_\eta^*\zeta'}\in I,
\]
where we have used the forward implication of (\ref{Eq: comp}) applied to the C*-correspondence $Y$ and $k=\phi_Y(a)\in\K(YI)$ to establish the membership to $I$.
An application of the converse implication of (\ref{Eq: comp}) gives that $\Theta_{\xi a, \eta}^X \otimes \id_Y \in \K((X\otimes_A Y)I)$, as required.
\end{proof}

\begin{corollary}\label{C:comp to comp}
Let $X$ and $Y$ be C*-correspondences over a C*-algebra $A$.
Let $I\subseteq A$ be an ideal and suppose that $\phi_Y(I)\subseteq\K(YI)$.
Then $k\otimes\id_Y\in\K((X\otimes_A Y)I)$ for all $k\in\K(XI)$.
\end{corollary}

\begin{proof}
The claim follows by applying Lemma \ref{L:com id} to finite linear combinations of rank-one operators in $\K(XI)$ and their norm-limits.
\end{proof}

If $X$ is a C*-correspondence over $A$, then we need to make an additional imposition on $I$ in order for $[X]_I$ to carry a canonical structure as a C*-correspondence over $[A]_I$. 
More specifically, we say that $I$ is \emph{positively invariant for $X$} if it satisfies
\[ 
X(I):=[\sca{X,IX}]\subseteq I. 
\]
Notice that the space $X(I)$ is an ideal of $A$ (in fact, this is true even when $I$ is replaced by any subset of $A$). 
When $I$ is positively invariant, we can define a left action on $[X]_I$ by
\[
\phi_{[X]_I} \colon [A]_I \to \L([X]_I); [a]_I \mapsto [\phi_X(a)]_I\foral a\in A.
\] 
To ease notation, we will denote $\phi_{[X]_I}$ by $[\phi_X]_I$.

\begin{lemma}\label{L:IX to XI}
Let $X$ be a C*-correspondence over a C*-algebra $A$ and let $I\subseteq A$ be an ideal.
Then the following are equivalent:
\begin{enumerate}
\item $I$ is positively invariant for $X$;
\item $IX \subseteq XI$;
\item $IXI = IX$.
\end{enumerate}
\end{lemma}

\begin{proof}
\noindent
[(i) $\Leftrightarrow$ (ii)]:
This follows from \cite[Proposition 1.3]{Kat07}.

\smallskip

\noindent
[(ii) $\Rightarrow$ (iii)]:
Trivially $IXI \subseteq IX$.
On the other hand, we have $IX = IIX \subseteq IXI$ by assumption, and thus we obtain the required equality.

\smallskip

\noindent
[(iii) $\Rightarrow$ (i)]: 
A direct computation yields $[\sca{X, IX}] = [\sca{X, IXI}] = [\sca{X, IX}I] \subseteq I$, and thus $I$ is positively invariant.
\end{proof}

\begin{corollary}\label{C:compconj}
Let $X$ be a C*-correspondence over a C*-algebra $A$ and let $I\subseteq A$ be an ideal that is positively invariant for $X$.
Then we have
\[
\phi_X(I)\K(X)\phi_X(I)\subseteq\ol{\spn}\{\Theta_{\xi,\eta}^X \mid \xi, \eta \in IXI\}.
\]
\end{corollary}

\begin{proof}
The claim follows by checking on rank-one operators and using that $IXI = IX$ by Lemma \ref{L:IX to XI}.
\end{proof}

The following lemma will be useful in Section \ref{S:g inv struc}.

\begin{lemma}\label{L:newquot}
Let $X$ be a C*-correspondence over a C*-algebra $A$ such that $\phi_X(A)\subseteq\K(X)$, and let $I,J\subseteq A$ be ideals. 
If $I$ is positively invariant for $X$ and $I\subseteq J$, then
\[
\|[\phi_X]_I([a]_I)+\K([X]_I[J]_I)\|=\|\phi_X(a)+\K(XJ)\|\foral a\in A.
\]
\end{lemma}
\begin{proof}
Fix $a\in A$.
Recall that
\[
M:=\|\phi_X(a)+\K(XJ)\|=\inf\{\|\phi_X(a)+k\|\mid k\in\K(XJ)\},
\]
where the norm on the right hand side is taken in $\K(X)$ since $\phi_X(A)\subseteq\K(X)$.
Likewise, we have
\[
N:=\|[\phi_X]_I([a]_I)+\K([X]_I[J]_I)\|=\inf\{\|[\phi_X]_I([a]_I)+\dot{k}\|\mid \dot{k}\in\K([X]_I[J]_I)\},
\]
where the norm on the right hand side is taken in $\K([X]_I)$ by Lemma \ref{L:Kat07}.
First we show that $N\leq M$.
To this end, it suffices to show that 
\[
N\leq\|\phi_X(a)+k\|\foral k\in\K(XJ).
\]
Accordingly, fix $k\in\K(XJ)$.
Note that 
\[
\K([X]_I[J]_I)=\K([XJ]_I)=[\K(XJ)]_I
\]
for the usual $*$-epimorphism $[ \hspace{1pt} \cdot \hspace{1pt} ]_I\colon\K(X)\to\K([X]_I)$.
Thus $[k]_I\in\K([X]_I[J]_I)$.
It follows that
\[
N\leq\|[\phi_X]_I([a]_I)+[k]_I\|=\|[\phi_X(a)+k]_I\|\leq\|\phi_X(a)+k\|,
\]
using that $[ \hspace{1pt} \cdot \hspace{1pt} ]_I$ is contractive in the second inequality, as required.

Finally, to prove that $M\leq N$ it suffices to show that
\[
M\leq\|[\phi_X]_I([a]_I)+\dot{k}\|\foral \dot{k}\in\K([X]_I[J]_I).
\]
Accordingly, fix $\dot{k}\in\K([X]_I[J]_I)$ and note that $\dot{k}=[k]_I$ for some $k\in\K(XJ)$ since $\K([X]_I[J]_I)=[\K(XJ)]_I$.
By Lemma \ref{L:Kat07}, we have a $*$-isomorphism
\[
j\colon\K(X)/\K(XI)\to\K([X]_I); k'+\K(XI)\mapsto [k']_I\foral k'\in\K(X).
\]
In turn, we deduce that
\begin{align*}
\|[\phi_X]_I([a]_I)+\dot{k}\| & =\|[\phi_X]_I([a]_I)+[k]_I\|=\|[\phi_X(a)+k]_I\| \\
				    & =\|j((\phi_X(a)+k)+\K(XI))\|=\|(\phi_X(a)+k)+\K(XI)\| \\
				    & =\inf\{\|\phi_X(a)+k+k'\|\mid k'\in\K(XI)\}.
\end{align*}
Note that we use the assumption that $\phi_X(A)\subseteq\K(X)$ to pass to the second line.
Thus it suffices to show that 
\[
M\leq\|\phi_X(a)+k+k'\|\foral k'\in\K(XI).
\]
However, this is immediate since $I\subseteq J$ and so $\K(XI)\subseteq\K(XJ)$.
In total, we have $M=N$, finishing the proof.
\end{proof}

Moreover, we define two ideals of $A$ that are related to $I$ and $X$, namely
\[
X^{-1}(I):=\{a\in A \mid \; \sca{X,aX}\subseteq I\},
\]
and
\[
J(I,X):=\{a\in A \mid [\phi_X(a)]_I\in \K([X]_I), aX^{-1}(I)\subseteq I\}.
\]
The use of the ideal $J(I,X)$ is pivotal in the work of Katsura \cite{Kat07} for accounting for $*$-representations of $\T_X$ that may not be injective on $X$.
When $I$ is positively invariant and $J\subseteq A$ is an ideal satisfying $I\subseteq J$, the following lemma, which builds on \cite[Lemma 5.2]{Kat07}, illustrates the relationship between $X^{-1}(J)$ and $[X]_I^{-1}([J]_I)$.

\begin{lemma}\label{L:quotinv}
Let $X$ be a C*-correspondence over a C*-algebra $A$ and let $I,J\subseteq A$ be ideals.
If $I$ is positively invariant for $X$ and $I\subseteq J$, then 
\[
X^{-1}(J)=[ \hspace{1pt} \cdot \hspace{1pt} ]_I^{-1}([X]_I^{-1}([J]_I)).
\]
\end{lemma}
\begin{proof}
The forward inclusion is immediate by definition of the C*-correspondence operations of $[X]_I$.
For the reverse inclusion, take $a\in A$ such that $[a]_I\in[X]_I^{-1}([J]_I)$.
Then we have
\[
[\sca{X,aX}]_I=\sca{[X]_I,[aX]_I}=\sca{[X]_I,[a]_I[X]_I}\subseteq[J]_I.
\]
In turn, we have $\sca{X,aX}\subseteq J+I$, and the fact that $I\subseteq J$ then yields $\sca{X,aX}\subseteq J$.
Hence $a\in X^{-1}(J)$ by definition, completing the proof.
\end{proof}

As per \cite[Definition 5.6, Definition 5.12]{Kat07}, we define a \emph{T-pair} of $X$ to be a pair $\L=\{\L_\mt,\L_{\{1\}}\}$ of ideals of $A$ such that $\L_\mt$ is positively invariant for $X$ and $\L_\mt \subseteq \L_{\{1\}} \subseteq J(\L_\mt,X)$; a T-pair $\L$ that satisfies $J_X \subseteq \L_{\{1\}}$ is called an \emph{O-pair}.
Our choice of notation here will be clarified in the sequel.
Proposition 8.8 of \cite{Kat07} is fundamental to the current work.

\begin{theorem}\label{T:Kat par} \cite[Theorem 8.6, Proposition 8.8]{Kat07}
Let $X$ be a C*-correspondence over a C*-algebra $A$.
Then there is a bijection between the set of T-pairs (resp. O-pairs) of $X$ and the set of gauge-invariant ideals of $\T_X$ (resp. $\O_X$) that preserves inclusions and intersections.
\end{theorem}

The bijection of Theorem \ref{T:Kat par} restricts appropriately to a parametrisation of the gauge-invariant ideals of any relative Cuntz-Pimsner algebra \cite[Proposition 11.9]{Kat07}.
The parametrisation of the gauge-invariant ideals of $\T_X$ can be implemented as follows.
Firstly, if $\mathfrak{J}\subseteq\T_X$ is a gauge-invariant ideal, then we consider the representation $(Q_\mathfrak{J}\circ\ol{\pi}_X,Q_\mathfrak{J}\circ\ol{t}_X)$, where $Q_\mathfrak{J}\colon\T_X\to\T_X/\mathfrak{J}$ is the quotient map and $(\ol{\pi}_X,\ol{t}_X)$ is the universal representation of $X$.
We set
\[
\L_\mt^{(Q_\mathfrak{J}\circ\ol{\pi}_X,Q_\mathfrak{J}\circ\ol{t}_X)}:=\ker Q_\mathfrak{J}\circ\ol{\pi}_X \qand \L_{\{1\}}^{(Q_\mathfrak{J}\circ\ol{\pi}_X,Q_\mathfrak{J}\circ\ol{t}_X)}:=(Q_\mathfrak{J}\circ\ol{\pi}_X)^{-1}( (Q_\fJ \circ \ol{\psi}_X)(\K(X))).
\]
It follows that the pair
\[
\L^\mathfrak{J}:=\{\L_\mt^{(Q_\mathfrak{J}\circ\ol{\pi}_X,Q_\mathfrak{J}\circ\ol{t}_X)},\L_{\{1\}}^{(Q_\mathfrak{J}\circ\ol{\pi}_X,Q_\mathfrak{J}\circ\ol{t}_X)}\}
\]
is a T-pair of $X$ \cite[Proposition 5.11]{Kat07}.
Next, let $\L=\{\L_\mt,\L_{\{1\}}\}$ be a T-pair of $X$.
Then $[\L_{\{1\}}]_{\L_\mt}$ is an ideal of $[A]_{\L_\mt}$.
Additionally, we have$[\L_{\{1\}}]_{\L_\mt}\subseteq J_{[X]_{\L_\mt}}$ by using \cite[Lemma 5.2]{Kat07}.
Let $(\tilde{\pi},\tilde{t})$ be the universal $[\L_{\{1\}}]_{\L_\mt}$-covariant representation of $[X]_{\L_\mt}$.
Then we may form a representation $(\pi^\L,t^\L)$ of $X$ generating $\O([\L_{\{1\}}]_{\L_\mt},[X]_{\L_\mt})$ via
\[
\pi^\L(a)=\tilde{\pi}([a]_{\L_\mt}) \; \text{and} \; t^\L(\xi)=\tilde{t}([\xi]_{\L_\mt}) \foral a\in A,\xi\in X.
\]
The universal property of $\T_X$ then guarantees a canonical $*$-epimorphism 
\[
\pi^\L\times t^\L\colon\T_X\to\O([\L_{\{1\}}]_{\L_\mt},[X]_{\L_\mt}).
\]
We set 
\[
\mathfrak{J}^\L:=\ker\pi^\L\times t^\L
\]
and observe that $\mathfrak{J}^\L$ is a gauge-invariant ideal of $\T_X$.
By \cite[Proposition 8.8]{Kat07}, the maps 
\begin{align*}
\mathfrak{J}\mapsto\L^\mathfrak{J} & \; \text{for all gauge-invariant ideals $\mathfrak{J}$ of $\T_X$}, \\
\L\mapsto\mathfrak{J}^\L & \; \text{for all T-pairs $\L$ of $X$},
\end{align*}
are mutually inverse.
The notation we have used here will be revisited in the sequel.

\subsection{Product systems}\label{Ss:prod sys}

Let $P$ be a subsemigroup of a discrete group $G$ that contains the identity $e$ of $G$ (i.e., $P$ is \emph{unital}). 
A \emph{product system $X$ over $P$ with coefficients in a C*-algebra $A$} is a family $\{X_p\}_{p\in P}$ of C*-correspondences over $A$ together with multiplication maps $u_{p,q} \colon X_p\otimes_A X_q\to X_{pq}$ for all $p,q\in P$, such that:
\begin{enumerate}
\item $X_e=A$, viewing $A$ as a C*-correspondence over itself; 
\item if $p=e$, then $u_{e,q} \colon A\otimes_A X_q\to[\phi_q(A)X_q]$ is the unitary implementing the left action of $A$ on $X_q$;
\item if $q=e$, then $u_{p,e} \colon X_p\otimes_A A\to X_p$ is the unitary implementing the right action of $A$ on $X_p$;
\item if $p,q\in P\setminus\{e\}$, then $u_{p,q} \colon X_p\otimes_A X_q\to X_{pq}$ is a unitary;
\item the multiplication maps are associative in the sense that
\[ u_{pq,r}(u_{p,q}\otimes\text{id}_{X_r})=u_{p,qr}(\text{id}_{X_p}\otimes u_{q,r})\foral p,q,r\in P. \]
\end{enumerate}

Note that we use $\phi_p$ to denote the left action $\phi_{X_p}$ of $X_p$ for each $p\in P$.
We refer to the C*-correspondences $X_p$ as the \emph{fibres} of $X$.
We do \emph{not} assume that the fibres are non-degenerate.
If $X_p$ is injective (resp. regular) for all $p\in P$, then we say that $X$ is \emph{injective} (resp. \emph{regular}).
Axioms (i) and (ii) imply that the unitary $u_{e,e} \colon A\otimes_A A\to A$ is simply multiplication in $A$.
For brevity, we will write $u_{p,q}(\xi_p\otimes\xi_q)$ as $\xi_p\xi_q$ for all $p, q\in P$, $\xi_p\in X_p$ and $\xi_q\in X_q$, with the understanding that $\xi_p$ and $\xi_q$ are allowed to differ when $p=q$.
Axioms (ii) and (v) imply that
\[
\phi_{pq}(a)(\xi_p\xi_q) = (\phi_p(a)\xi_p)\xi_q \foral \xi_p\in X_p, \xi_q\in X_q, p,q\in P.
\]

For $p\in P\setminus\{e\}$ and $q\in P$, we use the product system structure to define a $*$-homomorphism $\iota_p^{pq} \colon \L(X_p)\to\L(X_{pq})$ by 
\[
\iota_p^{pq}(S)=u_{p,q}(S\otimes\text{id}_{X_q})u_{p,q}^* \foral S \in \L(X_p).
\]
This is also captured by the formula 
\[
\textup{$\iota_p^{pq}(S)(\xi_p\xi_q)=(S\xi_p)\xi_q$ for all $\xi_p\in X_p$ and $\xi_q\in X_q$. }
\]
For each $q\in P$ we also define a $*$-homomorphism $\iota_e^q\colon\K(A)\to\L(X_q)$ by $\iota_e^q( \phi_e(a))=\phi_q(a)$ for all $a\in A$.
Moreover, we have
\[
\iota_p^p=\text{id}_{\L(X_p)} \foral p\in P\setminus\{e\} \qand \iota_e^e=\id_{\K(A)}.
\]

The theory of product systems includes that of C*-correspondences in the sense that every C*-correspondence $X$ over a C*-algebra $A$ can be viewed as the product system $\{X_n\}_{n \in \bZ_+}$ with 
\[
X_0=A \qand X_n = X^{\otimes n} \foral n\in\bN,
\] 
and multiplication maps $u_{n,m}$ for $n,m \neq 0$ given by the natural inclusions.

Let $X$ and $Y$ be product systems over $P$ with coefficients in $A$ and $B$, respectively.
Denote the multiplication maps of $X$ by $\{u_{p,q}^X\}_{p,q\in P}$ and the multiplication maps of $Y$ by $\{u_{p,q}^Y\}_{p,q\in P}$.
We say that $X$ and $Y$ are \emph{unitarily equivalent} (symb. $X\cong Y$) if there exist surjective linear maps $W_p\colon X_p\to Y_p$ for all $p\in P$ with the following properties:
\begin{enumerate}
\item $W_e\colon A\to B$ is a $*$-isomorphism;
\item $\sca{W_p(\xi_p), W_p(\xi_p')} = W_e(\sca{\xi_p,\xi_p'})$ for all $\xi_p,\xi_p'\in X_p$ and $p\in P\setminus\{e\}$;
\item $\phi_{Y_p}(W_e(a)) W_p(\xi_p)=W_p(\phi_{X_p}(a)\xi_p)$ for all $a\in A, \xi_p\in X_p$ and $p\in P\setminus\{e\}$;
\item $W_p(\xi_p) W_e(a) = W_p(\xi_pa)$ for all $a\in A, \xi_p\in X_p$ and $p\in P\setminus\{e\}$;
\item $u_{p,q}^Y \circ ( W_p\otimes W_q) = W_{pq} \circ u_{p,q}^X$ for all $p,q\in P$.
\end{enumerate}
In this case, we say that $\{W_p\}_{p \in P}$ implements a unitary equivalence between $X$ and $Y$.

\begin{remark}\label{R:prodsysisoprop}
Virtually all structural properties of product systems are preserved under unitary equivalence.
Let $\{W_p\}_{p\in P}$ implement a unitary equivalence between product systems $X$ and $Y$.
Item (i) guarantees that items (ii)-(iv) hold when $p=e$.
Likewise, items (i) and (ii) imply that $W_p$ is injective for all $p\in P$.
Hence $W_p$ is bijective for all $p\in P$.
Consequently, the collection $\{W_p^{-1}\}_{p \in P}$ defines a unitary equivalence between $Y$ and $X$.

For all $p,q\in P$, the map $W_p\otimes W_q\in\B(X_p\otimes_A X_q, Y_p\otimes_B Y_q)$ is defined on simple tensors via
\[
(W_p\otimes W_q)(\xi_p\otimes \xi_q)=W_p(\xi_p)\otimes W_q(\xi_q)\foral \xi_p\in X_p, \xi_q\in X_q.
\]
It follows that $(W_p\otimes W_q)^{-1}=W_p^{-1}\otimes W_q^{-1}$, and that
\begin{equation}\label{Eq:unitcomp}
W_p\K(X_p)W_p^{-1}=\K(Y_p).
\end{equation}
Fixing $a\in A$, we also have
\begin{equation}\label{Eq:unitleft}
W_p\phi_{X_p}(a)W_p^{-1}=\phi_{Y_p}(W_e(a)).
\end{equation}
Finally, let $\{\iota_p^{pq}\}_{p,q\in P}$ denote the connecting $*$-homomorphisms of $X$ and let $\{j_p^{pq}\}_{p,q\in P}$ denote the connecting $*$-homomorphisms of $Y$. 
For $p\in P\setminus\{e\}$ and $q\in P$, we have 
\begin{equation}\label{Eq:unitmult}
\iota_p^{pq}(\Theta_{\xi_p,\xi_p'}^{X_p}) = W_{pq}^{-1}j_p^{pq}(\Theta_{W_p(\xi_p),W_p(\xi_p')}^{Y_p})W_{pq}
\foral \xi_p,\xi_p'\in X_p.
\end{equation}
\end{remark}

A \emph{(Toeplitz) representation} $(\pi,t)$ of $X$ on $\B(H)$ consists of a family $\{(\pi,t_p)\}_{p\in P}$, where $(\pi,t_p)$ is a representation of $X_p$ on $\B(H)$ for all $p\in P$, $t_e=\pi$ and
\[ 
t_p(\xi_p)t_q(\xi_q)=t_{pq}(\xi_p\xi_q) \foral \xi_p\in X_p, \xi_q\in X_q, p, q\in P.
\]
We write $\psi_p$ for the induced $*$-homomorphism $\K(X_p) \to \B(H)$ for all $p\in P$. 
We say that $(\pi,t)$ is \emph{injective} if $\pi$ is injective; in this case $t_p$ and $\psi_p$ are isometric for all $p\in P$.
We denote the C*-algebra generated by $\pi(A)$ and every $t_p(X_p)$ by $\ca(\pi,t)$.
We write $\T_X$ for the universal C*-algebra with respect to the Toeplitz representations of $X$, and refer to it as the \emph{Toeplitz algebra of $X$}.
The following proposition 
is well known. 

\begin{proposition}\label{P:T univ}
Let $P$ be a unital subsemigroup of a discrete group $G$. 
Let $X$ be a product system over $P$ with coefficients in a C*-algebra $A$. 
Then there exists a C*-algebra $\T_X$ and a representation $(\pi_X, t_X)$ of $X$ on $\T_X$ such that:
\begin{enumerate}
\item $\T_X=\ca(\pi_X,t_X)$;
\item if $(\pi,t)$ is a representation of $X$, then there exists a canonical $*$-epimorphism $\pi\times t\colon\T_X\to\ca(\pi,t)$, i.e., $(\pi\times t)(t_{X,p}(\xi_p))=t_p(\xi_p)$ for all $\xi_p\in X_p$ and $p\in P$.
\end{enumerate}
The pair $(\T_X,(\pi_X,t_X))$ is unique up to canonical $*$-isomorphism.
\end{proposition}

\begin{remark}\label{R:tpiso}
Let $X$ and $Y$ be product systems over $P$ with coefficients in $A$ and $B$, respectively.
Suppose that $X$ and $Y$ are unitarily equivalent by a collection $\{W_p\colon X_p\to Y_p\}_{p \in P}$.
Then it is easy to see that there is a bijection between the sets of their representations, in the sense that $\{(\pi, t_p)\}_{p \in P}$ is a representation of $X$ if and only if $\{(\pi \circ W_e^{-1}, t_p \circ W_p^{-1})\}_{p \in P}$ is a representation of $Y$. 
Therefore $\T_X$ and $\T_Y$ are universal with respect to the \lq\lq same" representations, and hence $\T_X\cong\T_Y$ canonically.
\end{remark}

The quintessential example of a representation that we will be using throughout is the Fock representation.
Here we define this representation and collect its basic properties.
Firstly, we set $\F X := \sum_{p\in P}^\oplus X_p$, as the direct sum of right Hilbert $A$-modules, e.g., \cite[p. 6]{Lan95}.
For each $p\in P$, we identify $X_p$ with the $p$-th direct summand of $\F X$.
For each $p,q\in P$, we define $\iota_{p,q} \colon \L(X_p,X_q)\to\L(\F X)$ by 
\[ 
\iota_{p,q}(S)\xi_r
=
\begin{cases} 
S \xi_p & \text{if } r=p, \\ 
0 & \text{otherwise}, 
\end{cases} 
\]
for all $S \in \L(X_p,X_q), \xi_r\in X_r$ and $r \in P$, and note that $\iota_{p,q}(S)^*= \iota_{q,p}(S^*)$.
It follows that $\iota_{p,q}$ is an isometric linear map, and hence we may identify $\L(X_p,X_q)$ as a subspace of $\L(\F X)$.
When $p=q$, we have that $\iota_{p,p}$ is a $*$-homomorphism, and thus we may identify $\L(X_p)$ as a C*-subalgebra of $\L(\F X)$.
More generally, we may identify $\bigoplus_{p\in Q}\L(X_p)$ as a C*-subalgebra of $\L(\F X)$, for any subset $Q$ of $P$.

The \emph{Fock representation} $(\ol{\pi},\ol{t})$ is the representation of $X$ on $\L(\F X)$ determined by 
\begin{align*}
\ol{\pi}(a)\xi_q=\phi_q(a)\xi_q \qand \ol{t}_p(\xi_p)\xi_q=\xi_p\xi_q,
\end{align*}
for all $a\in A, \xi_p\in X_p, \xi_q\in X_q$ and $p,q\in P$. 
For $\xi_p\in X_p$ with $p \neq e$, we have that $\ol{t}_p(\xi_p)^*$ maps $X_r$ to $0$ whenever $r\notin pP$.
Conversely, if $r=pq$ for some $q\in P$, then
\[
\ol{t}_p(\xi_p)^*(\eta_p\eta_q)=\phi_q(\sca{\xi_p,\eta_p})\eta_q \foral \eta_p\in X_p,\eta_q\in X_q.
\]
Since $u_{p,q}(X_p \odot X_q)$ spans a dense subset of $X_{pq}$, and $X_{pq}$ embeds isometrically in $\F X$, this formula is sufficient to determine $\ol{t}_p(\xi_p)^*$ on all of $X_{pq}$ and hence on all of $\F X$.
The Fock representation is injective, for if $\ol{\pi}(a)=0$ for some $a \in A$ then in particular $a a^* = \ol{\pi}(a) a^*=0$, and thus $a=0$.

Next we illustrate how the quotient construction for C*-correspondences extends to the setting of product systems.
The following definition is motivated by the notion of positive invariance for C*-correspondences.

\begin{definition}\label{D:pos inv}
Let $P$ be a unital subsemigroup of a discrete group $G$. 
Let $X$ be a product system over $P$ with coefficients in a C*-algebra $A$ and let $I\subseteq A$ be an ideal. 
We say that $I$ is \emph{positively invariant for $X$} if it satisfies
\[ X(I):=\ol{\text{span}}\{\sca{X_p, I X_p}\mid p\in P \}\subseteq I.\]
\end{definition}

Notice that the space $X(I)$ is an ideal of $A$ (and in fact, this is true even when $I$ is replaced by any subset of $A$), and that an ideal $I$ is positively invariant for $X$ if and only if it is positively invariant for every fibre of $X$.
This observation lies at the heart of the following proposition.

\begin{proposition}\label{P:qnt}
Let $P$ be a unital subsemigroup of a discrete group $G$. 
Let $X$ be a product system over $P$ with coefficients in a C*-algebra $A$ and let $I\subseteq A$ be an ideal that is positively invariant for $X$.
Set 
\[
[X]_I:=\{[X_p]_I\}_{p\in P},
\; \text{where } \;
[X_p]_I = X_p/X_pI \foral p \in P.
\]
Then $[X]_I$, endowed with the multiplication maps 
\[
[X_p]_I\otimes_{[A]_I}[X_q]_I\to[X_{pq}]_I; [\xi_p]_I \otimes [\xi_q]_I \mapsto [\xi_p \xi_q]_I \foral p,q \in P,
\] 
is a product system over $P$ with coefficients in $[A]_I$.
\end{proposition}

\begin{proof}
We will denote the multiplication maps of $X$ by $\{u_{p,q}\}_{p,q\in P}$.
The quotient construction for C*-correspondences covered in Subsection \ref{Ss:C*-cor} renders a canonical structure on $[X_p]_I$ as a C*-correspondence over $[A]_I$ for all $p\in P\setminus\{e\}$.
Thus $[X]_I$ constitutes a family of C*-correspondences over $[A]_I$.
It remains to show that $[X]_I$, together with the maps of the statement, satisfies axioms (i)-(v) of a product system.

First we show well-definedness of the multiplication maps for $[X]_I$.
This is immediate for $p=e$ or $q=e$.
So fix $p,q \in P \setminus \{e\}$ and define the map
\[ 
v_{p,q} \colon [X_p]_I\times[X_q]_I\to[X_{pq}]_I; ([\xi_p]_I,[\xi_q]_I) \mapsto [\xi_p\xi_q]_I\equiv [u_{p,q}(\xi_p\otimes\xi_q)]_I \foral \xi_p \in X_p, \xi_q \in X_q.
\] 
To see that $v_{p,q}$ is well-defined, fix $\xi_p\in X_p, \xi_q\in X_q$ and suppose that
\[
[\xi_p]_I=[\eta_p]_I \qand [\xi_q]_I=[\eta_q]_I
\]
for some $\eta_p\in X_p$ and $\eta_q\in X_q$.
Then $\xi_p = \eta_p + \zeta_p a$ and $\xi_q = \eta_q + \zeta_q b$ for some $\zeta_p\in X_p$, $\zeta_q\in X_q$ and $a,b \in I$.
A direct computation yields
\begin{align*}
[\xi_p\xi_q]_I 
=
[(\eta_p+\zeta_p a)(\eta_q+\zeta_q b)]_I
=
[\eta_p\eta_q+\eta_p(\zeta_q b)+(\zeta_p a)\eta_q+(\zeta_p a)(\zeta_q b)]_I.
\end{align*}
Observe that 
\[
\eta_p(\zeta_qb)=(\eta_p\zeta_q) b\in X_{pq}I
\qand
(\zeta_p a)(\zeta_q b)=((\zeta_p a)\zeta_q) b\in X_{pq}I.
\]
Moreover, by Lemma \ref{L:IX to XI} we have
\[
(\zeta_p a)\eta_q = \zeta_p (a \eta_q) \in u_{p,q}(X_p \otimes_A I X_q) \subseteq u_{p,q}(X_p \otimes_A X_q I) = X_{pq}I
\]
using positive invariance of $I$.
Hence 
\[
[\eta_p(\zeta_qb)]_I=[(\zeta_pa)\eta_q]_I=[(\zeta_pa)(\zeta_qb)]_I=0,
\]
and thus $[\xi_p\xi_q]_I=[\eta_p\eta_q]_I$, as required.

Next observe that $v_{p,q}$ is bilinear and $[A]_I$-balanced, and hence it induces a unique linear map
\[
v_{p,q} \colon [X_p]_I\odot_{[A]_I}[X_q]_I \to [X_{pq}]_I; [\xi_p]_I\otimes [\xi_q]_I \mapsto [\xi_p\xi_q]_I \foral \xi_p \in X_p, \xi_q \in X_q.
\]
For $\xi_p,\eta_p\in X_p$ and $\xi_q,\eta_q\in X_q$, we have
\begin{align*}
\sca{v_{p,q}([\xi_p]_I\otimes[\xi_q]_I), v_{p,q}([\eta_p]_I\otimes[\eta_q]_I)} 
& =
\sca{[\xi_p\xi_q]_I,[\eta_p\eta_q]_I} 
=
[\sca{\xi_p\otimes\xi_q,\eta_p\otimes\eta_q}]_I \\
& =
[\sca{\xi_q,\phi_q(\sca{\xi_p,\eta_p})\eta_q}]_I 
=
\sca{[\xi_p]_I\otimes[\xi_q]_I,[\eta_p]_I\otimes[\eta_q]_I},
\end{align*}
and thus $\sca{v_{p,q}(\zeta),v_{p,q}(\zeta')}=\sca{\zeta,\zeta'}$ for all $\zeta, \zeta' \in [X_p]_I\odot_{[A]_I}[X_q]_I$.
In particular, $v_{p,q}$ is bounded with respect to the norm on $[X_p]_I\odot_{[A]_I}[X_q]_I$ induced by the $[A]_I$-valued inner product.
Hence it extends to a bounded linear map given  by
\[ 
v_{p,q} \colon [X_p]_I\otimes_{[A]_I}[X_q]_I\to[X_{pq}]_I; [\xi_p]_I \otimes [\xi_q]_I \mapsto [\xi_p\xi_q]_I \foral \xi_p \in X_p, \xi_q \in X_q.
\]

From the preceding calculations it also follows that each $v_{p,q}$ is isometric, and similar reasoning shows that each $v_{p,q}$ is an $[A]_I$-bimodule map.
Surjectivity of $v_{p,q}$ onto the corresponding space follows from surjectivity of $u_{p,q}$.
Associativity of the multiplication maps $\{v_{p,q}\}_{p,q \in P}$ follows from associativity of $\{u_{p,q}\}_{p,q \in P}$, and the proof is complete.
\end{proof}

We will use the notation $[\hspace{1pt}\cdot\hspace{1pt}]_I \colon X \to [X]_I$ as shorthand for the family of quotient maps $[\hspace{1pt}\cdot\hspace{1pt}]_I \colon X_p \to [X_p]_I$ for all $p \in P$.

\subsection{Product systems over right LCM semigroups}\label{Ss: ca prod sys}

A unital semigroup $P$ is said to be a \emph{right LCM semigroup} if it is left cancellative and satisfies Clifford's condition, that is:
\[
\text{for every} \; p,q\in P \; \text{with} \; pP\cap qP\neq\mt, \; \text{there exists} \; w\in P \; \text{such that} \; pP\cap qP=wP.
\]
The element $w$ is referred to as a \emph{right least common multiple} or \emph{right LCM} of $p$ and $q$.

Group embeddable right LCM semigroups include as a special case the quasi-lattice ordered groups considered in \cite{Nic92}, or more generally the weakly quasi-lattice ordered groups considered in \cite{Exe17}.
Indeed, both cases form right LCM semigroups with the property that the only invertible element in $P$ is the unit.
Further examples include the Artin monoids \cite{BS72}, the Baumslag-Solitar monoids $B(m,n)^+$ \cite{HNSY19, Li20, Spi12}, and the semigroup $R \rtimes R^\times$ of affine transformations of an integral domain $R$ that satisfies the GCD condition \cite{Li13, Nor14}.

Compactly aligned product systems over right LCM semigroups were introduced and studied by Kwa\'{s}niewski and Larsen \cite{KL19a, KL19b}, extending the construction of Fowler \cite{Fow02}.
They have been investigated further in \cite{DKKLL20, KKLL21b}.
The interest lies in that they retain several of the structural properties from the single C*-correspondence case.

Let $A$ be a C*-algebra and let $P$ be a unital right LCM subsemigroup of a discrete group $G$. 
Let $X$ be a product system over $P$ with coefficients in $A$. 
We say that $X$ is \emph{compactly aligned} if, for all $p,q\in P\setminus\{e\}$ with the property that $pP\cap qP=wP$ for some $w\in P$, we have
\[ 
\iota_p^w(\K(X_p)) \iota_q^w(\K(X_q))\subseteq \K(X_w).
\]
This condition is independent of the choice of right LCM $w\in P$ of $p,q\in P$, e.g., \cite[p. 11]{DKKLL20}.
Notice that we disregard the case where $p$ or $q$ equals $e$, as the compact alignment condition holds automatically in this case.

When $(G, P)$ is totally ordered, $X$ is automatically compactly aligned, e.g., when $P = \bZ_+$.
Moreover, it is a standard fact that if $\phi_p(A)\subseteq\K(X_p)$ for all $p\in P$, then $X$ is automatically compactly aligned.
We provide a short proof.

\begin{proposition}\label{P:autcaps}
Let $P$ be a unital subsemigroup of a discrete group $G$.
Let $X$ be a product system over $P$ with coefficients in a C*-algebra $A$.
If $\phi_p(A)\subseteq\K(X_p)$ for all $p\in P$, then $\iota_p^{pq}(\K(X_p))\subseteq\K(X_{pq})$ for all $p,q \in P$, and thus if $P$ is a right LCM semigroup, then $X$ is automatically compactly aligned.
\end{proposition}

\begin{proof}
First note that $\iota_e^q(\phi_e(a))=\phi_q(a)\in\K(X_q)$ for all $q\in P$ and $a\in A$.
Now fix $p\in P\setminus\{e\}, q\in P$ and $k_p\in\K(X_p)$.
Since $\phi_q(A)\subseteq\K(X_q)$ and $k_p\in\K(X_p)$, an application of \cite[Proposition 4.7]{Lan95} gives that $k_p\otimes\id_{X_q}\in\K(X_p\otimes_A X_q)$.
Hence an application of \cite[p. 9, (1.6)]{Lan95} yields $\iota_p^{pq}(k_p)= u_{p,q}(k_p\otimes\id_{X_q})u_{p,q}^* \in\K(X_{pq})$, as required.
\end{proof}

\begin{remark}\label{R:unitca}
Let $X$ and $Y$ be unitarily equivalent product systems over $P$ with coefficients in $A$ and $B$, respectively.
An application of (\ref{Eq:unitcomp}) and (\ref{Eq:unitmult}) then gives that $X$ is compactly aligned if and only if $Y$ is compactly aligned.
\end{remark}

We say that a representation $(\pi,t)$ of a compactly aligned product system $X$ is \emph{Nica-covariant} if for all $p,q\in P\setminus\{e\}, k_p\in\K(X_p)$ and $k_q\in\K(X_q)$ we have
\begin{equation}\label{Eq:nc}
\psi_p(k_p)\psi_q(k_q)
=
\begin{cases} 
\psi_w(\iota_p^w(k_p) \iota_q^w(k_q)) & \text{if } pP\cap qP=wP \text{ for some } w\in P, \\ 
0 & \text{otherwise}. 
\end{cases} 
\end{equation}
This condition is independent of the choice of right LCM $w\in P$ of $p,q\in P$, e.g., \cite[Proposition 2.4]{DKKLL20} or \cite[Lemma 3.10]{KL19a}.
We disregard the case where $p$ or $q$ equals $e$, as (\ref{Eq:nc}) holds automatically in this case.
It is straightforward to see that the Fock representation is Nica-covariant.

The Nica-covariance condition induces a Wick ordering on $\ca(\pi,t)$, e.g, \cite{DKKLL20, Fow02, KL19a, KL19b}.
More specifically, let $p,q\in P$ and suppose firstly that $pP\cap qP=wP$ for some $w\in P$. 
Then
\[ 
t_p(X_p)^*t_q(X_q) \subseteq [t_{p'}(X_{p'})t_{q'}(X_{q'})^*], \; \text{where} \; p'=p^{-1}w, q'=q^{-1}w. 
\]
On the other hand, if $pP\cap qP=\mt$, then $t_p(X_p)^*t_q(X_q) = \{0\}$.
From this it follows that 
\[ 
\ca(\pi,t) = \ol{\text{span}}\{t_p(X_p)t_q(X_q)^* \mid p,q\in P\}. 
\]

We write $\N\T_X$ for the universal C*-algebra with respect to the Nica-covariant representations of $X$, and refer to it as the \emph{Toeplitz-Nica-Pimsner algebra of $X$}.
Since the Nica-covariance relations are graded, the existence of $\N\T_X$ and its universal property follow from Proposition \ref{P:T univ}, see also \cite[Proposition 4.3]{DKKLL20}.
We write $(\ol{\pi}_X, \ol{t}_X)$ for the \emph{universal Nica-covariant representation of $X$}.
If $(\pi,t)$ is a Nica-covariant representation of $X$, we will write (in a slight abuse of notation) $\pi \times t$ for the canonical $*$-epimorphism $\N\T_X \to \ca(\pi,t)$.
Injectivity of the Fock representation $(\ol{\pi},\ol{t})$ implies that $(\ol{\pi}_X,\ol{t}_X)$ is injective.
In fact, when $P$ is contained in an amenable discrete group, the $*$-representation $\ol{\pi} \times \ol{t}$ is faithful, e.g., \cite{KKLL21b}.

\begin{remark}\label{R:ntpiso}
If $X$ and $Y$ are unitarily equivalent product systems, then $\T_X\cong\T_Y$ by Remark \ref{R:tpiso}.
When $X$ or $Y$ is compactly aligned (and hence both are by Remark \ref{R:unitca}), we moreover have that the bijection of Remark \ref{R:tpiso} preserves the Nica-covariant representations. 
This is easily seen by Remark \ref{R:prodsysisoprop}.
Therefore $\N\T_X\cong\N\T_Y$ canonically.
\end{remark}

\begin{proposition}\label{P:qnt ca}
Let $P$ be a unital right LCM subsemigroup of a discrete group $G$.
Let $X$ be a compactly aligned product system over $P$ with coefficients in a C*-algebra $A$ and let $I\subseteq A$ be an ideal that is positively invariant for $X$.
Let  $\{\iota_p^{pq}\}_{p,q\in P}$ denote the connecting $*$-homomorphisms of $X$ and $\{j_p^{pq}\}_{p,q\in P}$ denote the connecting $*$-homomorphisms of $[X]_I$.
Then 
\[
j_e^{q}([\phi_e(a)]_I) = [\iota_e^{q}(\phi_e(a))]_I, \foral a \in A,
\qand
j_p^{pq}([S_p]_I) = [\iota_p^{pq}(S_p)]_I, \foral S_p \in \L(X_p),
\]
for all $p \in P \setminus \{e\}$ and $q \in P$, and thus $[X]_I$ is compactly aligned.
\end{proposition}

\begin{proof}
Fix $q\in P$ and $a\in A$, and observe that
\[ 
j_e^q([\phi_e(a)]_I)=[\phi_q]_I([a]_I)=[\phi_q(a)]_I=[\iota_e^q(\phi_e(a))]_I. 
\]
Next, fix $p\in P\setminus\{e\}, q\in P$ and $S_p\in\L(X_p)$.
For $\xi_p\in X_p$ and $\xi_q\in X_q$, we have
\begin{align*}
j_p^{pq}([S_p]_I)([\xi_p]_I[\xi_q]_I) 
= [(S_p\xi_p)\xi_q]_I 
= [\iota_p^{pq}(S_p)(\xi_p\xi_q)]_I 
= [\iota_p^{pq}(S_p)]_I([\xi_p]_I[\xi_q]_I),
\end{align*}
from which it follows that $j_p^{pq}([S_p]_I)=[\iota_p^{pq}(S_p)]_I$.
This proves the first claim.

Finally, take $p,q\in P\setminus\{e\}$ such that $pP\cap qP=wP$ for some $w\in P$, and fix $\dot{k}_p\in\K([X_p]_I)$ and $\dot{k}_q\in\K([X_q]_I)$.
By Lemma \ref{L:Kat07}, we have
\[
\textup{$\dot{k}_p=[k_p]_I$ and $\dot{k}_q=[k_q]_I$ for some $k_p\in\K(X_p)$ and $k_q\in\K(X_q)$,}
\]
and therefore we obtain
\begin{align*}
j_p^w(\dot{k}_p)j_q^w(\dot{k}_q) & =j_p^w([k_p]_I)j_q^w([k_q]_I) 
					   =[\iota_p^w(k_p)]_I[\iota_q^w(k_q)]_I 
					   =[\iota_p^w(k_p)\iota_q^w(k_q)]_I.
\end{align*}
Since $X$ is compactly aligned, we have $\iota_p^w(k_p)\iota_q^w(k_q)\in\K(X_w)$. 
Another application of Lemma \ref{L:Kat07} then gives that $j_p^w(\dot{k}_p)j_q^w(\dot{k}_q) \in \K([X_w]_I)$. 
Hence $[X]_I$ is compactly aligned and the proof is complete.
\end{proof}

\subsection{Product systems over $\bZ_+^d$}\label{Ss: scaps}

For $d\in\bN$, we write $[d] := \{1, 2, \dots, d\}$. 
We denote the usual free generators of $\bZ_+^d$ by $\un{1}, \dots, \un{d}$, and we set $\un{0} = (0, \dots, 0)$. 
For $\un{n}=(n_1,\dots,n_d)\in\bZ_+^d,$ we define the \emph{length} of $\un{n}$ by $|\un{n}|:= \sum_{i \in [d]} n_i$.
For $\mt\neq F \subseteq [d]$, we write
\[ 
\un{1}_F := \sum_{i \in F} \un{i} \; \text{and} \; \un{1}_\mt:=\un{0}. 
\]
We consider the lattice structure on $\bZ_+^d$ given by
\[ 
\un{n} \vee \un{m} := (\max\{n_i, m_i\})_{i=1}^d \qand \un{n} \wedge \un{m} := (\min\{n_i, m_i\})_{i=1}^d. 
\]
The semigroup $\bZ_+^d$ imposes a partial order on $\bZ^d$ that is compatible with the lattice structure. 
Specifically, we say that $\un{n}\leq\un{m}$ (resp. $\un{n} < \un{m}$) if and only if $n_i\leq m_i$ for all $i\in[d]$ (resp. $\un{n} \leq \un{m}$ and $\un{n} \neq \un{m}$). 
We denote the \emph{support} of $\un{n}$ by 
\[
\supp \un{n} := \{i \in [d] \mid n_i \neq 0\},
\]
and we write
\[ 
\un{n} \perp \un{m} \iff \supp \un{n} \cap \supp \un{m} = \mt.
\]
We write $\un{n} \perp F$ if $\supp \un{n} \cap F = \mt$.
Notice that the set $\{\un{n}\in\bZ_+^d\mid \un{n}\perp F\}$ is directed; indeed, if $\un{n},\un{m}\perp F$ then $\un{n},\un{m}\leq\un{n}\vee\un{m}\perp F$.
Therefore we can make sense of limits with respect to $\un{n} \perp F$.

Many desirable properties of the fibres of a product system $X = \{X_{\un{n}}\}_{\un{n} \in \bZ_+^d}$ are inherited from the corresponding properties of the fibres $X_\un{i}$, where $i\in[d]$.

\begin{proposition}\label{P:inj+comp}
Let $X$ be a product system over $\bZ_+^d$ with coefficients in a C*-algebra $A$.
Then the following hold:
\begin{enumerate}
\item $X_\un{i}$ is injective for all $i\in[d]$ if and only if $X_\un{n}$ is injective for all $\un{n}\in\bZ_+^d$.
\item $\phi_\un{i}(A)\subseteq\K(X_\un{i})$ for all $i\in[d]$ if and only if $\phi_\un{n}(A)\subseteq\K(X_\un{n})$ for all $\un{n}\in\bZ_+^d$.
\end{enumerate}
In particular, $X_\un{i}$ is regular for all $i\in[d]$ if and only if $X$ is regular.
\end{proposition}

\begin{proof}
Both items follow by induction on $|\un{n}|$.
Specifically, item (i) follows by \cite[p. 42]{Lan95}, and item (ii) follows by \cite[Proposition 4.7 and p. 9, (1.6)]{Lan95}.
\end{proof}

The following lemma will be used frequently in the sequel.

\begin{lemma}\label{L:inv comp}
Let $X$ be a product system over $\bZ_+^d$ with coefficients in a C*-algebra $A$.
Fix $F \subseteq [d]$ and let $I \subseteq A$ be an ideal such that
\[
I \subseteq \bigcap \{\phi_{\un{i}}^{-1}(\K(X_{\un{i}}))\mid i\in F^c\}
\qand
\sca{ X_{\un{n}}, I X_{\un{n}} } \subseteq I \foral \un{n} \perp F.
\]
Then
\[
\phi_{\un{n}}(I) \subseteq \K(X_\un{n} I) \foral \un{n}\perp F.
\]
\end{lemma}

\begin{proof}
We will proceed by induction on $|\un{n}|$. 
The cases where $|\un{n}| = 0$ and $F = [d]$ follow trivially.
The statement holds when $|\un{n}| = 1$ by (\ref{Eq: comp}).
Now assume that the claim holds for all $\un{n}\perp F$ with $|\un{n}|=N$, where $N\geq1$. 
Fix $\un{m}\perp F$ with $|\un{m}|=N+1$, so that $\un{m}=\un{n}+\un{i}$ for some $\un{n} \perp F$ and $i\in F^c$ such that $|\un{n}|=N$. 
We must show that $\phi_\un{m}(a)\in\K(X_\un{m}I)$ for all $a \in I$. 
By definition we have
\[
\phi_\un{m}(a)=\phi_{\un{n}+\un{i}}(a)=\iota_\un{n}^{\un{n}+\un{i}}(\phi_\un{n}(a))=u_{\un{n},\un{i}}(\phi_\un{n}(a)\otimes\text{id}_{X_\un{i}})u_{\un{n},\un{i}}^*.
\]

Consider the element $k_\un{n}\otimes\id_{X_\un{i}}$ for some $k_\un{n}\in\K(X_\un{n}I)$.
From the length $1$ case, we have $\phi_\un{i}(I)\subseteq\K(X_\un{i}I)$.
An application of Corollary \ref{C:comp to comp} then yields $k_\un{n}\otimes\id_{X_\un{i}}\in\K((X_\un{n}\otimes_A X_\un{i})I)$, and thus $u_{\un{n},\un{i}}(k_\un{n}\otimes\id_{X_\un{i}})u_{\un{n},\un{i}}^*\in\K(X_\un{m} I)$.

Returning to the proof, by the inductive hypothesis we have $\phi_\un{n}(a)\in\K(X_\un{n}I)$ and hence, applying the preceding comment for $k_\un{n}=\phi_\un{n}(a)$, we deduce that $\phi_\un{m}(a)\in\K(X_\un{m}I)$.
By induction, the proof is complete.
\end{proof}

Let $X$ be a compactly aligned product system over $\bZ_+^d$ with coefficients in a C*-algebra $A$ and let $(\pi,t)$ be a Nica-covariant representation of $X$. 
We say that $(\pi,t)$ \emph{admits a gauge action} $\ga \colon \bT^d\to\text{Aut}(\ca(\pi,t))$ if $\ga$ is a group homomorphism, $\{\ga_\un{z}\}_{\un{z}\in\bT^d}$ is point-norm continuous and
\[ \ga_\un{z}(\pi(a))=\pi(a) \foral a\in A \; \text{and} \; \ga_\un{z}(t_\un{n}(\xi_\un{n}))=\un{z}^\un{n}t_\un{n}(\xi_\un{n}) \foral \xi_\un{n}\in X_\un{n}, \un{n}\in\bZ_+^d\setminus\{\un{0}\}, \]
for each $\un{z}\in\bT^d$.
If $\un{z}=(z_1,\dots,z_d)\in\bT^d$ and $\un{n}=(n_1,\dots,n_d)\in\bZ_+^d$, then we write 
\[
\un{z}^\un{n}:=z_1^{n_1}\dots z_d^{n_d}.
\]
When such a gauge action $\ga$ exists, it is necessarily unique.
The Fock representation and the universal Nica-covariant representation admit gauge actions.
We say that an ideal $\mathfrak{J}\subseteq\ca(\pi,t)$ is \emph{gauge-invariant} or \emph{equivariant} if $\ga_\un{z}(\mathfrak{J})\subseteq\mathfrak{J}$ for all $\un{z}\in\bT^d$ (and so $\ga_\un{z}(\mathfrak{J})=\mathfrak{J}$ for all $\un{z}\in\bT^d$).

Given $\un{m},\un{m}'\in\bZ_+^d$ with $\un{m} \leq \un{m}'$, we write
\begin{align*}
B_{[\un{m}, \un{m}']}^{(\pi,t)} := \spn\{ \psi_{\un{n}}(\K(X_{\un{n}})) \mid \un{m} \leq \un{n} \leq \un{m}'\}
\qand
B_{(\un{m}, \un{m}']}^{(\pi,t)} := \spn\{ \psi_{\un{n}}(\K(X_{\un{n}})) \mid \un{m} < \un{n} \leq \un{m}'\}.
\end{align*}
These spaces are in fact C*-subalgebras of $\ca(\pi,t)$, e.g, \cite{CLSV11}.
By convention we take the linear span of $\mt$ to be $\{0\}$, so that $B^{(\pi,t)}_{(\un{m},\un{m}]}=\{0\}$ for all $\un{m}\in\bZ_+^d$.
We also define
\begin{align*}
B_{[\un{m}, \infty]}^{(\pi,t)} := \ol{\spn} \{ \psi_{\un{n}}(\K(X_{\un{n}})) \mid \un{m} \leq \un{n} \}
\qand
B_{(\un{m}, \infty]}^{(\pi,t)} := \ol{\spn} \{ \psi_{\un{n}}(\K(X_{\un{n}})) \mid \un{m} < \un{n} \}.
\end{align*}
We refer to these C*-algebras as the \emph{cores} of $(\pi,t)$. 
When $(\pi,t)$ admits a gauge action $\ga$, we have that
\[ 
B_{[\un{0},\infty]}^{(\pi,t)}=\ca(\pi,t)^\ga:=\{f\in\ca(\pi,t) \mid \ga_{\un{z}}(f)=f \foral \un{z}\in\bT^d\}
\]
is the \emph{fixed point algebra} of $\ca(\pi,t)$ with respect to $\gamma$.

In the current work we will consider a subclass of compactly aligned product systems over $\bZ_+^d$ introduced by Dor-On and the second-named author \cite{DK18}.
Let $X$ be a product system over $\bZ_+^d$ with coefficients in a C*-algebra $A$.
We say that $X$ is \emph{strong compactly aligned} if it is compactly aligned and satisfies:
\begin{equation}\label{Eq:sca}
\iota_\un{n}^{\un{n}+\un{i}}(\K(X_\un{n}))\subseteq\K(X_{\un{n}+\un{i}}) 
\textup{ whenever $\un{n}\perp\un{i}$, where $i\in[d],\un{n}\in\bZ_+^d\setminus\{\un{0}\}$}.
\end{equation}
We disallow $\un{n}=\un{0}$, as then (\ref{Eq:sca}) would imply that the left action of each fibre of $X$ is by compact operators. 
Note that (\ref{Eq:sca}) does not imply compact alignment (rather, a strong compactly aligned product system is \emph{a priori} assumed to be compactly aligned). 
Any C*-correspondence, when viewed as a product system over $\bZ_+$, is vacuously strong compactly aligned.
Moreover, strong compactly aligned product systems include those product systems over $\bZ_+^d$ whose left actions are by compacts.

\begin{corollary}\label{C:autscaps}
Let $X$ be a product system over $\bZ_+^d$ with coefficients in a C*-algebra $A$.
If $\phi_\un{n}(A)\subseteq\K(X_\un{n})$ for all $\un{n}\in\bZ_+^d$, then $X$ is strong compactly aligned.
\end{corollary}

\begin{proof}
This is immediate by Proposition \ref{P:autcaps}.
\end{proof}

\begin{remark}\label{R:unitsca}
Let $X$ and $Y$ be unitarily equivalent product systems over $\bZ_+^d$ with coefficients in C*-algebras $A$ and $B$, respectively.
An application of Remark \ref{R:prodsysisoprop} and Remark \ref{R:unitca} implies that $X$ is strong compactly aligned if and only if $Y$ is strong compactly aligned.
\end{remark}

The quotient construction of Proposition \ref{P:qnt} also preserves strong compact alignment.

\begin{proposition}\label{P:qnt sca}
Let $X$ be a strong compactly aligned product system with coefficients in a C*-algebra $A$ and let $I\subseteq A$ be an ideal that is positively invariant for $X$.
Then $[X]_I$ is strong compactly aligned.
\end{proposition}

\begin{proof}
We use $\{\iota_\un{n}^{\un{n}+\un{m}}\}_{\un{n},\un{m}\in \bZ_+^d}$ (resp. $\{j_\un{n}^{\un{n}+\un{m}}\}_{\un{n},\un{m}\in \bZ_+^d}$) to denote the connecting $*$-homomorphisms of $X$ (resp. $[X]_I$).
Proposition \ref{P:qnt ca} guarantees that $[X]_I$ is compactly aligned, so it remains to check that $[X]_I$ is strong compactly aligned.
To this end, fix $\un{n}\in\bZ_+^d\setminus\{\un{0}\}, i\in[d]$ and $\dot{k}_\un{n}\in\K([X_\un{n}]_I)$, and suppose that $\un{n}\perp\un{i}$.
By Lemma \ref{L:Kat07}, we have $\dot{k}_\un{n}=[k_\un{n}]_I$ for some $k_\un{n}\in\K(X_\un{n})$. Thus, by Proposition \ref{P:qnt ca}, we have 
\[
 j_\un{n}^{\un{n}+\un{i}}(\dot{k}_\un{n})
=
 j_\un{n}^{\un{n}+\un{i}}([k_\un{n}]_I)
=
[\iota_\un{n}^{\un{n}+\un{i}}(k_\un{n})]_I 
\in 
\K([X_{\un{n}+\un{i}}]_I),
\]
where we have used the strong compact alignment of $X$ together with Lemma \ref{L:Kat07} in the final assertion.
Hence $j_\un{n}^{\un{n}+\un{i}}(\K([X_\un{n}]_I))\subseteq\K([X_{\un{n}+\un{i}}]_I)$, completing the proof.
\end{proof}

We will require some notation and results from \cite{DK18}.
Henceforth we assume that $X$ is strong compactly aligned.
For each $\mt\neq F\subseteq[d]$, we define
\[ 
\J_F :=
(\bigcap_{i\in F}\ker\phi_{\un{i}})^\perp\cap(\bigcap_{i\in [d]}\phi_{\un{i}}^{-1}(\K(X_{\un{i}})))\qand \J_\mt:=\{0\},
\]
which are ideals of $A$.
Notice that strong compact alignment of $X$ implies that 
\[ 
\bigcap_{i\in F}\phi_{\un{i}}^{-1}(\K(X_{\un{i}}))=\bigcap\{\phi_{\un{n}}^{-1}(\K(X_{\un{n}}))\mid \un{0}\leq\un{n}\leq\un{1}_F \}\foral \mt\neq F\subseteq[d].
\]
In turn, for each $F\subseteq [d]$, we define
\[ 
\I_F :=
\{a \in A \mid \sca{X_{\un{n}},aX_{\un{n}}}\subseteq \J_F \; \text{for all} \; \un{n}\perp F\}=\bigcap\{X_\un{n}^{-1}(\J_F)\mid \un{n}\perp F\}.
\]
In particular, we have $\I_\mt=\{0\}$ and $\I_F \subseteq \J_F$ for all $F \subseteq [d]$.
The ideal $\I_F$ is the largest ideal in $\J_F$ that is \emph{$F^\perp$-invariant}, i.e, 
\[
\sca{X_\un{n}, \I_F X_\un{n}}\subseteq\I_F
\foral \un{n}\perp F,
\]
see \cite[Proposition 2.7]{DK18}.
To avoid ambiguity, given two strong compactly aligned product systems $X$ and $Y$, we will denote the ideals $\J_F$ for $X$ and $Y$ by $\J_F(X)$ and $\J_F(Y)$, respectively.
We will use the same convention for the ideals $\I_F$.
The ideals $\J_F$ and $\I_F$ are preserved under unitary equivalence.

\begin{remark}\label{R:JIunit}
Let $X$ and $Y$ be strong compactly aligned product systems over $\bZ_+^d$ with coefficients in C*-algebras $A$ and $B$, respectively.
Suppose that $X$ and $Y$ are unitarily equivalent by a collection $\{W_{\un{n}}\colon X_\un{n}\to Y_\un{n}\}_{\un{n} \in \bZ_+^d}$.
The properties of $\{W_\un{n}\}_{\un{n}\in\bZ_+^d}$ from Remark \ref{R:prodsysisoprop} then yield
\[
W_\un{0}(\J_F(X))=\J_F(Y) \qand W_\un{0}(\I_F(X))=\I_F(Y) \foral F\subseteq[d].
\]
\end{remark}

The ideals $\I_F$ emerge naturally when solving polynomial equations, originating in \cite{DFK17} in the case of C*-dynamical systems.
In order to make this precise, we require the following notation.
Following the conventions of \cite[Section 3]{DK18}, we introduce an approximate unit $(k_{\un{i},\la})_{\la\in\La}$ of $\K(X_{\un{i}})$ for each generator $\un{i}$ of $\bZ_+^d$.
Without loss of generality, we may assume that these approximate units are indexed by the same directed set $\La$, by replacing with their product.

\begin{proposition}\label{P:sca ai} \cite[Proposition 2.4]{DK18}
Let $X$ be a strong compactly aligned product system with coefficients in a C*-algebra $A$.
Let $(k_{\un{i},\la})_{\la\in\La}$ be an approximate unit of $\K(X_\un{i})$ for all $i\in[d]$.
Fix $\mt\neq F \subseteq [d]$ and $\un{0} \neq \un{n} \in \bZ_+^d$, and set $\un{m} = \un{n} \vee \un{1}_F$.
Then the net $(e_{F,\la})_{\la\in\La}$ defined by
\[
e_{F, \la}: = \prod \{ \iota_{\un{i}}^{\un{1}_F}(k_{\un{i}, \la}) \mid i \in F \} \foral \la\in\La,
\]
is contained in $\K(X_{\un{1}_F})$, and we have
\begin{equation}\label{eq2-1}
\nor{\cdot}\text{-}\lim_\la \iota_{\un{1}_F}^{\un{m}}(e_{F, \la})\iota_{\un{n}}^{\un{m}}(k_{\un{n}}) 
= 
\iota_{\un{n}}^{\un{m}}(k_{\un{n}})  \foral k_\un{n} \in \K(X_{\un{n}}).
\end{equation}
Moreover, it follows that $\iota_{\un{n}}^{\un{m}}(k_{\un{n}}) \in \K(X_{\un{m}})$ for all $k_\un{n} \in \K(X_{\un{n}})$.
\end{proposition}

Note that we did not specify an order for the product in $e_{F,\la}$ above; that is because (\ref{eq2-1}) holds independently of a chosen order defining the $e_{F,\la}$.
Let $(\pi,t)$ be a Nica-covariant representation of $X$ in some $\B(H)$. 
Fixing $i \in [d]$ and an approximate unit $(k_{\un{i},\la})_{\la\in\La}$ of $\K(X_{\un{i}})$, we define
\begin{equation}\label{eq:pi defn}
p_{\un{i},\la}:=\psi_{\un{i}}(k_{\un{i},\la}) \foral \la\in\La,\; \text{and} \; p_{\un{i}}:=\textup{w*-}\lim_\la p_{\un{i},\la}; 
\end{equation}
i.e., $p_{\un{i}}$ is the projection on the space $[\psi_{\un{i}}(\K(X_{\un{i}}) ) H]$ for the Hilbert space $H$ on which $(\pi,t)$ acts.
In turn, we set
\begin{equation}\label{eq:qi defn}
q_\mt:= I_H,
q_{\un{i}} := I_H - p_{\un{i}},
\text{ and }
q_F:=\prod_{i\in F}(I_H - p_{\un{i}}) \textup{ for all $\mt\neq F\subseteq [d]$}.
\end{equation}

\begin{remark} \label{R:pi proj}
Note that the projections $p_{\un{i}}$ can be defined for a (just) compactly aligned product system $X$.
It is shown in \cite[Proposition 5.6]{Fow02} that if $(\pi,t)$ is Nica-covariant then the projections $p_{\un{i}}$ commute, and so there is no ambiguity regarding the order of the product defining each $q_F$. 
In particular, the operator $p_{\un{i}} p_{\un{j}}$ coincides with the projection on the space $[\psi_{\un{i}+\un{j}}(\K(X_{\un{i} + \un{j}})) H]$ whenever $i \neq j$.
We provide a straightforward proof of this fact for strong compactly aligned product systems for the sake of completeness.

For $i, j \in [d]$ with $i \neq j$, consider the approximate units $(k_{\un{i}, \la})_{\la \in \La}$ and $(k_{\un{j},\la})_{\la \in \La}$ of $\K(X_{\un{i}})$ and $\K(X_{\un{j}})$, respectively. 
We see that the nets
\begin{equation*}
\big( \iota_{\un{i}}^{\un{i} + \un{j}}(k_{\un{i}, \la}) \big)_{\la \in \La}
\qand
\big( \iota_{\un{j}}^{\un{i} + \un{j}}(k_{\un{j}, \la}) \big)_{\la \in \La}
\end{equation*}
are contained in $\K(X_{\un{i} + \un{j}})$ due to strong compact alignment.
In particular, we claim that they provide approximate units for $\K(X_{\un{i} + \un{j}})$.
Indeed, we have $\|\iota_\un{i}^{\un{i}+\un{j}}(k_{\un{i},\la})\|\leq\|k_{\un{i},\la}\|\leq 1$ for all $\la\in\La$, and for $\xi_{\un{i}}, \eta_{\un{i}} \in X_{\un{i}}$ and $\xi_{\un{j}}, \eta_{\un{j}} \in X_{\un{j}}$, we have
\begin{align*}
\| \iota_{\un{i}}^{\un{i} + \un{j}}(k_{\un{i}, \la}) \Theta_{\xi_{\un{i}}\xi_{\un{j}}, \eta_{\un{i}} \eta_{\un{j}}}^{X_{\un{i}+\un{j}}} - \Theta_{\xi_{\un{i}}\xi_{\un{j}}, \eta_{\un{i}} \eta_{\un{j}}}^{X_{\un{i}+\un{j}}} \|
& =
\|\Theta_{(k_{\un{i}, \la} \xi_{\un{i}} - \xi_{\un{i}}) \otimes \xi_{\un{j}}, \eta_{\un{i}} \otimes \eta_{\un{j}} }^{X_\un{i}\otimes_A X_\un{j}} \| 
 \leq
\|k_{\un{i}, \la} \xi_{\un{i}} - \xi_{\un{i}}\| \cdot \|\xi_{\un{j}}\| \cdot \|\eta_{\un{i}}\| \cdot \|\eta_{\un{j}}\| \stackrel{\la}{\to} 0.
\end{align*}
Taking finite linear combinations of rank-one operators and their norm-limits establishes the claim for $\un{i}$, and the case of $\un{j}$ is dealt with by symmetry.
Therefore, due to Nica-covariance, we have
\[
\nor{\cdot}\text{-}\lim_\la \psi_{\un{i}}(k_{\un{i}, \la}) \psi_{\un{i}+ \un{j}}(k_{\un{i} + \un{j}})
=
\nor{\cdot}\text{-}\lim_\la \psi_{\un{i} + \un{j}}( \iota_{\un{i}}^{\un{i} + \un{j}}(k_{\un{i}, \la})k_{\un{i} + \un{j}})
=
\psi_{\un{i}+ \un{j}}(k_{\un{i} + \un{j}}) \foral k_{\un{i}+\un{j}}\in\K(X_{\un{i}+\un{j}}),
\]
and likewise for $\un{j}$.
Thus, for every $h \in H$ and $k_{\un{i} + \un{j}} \in \K(X_{\un{i} + \un{j}})$, we deduce that
\begin{align*}
p_{\un{i}} \psi_{\un{i} + \un{j}}(k_{\un{i} + \un{j}}) h
& =
(\textup{w*-}\lim_{\la} \psi_{\un{i}}(k_{\un{i}, \la}))  \psi_{\un{i} + \un{j}}(k_{\un{i} + \un{j}}) h 
=
\big[\nor{\cdot}\text{-}\lim_{\la} \psi_{\un{i}}(k_{\un{i}, \la})  \psi_{\un{i} + \un{j}}(k_{\un{i} + \un{j}})\big]h 
 =
\psi_{\un{i} + \un{j}}(k_{\un{i} + \un{j}}) h.
\end{align*}
Likewise for $\un{j}$, we have $p_{\un{j}} \psi_{\un{i} + \un{j}}(k_{\un{i} + \un{j}}) h = \psi_{\un{i} + \un{j}}(k_{\un{i} + \un{j}}) h$, and therefore
\[
p_{\un{i}} p_{\un{j}} h = h \foral h \in [\psi_{\un{i} + \un{j}}(\K(X_{\un{i} + \un{j}})) H].
\]
On the other hand, for $h \perp [\psi_{\un{i} + \un{j}}(\K(X_{\un{i} + \un{j}})) H]$ and $h' \in H$, we have
\begin{align*}
\sca{p_{\un{i}} p_{\un{j}} h, h'}
& =
\sca{h, p_{\un{j}} p_{\un{i}} h'} 
=
\lim_\la\lim_{\la'} \sca{h, \psi_{\un{j}}(k_{\un{j}, \la}) \psi_{\un{i}}(k_{\un{i}, \la'}) h'}
=
0,
\end{align*}
where we have used that
\[
\psi_{\un{j}}(k_{\un{j}, \la}) \psi_{\un{i}}(k_{\un{i}, \la'}) h'
=
\psi_{\un{i} + \un{j}}( \iota_{\un{j}}^{\un{i} + \un{j}}(k_{\un{j}, \la})\iota_{\un{i}}^{\un{i} + \un{j}}(k_{\un{i}, \la'}) ) h'
\in 
\psi_{\un{i} + \un{j}}(\K(X_{\un{i} + \un{j}}))H
\foral \la,\la'\in\La,
\]
due to Nica-covariance.
Hence we have
\[
p_{\un{i}} p_{\un{j}} h = 0 \foral h \perp [\psi_{\un{i} + \un{j}}(\K(X_{\un{i} + \un{j}})) H].
\]
Consequently $p_{\un{i}} p_{\un{j}}$ coincides with the projection on the space $[\psi_{\un{i}+\un{j}}(\K(X_{\un{i} + \un{j}})) H]$, as required.
\end{remark}

We gather some useful algebraic relations proved in \cite{DK18}.

\begin{proposition}\label{P:pf reducing} \cite[Proposition 4.4, proof of Proposition 4.6]{DK18}
Let $X$ be a strong compactly aligned product system with coefficients in a C*-algebra $A$.
Let $(\pi,t)$ be a Nica-covariant representation of $X$ and fix $F \subseteq [d]$.
Then for all $\un{m}\in\bZ_+^d$ and $\xi_{\un{m}} \in X_{\un{m}}$, we have
\[
 q_Ft_{\un{m}}(\xi_{\un{m}})
=
\begin{cases}
t_{\un{m}}(\xi_{\un{m}})q_F & \text{ if } \un{m} \perp F, \\
0 & \text{ if } \un{m} \not\perp F,
\end{cases}
\]
so that in particular $q_F \in \pi(A)'$.
Moreover, for the ideal $\I_F$ we have
\[
\pi(\I_F) t_{\un{m}}(X_{\un{m}}) \subseteq t_{\un{m}}(X_{\un{m}}) \pi(\I_F)
\foral
\un{m} \perp F.
\]
\end{proposition}

\begin{proposition}\label{P:prod cai} \cite[Section 3]{DK18}
Let $X$ be a strong compactly aligned product system with coefficients in a C*-algebra $A$ and let $(\pi,t)$ be a Nica-covariant representation of $X$.
Let $p_{\un{i}, \la}$, $p_{\un{i}}$ and $q_F$ be the associated operators of (\ref{eq:pi defn}) and (\ref{eq:qi defn}), and fix $\mt \neq F \subseteq [d]$.
Then
\[
\nor{\cdot}\text{-}\lim_\la  \psi_{\un{n}}(k_{\un{n}}) \prod_{i \in F} p_{\un{i}, \la}
=
\psi_{\un{n}}(k_{\un{n}}) \prod_{i \in F} p_{\un{i}}
\foral \un{n}\in\bZ_+^d\setminus\{\un{0}\}, k_{\un{n}} \in \K(X_{\un{n}}).
\]
If $a \in \bigcap\{ \phi_{\un{i}}^{-1}( \K(X_{\un{i}})) \mid i \in F \}$, then 
\[
\pi(a) \prod_{i \in D} p_{\un{i}} = \psi_{\un{1}_D}(\phi_{\un{1}_D}(a)) \foral \mt\neq D \subseteq F,
\]
and so
\[
\pi(a)q_F
=
\pi(a) + \sum \{ (-1)^{|\un{n}|} \psi_{\un{n}}(\phi_{\un{n}}(a)) \mid \un{0} \neq \un{n} \leq \un{1}_F \}
\in \ca(\pi,t).
\]
\end{proposition}

Note that the product of the $p_{\un{i},\la}$ in the first statement of Proposition \ref{P:prod cai} can be taken in any order.
To elaborate on the second point, let $a\in\bigcap\{\phi_\un{i}^{-1}(\K(X_\un{i}))\mid i\in F\}$.
Then $a \in \bigcap\{\phi_\un{i}^{-1}(\K(X_\un{i}))\mid i\in D\}$ for a fixed $\mt\neq D \subseteq F$, and so we have
\[
\pi(a)\prod_{i\in D}p_{\un{i},\la}=\pi(a)\prod_{i\in D}\psi_\un{i}(k_{\un{i},\la})=\pi(a)\psi_{\un{1}_D}(e_{D,\la})=\psi_{\un{1}_D}(\phi_{\un{1}_D}(a)e_{D,\la}) \foral\la\in\La,
\]
using Nica-covariance in the second equality, as well as the notation of Proposition \ref{P:sca ai}.
We have
\[
\nor{\cdot}\text{-}\lim_\la\psi_{\un{1}_D}(\phi_{\un{1}_D}(a)e_{D,\la})=\psi_{\un{1}_D}(\nor{\cdot}\text{-}\lim_\la\phi_{\un{1}_D}(a)e_{D,\la})=\psi_{\un{1}_D}(\phi_{\un{1}_D}(a)),
\]
by \cite[(2.4)]{DK18}, where we are now using the fact that $\phi_{\un{1}_D}(a)\in\K(X_{\un{1}_D})$ by strong compact alignment.
Note that the w*-limit of the net $(\prod_{i\in D}p_{\un{i},\la})_{\la\in\La}$ is the projection on the space $[\psi_{\un{1}_D}(\K(X_{\un{1}_D}))H]$ and coincides with $\prod_{i\in D}p_\un{i}$, as in Remark \ref{R:pi proj}.
Hence we obtain
\begin{align*}
\pi(a)\prod_{i\in D}p_\un{i}
& =
\textup{w*-}\lim_\la \pi(a)\prod_{i\in D}p_{\un{i},\la}
=
\textup{w*-}\lim_\la \psi_{\un{1}_D}(\phi_{\un{1}_D}(a)e_{D,\la}) \\
& =
\nor{\cdot}\text{-}\lim_\la \psi_{\un{1}_D}(\phi_{\un{1}_D}(a)e_{D,\la})
=
\nor{\cdot}\text{-}\lim_\la\pi(a)\prod_{i\in D}p_{\un{i},\la}
=
\psi_{\un{1}_D}(\phi_{\un{1}_D}(a)),
\end{align*}
as claimed.

Let $(\pi,t)$ be a Nica-covariant representation of $X$. 
In the process of studying the kernel of $\pi \times t$, one needs to solve equations of the form
\begin{equation}\label{eq:in com}
\pi(a) \in B_{(\un{0}, \un{m}]}^{(\pi,t)} \text{ for } a\in A, \un{m} \in \bZ_+^d \setminus \{\un{0}\}.
\end{equation}
Due to the structure of the cores, an element $\pi(a)$ satisfies (\ref{eq:in com}) if and only if there are $k_{\un{n}} \in \K(X_{\un{n}})$ for all $\un{0}\neq\un{n}\leq \un{m}$ such that
\begin{equation}\label{eq:in}
\pi(a) + \sum \{ \psi_{\un{n}}(k_{\un{n}}) \mid \un{0} \neq \un{n} \leq \un{m} \} = 0.
\end{equation}
Recall that $\psi_{\un{n}}(k_{\un{n}}) p_{\un{i}} = \psi_{\un{n}}(k_{\un{n}})$ when $i \in \supp \un{n}$.
Let $F := \supp \un{m}$.
Since the $p_{\un{i}}$ for $i \in [d]$ are commuting projections, we have $\psi_{\un{n}}(k_{\un{n}}) q_F = 0$ for all $\un{0} \neq \un{n} \leq \un{m}$, and therefore (\ref{eq:in}) implies that
\begin{equation}\label{eq:out}
\pi(a)q_F=0.
\end{equation}
Moreover, by \cite[Proposition 3.2]{DK18}, if $(\pi,t)$ is injective then (under the assumption of strong compact alignment) (\ref{eq:in}) implies that
\begin{equation}\label{eq:comp}
\phi_{\un{n}}(a) \in \K(X_{\un{n}}) \foral \un{0}\leq\un{n}\leq \un{1}_{[d]}.
\end{equation}
Conversely, if (\ref{eq:out}) and (\ref{eq:comp}) hold for some $F \subseteq [d]$ and $a \in A$, then Proposition \ref{P:prod cai} gives that
\begin{equation}\label{eq:expphi}
\pi(a) + \sum \{ (-1)^{|\un{n}|} \psi_{\un{n}}(\phi_{\un{n}}(a)) \mid \un{0} \neq \un{n} \leq \un{1}_F \} = \pi(a) q_F = 0,
\end{equation}
and so $\pi(a) \in B_{(\un{0}, \un{m}]}^{(\pi,t)}$ for any $\un{m} \geq \un{1}_F$.
The following proposition justifies the usage of $\{\I_F\}_{F \subseteq [d]}$.

\begin{proposition}\label{P:in IF} \cite[Proposition 3.4]{DK18}
Let $X$ be a strong compactly aligned product system with coefficients in a C*-algebra $A$.
Suppose that $(\pi,t)$ is an injective Nica-covariant representation of $X$ and fix $a\in A$.
If $\pi(a) \in B_{(\un{0}, \un{m}]}^{(\pi,t)}$ for some $\un{m}\in\bZ_+^d$, then $a \in \I_F$ for $F = \supp \un{m}$.
\end{proposition}

We define the \emph{ideal of the CNP-relations with respect to $(\pi,t)$} as follows:
\[ 
\mathfrak{J}_\I^{(\pi,t)}:=\sca{\pi(\I_F)q_F\mid F\subseteq[d]}\subseteq\ca(\pi,t). 
\]
Following \cite[Definition 2.8]{DK18}, we will say that $(\pi,t)$ is a \emph{CNP-representation of $X$} if it satisfies 
\[
\pi(a)q_F = \pi(a) + \sum \{ (-1)^{|\un{n}|} \psi_{\un{n}}(\phi_{\un{n}}(a)) \mid \un{0} \neq \un{n} \leq \un{1}_F \}=0  \foral a\in\I_F, F\subseteq[d].
\]
It follows that $\mathfrak{J}_\I^{(\pi,t)}=\{0\}$ if and only if $(\pi,t)$ is a CNP-representation.
We write $\N\O_X$ for the universal C*-algebra with respect to the CNP-representations, and refer to it as the \emph{Cuntz-Nica-Pimsner algebra of $X$}, i.e.,
\[
\N\O_X \equiv \N\T_X/\mathfrak{J}_\I^{(\ol{\pi}_X,\ol{t}_X)}.
\] 
We write $(\pi_X^\I,t_X^\I)$ for the \emph{universal CNP-representation of $X$}, i.e., $(\pi_X^\I,t_X^\I) = (Q\circ\ol{\pi}_X,Q\circ\ol{t}_X)$, where $Q \colon \N\T_X\to\N\O_X$ is the canonical quotient map.
Since $\N\O_X$ is an equivariant quotient of $\N\T_X$, the representation $(\pi_X^\I,t_X^\I)$ admits a gauge action.

In \cite{DK18} it is shown that $\N\O_X$ coincides with the Cuntz-Nica-Pimsner algebra of Sims and Yeend \cite{SY10}, and thus with the strong covariance algebra of Sehnem \cite{Seh18}.
In particular, $(\pi_X^\I,t_X^\I)$ is injective by \cite[Theorem 4.1]{SY10}, since $(\bZ^d,\bZ_+^d)$ satisfies \cite[Condition (3.5)]{SY10}.
Moreover, $\N\O_X$ is co-universal with respect to the injective Nica-covariant representations of $X$ that admit a gauge action \cite{SY10}.
The co-universal property of $\N\O_X$ has been verified in several works \cite{CLSV11, DK18, DKKLL20, Seh21} in more general contexts.

\begin{remark}\label{R:noiso}
Let $X$ and $Y$ be unitarily equivalent strong compactly aligned product systems with coefficients in C*-algebras $A$ and $B$, respectively.
By Remarks \ref{R:ntpiso}, \ref{R:JIunit} and \ref{R:prodsysisoprop}, we have that the bijection of Remark \ref{R:tpiso} preserves the CNP-representa\-tions.
Therefore $\N\O_X\cong\N\O_Y$ canonically.
\end{remark}

We finish the subsection with a proposition concerning canonical $*$-epimorphisms arising from injective Nica-covariant representations.
This trick was used implicitly in \cite{DK18, KKLL21b}.

\begin{proposition}\label{P:inj fp}
Let $X$ be a strong compactly aligned product system with coefficients in a C*-algebra $A$.
Let $(\tilde{\pi}, \tilde{t})$ and $(\pi,t)$ be Nica-covariant representations of $X$ such that $(\pi,t)$ is injective, and suppose there exists a canonical $*$-epimorphism $\Phi \colon \ca(\tilde{\pi}, \tilde{t}) \to \ca(\pi,t)$.
Then the following are equivalent:

\vspace{1pt}

\begin{enumerate}
\item $\ker\Phi \cap B_{[\un{0}, \infty]}^{(\tilde{\pi}, \tilde{t})} = \{0\}$;

\vspace{2pt}

\item $\ker\Phi \cap B_{[\un{0}, \un{m}]}^{(\tilde{\pi}, \tilde{t})} = \{0\}$ for all $\un{m} \in \bZ_+^d$;

\vspace{2pt}

\item $\ker\Phi \cap B_{[\un{0}, \un{1}_{[d]}]}^{(\tilde{\pi}, \tilde{t})} = \{0\}$;

\vspace{2pt}

\item $\ker\Phi \cap B_{[\un{0}, \un{1}_F]}^{(\tilde{\pi}, \tilde{t})} = \{0\}$ for all $F \subseteq [d]$.
\end{enumerate}
\end{proposition}

\begin{proof}
First note that $(\tilde{\pi},\tilde{t})$ is injective since $(\pi,t)$ is injective.
By definition we see that $B_{[\un{0},\infty]}^{(\tilde{\pi},\tilde{t})}$ is the inductive limit of the cores $B_{[\un{0}, \un{m}]}^{(\tilde{\pi},\tilde{t})}$ for $\un{m} \in \bZ_+^d$, and thus \cite[Lemma III.4.1]{KD96} yields the equivalence of item (i) with item (ii).

If item (ii) holds then it implies item (iii) by applying for $\un{m} = \un{1}_{[d]}$.
If item (iii) holds then it implies item (iv) since $B_{[\un{0}, \un{1}_F]}^{(\tilde{\pi}, \tilde{t})} \subseteq B_{[\un{0}, \un{1}_{[d]}]}^{(\tilde{\pi}, \tilde{t})}$ for all $F \subseteq [d]$.

It now suffices to show that item (iv) implies item (ii), so fix $\un{m}\in\bZ_+^d$. 
Without loss of generality, assume that $\un{m}$ has at least one entry greater than $1$ (the assumption deals with the case where $\un{m}$ has no entries greater than $1$). 
To reach contradiction, assume that $\ker\Phi\cap B_{[\un{0},\un{m}]}^{(\tilde{\pi},\tilde{t})}\neq\{0\}$. 
Take $0\neq f\in\ker\Phi\cap B_{[\un{0},\un{m}]}^{(\tilde{\pi},\tilde{t})}$, so that we may write
\[
f=\tilde{\pi}(a)+\sum\{\tilde{\psi}_\un{n}(k_\un{n})\mid \un{0}\neq\un{n}\leq\un{m}\},
\]
for some $a\in A$ and $k_\un{n}\in\K(X_\un{n})$ for all $\un{0}\neq\un{n}\leq\un{m}$.
Note that we can write $\tilde{\pi}(a)=\tilde{\psi}_\un{0}(k_\un{0})$ for $k_\un{0}:=\phi_\un{0}(a)$.
Without loss of generality, we may assume that $f$ is written irreducibly, so that we may choose a minimal $\un{0}\leq\un{r}\leq\un{m}$ such that $k_\un{r}\neq 0$, and $\tilde{\psi}_\un{r}(k_\un{r})\notin B_{(\un{r},\un{m}]}^{(\tilde{\pi},\tilde{t})}$.
The minimality of $\un{r}$ means that if we have $\un{0}\leq\un{n}\leq\un{m}$ such that $k_\un{n}\neq 0$ and $\un{n}\leq\un{r}$, then $\un{n}=\un{r}$.
If $\un{r}=\un{m}$, then $f=\tilde{\psi}_\un{m}(k_\un{m})$ and $\Phi(f)=\psi_\un{m}(k_\un{m})=0$.
Injectivity of $(\pi,t)$ then implies that $k_\un{m}=0$ and hence $f=0$, a contradiction.
So without loss of generality assume that $\un{r} < \un{m}$.
Fixing $\xi_\un{r},\eta_\un{r}\in X_\un{r}$, we have
\[
\tilde{t}_\un{r}(\xi_\un{r})^*f\tilde{t}_\un{r}(\eta_\un{r})=\tilde{\pi}(b)+\sum\{\tilde{\psi}_\un{n}(k_\un{n}')\mid \un{0}\neq\un{n}\leq\un{m}-\un{r}\},
\]
where $k_\un{n}'$ is a suitably defined element of $\K(X_\un{n})$ for all $\un{0}\neq\un{n}\leq\un{m}-\un{r}$ and $b := \sca{\xi_\un{r},k_\un{r}\eta_\un{r}}$, due to the minimality of $\un{r}$.
Applying $\Phi$, we obtain
\[
\pi(b)+\sum\{\psi_\un{n}(k_\un{n}')\mid \un{0}\neq\un{n}\leq\un{m}-\un{r}\}=0,
\]
which in turn implies that $b\in\pi^{-1}(B_{(\un{0},\un{m}-\un{r}]}^{(\pi,t)})$. 
Let $F := \supp(\un{m}-\un{r})$, which is non-empty since $\un{m}\neq\un{r}$. 
An application of (\ref{eq:out}) then yields $\pi(b)q_F=0$. 
Note also that $\phi_{\un{n}}(b)\in \K(X_\un{n})$ for all $\un{0}\leq\un{n}\leq\un{1}_{[d]}$ using (\ref{eq:comp}), which applies since $(\pi,t)$ is injective. 
Therefore we obtain
\[
\tilde{\pi}(b)\tilde{q}_F 
= 
\tilde{\pi}(b)+\sum\{(-1)^{|\un{n}|}\tilde{\psi}_\un{n}(\phi_\un{n}(b))\mid \un{0}\neq\un{n}\leq\un{1}_F\}
\in
B_{[\un{0},\un{1}_F]}^{(\tilde{\pi},\tilde{t})}.
\]
It then follows that
\begin{align*}
\Phi(\tilde{\pi}(b)\tilde{q}_F) & =\Phi(\tilde{\pi}(b)+\sum\{(-1)^{|\un{n}|}\tilde{\psi}_\un{n}(\phi_\un{n}(b))\mid \un{0}\neq\un{n}\leq\un{1}_F\}) \\
						     & =\pi(b)+\sum\{(-1)^{|\un{n}|}\psi_\un{n}(\phi_\un{n}(b))\mid \un{0}\neq\un{n}\leq\un{1}_F\} 
						     =\pi(b)q_F=0.
\end{align*}
Hence $\tilde{\pi}(b)\tilde{q}_F\in\ker\Phi\cap B_{[\un{0},\un{1}_F]}^{(\tilde{\pi},\tilde{t})}$ and so $\tilde{\pi}(b)\tilde{q}_F = 0$ by assumption. 
So we have
\[
\tilde{t}_\un{r}(\xi_\un{r})^*\tilde{\psi}_\un{r}(k_\un{r})\tilde{t}_\un{r}(\eta_\un{r})=\tilde{\pi}(b)\in B_{(\un{0},\un{1}_F]}^{(\tilde{\pi},\tilde{t})}
\]
for all $\xi_\un{r},\eta_\un{r}\in X_\un{r}$. 
Consequently, we have
\begin{align*}
\tilde{\psi}_\un{r}(\K(X_\un{r}))\tilde{\psi}_\un{r}(k_\un{r})\tilde{\psi}_\un{r}(\K(X_\un{r}))
\subseteq
[\tilde{t}_\un{r}(X_\un{r})B_{(\un{0},\un{1}_F]}^{(\tilde{\pi},\tilde{t})}\tilde{t}_\un{r}(X_\un{r})^*]
\subseteq 
[\tilde{t}_\un{r}(X_\un{r})B_{(\un{0},\un{m}-\un{r}]}^{(\tilde{\pi},\tilde{t})}\tilde{t}_\un{r}(X_\un{r})^*]
\subseteq 
B_{(\un{r},\un{m}]}^{(\tilde{\pi},\tilde{t})}.
\end{align*}
By using an approximate unit of $\tilde{\psi}_\un{r}(\K(X_\un{r}))$ and the fact that $B_{(\un{r},\un{m}]}^{(\tilde{\pi},\tilde{t})}$ is closed in $\ca(\tilde{\pi}, \tilde{t})$, we deduce that $\tilde{\psi}_\un{r}(k_\un{r})\in B_{(\un{r},\un{m}]}^{(\tilde{\pi},\tilde{t})}$, which contradicts the irreducibility assumption.
Therefore we obtain $\ker\Phi\cap B_{[\un{0},\un{m}]}^{(\tilde{\pi},\tilde{t})}=\{0\}$ for all $\un{m}\in\bZ_+^d$, as required.
\end{proof}

\subsection{The product system $IXI$}\label{Ss: IXI}

In this subsection we present a product system construction that will be useful in Section \ref{S:app}, and which extends the results of \cite[Section 9]{Kat07}.

Let $X$ be a C*-correspondence over a C*-algebra $A$ and let $I\subseteq A$ be an ideal.
Consider $IXI\subseteq XI$ and recall that $XI$ is closed.
An application of the Hewitt-Cohen Factorisation Theorem gives that $IXI$ is a closed linear subspace of $XI$, i.e., $IXI=[IXI]$.
It is clear that $IXI$ is also closed under the right action of $A$, and hence $IXI$ carries the structure of a right Hilbert $A$-module.
Next, we define a left action $\phi_{IXI}$ of $A$ on $IXI$ by
\[
\phi_{IXI}\colon A\to\L(IXI); \phi_{IXI}(a)=\phi_X(a)|_{IXI} \foral a \in A.
\]
Hence $IXI$ inherits the structure of a C*-correspondence over $A$ from $X$.
By restricting $\phi_{IXI}$ to $I$, we may view $IXI$ as a C*-correspondence over $I$.

As was the case with $\K(XI)$, we have a natural embedding of $\K(IXI)$ in $\K(X)$.
The proof of the following result follows an almost identical trajectory to that of \cite[Lemma 2.6 (1)]{FMR03}, and it is left to the reader.

\begin{lemma}\label{L:comp IXI embed}
Let $X$ be a C*-correspondence over a C*-algebra $A$ and let $I\subseteq A$ be an ideal.
Then there exists an embedding $\iota\colon \K(IXI)\to\K(X)$ such that
\[
\iota(\Theta_{\xi, \eta}^{IXI})=\Theta_{\xi, \eta}^X\foral \xi, \eta \in IXI.
\]
Consequently $\K(IXI)$ is $*$-isomorphic to $\ol{\spn}\{\Theta_{\xi, \eta}^X\mid \xi, \eta \in IXI\}$ via the map $\iota$, with
\[
\iota^{-1}(k)=k|_{IXI}\foral k\in\ol{\spn}\{\Theta_{\xi, \eta}^X\mid \xi, \eta \in IXI\}.
\]
\end{lemma}

Now let $X$ be a product system over $P$.
Recall that if $I$ is positively invariant for $X$, then $IX_pI = IX_p$ for all $p \in P$ by Lemma \ref{L:IX to XI}.
In this case we obtain a product system $IXI$, with a construction that is compatible with the product system structure of $X$.

\begin{proposition}\label{P:IXIprodsys}
Let $P$ be a unital subsemigroup of a discrete group $G$.
Let $X$ be a product system over $P$ with coefficients in a C*-algebra $A$ and let $I\subseteq A$ be an ideal that is positively invariant for $X$.
Then $IXI:=\{IX_pI\}_{p\in P}$ inherits from $X$ a canonical structure as a product system over $P$ with coefficients in $I$, identifying each $I X_p I$ as a sub-C*-correspondence of $X_p$.
\end{proposition}

\begin{proof}
We will denote the multiplication maps of $X$ by $\{u_{p,q}^X\}_{p,q\in P}$.
The discussion preceding Lemma \ref{L:comp IXI embed} gives that $IX_pI$ is a C*-correspondence over $I$ for all $p\in P$.
For $p,q \in P$, let the unitary map $\iota_{p,q}$ be induced by the canonical identifications
\[
IX_pI \otimes_I IX_qI \cong IX_pI \otimes_A IX_qI \cong IX_p \otimes_A X_qI,
\]
where we use Lemma \ref{L:IX to XI} in the second identification and write (in an abuse of notation) $IX_pI \otimes_A IX_qI$ for the closed linear space generated by $IX_pI \odot_A IX_qI$ in $X_p \otimes_A X_q$.
We then define the multiplication map $u_{p,q}^{IXI}$ for $IXI$ by $u_{p,q}^{IXI} := u_{p,q}^X \circ \iota_{p,q}$.
It is clear that $u_{p,q}^{IXI}$ is a linear isometry and an $I$-bimodule map onto $u_{p,q}^X(IX_p \otimes_A X_qI) = I X_{pq} I$.
It follows that $IXI$ together with the multiplication maps $\{u_{p,q}^{IXI}\}_{p,q\in P}$ becomes a product system, using the product system axioms of $X$.
\end{proof}

\begin{proposition}\label{P:IXIiota}
Let $P$ be a unital subsemigroup of a discrete group $G$.
Let $X$ be a product system over $P$ with coefficients in a C*-algebra $A$ and let $I\subseteq A$ be an ideal that is positively invariant for $X$.
Let $\{\iota_p^{pq}\}_{p,q\in P}$ denote the connecting $*$-homomorphisms of $X$ and let $\{j_p^{pq}\}_{p,q\in P}$ denote the connecting $*$-homomorphisms of $IXI$.
Then
\[
j_p^{pq}(k_p|_{IX_pI})=\iota_p^{pq}(k_p)|_{IX_{pq}I}
\foral  
p,q\in P\; \text{and} \; k_p\in\ol{\spn}\{\Theta_{\xi_p,\eta_p}^{X_p}\mid \xi_p,\eta_p\in IX_pI\}.
\]
\end{proposition}

\begin{proof}
Without loss of generality we may assume that $p \neq e$ and $q\neq e$, as the claim is straightforward otherwise.
By Lemma \ref{L:comp IXI embed} we have $k_p|_{IX_pI}\in\K(IX_pI)$.
Recalling that $IX_{pq}I$ is a closed linear subspace of $X_{pq}$, it suffices to show that
\[
j_p^{pq}(k_p|_{IX_pI})u_{p,q}^{IXI}(\xi_p\otimes\xi_q)=\iota_p^{pq}(k_p)u_{p,q}^{IXI}(\xi_p\otimes\xi_q),
\foral
\xi_p\in IX_pI, \xi_q\in IX_qI.
\]
Indeed, fixing $\xi_p\in IX_pI$ and $\xi_q\in IX_qI$, we have
\begin{align*}
j_p^{pq}(k_p|_{IX_pI})u_{p,q}^{IXI}(\xi_p\otimes\xi_q) 
& =
u_{p,q}^{IXI}((k_p\xi_p)\otimes\xi_q) 
=
u_{p,q}^X((k_p\xi_p)\otimes\xi_q) \\
& =
\iota_p^{pq}(k_p)u_{p,q}^X(\xi_p\otimes\xi_q)
=
\iota_p^{pq}(k_p)u_{p,q}^{IXI}(\xi_p\otimes\xi_q),
\end{align*}
as required.
\end{proof}

Next we show that the properties of compact alignment and strong compact alignment of $IXI$ are inherited from $X$.

\begin{proposition}\label{P:IXIca}
Let $P$ be a unital right LCM subsemigroup of a discrete group $G$.
Let $X$ be a compactly aligned product system over $P$ with coefficients in a C*-algebra $A$ and let $I\subseteq A$ be an ideal that is positively invariant for $X$.
Then the product system $IXI$ is compactly aligned.
\end{proposition}

\begin{proof}
Let $\{\iota_p^{pq}\}_{p,q\in P}$ (resp. $\{j_p^{pq}\}_{p,q\in P}$) denote the connecting $*$-homomorphisms of $X$ (resp. $IXI$).
It suffices to show that the compact alignment condition holds for rank-one operators.
To this end, let $p,q\in P\setminus\{e\}$ be such that $pP\cap qP=wP$ for some $w\in P$, and fix $\xi_p,\eta_p\in IX_pI$ and $\xi_q,\eta_q\in IX_qI$.
We must show that
\[
j_p^w(\Theta_{\xi_p,\eta_p}^{IX_pI})j_q^w(\Theta_{\xi_q,\eta_q}^{IX_qI})\in\K(IX_wI).
\]
Recall that 
\[
\Theta_{\xi_p,\eta_p}^{IX_pI}=\Theta_{\xi_p,\eta_p}^{X_p}|_{IX_pI}
\qand
\Theta_{\xi_q,\eta_q}^{IX_qI}=\Theta_{\xi_q,\eta_q}^{X_q}|_{IX_qI},
\]
by Lemma \ref{L:comp IXI embed}.
An application of Proposition \ref{P:IXIiota} then gives that
\begin{equation}\label{Eq:IXIca 1}
j_p^w(\Theta_{\xi_p,\eta_p}^{IX_pI})j_q^w(\Theta_{\xi_q,\eta_q}^{IX_qI})=\iota_p^w(\Theta_{\xi_p,\eta_p}^{X_p})|_{IX_wI}\iota_q^w(\Theta_{\xi_q,\eta_q}^{X_q})|_{IX_wI}=[\iota_p^w(\Theta_{\xi_p,\eta_p}^{X_p})\iota_q^w(\Theta_{\xi_q,\eta_q}^{X_q})]|_{IX_wI}.
\end{equation}
By Lemma \ref{L:comp IXI embed}, it suffices to show that
\[
\iota_p^w(\Theta_{\xi_p,\eta_p}^{X_p})\iota_q^w(\Theta_{\xi_q,\eta_q}^{X_q})\in\ol{\spn}\{\Theta_{\xi_w,\eta_w}^{X_w}\mid \xi_w,\eta_w\in IX_wI\}.
\]
Compact alignment of $X$ gives that
\[
\iota_p^w(\Theta_{\xi_p,\eta_p}^{X_p})\iota_q^w(\Theta_{\xi_q,\eta_q}^{X_q})\in\K(X_w).
\]
Next, let $(u_\la)_{\la\in\La}$ denote an approximate unit of $I$.
For each $\la\in\La$, we have
\[
\phi_w(u_\la)\iota_p^w(\Theta_{\xi_p,\eta_p}^{X_p})=\iota_p^w(\phi_p(u_\la)\Theta_{\xi_p,\eta_p}^{X_p})=\iota_p^w(\Theta_{\phi_p(u_\la)\xi_p,\eta_p}^{X_p}).
\]
Using this and the fact that $\nor{\cdot}\text{-}\lim_\la \phi_p(u_\la)\xi_p=\xi_p$ (as $\xi_p \in IX_p I$), we obtain
\[
\nor{\cdot}\text{-}\lim_\la\phi_w(u_\la)\iota_p^w(\Theta_{\xi_p,\eta_p}^{X_p})=\iota_p^w(\Theta_{\xi_p,\eta_p}^{X_p}).
\]
By analogous reasoning, we have
\[
\nor{\cdot}\text{-}\lim_\la\iota_q^w(\Theta_{\xi_q,\eta_q}^{X_q})\phi_w(u_\la)=\iota_q^w(\Theta_{\xi_q,\eta_q}^{X_q}).
\]
Therefore, we obtain
\[
\iota_p^w(\Theta_{\xi_p,\eta_p}^{X_p})\iota_q^w(\Theta_{\xi_q,\eta_q}^{X_q})
=
\nor{\cdot}\text{-}\lim_\la \phi_w(u_\la)\iota_p^w(\Theta_{\xi_p,\eta_p}^{X_p})\iota_q^w(\Theta_{\xi_q,\eta_q}^{X_q})\phi_w(u_\la).
\]
Consequently, we have that $\iota_p^w(\Theta_{\xi_p,\eta_p}^{X_p})\iota_q^w(\Theta_{\xi_q,\eta_q}^{X_q})$ is expressible as the norm-limit of a net that is contained in $\phi_w(I)\K(X_w)\phi_w(I)$.
We also have 
\[
\phi_w(I)\K(X_w)\phi_w(I)\subseteq\ol{\spn}\{\Theta_{\xi_w,\eta_w}^{X_w}\mid \xi_w,\eta_w\in IX_wI\}
\]
by Corollary \ref{C:compconj}.
Thus 
\begin{equation}\label{Eq:IXIca 2}
\iota_p^w(\Theta_{\xi_p,\eta_p}^{X_p})\iota_q^w(\Theta_{\xi_q,\eta_q}^{X_q})\in\ol{\spn}\{\Theta_{\xi_w,\eta_w}^{X_w}\mid \xi_w,\eta_w\in IX_wI\},
\end{equation}
as required.
\end{proof}

\begin{proposition}\label{P:IXIsca}
Let $X$ be a strong compactly aligned product system with coefficients in a C*-algebra $A$ and let $I\subseteq A$ be an ideal that is positively invariant for $X$.
Then the product system $IXI$ is strong compactly aligned.
\end{proposition}

\begin{proof}
Let $\{\iota_\un{n}^{\un{n}+\un{m}}\}_{\un{n},\un{m}\in \bZ_+^d}$ (resp. $\{j_\un{n}^{\un{n}+\un{m}}\}_{\un{n},\un{m}\in \bZ_+^d}$) denote the connecting $*$-homomorphisms of $X$ (resp. $IXI$).
Proposition \ref{P:IXIca} asserts that $IXI$ is compactly aligned, so it remains to check that $IXI$ satisfies
\[
j_\un{n}^{\un{n}+\un{i}}(\K(IX_\un{n}I))\subseteq\K(IX_{\un{n}+\un{i}}I)
\textup{ whenever $\un{n}\perp\un{i}$, where $i\in[d],\un{n}\in\bZ_+^d\setminus\{\un{0}\}$}.
\]
It suffices to show that this holds for rank-one operators.
To this end, fix $\xi_\un{n},\eta_\un{n}\in IX_\un{n}I$.
An application of Proposition \ref{P:IXIiota} then yields
\[
j_\un{n}^{\un{n}+\un{i}}(\Theta_{\xi_\un{n},\eta_\un{n}}^{IX_\un{n}I})
=
j_\un{n}^{\un{n}+\un{i}}(\Theta_{\xi_\un{n},\eta_\un{n}}^{X_\un{n}} |_{IX_{\un{n}}I} )
=
\iota_\un{n}^{\un{n}+\un{i}}(\Theta_{\xi_\un{n},\eta_\un{n}}^{X_\un{n}})|_{IX_{\un{n}+\un{i}}I}.
\]
By strong compact alignment of $X$, we have $\iota_\un{n}^{\un{n}+\un{i}}(\Theta_{\xi_\un{n},\eta_\un{n}}^{X_\un{n}}) \in \K( X_{\un{n} + \un{i}})$.
By using an approximate unit $(u_\la)_{\la \in \La}$ of $I$ and Corollary \ref{C:compconj}, we then obtain
\begin{align*}
\iota_\un{n}^{\un{n}+\un{i}}(\Theta_{\xi_\un{n},\eta_\un{n}}^{X_\un{n}})
& =
\nor{\cdot}\text{-}\lim_\la \phi_{\un{n} + \un{i}}(u_\la) \iota_\un{n}^{\un{n}+\un{i}}(\Theta_{\xi_\un{n},\eta_\un{n}}^{X_\un{n}}) \phi_{\un{n} + \un{i}}(u_\la) \\
& \in
[\phi_{\un{n} + \un{i}}(I) \K( X_{\un{n} + \un{i}}) \phi_{\un{n} + \un{i}}(I)] \\
& \subseteq
\ol{\spn}\{\Theta_{\xi_{\un{n}+\un{i}},\eta_{\un{n}+\un{i}}}^{X_{\un{n}+\un{i}}}\mid \xi_{\un{n}+\un{i}},\eta_{\un{n}+\un{i}}\in IX_{\un{n}+\un{i}}I\}.
\end{align*}
An application of Lemma \ref{L:comp IXI embed} finishes the proof.
\end{proof}

Next we study the representations of $IXI$.
Given a representation $(\pi,t)$ of $X$,
we write $(\pi|_I,t|_{IXI})$ for the family $\{(\pi|_I, t_p|_{IX_pI})\}_{p\in P}$.
It is routine to check that $(\pi|_I,t|_{IXI})$ is a representation of $IXI$.
For each $p\in P$, let 
\[
\tilde{\psi}_p \colon \K(IX_pI)\to\ca(\pi,t)
\]
be the $*$-homomorphism induced by $(\pi|_I,t_p|_{IX_pI})$.
Then for all $\xi_p,\eta_p\in IX_pI$, we have
\[
\tilde{\psi}_p(\Theta_{\xi_p,\eta_p}^{X_p}|_{IX_pI})
=
\tilde{\psi}_p(\Theta_{\xi_p,\eta_p}^{IX_pI})
=
t_p|_{IX_pI}(\xi_p)t_p|_{IX_pI}(\eta_p)^*
=
t_p(\xi_p)t_p(\eta_p)^*
=
\psi_p(\Theta_{\xi_p,\eta_p}^{X_p}),
\]
from which it follows that
\begin{equation}\label{Eq:IXIcomprest}
\tilde{\psi}_p(k_p|_{IX_pI})=\psi_p(k_p)\foral k_p\in\ol{\spn}\{\Theta_{\xi_p,\eta_p}^{X_p}\mid \xi_p,\eta_p\in IX_pI\}.
\end{equation}
When $P$ is a right LCM semigroup and $X$ is compactly aligned (and thus $IXI$ is compactly aligned by Proposition \ref{P:IXIca}), Nica-covariance of $(\pi,t)$ is inherited by $(\pi|_I,t|_{IXI})$.

\begin{proposition}\label{P:IXIrepnNC}
Let $P$ be a unital right LCM subsemigroup of a discrete group $G$.
Let $X$ be a compactly aligned product system over $P$ with coefficients in a C*-algebra $A$ and let $I\subseteq A$ be an ideal that is positively invariant for $X$.
Let $(\pi,t)$ be a Nica-covariant representation of $X$.
Then $(\pi|_I,t|_{IXI})$ is a Nica-covariant representation of $IXI$, and
\begin{equation}\label{Eq:restrepran}
\begin{aligned}
\ca(\pi|_I,t|_{IXI})
& = \ol{\spn}\{\pi(I)t_p(X_p)\pi(I)t_q(X_q)^*\pi(I)\mid p,q\in P\} \\
& = \ol{\spn}\{\pi(I)t_p(X_p)t_q(X_q)^*\pi(I)\mid p,q\in P\}.
\end{aligned}
\end{equation}
\end{proposition}

\begin{proof}
Let $\{\iota_p^{pq}\}_{p,q\in P}$ (resp. $\{j_p^{pq}\}_{p,q\in P}$) denote the connecting $*$-homomorphisms of $X$ (resp. $IXI$).
For each $p\in P$, let 
\[
\tilde{\psi}_p \colon \K(IX_pI)\to\ca(\pi,t)
\]
be the $*$-homomorphism induced by $(\pi|_I,t_p|_{IX_pI})$.
Now fix $p,q\in P\setminus\{e\}, k_p\in\K(IX_pI)$ and $k_q\in\K(IX_qI)$.
We must show that
\[
\tilde{\psi}_p(k_p)\tilde{\psi}_q(k_q)=\begin{cases} \tilde{\psi}_w(j_p^w(k_p)j_q^w(k_q)) & \text{if} \; pP\cap qP=wP \; \text{for some} \; w\in P, \\ 0 & \text{otherwise}.\end{cases}
\]
It suffices to show this for $k_p=\Theta_{\xi_p,\eta_p}^{IX_pI}$ and $k_q=\Theta_{\xi_q,\eta_q}^{IX_qI}$ for some $\xi_p,\eta_p\in IX_pI$ and $\xi_q,\eta_q\in IX_qI$.
By (\ref{Eq:IXIcomprest}) and Nica-covariance of $(\pi,t)$, we have
\begin{align*}
\tilde{\psi}_p(\Theta_{\xi_p,\eta_p}^{IX_pI})\tilde{\psi}_q(\Theta_{\xi_q,\eta_q}^{IX_qI}) & =\psi_p(\Theta_{\xi_p,\eta_p}^{X_p})\psi_q(\Theta_{\xi_q,\eta_q}^{X_q}) \\
																		     & =\begin{cases} \psi_w(\iota_p^w(\Theta_{\xi_p,\eta_p}^{X_p})\iota_q^w(\Theta_{\xi_q,\eta_q}^{X_q})) & \text{if} \; pP\cap qP=wP \; \text{for some} \; w\in P, \\ 0 & \text{otherwise}. \end{cases}
\end{align*}
Suppose $pP\cap qP=wP$ for some $w\in P$.
Since the assumptions of Proposition \ref{P:IXIca} are satisfied, by (\ref{Eq:IXIca 1}) we have
\[
[\iota_p^w(\Theta_{\xi_p,\eta_p}^{X_p})\iota_q^w(\Theta_{\xi_q,\eta_q}^{X_q})]|_{IX_wI}=j_p^w(\Theta_{\xi_p,\eta_p}^{IX_pI})j_q^w(\Theta_{\xi_q,\eta_q}^{IX_qI}),
\]
and by (\ref{Eq:IXIca 2}) we have
\[
\iota_p^w(\Theta_{\xi_p,\eta_p}^{X_p})\iota_q^w(\Theta_{\xi_q,\eta_q}^{X_q})\in\ol{\spn}\{\Theta_{\xi_w,\eta_w}^{X_w}\mid \xi_w,\eta_w\in IX_wI\}.
\]
Another application of (\ref{Eq:IXIcomprest}) then yields
\[
\tilde{\psi}_p(\Theta_{\xi_p,\eta_p}^{IX_pI})\tilde{\psi}_q(\Theta_{\xi_q,\eta_q}^{IX_qI})=\begin{cases} \tilde{\psi}_w(j_p^w(\Theta_{\xi_p,\eta_p}^{IX_pI})j_q^w(\Theta_{\xi_q,\eta_q}^{IX_qI})) & \text{if} \; pP\cap qP=wP \; \text{for some} \; w\in P, \\ 0 & \text{otherwise}, \end{cases}
\]
as required.
The second claim is a direct consequence of Nica-covariance.
\end{proof}

The structure of $\ca(\pi|_I,t|_{IXI})$ is linked with that of $\sca{\pi(I)}\subseteq\ca(\pi,t)$.

\begin{proposition}\label{P:IXIherfull}
Let $P$ be a unital right LCM subsemigroup of a discrete group $G$.
Let $X$ be a compactly aligned product system over $P$ with coefficients in a C*-algebra $A$ and let $I\subseteq A$ be an ideal that is positively invariant for $X$. 
Let $(\pi,t)$ be a Nica-covariant representation of $X$.
Then the following hold:
\begin{enumerate}
\item $\sca{\pi(I)}=\ol{\spn}\{t_p(X_p)\pi(I)t_q(X_q)^*\mid p,q\in P\}$;
\item $\ca(\pi|_I,t|_{IXI})=[\pi(I)\sca{\pi(I)}\pi(I)]$, and thus $\ca(\pi|_I,t|_{IXI})$ is a hereditary, full C*-subalgebra of $\sca{\pi(I)}$.
\end{enumerate}
\end{proposition}

\begin{proof}
Item (i) follows by Nica-covariance of $(\pi,t)$.
The first part of item (ii) follows from (\ref{Eq:restrepran}).
The second part of item (ii) is a standard C*-result.
\end{proof}

Suppose now that $P=\bZ_+^d$ and $X$ is strong compactly aligned.
An application of the Gauge-Invariant Uniqueness Theorem permits the identification of $\N\T_{IXI}$ and $\N\O_{IXI}$ with C*-subalgebras of $\N\T_X$ and $\N\O_X$, respectively.
We note that the Toeplitz-Nica-Pimsner case extends to right LCM semigroups using \cite[Theorem 6.4]{KKLL21b} (but we will not require it here).

\begin{proposition}\label{P:NTIXIembed}
Let $X$ be a strong compactly aligned product system with coefficients in a C*-algebra $A$ and let $I\subseteq A$ be an ideal that is positively invariant for $X$.
Then the inclusion $IXI \hookrightarrow X$ induces a canonical embedding
\[
\N\T_{IXI}\cong\ca(\ol{\pi}_X|_I,\ol{t}_X|_{IXI})\subseteq\N\T_X,
\]
and thus $\N\T_{IXI}$ is a hereditary, full C*-subalgebra of $\sca{\ol{\pi}_X(I)}$.
\end{proposition}

\begin{proof}
Note that $IXI$ is strong compactly aligned by Proposition \ref{P:IXIsca}, and $(\ol{\pi}_X|_I,\ol{t}_X|_{IXI})$ is a Nica-covariant representation of $IXI$ by Proposition \ref{P:IXIrepnNC}.
Injectivity of $(\ol{\pi}_X|_I,\ol{t}_X|_{IXI})$ is inherited from $(\ol{\pi}_X,\ol{t}_X)$, and $(\ol{\pi}_X|_I,\ol{t}_X|_{IXI})$ admits a gauge action by restriction.
Hence it suffices to show that $\ol{\pi}_X|_I\times\ol{t}_X|_{IXI}\colon \N\T_{IXI}\to\ca(\ol{\pi}_X|_I,\ol{t}_X|_{IXI})$ is injective on the fixed point algebra; equivalently on the $[\un{0}, \un{1}_{[d]}]$-core by Proposition \ref{P:inj fp}.

To this end, first note that we may take $(\ol{\pi}_X,\ol{t}_X)$ to be the Fock representation without loss of generality. 
Let $\psi_\un{n} \colon \K(IX_\un{n}I)\to\N\T_X$ be the $*$-homomorphism induced by $(\ol{\pi}_X|_I,\ol{t}_{X,\un{n}}|_{IX_{\un{n}}I})$ for all $\un{n}\in\bZ_+^d$.
Take $f\in\ker\ol{\pi}_X|_I\times\ol{t}_X|_{IXI}\cap B_{[\un{0},\un{1}_{[d]}]}^{(\ol{\pi}_{IXI},\ol{t}_{IXI})}$, so that we may write
\[
f=\ol{\pi}_{IXI}(a)+\sum\{\ol{\psi}_{IXI,\un{n}}(k_\un{n})\mid \un{0}\neq\un{n}\leq\un{1}_{[d]}\},
\]
for some $a\in I$ and $k_\un{n}\in\K(IX_\un{n}I)$ for all $\un{0}\neq\un{n}\leq\un{1}_{[d]}$.
Recall that we write $\ol{\pi}_{IXI}(a)=\ol{\psi}_{IXI,\un{0}}(k_\un{0})$ for $k_\un{0}:=\phi_{IX_\un{0}I}(a)$, and in turn $\ol{\pi}_X|_I(a)=\psi_\un{0}(k_\un{0})$.
Towards contradiction, suppose that $f\neq 0$.
Then we may choose $\un{0}\leq\un{r} \leq \un{1}_{[d]}$ minimal such that $k_{\un{r}}\neq 0$.
Let $P_{\un{r}} \colon \F X \to X_{\un{r}}$ be the canonical projection.
Then
\[
\iota(k_{\un{r}}) = P_{\un{r}} \left(\sum\{\psi_\un{n}(k_\un{n})\mid \un{0}\leq\un{n}\leq\un{1}_{[d]}\} \right) P_{\un{r}} = 0,
\]
where $\iota \colon \K(I X_{\un{r}} I) \to \K(X_{\un{r}})$ is the embedding guaranteed by Lemma \ref{L:comp IXI embed}.
Hence $k_{\un{r}} = 0$, which is a contradiction.
We conclude that $\ol{\pi}_X|_I\times\ol{t}_X|_{IXI}$ is injective on the $[\un{0},\un{1}_{[d]}]$-core, as required.
Proposition \ref{P:IXIherfull} then completes the proof.
\end{proof}

To establish the corresponding result for $\N\O_{IXI}$, we first generalise \cite[Proposition 9.2]{Kat07}.

\begin{lemma}\label{L:IXIJandI}
Let $X$ be a strong compactly aligned product system with coefficients in a C*-algebra $A$ and let $I\subseteq A$ be an ideal that is positively invariant for $X$.
Then, for all $F\subseteq[d]$, we have:
\begin{enumerate}
\item $\J_F(IXI)=I\cap\J_F(X)$;
\item $\I_F(IXI)=I\cap\I_F(X)$.
\end{enumerate}
\end{lemma}

\begin{proof}
(i) The claim holds trivially for $F=\mt$, so assume that $F\neq\mt$.
Let $a\in\J_F(IXI) \subseteq I$.
It suffices to show that $a\in\J_F(X)$.
To this end, fix $i\in[d]$.
By definition, we have $a\in(\phi_{I X_\un{i} I})^{-1}(\K(IX_\un{i}I))$.
Since $I$ is positively invariant for $X_\un{i}$, we may apply Lemma \ref{L:IX to XI} to deduce that $IX_\un{i}I=IX_\un{i}$, and an application of \cite[Lemma 9.1]{Kat07} then yields
\[
(\phi_{I X_\un{i} I})^{-1}(\K(IX_\un{i}I))=I\cap\phi_\un{i}^{-1}(\K(X_\un{i})).
\]
Hence $a\in\phi_\un{i}^{-1}(\K(X_\un{i}))$ for all $i\in[d]$, as required.
Moreover, since $\ker \phi_{\un{i}} \cap I \subseteq \ker \phi_{I X_\un{i} I}$ for every $i \in [d]$, we have $(\bigcap_{i\in F}\ker\phi_{I X_\un{i} I})^\perp \subseteq (\bigcap_{i\in F}\ker\phi_\un{i})^\perp \cap I$.
Hence $a\in(\bigcap_{i\in F}\ker\phi_\un{i})^\perp$ and thus $a\in\J_F(X)$, as required.

For the reverse inclusion, let $a\in I\cap\J_F(X)$.
For each $i\in[d]$, we have
\[
a\in I\cap\phi_\un{i}^{-1}(\K(X_\un{i}))=(\phi_{I X_\un{i} I})^{-1}(\K(IX_\un{i}I))
\]
by Lemma \ref{L:IX to XI} and \cite[Lemma 9.1]{Kat07}.
Therefore, it suffices to show that $a\in(\bigcap_{i\in F}\ker\phi_{I X_\un{i} I})^\perp$.
To this end, fix $b\in\bigcap_{i\in F}\ker\phi_{I X_\un{i} I}$.
Another application of \cite[Lemma 9.1]{Kat07} gives that
\[
b\in I\cap(\bigcap_{i\in F}\ker\phi_\un{i}).
\]
Since $a\in(\bigcap_{i\in F}\ker\phi_\un{i})^\perp$, we therefore have $ab=0$, as required.

\smallskip

\noindent
(ii) Once again, we may assume without loss of generality that $F\neq\mt$.
Take $a\in\I_F(IXI)\subseteq I$ and fix $\un{n}\perp F$.
We have
\begin{align*}
\sca{X_\un{n},aX_\un{n}} 
& \subseteq
[\sca{X_\un{n},IaIX_\un{n}}] 
\subseteq
[\sca{IX_\un{n},aIX_\un{n}}] \\
& \subseteq
[\sca{IX_\un{n}I,\phi_{IX_\un{n}I}(a)(IX_\un{n}I)}] 
\subseteq
\J_F(IXI) \subseteq \J_F(X),
\end{align*}
by using Lemma \ref{L:IX to XI} and item (i).
Thus $a\in I\cap\I_F(X)$, as required.

For the reverse inclusion, fix $a\in I\cap\I_F(X)$ and $\un{n}\perp F$. 
We have
\[
\sca{IX_\un{n}I,\phi_{IX_\un{n}I}(a)(IX_\un{n}I)}\subseteq\sca{X_\un{n},aX_\un{n}}\subseteq I\cap\J_F(X)=\J_F(IXI),
\]
using positive invariance of $I$, the fact that $a\in\I_F(X)$ and item (i).
Thus $a\in\I_F(IXI)$, completing the proof.
\end{proof}

\begin{proposition}\label{P:NOIXIembed}
Let $X$ be a strong compactly aligned product system with coefficients in a C*-algebra $A$ and let $I\subseteq A$ be an ideal that is positively invariant for $X$.
Then the inclusion $IXI \hookrightarrow X$ induces a canonical embedding
\[
\N\O_{IXI}\cong\ca(\pi_X^\I|_I,t_X^\I|_{IXI})\subseteq\N\O_X,
\]
and thus $\N\O_{IXI}$ is a hereditary, full C*-subalgebra of $\sca{\pi_X^\I(I)}$.
\end{proposition}

\begin{proof}
Being the restriction of an injective Nica-covariant representation of $X$ that admits a gauge action, $(\pi_X^\I|_I,t_X^\I|_{IXI})$ is an injective Nica-covariant representation of $IXI$ that admits a gauge action.
Thus it suffices to show that $(\pi_X^\I|_I,t_X^\I|_{IXI})$ is a CNP-representation of $IXI$ by \cite[Theorem 4.2]{DK18}.

To this end, fix $\mt\neq F\subseteq[d]$ and $a\in\I_F(IXI)$.
For each $\un{n}\in\bZ_+^d$, let $\tilde{\psi}_\un{n} \colon \K(IX_\un{n}I)\to\N\O_X$ be the $*$-homomorphism induced by $(\pi_X^\I|_I,t_{X,\un{n}}^\I|_{IX_\un{n}I})$.
We must show that
\[
\pi_X^\I|_I(a)+\sum\{(-1)^{|\un{n}|}\tilde{\psi}_\un{n}(\phi_{IX_\un{n}I}(a))\mid \un{0}\neq\un{n}\leq\un{1}_F\}=0.
\]
Fix $\un{0}\neq\un{n}\leq\un{1}_F$.
We claim that 
\[
\phi_\un{n}(a)\in\ol{\spn}\{\Theta_{\xi_\un{n},\eta_\un{n}}^{X_\un{n}}\mid \xi_\un{n},\eta_\un{n}\in IX_\un{n}I\}.
\]
To see this, first note that $a\in I\cap\I_F(X)$ by item (ii) of Lemma \ref{L:IXIJandI}, and so $\phi_\un{n}(a)\in\K(X_\un{n})$.
Let $(u_\la)_{\la\in\La}$ be an approximate unit of $I$.
Then an application of Corollary \ref{C:compconj} yields
\[
\phi_\un{n}(u_\la)\phi_\un{n}(a)\phi_\un{n}(u_\la)\in\ol{\spn}\{\Theta_{\xi_\un{n},\eta_\un{n}}^{X_\un{n}}\mid \xi_\un{n},\eta_\un{n}\in IX_\un{n}I\}\foral \la\in\La.
\]
Since $a\in I$, we also have
\[
\nor{\cdot}\text{-}\lim_\la \phi_\un{n}(u_\la)\phi_\un{n}(a)\phi_\un{n}(u_\la)=\phi_\un{n}(a),
\]
and consequently 
\[
\phi_\un{n}(a)\in\ol{\spn}\{\Theta_{\xi_\un{n},\eta_\un{n}}^{X_\un{n}}\mid \xi_\un{n},\eta_\un{n}\in IX_\un{n}I\},
\]
as claimed.
Note also that $\phi_\un{n}(a)|_{IX_\un{n}I}=\phi_{IX_\un{n}I}(a)$ by definition.
An application of (\ref{Eq:IXIcomprest}) then gives that
\[
\tilde{\psi}_\un{n}(\phi_{IX_\un{n}I}(a))=\psi_{X,\un{n}}^\I(\phi_\un{n}(a))\foral\un{0}\neq\un{n}\leq\un{1}_F.
\]
From this we obtain
\begin{align*}
\pi_X^\I|_I(a)+\sum & \{(-1)^{|\un{n}|}\tilde{\psi}_\un{n}(\phi_{IX_\un{n}I}(a))\mid \un{0}\neq\un{n}\leq\un{1}_F\} = \\
& =
\pi_X^\I(a)+\sum\{(-1)^{|\un{n}|}\psi_{X,\un{n}}^\I(\phi_\un{n}(a))\mid \un{0}\neq\un{n}\leq\un{1}_F\}=0,
\end{align*}
where the final equality follows from the fact that $a\in\I_F(X)$, together with the fact that $(\pi_X^\I,t_X^\I)$ is a CNP-representation of $X$.
Thus $(\pi_X^\I|_I,t_X^\I|_{IXI})$ is a CNP-representation of $IXI$, as required.
Proposition \ref{P:IXIherfull} then completes the proof.
\end{proof}

\section{Relative $2^d$-tuples and relative Cuntz-Nica-Pimsner algebras}\label{S:relcnpalg}

Parametrising the gauge-invariant ideals of the Toeplitz-Pimsner algebra of a C*-correspon\-dence is facilitated using relative Cuntz-Pimsner algebras.
Hence our first goal is to suitably adapt the relative Cuntz-Pimsner algebra construction for a strong compactly aligned product system $X$.
We do this by mimicking the construction of $\N\O_X$, and exploiting strong compact alignment in order to define our relative Cuntz-Nica-Pimsner algebras in terms of simple algebraic relations induced by finitely many subsets of the coefficient algebra.
As an intermediate step towards the parametrisation, we study further quotients in-between $\N\T_X$ and $\N\O_X$.

\subsection{Relative $2^d$-tuples and induced ideals}

We start by paring down the properties of the family $\I:=\{\I_F\}_{F\subseteq[d]}$ and the CNP-relations of \cite{DK18} covered in Subsection \ref{Ss: scaps}. 

\begin{definition}\label{D:CNP rel}
Let $X$ be a strong compactly aligned product system with coefficients in a C*-algebra $A$. 
A \emph{$2^d$-tuple of $X$} is a family $\L := \{\L_F\}_{F \subseteq [d]}$ such that $\L_F$ is a non-empty subset of $A$ for all $F\subseteq[d]$.
A $2^d$-tuple $\L$ of $X$ is called \emph{relative} if
\[
\L_F\subseteq \bigcap\{\phi_{\un{i}}^{-1}(\K(X_{\un{i}}))\mid i\in F\} 
\foral
\mt\neq F\subseteq [d].
\]
\end{definition}

\begin{remark}\label{R:relrem}
We stipulate that the sets $\L_F$ are non-empty for convenience.
More specifically, the sets $\L_F$ are designed to generate certain ``relative CNP-relations", and so $\L_F=\{0\}$ plays the same role as $\L_F=\mt$.
In the former case the relative CNP-relations are satisfied trivially, and in the latter case they are satisfied vacuously.
Functionally there is no difference, and it is more convenient to take $\L_F=\{0\}$ as in this case $\L_F$ is an ideal.
\end{remark}

We write $\L\subseteq\L'$ for $2^d$-tuples $\L$ and $\L'$ if and only if $\L_F\subseteq\L_F'$ for all $F\subseteq[d]$.
This defines a partial order on the set of $2^d$-tuples of $X$.
We say that $\L=\L'$ if and only if $\L_F = \L'_F$ for all $F \subseteq [d]$.

Let $(\pi,t)$ be a Nica-covariant representation of $X$.
The crucial property of a relative $2^d$-tuple $\L$ is that
\[
\pi(a)q_F=\pi(a)+\sum\{(-1)^{|\un{n}|}\psi_\un{n}(\phi_\un{n}(a))\mid \un{0}\neq\un{n}\leq\un{1}_F\}\in\ca(\pi,t)\foral a\in\L_F, F\subseteq[d],
\]
using Proposition \ref{P:prod cai}.
This allows us to extend the ideas of \cite{DK18} in a natural way.

\begin{definition}\label{D:LCNP}
Let $X$ be a strong compactly aligned product system with coefficients in a C*-algebra $A$. 
Let $\L$ be a relative $2^d$-tuple of $X$ and let $(\pi,t)$ be a Nica-covariant representation of $X$. 
We define the \emph{ideal of the $\L$-relative CNP-relations with respect to $(\pi,t)$} to be
\[
\fJ_{\L}^{(\pi,t)} := \sum\{\fJ_{\L, F}^{(\pi,t)}\mid F\subseteq[d]\} \subseteq \ca(\pi,t),
\; \text{where} \;
\mathfrak{J}_{\L,F}^{(\pi,t)}:=\sca{\pi(\L_F)q_F}\foral F\subseteq[d].
\]
We say that \emph{$\L$ induces $\fJ_{\L}^{(\pi,t)}$}.
\end{definition}

Being an algebraic sum of ideals in $\ca(\pi,t)$, the space $\mathfrak{J}_\L^{(\pi,t)}$ is itself an ideal in $\ca(\pi,t)$.
When $(\pi,t)$ admits a gauge action, the ideal $\mathfrak{J}_\L^{(\pi,t)}$ is gauge-invariant.
By setting $\L=\I$, we recover $\mathfrak{J}_\I^{(\pi,t)}$ as defined in \cite{DK18}.
In many cases, we may assume that $\L$ is a family of ideals without loss of generality.

\begin{lemma}\label{L: ideal}
Let $X$ be a strong compactly aligned product system with coefficients in a C*-algebra $A$.
Let $\L$ be a relative $2^d$-tuple of $X$ and let $(\pi,t)$ be a Nica-covariant representation of $X$.
Then $\sca{\L}:=\{\sca{\L_F}\}_{F\subseteq[d]}$ is a relative $2^d$-tuple of $X$ such that $\L\subseteq\sca{\L}$ and $\mathfrak{J}_{\L,F}^{(\pi,t)}=\mathfrak{J}_{\sca{\L},F}^{(\pi,t)}$ for all $F\subseteq[d]$.
In particular, we have $\mathfrak{J}_\L^{(\pi,t)}=\mathfrak{J}_{\sca{\L}}^{(\pi,t)}$.
\end{lemma}

\begin{proof}
It is immediate that $\L\subseteq\sca{\L}$, and that $\sca{\L}$ is a relative $2^d$-tuple of $X$. 
It remains to see that $\mathfrak{J}_{\L,F}^{(\pi,t)}=\mathfrak{J}_{\sca{\L},F}^{(\pi,t)}$ for all $F\subseteq[d]$, as the final claim follows as an immediate consequence.
To this end, fix $F\subseteq[d]$. 
Since $\L_F\subseteq\sca{\L_F}$, we have $\mathfrak{J}_{\L,F}^{(\pi,t)}\subseteq \mathfrak{J}_{\sca{\L},F}^{(\pi,t)}$. 
For the reverse inclusion, for $b,c\in A$ and $a\in\L_F$, we have
\[
\pi(bac)q_F=\pi(b)\pi(a)\pi(c)q_F=\pi(b)(\pi(a)q_F)\pi(c) \in \mathfrak{J}_{\L,F}^{(\pi,t)},
\] 
using Proposition \ref{P:pf reducing} in the second equality. 
It follows that $\mathfrak{J}_{\sca{\L},F}^{(\pi,t)}\subseteq \mathfrak{J}_{\L,F}^{(\pi,t)}$, as required.
\end{proof}

Let $\L_1$ and $\L_2$ be relative $2^d$-tuples of $X$.
It is straightforward to see that their sum $\L_1+\L_2$, defined by $(\L_1 + \L_2)_F := \L_{1,F}+\L_{2,F}$ for all $F \subseteq [d]$, is also a relative $2^d$-tuple.
When $\L_1$ and $\L_2$ consist of ideals, summing respects the induced ideals.

\begin{lemma}\label{L:sumofe}
Let $X$ be a strong compactly aligned product system with coefficients in a C*-algebra $A$.
Let $\L_1$ and $\L_2$ be relative $2^d$-tuples of $X$ that consist of ideals and let $(\pi,t)$ be a Nica-covariant representation of $X$.
Then
\[
\fJ_{\L_1+\L_2, F}^{(\pi,t)}=\fJ_{\L_1, F}^{(\pi,t)}+\fJ_{\L_2, F}^{(\pi,t)} \foral F \subseteq [d],
\]
and thus $\fJ_{\L_1+\L_2}^{(\pi,t)}=\fJ_{\L_1}^{(\pi,t)}+\fJ_{\L_2}^{(\pi,t)}$.
\end{lemma}

\begin{proof}
This is immediate since $0 \in \L_{1,F}, \L_{2,F}$ for all $F \subseteq [d]$.
\end{proof}

A first approach towards the parametrisation of the gauge-invariant ideals of $\N\T_X$ would be to establish a correspondence between the relative $2^d$-tuples of $X$ and the gauge-invariant ideals of $\N\T_X$ that they induce.
However, this is insufficient as different relative $2^d$-tuples may induce the same gauge-invariant ideal of $\N\T_X$.
We provide an example to this effect in Remark \ref{R:max} by using the next lemma.

\begin{lemma}\label{L:max}
Let $X$ be a strong compactly aligned product system with coefficients in a C*-algebra $A$.
Let $\L$ be a relative $2^d$-tuple of $X$ and let the relative $2^d$-tuple $\L'$ be defined by 
\[
\L'_F := 
\begin{cases}
\L_\mt & \text{if } F= \mt, \\
\L_F + \L_\mt \cap (\bigcap \{\phi_{\un{i}}^{-1}(\K(X_{\un{i}}))\mid i\in F\}) & \text{if } \mt \neq F \subseteq [d].
\end{cases}
\]
Then we have $\fJ_{\L'}^{(\pi,t)} = \fJ_{\L}^{(\pi,t)}$ for every Nica-covariant representation $(\pi,t)$ of $X$.
\end{lemma}

\begin{proof}
On one hand we have $\mathfrak{J}_\L^{(\pi,t)}\subseteq\mathfrak{J}_{\L'}^{(\pi,t)}$ since $\L \subseteq \L'$.
On the other hand, fix $\mt\neq F\subseteq[d]$ and $a \in \L_\mt$ satisfying $\phi_{\un{i}}(a) \in \K(X_{\un{i}})$ for all $i \in F$.
We claim that
\[
\psi_{\un{n}}(\phi_{\un{n}}(a)) \in \sca{\pi(\L_\mt)} = \fJ_{\L, \mt}^{(\pi,t)} \foral \un{0}\neq\un{n} \leq \un{1}_F.
\]
We will prove this by induction on $|\un{n}|$.
When $|\un{n}|=1$, we must have $\un{n}=\un{i}$ for some $i\in F$.
In turn, for an approximate unit $(k_{\un{i}, \la})_{\la\in\La} \subseteq \K(X_{\un{i}})$, we have
\[
\psi_\un{i}(\phi_\un{i}(a))
=
\nor{\cdot}\text{-}\lim_\la \psi_\un{i}(\phi_\un{i}(a)k_{\un{i},\la})
=
\nor{\cdot}\text{-}\lim_\la \pi(a)\psi_\un{i}(k_{\un{i},\la})
\in 
\sca{\pi(\L_\mt)},
\]
as required.
Now suppose that the claim holds for all $\un{n}$ such that $\un{0}\neq\un{n}\leq\un{1}_F$ and $|\un{n}|=N$ for some $1\leq N<|F|$.
Fix $\un{m}$ such that $\un{0}\neq\un{m}\leq\un{1}_F$ and $|\un{m}|=N+1$.
We may express $\un{m}$ in the form $\un{m}=\un{n}+\un{i}$ for some $\un{0}\neq\un{n}\leq\un{1}_F$ and $i\in F$ such that $\un{n}\perp\un{i}$ and $|\un{n}|=N$.
We have 
\begin{align*}
\psi_\un{m}(\phi_\un{m}(a))
& =
\nor{\cdot}\text{-}\lim_\la \psi_\un{m}(\iota_\un{n}^\un{m}(\phi_\un{n}(a))\iota_\un{i}^\un{m}(k_{\un{i},\la})) 
=
\nor{\cdot}\text{-}\lim_\la \psi_\un{n}(\phi_\un{n}(a))\psi_\un{i}(k_{\un{i},\la})
\in \sca{\pi(\L_\mt)},
\end{align*}
by the inductive hypothesis, Proposition \ref{P:sca ai}, and Nica-covariance (noting that $\un{n}\vee\un{i}=\un{n}+\un{i}=\un{m}$ since $\un{n}\perp\un{i}$).
This finishes the proof of the claim.
Therefore, we have
\[
\pi(a) q_{F} = \pi(a) + \sum \{ (-1)^{|\un{n}|} \psi_{\un{n}}(\phi_{\un{n}}(a)) \mid \un{0} \neq \un{n} \leq \un{1}_F \}
\in 
\sca{\pi(\L_\mt)}
\subseteq
\mathfrak{J}_\L^{(\pi,t)}.
\]
Hence
\[
\pi \left( \L_\mt \cap (\bigcap \{\phi_{\un{i}}^{-1}(\K(X_{\un{i}}))\mid i\in F\}) \right) q_F
\subseteq
\mathfrak{J}_\L^{(\pi,t)},
\]
from which it follows that $\mathfrak{J}_{\L'}^{(\pi,t)} \subseteq \mathfrak{J}_\L^{(\pi,t)}$, as required.
\end{proof}

\begin{remark}\label{R:max}
Notice that $\L$ and $\L'$ in Lemma \ref{L:max} may differ in general.
For example, assume that the left actions of the fibres of $X$ are by compact operators.
Then any $2^d$-tuple of $X$ is automatically relative.
Let $\L_\mt$ be a non-zero ideal and set $\L_F=\{0\}$ for all $\mt\neq F\subseteq[d]$. 
Then $\L \neq \L'$, as claimed.
\end{remark}

To remedy this issue, we will instead look for the largest relative $2^d$-tuple inducing a fixed gauge-invariant ideal of $\N\T_X$.

\begin{definition}\label{D:maximal}
Let $X$ be a strong compactly aligned product system with coefficients in a C*-algebra $A$.
Let $\M$ be a relative $2^d$-tuple of $X$ and let $(\pi,t)$ be a Nica-covariant representation of $X$.
We say that $\M$ is a \emph{maximal $2^d$-tuple of $X$ with respect to $(\pi,t)$} if whenever $\L$ is a relative $2^d$-tuple of $X$ such that $\mathfrak{J}_\L^{(\pi,t)}=\mathfrak{J}_\M^{(\pi,t)}$ and $\M\subseteq\L$, we have $\L=\M$.
When we replace $(\pi,t)$ by $(\ol{\pi}_X,\ol{t}_X)$, we will refer to $\M$ simply as a \emph{maximal $2^d$-tuple of $X$}.
\end{definition}

The following proposition establishes existence and uniqueness of maximal $2^d$-tuples.

\begin{proposition} \label{P:maximal}
Let $X$ be a strong compactly aligned product system with coefficients in a C*-algebra $A$. 
Let $\L$ be a relative $2^d$-tuple of $X$ and let $(\pi,t)$ be a Nica-covariant representation of $X$.
Then there exists a unique relative $2^d$-tuple $\M$ of $X$ such that:
\begin{enumerate}
\item $\fJ_{\L}^{(\pi,t)} = \fJ_{\M}^{(\pi,t)}$, and
\item $\L' \subseteq \M$ for every relative $2^d$-tuple $\L'$ of $X$ satisfying $\fJ_{\L'}^{(\pi,t)} = \fJ_{\L}^{(\pi,t)}$.
\end{enumerate}
In particular, $\M$ is the largest $2^d$-tuple of $X$ with respect to $(\pi,t)$ which consists of ideals.
\end{proposition}
\begin{proof}
For each $F\subseteq[d]$, define
\[
\M_F:=\bigcup\{\L_F'\mid \L' \; \text{is a relative $2^d$-tuple of $X$ such that} \; \mathfrak{J}_{\L'}^{(\pi,t)}=\mathfrak{J}_\L^{(\pi,t)}\}.
\]
The union is well-defined, since it includes $\L_F$ by assumption, and the index takes values in the set $\mathcal{P}(A)^{2^d}$.
Each $\M_F$ is a non-empty subset of $A$ by construction, so $\M$ is a $2^d$-tuple of $X$.
Fix $\mt\neq F\subseteq[d]$ and take $a\in\M_F$.
Then $a\in\L_F'$ for some relative $2^d$-tuple $\L'$, and hence $\phi_\un{i}(a)\in\K(X_\un{i})$ for all $i\in F$. 
Therefore, setting $\M:=\{\M_F\}_{F\subseteq[d]}$, we have that $\M$ is a relative $2^d$-tuple of $X$.

Because $\L\subseteq\M$, we have $\mathfrak{J}_\L^{(\pi,t)}\subseteq\mathfrak{J}_\M^{(\pi,t)}$ trivially.
For the other inclusion, fix $F\subseteq[d]$ and $a\in\M_F$.
It suffices to show that $\pi(a)q_F\in\mathfrak{J}_\L^{(\pi,t)}$.
To this end, we have $a\in\L_F'$ for some relative $2^d$-tuple $\L'$ with the property that $\mathfrak{J}_{\L'}^{(\pi,t)}=\mathfrak{J}_\L^{(\pi,t)}$.
By definition, we have $\pi(a)q_F\in\mathfrak{J}_{\L'}^{(\pi,t)}=\mathfrak{J}_\L^{(\pi,t)}$, as required.

By construction, we have $\L'\subseteq\M$ for every relative $2^d$-tuple $\L'$ satisfying $\mathfrak{J}_{\L'}^{(\pi,t)}=\mathfrak{J}_\L^{(\pi,t)}$. 
It is immediate that $\M$ is unique and maximal with respect to $(\pi,t)$, being the largest relative $2^d$-tuple inducing $\fJ_\L^{(\pi,t)}$.

Finally, by Lemma \ref{L: ideal} we have that $\sca{\M}:=\{\sca{\M_F}\}_{F\subseteq[d]}$ is a relative $2^d$-tuple such that $\M\subseteq\sca{\M}$ and $\mathfrak{J}_\M^{(\pi,t)}=\mathfrak{J}_{\sca{\M}}^{(\pi,t)}$.
Applying maximality of $\M$, we have $\M=\sca{\M}$.
This shows that $\M$ consists of ideals, finishing the proof.
\end{proof}

\begin{remark}\label{R:max I}
The motivating example of a maximal relative $2^d$-tuple is the family $\I$.
To see this, first note that $\I$ is a relative $2^d$-tuple by definition.
Let $\M$ be the maximal $2^d$-tuple inducing $\fJ_\I^{(\ol{\pi}_X, \ol{t}_X)}$, so that $\I\subseteq\M$.
If $a \in \M_F$ for $F \subseteq [d]$, then 
\[
\ol{\pi}_X(a) \ol{q}_{X,F} \in \fJ_\M^{(\ol{\pi}_X, \ol{t}_X)}=\fJ_\I^{(\ol{\pi}_X, \ol{t}_X)},
\]
and thus we have $a \in \I_F$ by \cite[Proposition 3.4]{DK18}.
Therefore $\M \subseteq \I$, showing that $\M = \I$.
\end{remark}

We wish to ascertain the conditions under which a relative $2^d$-tuple is promoted to a maximal one.
Using $\I$ as a prototype, we abstract two of its properties.

\begin{definition}\label{D:inv po}
Let $X$ be a strong compactly aligned product system with coefficients in a C*-algebra $A$. 
Let $\L$ be a $2^d$-tuple of $X$.
\begin{enumerate}
\item  We say that $\L$ is $X$-\emph{invariant} if $\left[\sca{X_\un{n},\L_F X_\un{n}}\right]\subseteq\L_F$, for all $\un{n}\perp F$, $F\subseteq[d]$.
\item We say that $\L$ is \emph{partially ordered} if $\L_{F_1} \subseteq \L_{F_2}$ whenever  $F_1 \subseteq F_2 \subseteq [d]$.
\end{enumerate}
\end{definition}

When the underlying product system $X$ is clear from the context, we will abbreviate ``$X$-invariant" as simply ``invariant".
Notice that when we take $F=\mt$, condition (i) implies that $\L_\mt$ is positively invariant for $X$ (provided that $\L_\mt$ is an ideal).
If $\L_F$ is an ideal, then we may drop the closed linear span in condition (i).
If $\L$ is a partially ordered relative $2^d$-tuple, then
\[
\L_F\subseteq\L_{[d]}\subseteq\bigcap\{\phi_\un{i}^{-1}(\K(X_\un{i}))\mid i\in[d]\}\foral F\subseteq[d].
\]
In particular, we have $\pi(a)q_F\in\ca(\pi,t)$ for all $a\in\L_D$ and $D,F\subseteq[d]$, where $(\pi,t)$ is a Nica-covariant representation.

The $2^d$-tuple $\I$ is invariant, partially ordered and consists of ideals.
Since $\I$ is maximal by Remark \ref{R:max I}, one may be tempted to assert that invariance and partial ordering suffice to capture maximality.
However, this is not true in general, as the following counterexample shows.

\begin{example} \label{E:not en}
Let $B$ be a non-zero C*-algebra and let $A=B \oplus \bC$ be its unitisation. 
Consider the semigroup action $\al\colon\bZ_+^2\to\text{End}(A)$ such that 
\[
\al_{(1,0)}=\text{id}_A
\qand 
\al_{(0,1)}(b,\la)=(0,\la) 
\foral b\in B,\la\in\bC.
\]
Applying the construction of Subsection \ref{Ss:dynsys} to the C*-dynamical system $(A,\al,\bZ_+^2)$, we obtain a strong compactly aligned product system $X_\al$ over $\bZ_+^2$ with coefficients in $A$.
In particular, the left action of every fibre is by compacts, and so any $2^2$-tuple of $X_\al$ is automatically a relative $2^2$-tuple.

Next we define the relative $2^2$-tuples $\L$ and $\L'$ of $X_\al$ by
\begin{center}
\begin{tikzcd}
\L_{\{2\}}=\{0\} \arrow[dash]{r}\arrow[dash]{d} & \L_{\{1,2\}}=B\oplus\{0\}\arrow[dash]{d} \\
\L_\mt=\{0\}\arrow[dash]{r} & \L_{\{1\}}=\{0\}
\end{tikzcd}
and
\begin{tikzcd}
\L_{\{2\}}'=\{0\} \arrow[dash]{r}\arrow[dash]{d} & \L_{\{1,2\}}'=B\oplus\{0\}\arrow[dash]{d} \\
\L_\mt'=\{0\}\arrow[dash]{r} & \L_{\{1\}}'=B\oplus\{0\}.
\end{tikzcd}
\end{center}
Notice that $\L$ and $\L'$ are invariant, partially ordered, consist of ideals and satisfy $\L \subsetneq \L'$.
Let $(\pi,t)$ be a Nica-covariant representation.
Since $\L\subseteq\L'$, it is clear that $\mathfrak{J}_\L^{(\pi,t)}\subseteq\mathfrak{J}_{\L'}^{(\pi,t)}$.
It is immediate that $\pi(\L_F')q_F\subseteq\mathfrak{J}_\L^{(\pi,t)}$ for all $F\in\{\mt,\{2\},\{1,2\}\}$.
Fix $(b,0)\in \L_{\{1\}}'$, and observe that
\begin{align*}
\mathfrak{J}_\L^{(\pi,t)}\ni\pi(b,0)q_{\{1,2\}} 
& =
\pi(b,0)-\psi_{(1,0)}(\phi_{(1,0)}(b,0))-\psi_{(0,1)}(\phi_{(0,1)}(b,0))+\psi_{(1,1)}(\phi_{(1,1)}(b,0)) \\
& =
\pi(b,0)-\psi_{(1,0)}(\phi_{(1,0)}(b,0))
=
\pi(b,0)q_{\{1\}},
\end{align*}
using the fact that $\phi_{(0,1)}(b,0)=\phi_{(1,1)}(b,0)=0$ by definition.
Therefore we conclude that $\pi(\L_{\{1\}}')q_{\{1\}}\subseteq\mathfrak{J}_\L^{(\pi,t)}$, and thus $\mathfrak{J}_{\L'}^{(\pi,t)}\subseteq\mathfrak{J}_\L^{(\pi,t)}$. 
Hence $\L$ is not maximal with respect to $(\pi,t)$.
\end{example}

To better understand maximality, we need to explore $2^d$-tuples induced by Nica-covariant representations, demonstrating the conditions of Definition \ref{D:inv po} in action.

\begin{definition}\label{D:rep fam}
Let $X$ be a strong compactly aligned product system with coefficients in a C*-algebra $A$ and let $(\pi,t)$ be a Nica-covariant representation of $X$.
We define $\L^{(\pi,t)}$ to be the $2^d$-tuple of $X$ given by
\[
\L_\mt^{(\pi,t)}:=\ker\pi \qand \L_F^{(\pi,t)}:=\pi^{-1}(B_{(\un{0},\un{1}_F]}^{(\pi,t)}) \foral \mt\neq F\subseteq[d].
\]
\end{definition}

\begin{proposition}\label{P:inj e fam}
Let $X$ be a strong compactly aligned product system with coefficients in a C*-algebra $A$ and let $(\pi,t)$ be a Nica-covariant representation of $X$.
Then $\L^{(\pi,t)}$ is invariant, partially ordered and consists of ideals. 

If $(\pi,t)$ is in addition injective, then $\L^{(\pi,t)} \subseteq \I$, and thus $\L^{(\pi,t)}$ is relative.
\end{proposition}

\begin{proof}
It is clear that $\L^{(\pi,t)}$ consists of ideals by construction and the properties of the representation $(\pi,t)$. 
For invariance, fix $F \subseteq [d], a\in\L_F^{(\pi,t)}$, and $\un{m} \perp F$.
Due to Nica-covariance, we have
\[
\pi(\sca{X_{\un{m}}, a X_{\un{m}}}) 
=
t_{\un{m}}(X_{\un{m}})^* \pi(a) t_{\un{m}}(X_{\un{m}})
\subseteq
t_{\un{m}}(X_{\un{m}})^* B_{(\un{0}, \un{1}_F]}^{(\pi,t)} t_{\un{m}}(X_{\un{m}}) \subseteq B_{(\un{0}, \un{1}_F]}^{(\pi,t)}.
\]
Hence $\L^{(\pi,t)}$ is invariant, as required.
To see that $\L^{(\pi,t)}$ is partially ordered, first note that the inclusion $\L_\mt^{(\pi,t)}\subseteq\L_F^{(\pi,t)}$ is immediate for all $F\subseteq[d]$.
Likewise, whenever $\mt\neq F\subseteq D\subseteq[d]$, we have $B_{(\un{0},\un{1}_F]}^{(\pi,t)}\subseteq B_{(\un{0},\un{1}_D]}^{(\pi,t)}$ and therefore $\L_F^{(\pi,t)}\subseteq\L_D^{(\pi,t)}$, as required.
The final claim follows by \cite[Proposition 3.4]{DK18} (see Proposition \ref{P:in IF}).
\end{proof}

In order to make further progress with maximality, we use relative $2^d$-tuples to define relative Cuntz-Nica-Pimsner algebras.

\begin{definition}\label{D:LCNP}
Let $X$ be a strong compactly aligned product system with coefficients in a C*-algebra $A$. 
Let $\L$ be a relative $2^d$-tuple of $X$ and let $(\pi,t)$ be a Nica-covariant representation of $X$. 
We say that $(\pi,t)$ is an \emph{$\L$-relative CNP-representation of $X$} if it satisfies
\[
\pi(\L_F)q_F = \{0\}
\foral
F \subseteq [d].
\]
The universal C*-algebra with respect to the $\L$-relative CNP-representations of $X$ is denoted by $\N\O(\L,X)$, and we refer to it as the \emph{$\L$-relative Cuntz-Nica-Pimsner algebra of $X$}.
\end{definition}

We have that $(\pi,t)$ is an $\L$-relative CNP-representation if and only if $\mathfrak{J}_\L^{(\pi,t)}=\{0\}$; equivalently if $\pi \times t$ factors through $\N\O(\L, X)$.
The following proposition is the analogue of Proposition \ref{P:T univ} for relative Cuntz-Nica-Pimsner algebras, and shows that $\N\O(\L, X)$ exists.

\begin{proposition}\label{P:NO}
Let $X$ be a strong compactly aligned product system with coefficients in a C*-algebra $A$ and let $\L$ be a relative $2^d$-tuple of $X$. 
Then there exists a C*-algebra $\N\O(\L, X)$ and an $\L$-relative CNP-representation $(\pi_X^\L,t_X^\L)$ of $X$ on $\N\O(\L,X)$ such that:
\begin{enumerate}
\item $\N\O(\L,X)=\ca(\pi_X^\L,t_X^\L)$;
\item if $(\pi,t)$ is an $\L$-relative CNP-representation of $X$, then there exists a canonical $*$-epimorphism $\Phi \colon \N\O(\L,X)\to\ca(\pi,t)$.
\end{enumerate}
The pair $(\N\O(\L,X),(\pi_X^\L,t_X^\L))$ is unique up to canonical $*$-isomorphism.
\end{proposition}

\begin{proof}
This follows immediately by considering the quotient of $\N\T_X$ by $\mathfrak{J}_{\L}^{(\ol{\pi}_X,\ol{t}_X)}$.
\end{proof}

We will refer to the representation $(\pi_X^\L,t_X^\L)$ as the \emph{universal $\L$-relative CNP-representation of $X$}.
Since $\N\O(\L,X)$ is an equivariant quotient of $\N\T_X$, the representation $(\pi_X^\L,t_X^\L)$ admits a gauge action.
Notice that $\{\{0\}\}_{F\subseteq[d]}$ and $\I$ both constitute relative $2^d$-tuples of $X$ and in particular 
\[
\N\O(\{\{0\}\}_{F\subseteq[d]},X) = \N\T_X
\qand 
\N\O(\I,X) = \N\O_X.
\] 
Likewise, when $X$ is a C*-correspondence we recover the relative Cuntz-Pimsner algebras.
Note that $\N\O(\L,X)=\N\O(\sca{\L},X)$ since the covariance is described C*-algebraically.

If $(\pi,t)$ is an injective Nica-covariant representation, then $\L^{(\pi,t)}$ is a relative $2^d$-tuple by Proposition \ref{P:inj e fam}, and thus we can make sense of $\mathfrak{J}_{\L^{(\pi,t)}}^{(\ol{\pi}_X,\ol{t}_X)}$.
The next proposition shows that $\mathfrak{J}_{\L^{(\pi,t)}}^{(\ol{\pi}_X,\ol{t}_X)}$ arises as a kernel when $(\pi,t)$ admits a gauge action.

\begin{proposition}\label{P:inj rel}
Let $X$ be a strong compactly aligned product system with coefficients in a C*-algebra $A$ and let $(\pi,t)$ be an injective Nica-covariant representation of $X$. 
Then $(\pi,t)$ is an $\L^{(\pi,t)}$-relative CNP-representation of $X$, and consequently there exists a canonical $*$-epimorphism
\[
\Phi \colon \N\O(\L^{(\pi,t)},X)\to\ca(\pi,t).
\]
Moreover, $\Phi$ is injective if and only if $(\pi,t)$ admits a gauge action, in which case 
\[
\ker \pi\times t = \mathfrak{J}_{\L^{(\pi,t)}}^{(\ol{\pi}_X,\ol{t}_X)}.
\]
\end{proposition}

\begin{proof}
Fix $F \subseteq [d]$ and $a \in \L_F^{(\pi,t)}=\pi^{-1}(B_{(\un{0},\un{1}_F]}^{(\pi,t)})$.
Then $\pi(a)q_F=0$ by (\ref{eq:out}).
This shows that $(\pi,t)$ is an $\L^{(\pi,t)}$-relative CNP-representation. 
Thus universality of $\N\O(\L^{(\pi,t)}, X)$ guarantees the existence of $\Phi$. 

For the second claim, we write $(\tilde{\pi},\tilde{t})$ for $(\pi_X^{\L^{(\pi,t)}},t_X^{\L^{(\pi,t)}})$, for notational convenience.
If $\Phi$ is injective, then $(\pi,t)$ inherits the gauge action $\be$ of $(\tilde{\pi},\tilde{t})$.
Conversely, suppose that $(\pi,t)$ admits a gauge action $\ga$, and note that $\Phi$ intertwines the gauge actions (i.e., $\Phi$ is equivariant).
Thus it suffices to show that $\Phi$ is injective on the fixed point algebra, and by Proposition \ref{P:inj fp} this amounts to showing that $\Phi$ is injective on the $[\un{0},\un{1}_{[d]}]$-core.
To reach contradiction, suppose there exists $0\neq f \in \ker\Phi\cap B_{[\un{0},\un{1}_{[d]}]}^{(\tilde{\pi},\tilde{t})}$, so that we may write
\[
f = \tilde{\pi}(a) + \sum \{ \tilde{\psi}_{\un{n}}(k_{\un{n}}) \mid \un{0}\neq\un{n} \leq \un{1}_{[d]} \},
\]
for some $a\in A$ and $k_\un{n}\in\K(X_\un{n})$ for all $\un{0}\neq\un{n}\leq\un{1}_{[d]}$.
Recall that we write $\tilde{\pi}(a)=\tilde{\psi}_\un{0}(k_\un{0})$ for $k_\un{0}:=\phi_\un{0}(a)$.
Without loss of generality, we may assume that $f$ is written irreducibly, so that we may choose $\un{0}\leq\un{r}\leq\un{1}_{[d]}$ minimal such that $k_\un{r}\neq 0$, and $\tilde{\psi}_{\un{r}}(k_{\un{r}}) \notin B_{(\un{r}, \un{1}_{[d]}]}^{(\tilde{\pi},\tilde{t})}$.
If $\un{r}=\un{1}_{[d]}$, then $f=\tilde{\psi}_{\un{1}_{[d]}}(k_{\un{1}_{[d]}})$ and $\Phi(f)=\psi_{\un{1}_{[d]}}(k_{\un{1}_{[d]}})=0$.
Injectivity of $(\pi,t)$ then implies that $k_{\un{1}_{[d]}}=0$ and hence $f=0$, a contradiction.
So without loss of generality assume that $\un{r}<\un{1}_{[d]}$.
Then for every $\xi_{\un{r}}, \eta_{\un{r}} \in X_{\un{r}}$, we have
\[
0 
= 
t_{\un{r}}(\xi_{\un{r}})^* \Phi(f) t_{\un{r}}(\eta_{\un{r}}) 
= 
\pi(\sca{\xi_{\un{r}}, k_{\un{r}} \eta_{\un{r}}})
+
\sum \{ \psi_{\un{n}}(k_{\un{n}}') \mid \un{0}\neq\un{n} \leq \un{1}_{[d]} - \un{r} \},
\]
where each $k_\un{n}'$ is a suitably defined element of $\K(X_\un{n})$ for all $\un{0}\neq\un{n}\leq\un{1}_{[d]} - \un{r}$.
Hence we have $\sca{X_{\un{r}}, k_{\un{r}} X_{\un{r}}} \subseteq \L_{F}^{(\pi,t)}$ for $F := \supp (\un{1}_{[d]} - \un{r})$ (which is non-empty since $\un{1}_{[d]}\neq\un{r}$), and thus we obtain $\tilde{\pi}(\sca{X_{\un{r}}, k_{\un{r}} X_{\un{r}}}) \subseteq B_{(\un{0}, \un{1}_{[d]} - \un{r}]}^{(\tilde{\pi},\tilde{t})}$ since $(\tilde{\pi},\tilde{t})$ is an $\L^{(\pi,t)}$-relative CNP-representation.
In particular, we have
\begin{align*}
\tilde{\psi}_{\un{r}}(\K(X_{\un{r}})) \tilde{\psi}_{\un{r}}(k_{\un{r}}) \tilde{\psi}_{\un{r}}(\K(X_{\un{r}}))
& \subseteq
[\tilde{t}_{\un{r}}(X_{\un{r}}) \tilde{\pi}(\sca{X_{\un{r}}, k_{\un{r}} X_{\un{r}}}) \tilde{t}_{\un{r}}(X_{\un{r}})^*] \\
& \subseteq
[\tilde{t}_{\un{r}}(X_{\un{r}}) B_{(\un{0}, \un{1}_{[d]} - \un{r}]}^{(\tilde{\pi},\tilde{t})} \tilde{t}_{\un{r}}(X_{\un{r}})^*]
\subseteq
B_{(\un{r}, \un{1}_{[d]}]}^{(\tilde{\pi},\tilde{t})},
\end{align*}
and by taking an approximate unit of $\tilde{\psi}_{\un{r}}(\K(X_{\un{r}}))$ we derive that $\tilde{\psi}_{\un{r}}(k_{\un{r}}) \in B_{(\un{r}, \un{1}_{[d]}]}^{(\tilde{\pi},\tilde{t})}$, which is a contradiction.
Thus $\Phi$ is injective on the $[\un{0},\un{1}_{[d]}]$-core, as required.

For the last assertion, we have $\pi\times t=\Phi\circ Q$ and thus in particular
\[
\ker \pi \times t = \ker \Phi \circ Q = \ker Q = \mathfrak{J}_{\L^{(\pi,t)}}^{(\ol{\pi}_X,\ol{t}_X)}
\]
for the quotient map $Q\colon\N\T_X\to\N\O(\L^{(\pi,t)},X)$, and the proof is complete.
\end{proof}

Injective Nica-covariant representations admitting a gauge action provide an ample supply of maximal relative $2^d$-tuples.

\begin{proposition}\label{P:rep m fam}
Let $X$ be a strong compactly aligned product system with coefficients in a C*-algebra $A$.
Let $(\pi,t)$ be an injective Nica-covariant representation of $X$ that admits a gauge action.
Then $\L^{(\pi,t)}$ is a maximal $2^d$-tuple of $X$ that is contained in $\I$.
\end{proposition}

\begin{proof}
By Proposition \ref{P:inj e fam}, we have $\L^{(\pi,t)} \subseteq \I$.
For maximality, let $\L$ be a relative $2^d$-tuple of $X$ such that $\mathfrak{J}_\L^{(\ol{\pi}_X,\ol{t}_X)}=\mathfrak{J}_{\L^{(\pi,t)}}^{(\ol{\pi}_X,\ol{t}_X)}$ and $\L^{(\pi,t)}\subseteq\L$.
It suffices to show that $\L \subseteq \L^{(\pi,t)}$.
Fix $F\subseteq[d]$ and $a\in\L_F$. 
Then 
\[
\ol{\pi}_X(a)\ol{q}_{X,F} \in \mathfrak{J}_\L^{(\ol{\pi}_X,\ol{t}_X)}=\mathfrak{J}_{\L^{(\pi,t)}}^{(\ol{\pi}_X,\ol{t}_X)}
\]
by definition.
Since $\mathfrak{J}_{\L^{(\pi,t)}}^{(\ol{\pi}_X,\ol{t}_X)}=\ker \pi\times t$ by Proposition \ref{P:inj rel}, we obtain
\begin{align*}
\pi(a)q_F
& =
\pi(a) + \sum \{ (-1)^{|\un{n}|} \psi_{\un{n}}(\phi_{\un{n}}(a)) \mid \un{0} \neq \un{n} \leq \un{1}_F \} \\
& =
(\pi\times t) (\ol{\pi}_X(a) + \sum \{ (-1)^{|\un{n}|} \ol{\psi}_{X,\un{n}}(\phi_{\un{n}}(a)) \mid \un{0} \neq \un{n} \leq \un{1}_F \} ) \\
& =
(\pi\times t)(\ol{\pi}_X(a)\ol{q}_{X,F})
=
0,
\end{align*}
using Proposition \ref{P:prod cai} in the first and third equalities.
Thus $a\in\pi^{-1}(B_{(\un{0},\un{1}_F]}^{(\pi,t)})=\L_F^{(\pi,t)}$ and hence $\L\subseteq\L^{(\pi,t)}$, as required.
\end{proof}

\subsection{(E)-$2^d$-tuples}\label{S:e fam}

We are interested in a special class of $2^d$-tuples that describes the relative Cuntz-Nica-Pimsner algebras $\N\O(\L,X)$ in-between $\N\T_X$ and $\N\O_X$.

\begin{proposition}\label{P:NO inj}
Let $X$ be a strong compactly aligned product system with coefficients in a C*-algebra $A$.
Let $\L$ be a relative $2^d$-tuple of $X$ and let $(\pi,t)$ be an $\L$-relative CNP-representation of $X$.
If $(\pi,t)$ is injective, then $\L \subseteq \I$.
Conversely, if $\L \subseteq \I$ then the universal $\L$-relative CNP-representation $(\pi_X^\L, t_X^\L)$ is injective.
\end{proposition}

\begin{proof}
First assume that $(\pi, t)$ is injective.
Fix $F\subseteq[d]$ and $a\in\L_F$.
Since $(\pi,t)$ is an $\L$-relative CNP-representation, we have $\pi(a)q_{F} = 0$ and hence $\pi(a) \in B_{(\un{0},\un{1}_F]}^{(\pi,t)}$.
Since $(\pi,t)$ is assumed to be injective, an application of \cite[Proposition 3.4]{DK18} gives that $a\in\I_F$.
Hence $\L\subseteq\I$.

Conversely, assume that $\L\subseteq\I$.
Then the canonical quotient map $\N\T_X \to \N\O_X$ (which is injective on $X$) factors through the canonical $*$-epimorphism $\N\O(\L,X)\to\N\O_X$.
Hence $(\pi_X^\L,t_X^\L)$ is injective, finishing the proof.
\end{proof}

We see that the $2^d$-tuples $\L\subseteq\I$ are special, and we give them their own name to reflect this.

\begin{definition}\label{D:e fam}
Let $X$ be a strong compactly aligned product system with coefficients in a C*-algebra $A$.
We say that a $2^d$-tuple $\L$ of $X$ is an \emph{(E)-$2^d$-tuple of $X$} if $\L \subseteq \I$.
\end{definition}

Every (E)-$2^d$-tuple is automatically a relative $2^d$-tuple with $\L_\mt = \{0\} = \I_\mt$.
The ``E" of ``(E)-$2^d$-tuple" stands for ``Embedding", since $X \hookrightarrow \N\O(\L,X)$ isometrically for a relative $2^d$-tuple $\L$ if and only if $\L \subseteq \I$, by Proposition \ref{P:NO inj}.
By Proposition \ref{P:inj e fam}, injective Nica-covariant representations provide the quintessential supply of (E)-$2^d$-tuples.
It is straightforward to check that the sum of two (E)-$2^d$-tuples is an (E)-$2^d$-tuple.
We also make the following observation.

\begin{lemma}\label{L:J+inv}
Let $X$ be a strong compactly aligned product system with coefficients in a C*-algebra $A$.
Let $\L$ be a $2^d$-tuple of $X$. 
If $\L$ is invariant and satisfies $\L\subseteq\J$, then $\L$ is an (E)-$2^d$-tuple of $X$.
\end{lemma}

\begin{proof}
Fix $F \subseteq [d]$ and $a\in\L_F$.
By assumption we have
\[
\sca{X_\un{n},aX_\un{n}}\subseteq\left[\sca{X_\un{n}, \L_F X_\un{n}}\right]\subseteq\L_F\subseteq\J_F\foral\un{n}\perp F.
\]
By definition this means that $a\in\I_F$ and hence $\L\subseteq\I$, as required.
\end{proof}

Let $(\pi,t)$ be a Nica-covariant representation and $\mathfrak{J}_\L^{(\pi,t)}$ be the ideal induced by an (E)-$2^d$-tuple $\L$.
Our goal is to show that we can choose $\L$ to be in addition invariant and partially ordered, and to consist of ideals.

\begin{lemma}\label{L:e inv}
Let $X$ be a strong compactly aligned product system with coefficients in a C*-algebra $A$. 
Let $\L$ be an (E)-$2^d$-tuple of $X$ and let $(\pi,t)$ be a Nica-covariant representation of $X$.
Then the $2^d$-tuple $\Inv(\L)$ of $X$ defined by
\[
\Inv(\L)_F:= \ol{\spn} \{X_\un{n}(\L_F)\mid \un{n}\perp F\} \foral F \subseteq [d]
\]
is an invariant (E)-$2^d$-tuple consisting of ideals, and is such that 
\[
\L\subseteq\Inv(\L) 
\qand 
\mathfrak{J}_{\L,F}^{(\pi,t)}=\mathfrak{J}_{\Inv(\L),F}^{(\pi,t)} \foral F\subseteq[d].
\]
In particular, we have $\mathfrak{J}_\L^{(\pi,t)}=\mathfrak{J}_{\Inv(\L)}^{(\pi,t)}$.
\end{lemma}

\begin{proof}
Recall that $X_\un{n}(\L_F) \equiv \left[\sca{X_\un{n},\L_F X_\un{n}}\right]$ for all $\un{n}\in\bZ_+^d$ and $F \subseteq [d]$.
By considering $\un{n} = \un{0}$, we see that $\L \subseteq \Inv(\L)$.

Fix $F \subseteq [d]$.
Then $X_\un{n}(\L_F)$ is an ideal of $A$ for all $\un{n}\perp F$, and hence $\Inv(\L)_F$ is itself an ideal of $A$.
Since $\L$ is an (E)-$2^d$-tuple, we have $\L_F\subseteq\I_F$ and thus $X_\un{n}(\L_F)\subseteq\I_F$ for all $\un{n}\perp F$ because $\I$ is invariant \cite[Proposition 2.7]{DK18}.
Thus $\Inv(\L)_F\subseteq\I_F$, and so $\Inv(\L)$ is an (E)-$2^d$-tuple that consists of ideals.

To see that $\Inv(\L)$ is invariant, fix $F\subseteq[d],\un{m}\perp F$ and $a\in\Inv(\L)_F$.
Since $\Inv(\L)_F$ is an ideal, it suffices to show that $\sca{X_\un{m},aX_\un{m}}\subseteq\Inv(\L)_F$.
Assume that $a=\sca{\xi_\un{n},b\eta_\un{n}}$ for some $\un{n}\perp F,\xi_\un{n},\eta_\un{n}\in X_\un{n}$ and $b\in\L_F$.
Then we have
\begin{align*}
\sca{X_\un{m},aX_\un{m}} 
& =
\sca{X_\un{m},\sca{\xi_\un{n},b\eta_\un{n}}X_\un{m}} 
\subseteq
\sca{X_{\un{n} + \un{m}}, b X_{\un{n} + \un{m}}}
\subseteq
X_{\un{n}+\un{m}}(\L_F)
\subseteq
\Inv(\L)_F,
\end{align*}
using that $\un{n}+\un{m}\perp F$ in the final inclusion.
It follows that $\sca{X_\un{m},X_\un{n}(\L_F)X_\un{m}}\subseteq\Inv(\L)_F$ for all $\un{n}\perp F$.
We obtain $\sca{X_\un{m}, \Inv(\L)_F X_\un{m}} \subseteq \Inv(\L)_F$ by linearity and continuity of the inner product, as required.

Since $\L\subseteq\Inv(\L)$, we have $\mathfrak{J}_{\L,F}^{(\pi,t)}\subseteq\mathfrak{J}_{\Inv(\L),F}^{(\pi,t)}$ for all $F\subseteq[d]$.
For the reverse inclusion, fix $F\subseteq[d]$ and $a\in\Inv(\L)_F$. 
Assume that $a=\sca{\xi_\un{n},b\eta_\un{n}}$ for some $\un{n}\perp F,\xi_\un{n},\eta_\un{n}\in X_\un{n}$ and $b\in\L_F$.
We have
\[
\pi(a)q_F=t_\un{n}(\xi_\un{n})^*\pi(b)t_\un{n}(\eta_\un{n})q_F=t_\un{n}(\xi_\un{n})^*(\pi(b)q_F)t_\un{n}(\eta_\un{n})\in\mathfrak{J}_{\L,F}^{(\pi,t)},
\]
where we have used Proposition \ref{P:pf reducing} in the second equality.
Therefore $\pi(X_\un{n}(\L_F))q_F\subseteq\mathfrak{J}_{\L,F}^{(\pi,t)}$ for all $\un{n}\perp F$, from which it follows that $\pi(\Inv(\L)_F)q_F \subseteq \mathfrak{J}_{\L,F}^{(\pi,t)}$.
In total we have $\mathfrak{J}_{\L,F}^{(\pi,t)}=\mathfrak{J}_{\Inv(\L),F}^{(\pi,t)}$ for all $F\subseteq[d]$, and thus in particular $\mathfrak{J}_\L^{(\pi,t)}=\mathfrak{J}_{\Inv(\L)}^{(\pi,t)}$, finishing the proof.
\end{proof}

In order to choose $\L$ to be partially ordered, we need the following auxiliary lemma.

\begin{lemma}\label{L:JL rising}
Let $X$ be a strong compactly aligned product system with coefficients in a C*-algebra $A$.
Let $\L$ be a relative $2^d$-tuple of $X$ and let $(\pi,t)$ be a Nica-covariant representation of $X$.
Fix $D \subseteq F \subseteq[d]$.
If $a\in\bigcap\{\phi_\un{i}^{-1}(\K(X_\un{i}))\mid i\in F\}$ and $\pi(a)q_D\in\mathfrak{J}_\L^{(\pi,t)}$, then $\pi(a)q_F\in\mathfrak{J}_\L^{(\pi,t)}$.
\end{lemma}

\begin{proof}
Without loss of generality, we may assume that $D=[k]$ (with the convention that if $k=0$ then $D=\mt$) and $F=[\ell]$ for some $0\leq k\leq\ell\leq d$.
Note that we have
\[
\pi(a)q_D=\pi(a)+\sum\{(-1)^{|\un{n}|}\psi_\un{n}(\phi_\un{n}(a))\mid \un{0}\neq\un{n}\leq\un{1}_D\}
\]
because $a\in\bigcap\{\phi_\un{i}^{-1}(\K(X_\un{i}))\mid i\in F\}$, and hence in particular $a\in\bigcap\{\phi_\un{i}^{-1}(\K(X_\un{i}))\mid i\in D\}$ since $D\subseteq F$.
If $k = \ell$, then there is nothing to show.
If $k < \ell$, then we have 
\begin{align*}
\pi(a)q_Dp_\un{k+1} 
& =
\pi(a)p_\un{k+1}+\sum\{(-1)^{|\un{n}|}\psi_\un{n}(\phi_\un{n}(a))p_{\un{k+1}}\mid \un{0}\neq\un{n}\leq\un{1}_D\} \\
& =
\nor{\cdot}\text{-}\lim_\la\pi(a)p_{\un{k+1},\la}+\sum\{(-1)^{|\un{n}|}(\nor{\cdot}\text{-}\lim_\la\psi_\un{n}(\phi_\un{n}(a))p_{\un{k+1},\la})\mid\un{0}\neq\un{n}\leq\un{1}_D\} \\
& =
\nor{\cdot}\text{-}\lim_\la(\pi(a)+\sum\{(-1)^{|\un{n}|}\psi_\un{n}(\phi_\un{n}(a))\mid \un{0}\neq\un{n}\leq\un{1}_D\})p_{\un{k+1},\la} \\
& =
\nor{\cdot}\text{-}\lim_\la\pi(a)q_Dp_{\un{k+1},\la}\in \mathfrak{J}_\L^{(\pi,t)},
\end{align*}
where in the second line we have used Proposition \ref{P:prod cai} and in the last line we have used that $\pi(a)q_D\in\mathfrak{J}_\L^{(\pi,t)}$ by assumption.
To express $\pi(a)p_{\un{k+1}}$ as a norm-limit, we use that $a\in\phi_\un{k+1}^{-1}(\K(X_{\un{k+1}}))$, and so
\[
\pi(a)p_{\un{k+1}} 
= \psi_{\un{k+1}}(\phi_{\un{k+1}}(a)) 
= \nor{\cdot}\text{-}\lim_\la \pi(a) p_{\un{k+1}, \la},
\]
by the comments succeeding Proposition \ref{P:prod cai}.
Hence we have
\[
\pi(a)q_{D\cup\{k+1\}} = \pi(a)q_D - \pi(a)q_Dp_{\un{k+1}} \in \mathfrak{J}_\L^{(\pi,t)}.
\]
Since $a\in\bigcap\{\phi_\un{i}^{-1}(\K(X_\un{i}))\mid i\in D\cup\{k+1\}\}$, we can express $\pi(a)q_{D\cup\{k+1\}}$ as an alternating sum by Proposition \ref{P:prod cai}.
Consequently, we may apply the preceding argument for $k+2$. 
The assumption that $a\in\bigcap\{\phi_\un{i}^{-1}(\K(X_\un{i}))\mid i\in F\}$ ensures that we may argue in this way until $\{k+1,\dots, \ell\}$ has been exhausted, and we deduce that $\pi(a)q_F\in \mathfrak{J}_\L^{(\pi,t)}$, as required. 
\end{proof}

\begin{lemma}\label{L:e po}
Let $X$ be a strong compactly aligned product system with coefficients in a C*-algebra $A$. 
Let $\L$ be an $(E)$-$2^d$-tuple of $X$ and let $(\pi,t)$ be a Nica-covariant representation of $X$.
Then the $2^d$-tuple $\PO(\L)$ of $X$ defined by
\[
\PO(\L)_F:= \sum\{\sca{\L_D} \mid D\subseteq F\} \foral F \subseteq [d]
\]
is a partially ordered (E)-$2^d$-tuple consisting of ideals, and is such that 
\[
\L\subseteq\PO(\L)
\qand 
\fJ_{\L}^{(\pi,t)} = \fJ_{\PO(\L)}^{(\pi,t)}.
\]
If $\L$ is in addition invariant, then so is $\PO(\L)$.
\end{lemma}

\begin{proof}
Every $\PO(\L)_F$ is an ideal, being a finite sum of ideals in a C*-algebra. 
Fix $F \subseteq [d]$.
For each $D\subseteq F$, we have $\sca{\L_D}\subseteq\I_D\subseteq\I_F$, using that $\L$ is an (E)-$2^d$-tuple (and that $\I_D$ is an ideal) in the first inclusion, and the fact that $\I$ is partially ordered in the second. 
Hence $\PO(\L)_F\subseteq\I_F$ for all $F\subseteq[d]$, and thus $\PO(\L)$ is an (E)-$2^d$-tuple.
It is clear that $\PO(\L)$ is partially ordered, and $\L\subseteq\PO(\L)$ by construction.

Next we check that $\mathfrak{J}_{\L}^{(\pi,t)}=\mathfrak{J}_{\PO(\L)}^{(\pi,t)}$.
Since $\L\subseteq\PO(\L)$, we have $\mathfrak{J}_{\L}^{(\pi,t)}\subseteq \mathfrak{J}_{\PO(\L)}^{(\pi,t)}$. 
For the reverse inclusion, fix $F\subseteq[d]$ and $a\in\PO(\L)_F$. 
It suffices to show that $\pi(a)q_F\in\mathfrak{J}_{\L}^{(\pi,t)}$.
Let $a = \sum\{a_D\mid D\subseteq F\}$, where each $a_D\in\sca{\L_D}$, so that $\pi(a_D)q_D\in\mathfrak{J}_{\sca{\L}}^{(\pi,t)}=\fJ_{\L}^{(\pi,t)}$ for all $D\subseteq F$ by Lemma \ref{L: ideal}.
For each $D\subseteq F$, an application of Lemma \ref{L:JL rising} gives that $\pi(a_D)q_F\in\mathfrak{J}_{\L}^{(\pi,t)}$ and hence $\pi(a)q_F\in\fJ_{\L}^{(\pi,t)}$, as required.

Now suppose that $\L$ is in addition invariant.
To establish invariance of the family $\PO(\L)$, fix $F\subseteq[d], \un{n}\perp F$ and $a\in\PO(\L)_F$. 
By definition, we can write $a=\sum\{a_D\mid D\subseteq F\}$ for some $a_D\in\sca{\L_D}$, for each $D\subseteq F$. 
Since $\L$ is invariant, it follows that $\sca{X_\un{n},\sca{\L_D}X_\un{n}}\subseteq \sca{\L_D}$ and hence in particular $\sca{X_\un{n},a_DX_\un{n}}\subseteq\sca{\L_D}$ for each $D\subseteq F$. 
Consequently, we have
\[
\sca{X_\un{n}, a X_\un{n}}
\subseteq \sum\{\sca{X_\un{n},a_D X_\un{n}}\mid D\subseteq F\}
\subseteq \sum\{\sca{\L_D}\mid D\subseteq F\}=\PO(\L)_F,
\]
and therefore $\PO(\L)$ is invariant, finishing the proof.
\end{proof}

Letting $\L$ be an (E)-$2^d$-tuple and $(\pi,t)$ be a Nica-covariant representation, Lemma \ref{L:e inv} and Lemma \ref{L:e po} combine to give that $\PO(\Inv(\L))$ is an (E)-$2^d$-tuple that is invariant, partially ordered, consists of ideals and satisfies $\L\subseteq\PO(\Inv(\L))$ and $\fJ_{\L}^{(\pi,t)}=\fJ_{\PO(\Inv(\L))}^{(\pi,t)}$.

Next we focus on the maximal (E)-$2^d$-tuples.

\begin{definition}\label{D:m fam}
Let $X$ be a strong compactly aligned product system with coefficients in a C*-algebra $A$ and let $(\pi,t)$ be a Nica-covariant representation of $X$.
An \emph{(M)-$2^d$-tuple $\M$ of $X$ with respect to $(\pi,t)$} is a maximal $2^d$-tuple of $X$ with respect to $(\pi,t)$ that is also an (E)-$2^d$-tuple of $X$.
When we replace $(\pi,t)$ by $(\ol{\pi}_X,\ol{t}_X)$, we will refer to $\M$ simply as an \emph{(M)-$2^d$-tuple of $X$}.
\end{definition}

The ``M" of ``(M)-$2^d$-tuple" stands for ``Maximal".
Note that (M)-$2^d$-tuples (with respect to $(\pi,t)$) maximalise collections of (E)-$2^d$-tuples, instead of just relative $2^d$-tuples.
Indeed, if $\M$ is an (M)-$2^d$-tuple with respect to $(\pi,t)$, then it contains every relative $2^d$-tuple $\L$ of $X$ satisfying $\fJ_\L^{(\pi,t)}=\fJ_\M^{(\pi,t)}$ by Proposition \ref{P:maximal}.
Each such $\L$ satisfies $\L\subseteq\M\subseteq\I$ and is therefore an (E)-$2^d$-tuple, as claimed.

\begin{remark}\label{R:0+I}
Both $\L := \{\{0\}\}_{F\subseteq[d]}$ and $\I$ are (M)-$2^d$-tuples.
For $\I$ this is shown in Remark \ref{R:max I}.
On the other hand, $\L$ is clearly an (E)-$2^d$-tuple.
To see that $\L$ is maximal, let $\L'$ be a relative $2^d$-tuple such that $\mathfrak{J}_{\L'}^{(\ol{\pi}_X,\ol{t}_X)} = \mathfrak{J}_\L^{(\ol{\pi}_X,\ol{t}_X)} = \{0\}$ and $\L\subseteq\L'$. 
We may replace $(\ol{\pi}_X, \ol{t}_X)$ by the Fock representation $(\ol{\pi}, \ol{t})$.
For the projection $P_{\un{0}} \colon \F X \to X_{\un{0}}$, we have
\[
\phi_{\un{0}}(\L'_F) = P_{\un{0}} \big(\ol{\pi}(\L'_F) \ol{q}_F \big) P_{\un{0}} \subseteq P_{\un{0}} \big( \mathfrak{J}_{\L'}^{(\ol{\pi},\ol{t})} \big) P_{\un{0}} = \{0\},
\]
for all $F \subseteq [d]$, and thus $\L=\L'$ as required.

Moreover, in Proposition \ref{P:rep m fam} we have shown that $\L^{(\pi,t)}$ is an (M)-$2^d$-tuple whenever $(\pi,t)$ is injective and admits a gauge action.
\end{remark}

The (M)-$2^d$-tuples of $X$ parametrise the gauge-invariant ideals $\mathfrak{J}$ of $\N\T_X$ for which $\N\T_X/\mathfrak{J}$ lies between $\N\T_X$ and $\N\O_X$.
To prove this, we will require the following proposition.

\begin{proposition}\label{P:inj on J}
Let $X$ be a strong compactly aligned product system with coefficients in a C*-algebra $A$ and let $\mathfrak{J}\subseteq\N\T_X$ be a gauge-invariant ideal.
Then the following are equivalent:
\begin{enumerate}
\item $\ol{\pi}_X^{-1}(\mathfrak{J})=\{0\}$;
\item there exists a unique (M)-$2^d$-tuple of $X$ inducing $\mathfrak{J}$;
\item there exists an (E)-$2^d$-tuple of $X$ inducing $\mathfrak{J}$.
\end{enumerate}
\end{proposition}

\begin{proof}
First assume that $\ol{\pi}_X^{-1}(\mathfrak{J})=\{0\}$.
Let $Q_\mathfrak{J}\colon\N\T_X\to\N\T_X/\mathfrak{J}$ denote the quotient map.
Then the Nica-covariant representation $(Q_\mathfrak{J}\circ\ol{\pi}_X,Q_\mathfrak{J}\circ\ol{t}_X)$ is injective and admits a gauge action.
By Proposition \ref{P:inj rel}, we have
\[
\fJ = \ker Q_{\fJ} = \ker(Q_\mathfrak{J}\circ\ol{\pi}_X)\times(Q_\mathfrak{J}\circ\ol{t}_X)=\mathfrak{J}_{\L^{(Q_\mathfrak{J}\circ\ol{\pi}_X,Q_\mathfrak{J}\circ\ol{t}_X)}}^{(\ol{\pi}_X,\ol{t}_X)},
\]
since $(Q_\mathfrak{J}\circ\ol{\pi}_X)\times(Q_\mathfrak{J}\circ\ol{t}_X)=Q_\mathfrak{J}$.
Thus $\L^{(Q_\mathfrak{J}\circ\ol{\pi}_X,Q_\mathfrak{J}\circ\ol{t}_X)}$ induces $\mathfrak{J}$, and by Proposition \ref{P:rep m fam} it is an (M)-$2^d$-tuple of $X$.
By Proposition \ref{P:maximal} it is also unique, proving that (i) implies (ii).

Item (ii) implies item (iii) trivially. 
So assume that there exists an (E)-$2^d$-tuple $\L$ such that $\mathfrak{J}_{\L}^{(\ol{\pi}_X, \ol{t}_X)} = \mathfrak{J}$.
In particular $(\pi_X^\L, t_X^\L)$ is injective by Proposition \ref{P:NO inj}.
Hence, letting 
\[
Q_\fJ\colon \N\T_X\to\N\O(\L,X)=\N\T_X/ \fJ
\]
denote the quotient map, we have
\begin{align*}
\{0\}=\ker\pi_X^\L 
& = \ker Q_\fJ\circ\ol{\pi}_X 
= \{a\in A\mid \ol{\pi}_X(a)\in\mathfrak{J}\} 
= \ol{\pi}_X^{-1}(\mathfrak{J}),
\end{align*}
showing that item (iii) implies item (i), finishing the proof.
\end{proof}

\begin{remark}
Proposition \ref{P:inj on J} implies that the mapping
\begin{align*}
\{\M\mid \M \; \text{is an (M)-$2^d$-tuple of} \; X\} & \to \{\mathfrak{J}\subseteq\N\T_X\mid \mathfrak{J} \; \text{is a gauge-invariant ideal,} \; \ol{\pi}_X^{-1}(\mathfrak{J})=\{0\}\} \\
									    \M & \mapsto\mathfrak{J}_\M^{(\ol{\pi}_X,\ol{t}_X)}
\end{align*}
is a bijection.
\noindent
More specifically, the mappings
\begin{align*}
\M\mapsto\mathfrak{J}_\M^{(\ol{\pi}_X,\ol{t}_X)} & \; \text{for all (M)-$2^d$-tuples $\M$ of $X$}, \\
\mathfrak{J}\mapsto\L^{(Q_\mathfrak{J}\circ\ol{\pi}_X,Q_\mathfrak{J}\circ\ol{t}_X)} & \; \text{for all gauge-invariant ideals $\mathfrak{J}\subseteq\N\T_X$ satisfying $\ol{\pi}_X^{-1}(\mathfrak{J})=\{0\}$},
\end{align*}
are mutually inverse, where $Q_\mathfrak{J}\colon\N\T_X\to\N\T_X/\mathfrak{J}$ is the quotient map.
\end{remark}

The following Gauge-Invariant Uniqueness Theorem recovers \cite[Corollary 11.8]{Kat07} when $d=1$, and \cite[Theorem 4.2]{DK18} when $\L = \I$.

\begin{theorem}[$\bZ_+^d$-GIUT for (M)-$2^d$-tuples]\label{T:d GIUT M}
Let $X$ be a strong compactly aligned product system with coefficients in a C*-algebra $A$. 
Let $\L$ be an (M)-$2^d$-tuple of $X$ and $(\pi,t)$ be a Nica-covariant representation of $X$. 
Then $\N\O(\L,X)\cong\ca(\pi,t)$ via a canonical $*$-isomorphism if and only if $(\pi,t)$ admits a gauge action and $\L^{(\pi,t)}=\L$.
\end{theorem}

\begin{proof}
We start with the following remark.
If $\N\O(\L,X)\cong\ca(\pi,t)$ canonically, then $(\pi,t)$ is injective as $X$ embeds isometrically in $\N\O(\L,X)$.
On the other hand, if $\L^{(\pi,t)}=\L$ then $\ker \pi \equiv \L_\mt^{(\pi,t)} = \L_\mt = \{0\}$, as $\L$ is in particular an (E)-$2^d$-tuple.
Hence in both directions we may use that $(\pi,t)$ is injective.

Assume that $\N\O(\L,X)\cong\ca(\pi,t)$ via a canonical $*$-isomorphism $\Phi\colon \N\O(\L,X)\to\ca(\pi,t)$. 
For notational convenience, we will abbreviate $(\pi_X^\L,t_X^\L)$ as $(\tilde{\pi},\tilde{t})$.
Let $\be$ denote the gauge action of $(\tilde{\pi},\tilde{t})$. 
By defining $\gamma_{\un{z}}:=\Phi\circ\beta_{\un{z}}\circ\Phi^{-1}$ for each $\un{z}\in\bT^d$, we obtain a gauge action $\ga$ of $(\pi,t)$.

Next we show that $\L^{(\pi,t)}=\L$.
To this end, it suffices to show that $\mathfrak{J}_{\L^{(\pi,t)}}^{(\ol{\pi}_X,\ol{t}_X)}=\mathfrak{J}_\L^{(\ol{\pi}_X,\ol{t}_X)}$.
Then we can apply maximality of $\L^{(\pi,t)}$ (which holds by Proposition \ref{P:rep m fam}) and of $\L$ (which holds by assumption), together with uniqueness from item (ii) of Proposition \ref{P:inj on J}, to obtain the result.
To this end, fix $F\subseteq[d]$ and $a\in\L_F$.
Then $\tilde{\pi}(a)\tilde{q}_F=0$ since $(\tilde{\pi},\tilde{t})$ is an $\L$-relative CNP-representation.
Applying $\Phi$, we obtain $\pi(a)q_F = 0$, and therefore $a\in\L_F^{(\pi,t)}$ by definition.
Hence $\mathfrak{J}_\L^{(\ol{\pi}_X,\ol{t}_X)}\subseteq\mathfrak{J}_{\L^{(\pi,t)}}^{(\ol{\pi}_X,\ol{t}_X)}$.
Next, fix $F\subseteq[d]$ and $a\in\L_F^{(\pi,t)}$.
An application of (\ref{eq:out}) yields $\pi(a)q_F=0$.
Recall that $\tilde{\pi}=Q\circ\ol{\pi}_X$ and $\tilde{t}_\un{n}=Q\circ\ol{t}_{X,\un{n}}$ for all $\un{n}\in\bZ_+^d\setminus\{\un{0}\}$, where $Q\colon \N\T_X\to\N\O(\L,X)$ is the quotient map.
Hence, applying $\Phi^{-1}$ to $\pi(a)q_F=0$, we obtain $Q(\ol{\pi}_X(a)\ol{q}_{X,F}) = 0$, and therefore 
\[
\ol{\pi}_X(a)\ol{q}_{X,F}\in\ker Q=\mathfrak{J}_\L^{(\ol{\pi}_X,\ol{t}_X)}.
\]
Hence we have $\mathfrak{J}_{\L^{(\pi,t)}}^{(\ol{\pi}_X,\ol{t}_X)}\subseteq\mathfrak{J}_\L^{(\ol{\pi}_X,\ol{t}_X)}$, and thus $\mathfrak{J}_{\L^{(\pi,t)}}^{(\ol{\pi}_X,\ol{t}_X)}=\mathfrak{J}_\L^{(\ol{\pi}_X,\ol{t}_X)}$, as required.

The converse is an immediate application of Proposition \ref{P:inj rel}.
Indeed, since $(\pi,t)$ is injective and admits a gauge action, we obtain $\N\O(\L^{(\pi,t)}, X) \cong \ca(\pi,t)$ canonically, and the fact that $\L^{(\pi,t)}=\L$ then gives that $\N\O(\L^{(\pi,t)}, X) = \N\O(\L,X)$, finishing the proof.
\end{proof}

\subsection{Properties of ideals induced by relative $2^d$-tuples}\label{S:ideal rel fam}

The definition of (M)-$2^d$-tuples itself depends extensively on the associated gauge-invariant ideals of $\N\T_X$.
Hence, to obtain a genuinely useful parametrisation, we must characterise (M)-$2^d$-tuples in a way that does not (explicitly) depend on any gauge-invariant ideal of $\N\T_X$.
To this end, we need to study the properties of the gauge-invariant ideals arising from relative $2^d$-tuples.

When a relative $2^d$-tuple $\L$ consists of ideals and is invariant, the ideals $\mathfrak{J}_{\L,F}^{(\pi,t)}$ take on a particularly simple form akin to the one obtained in \cite[Proposition 4.6]{DK18}.

\begin{proposition}\label{P: JLF}
Let $X$ be a strong compactly aligned product system with coefficients in a C*-algebra $A$. 
Let $\L$ be a relative $2^d$-tuple of $X$ and let $(\pi,t)$ be a Nica-covariant representation of $X$.
If $\L$ is invariant and consists of ideals, then
\begin{equation} \label{Eq: JLF}
\fJ_{\L, F}^{(\pi,t)}
= 
\ol{\spn} \{t_{\un{n}}(X_\un{n}) \pi(\L_F) q_F t_{\un{m}}(X_\un{m})^* \mid \un{n}, \un{m} \in \bZ_+^d\}
\foral
F \subseteq [d].
\end{equation}
\end{proposition}

\begin{proof}
The proof follows the arguments of \cite[Proposition 4.6]{DK18}.
In short, fix $F \subseteq [d]$ and denote the right hand side of (\ref{Eq: JLF}) by $\mathfrak{J}_F$. 
The fact that $\mathfrak{J}_F\subseteq\mathfrak{J}_{\L,F}^{(\pi,t)}$ is immediate by definition.
To prove that $\mathfrak{J}_{\L,F}^{(\pi,t)}\subseteq \mathfrak{J}_F$, it suffices to show that $\mathfrak{J}_F$ is an ideal of $\ca(\pi,t)$, as $\fJ_F$ contains the generators of $\mathfrak{J}_{\L,F}^{(\pi,t)}$.
Since $\mathfrak{J}_F$ is selfadjoint, it suffices to show that $\mathfrak{J}_F\ca(\pi,t)\subseteq \mathfrak{J}_F$. 

It is clear that $\mathfrak{J}_Ft_\un{r}(X_\un{r})^*\subseteq \mathfrak{J}_F$ for all $\un{r}\in\bZ_+^d$.
Hence it remains to show that $\mathfrak{J}_F t_\un{r}(X_\un{r}) \subseteq \mathfrak{J}_F$ for all $\un{r}\in\bZ_+^d$.
For $\un{n},\un{m},\un{r}\in\bZ_+^d$, we have
\[
t_\un{n}(X_\un{n})\pi(\L_F)q_Ft_\un{m}(X_\un{m})^*t_\un{r}(X_\un{r})
\subseteq
\left[t_\un{n}(X_\un{n})\pi(\L_F)q_F t_{\un{m}'}(X_{\un{m}'})t_{\un{r}'}(X_{\un{r}'})^*\right],
\]
where $\un{m}'=-\un{m}+\un{m}\vee\un{r}$ and $\un{r}'=-\un{r}+\un{m}\vee\un{r}$, due to Nica-covariance. 
If $\un{m}'\not\perp F$, then
\[
t_\un{n}(X_\un{n})\pi(\L_F)q_F t_{\un{m}'}(X_{\un{m}'}) t_{\un{r}'}(X_{\un{r}'})^* = \{0\}
\subseteq
\mathfrak{J}_F
\]
by Proposition \ref{P:pf reducing}.
If $\un{m}'\perp F$, then positive invariance of $\L_F$ for $X_{\un{m}'}$ and Lemma \ref{L:IX to XI} together yield
\[
\pi(\L_F)t_{\un{m}'}(X_{\un{m}'}) = t_{\un{m}'}( \L_F X_{\un{m}'}) \subseteq t_{\un{m}'}(X_{\un{m}'} \L_F) =  t_{\un{m}'}(X_{\un{m}'})\pi(\L_F).
\]
Then we have
\begin{align*}
\left[t_\un{n}(X_\un{n})\pi(\L_F)q_Ft_{\un{m}'}(X_{\un{m}'})t_{\un{r}'}(X_{\un{r}'})^*\right] 
& \subseteq
\left[t_\un{n}(X_\un{n})\pi(\L_F)t_{\un{m}'}(X_{\un{m}'})q_Ft_{\un{r}'}(X_{\un{r}'})^*\right] \\
& \subseteq
\left[t_{\un{n}+\un{m}'}(X_{\un{n}+\un{m}'})\pi(\L_F)q_Ft_{\un{r}'}(X_{\un{r}'})^*\right]\subseteq \mathfrak{J}_F,
\end{align*}
where we have used Proposition \ref{P:pf reducing} in the first inclusion.
Hence in all cases we have 
\[
t_\un{n}(X_\un{n}) \pi(\L_F)q_F t_\un{m}(X_\un{m})^*t_\un{r}(X_\un{r}) \subseteq \mathfrak{J}_F,
\]
and thus $\mathfrak{J}_Ft_\un{r}(X_\un{r})\subseteq \mathfrak{J}_F$ for all $\un{r}\in\bZ_+^d$, as required.
\end{proof}

By requiring that $\L$ is partially ordered, we can say more. 

\begin{proposition}\label{P: cond ex}
Let $X$ be a strong compactly aligned product system with coefficients in a C*-algebra $A$.
Let $\L$ be a relative $2^d$-tuple of $X$ that is invariant, partially ordered, and consists of ideals.
Let $(\pi,t)$ be a Nica-covariant representation of $X$.
Then
\begin{align*}
q_F\mathfrak{J}_{\L,D}^{(\pi,t)}q_F 
\subseteq 
\ol{\spn}\{t_\un{n}(X_\un{n})\pi(\L_{F \cup D})q_{F\cup D}t_\un{m}(X_\un{m})^*\mid \un{n},\un{m}\perp F\}
\subseteq 
\mathfrak{J}_{\L,F\cup D}^{(\pi,t)}
\foral F, D \subseteq [d].
\end{align*}
Moreover, we have
\[
q_F \fJ_{\L}^{(\pi,t)} q_F
\subseteq
\sum\{\fJ_{\L, D}^{(\pi, t)} \mid F\subseteq D\subseteq[d]\} \foral F\subseteq[d].
\]
\end{proposition}

\begin{proof}
For the first part, fix $F, D \subseteq[d]$. 
Then we have 
\begin{align*}
q_F\mathfrak{J}_{\L,D}^{(\pi,t)}q_F 
& \subseteq
\ol{\spn}\{q_Ft_\un{n}(X_\un{n})\pi(\L_D)q_Dt_\un{m}(X_\un{m})^*q_F\mid \un{n},\un{m}\in\bZ_+^d\} \\
& =\ol{\spn}\{t_\un{n}(X_\un{n})\pi(\L_D)q_{F\cup D}t_\un{m}(X_\un{m})^*\mid \un{n},\un{m}\perp F\}\\
& \subseteq 
\ol{\spn}\{t_\un{n}(X_\un{n})\pi(\L_{F \cup D})q_{F\cup D}t_\un{m}(X_\un{m})^*\mid \un{n},\un{m}\perp F\}
\subseteq 
\mathfrak{J}_{\L,F\cup D}^{(\pi,t)},
\end{align*}
where we have used Propositions \ref{P:pf reducing} and \ref{P: JLF}, that $q_F$ and $q_D$ are commuting projections, and that $\L$ is partially ordered (so $\L_D\subseteq\L_{F\cup D}$).

For the second part, fix $F\subseteq[d]$. 
Then applying the first part for all $D \subseteq [d]$ yields
\begin{align*}
q_F\mathfrak{J}_\L^{(\pi,t)}q_F
& =
\sum\{q_F\mathfrak{J}_{\L,D}^{(\pi,t)}q_F\mid D\subseteq[d]\}
\subseteq
\sum\{\mathfrak{J}_{\L,F\cup D}^{(\pi,t)}\mid D\subseteq[d]\}
=
\sum\{\mathfrak{J}_{\L,D}^{(\pi,t)}\mid F\subseteq D\subseteq[d]\},
\end{align*}
and the proof is complete.
\end{proof}

Let $I$ be an ideal of $A$ and fix $F \subseteq [d]$.
For $a\in A$, we write 
\[
\lim_{\un{m}\perp F}\|\phi_\un{m}(a)+\K(X_\un{m} I)\|=0
\]
if for any $\varepsilon>0$ there exists $\un{m}\perp F$ such that
\[
\|\phi_\un{n}(a)+\K(X_\un{n} I)\|<\varepsilon \foral \un{n}\geq\un{m} \; \text{satisfying} \; \un{n}\perp F.
\]

When working with a $2^d$-tuple $\L$ satisfying certain compatibility relations, the limit condition can be recast in a simpler form.

\begin{lemma}\label{L:lim perp}
Let $X$ be a strong compactly aligned product system with coefficients in a C*-algebra $A$.
Let $\L$ be an invariant $2^d$-tuple of $X$ which consists of ideals and satisfies
\[
\L_F\subseteq\bigcap\{\phi_\un{i}^{-1}(\K(X_\un{i}))\mid i\in[d]\} \foral F\subseteq[d].
\]
Then, for each $F\subseteq[d]$ and $a\in A$, we have that $\lim_{\un{m}\perp F}\|\phi_\un{m}(a)+\K(X_\un{m}\L_F)\|=0$ if and only if for any $\varepsilon>0$, there exists $\un{m}\perp F$ and $k_\un{m} \in\K(X_\un{m}\L_F)$ such that $\|\phi_\un{m}(a)+k_\un{m}\|<\varepsilon$.
\end{lemma}

\begin{proof}
The forward implication is immediate by definition. 
So assume that for any $\varepsilon>0$ there exists $\un{m}\perp F$ and $k_\un{m} \in \K(X_\un{m}\L_F)$ such that $\|\phi_\un{m}(a)+k_\un{m}\|<\varepsilon$. 
Fix $\varepsilon>0$ and a corresponding $\un{m}\perp F$ and $k_\un{m}\in\K(X_\un{m}\L_F)$. 
If $F=[d]$, then $\un{m} \perp F$ implies that $\un{m}=\un{0}$, and there is nothing to show. 
So assume that $F\subsetneq [d]$. 
Without loss of generality, we may assume that $\un{m}\neq\un{0}$. 
Indeed, if $\un{m}=\un{0}$ then $k_\un{0} \in\K(A\L_F)$ and hence $k_\un{0} = \phi_\un{0}(b)$ for some $b\in \L_F$.
Note that
\[
\|a+b\| = \|\phi_\un{0}(a)+\phi_\un{0}(b)\|=\|\phi_\un{0}(a)+k_\un{0}\|<\varepsilon. 
\]
Fix $i\in F^c$.
Since $b\in\L_F$, we have $\phi_\un{i}(b)\in\K(X_\un{i})$ by assumption, and $\sca{X_\un{i},bX_\un{i}}\subseteq\L_F$ by invariance of $\L$. 
Thus an application of (\ref{Eq: comp}) gives that $\phi_\un{i}(b)\in\K(X_\un{i}\L_F)$, and 
\[
\|\phi_\un{i}(a)+\phi_\un{i}(b)\|=\|\phi_\un{i}(a+b)\|\leq\|a+b\|<\varepsilon.
\]
So we may assume that $\un{m}\neq\un{0}$.

Next, take $\un{n}\geq\un{m}$ with $\un{n}\perp F$. 
We will show that $\|\phi_\un{n}(a)+\K(X_\un{n}\L_F)\|<\varepsilon$. 
We write $\un{n}=\un{m}+\un{r}$ for some $\un{r}\perp F$, and without loss of generality we may assume that $\un{r}\neq\un{0}$. 
By Lemma \ref{L:inv comp}, we have $\phi_\un{r}(\L_F)\subseteq\K(X_\un{r}\L_F)$.
An application of Corollary \ref{C:comp to comp} then yields $k_\un{m}\otimes\text{id}_{X_\un{r}}\in\K((X_\un{m}\otimes_AX_\un{r})\L_F)$, and so $\iota_\un{m}^\un{n}(k_\un{m})\in\K(X_\un{n} \L_F)$. 
We then obtain
\begin{align*}
\|\phi_\un{n}(a)+\K(X_\un{n}\L_F)\| 
& \leq 
\|\phi_\un{n}(a)+\iota_\un{m}^\un{n}(k_\un{m})\| 
=
\|\iota_\un{m}^\un{n}(\phi_\un{m}(a)+k_\un{m})\| 
\leq
\|\phi_\un{m}(a)+k_\un{m}\|<\varepsilon,
\end{align*}
from which it follows that $\lim_{\un{m}\perp F}\|\phi_\un{m}(a)+\K(X_\un{m}\L_F)\|=0$, as required.
\end{proof}

The limit condition appears naturally in the study of ideals of $\N\T_X$ that are induced by invariant, partially ordered relative $2^d$-tuples that consist of ideals.

\begin{proposition}\label{P:JL lim}
Let $X$ be a strong compactly aligned product system with coefficients in a C*-algebra $A$. 
Let $\L$ be a relative $2^d$-tuple of $X$ that is invariant, partially ordered and consists of ideals.
Fix $F\subsetneq[d]$ and let $a\in \bigcap\{\phi_{\un{i}}^{-1}(\K(X_{\un{i}})) \mid i\in[d] \}$.
Then the following are equivalent:
\begin{enumerate}
\item $\ol{\pi}_X(a)\ol{q}_{X,F}\in \fJ_\L^{(\ol{\pi}_X, \ol{t}_X)}$; 
\item  $\ol{\pi}_X(a)\ol{q}_{X,F}\in\sum\{\fJ_{\L, D}^{(\ol{\pi}_X, \ol{t}_X)}\mid F\subseteq D\subseteq[d]\}$.
\end{enumerate}
Furthermore, if one (and hence both) of (i) and (ii) holds, then
\[
\lim_{\un{m}\perp F}\|\phi_\un{m}(a)+\K(X_\un{m}\L_F)\|=0.
\]
\end{proposition}

\begin{proof}
Without loss of generality, we may replace $(\ol{\pi}_X,\ol{t}_X)$ by the Fock representation $(\ol{\pi},\ol{t})$ and write $\mathfrak{J}_\L \equiv \fJ_\L^{(\ol{\pi},\ol{t})}$, and $\mathfrak{J}_{\L,F} \equiv \fJ_{\L,F}^{(\ol{\pi},\ol{t})}$ for all $F \subseteq [d]$.
The implication [(ii)$\Rightarrow$(i)] is immediate.
Conversely, if $\ol{\pi}(a)\ol{q}_{F}\in \fJ_\L$, then we have 
\[
\ol{\pi}(a)\ol{q}_F = \ol{q}_F(\ol{\pi}(a)\ol{q}_F) \ol{q}_F \in \ol{q}_F \fJ_{\L} \ol{q}_F \subseteq \sum\{\mathfrak{J}_{\L,D}\mid F\subseteq D\subseteq[d]\},
\]
using Proposition \ref{P:pf reducing} (and the fact that $\ol{q}_F$ is a projection) in the first equality, and Proposition \ref{P: cond ex} in the final inclusion.
This establishes the equivalence of (i) and (ii).

Now suppose that one (and hence both) of (i) and (ii) holds.
Condition (ii) then yields
\begin{align*}
\ol{\pi}(a)\ol{q}_F
& \in
\sum\{\ol{q}_F \fJ_{\L, D} \ol{q}_F \mid F\subseteq D\subseteq[d]\} \\
& \subseteq
\ol{\spn}\{\ol{t}_\un{n}(X_\un{n})\ol{\pi}(\L_D)\ol{q}_D\ol{t}_\un{m}(X_\un{m})^*\mid \un{n},\un{m}\perp F, F\subseteq D\subseteq[d]\},
\end{align*}
using Proposition \ref{P: cond ex} for the containment.
Let $\ga$ be the gauge action of $(\ol{\pi},\ol{t})$ and let $E_\ga\colon \ca(\ol{\pi},\ol{t})\to\ca(\ol{\pi},\ol{t})^\ga$ be the associated faithful conditional expectation.
Therefore, since $\ol{\pi}(a)\ol{q}_F\in\ca(\ol{\pi},\ol{t})^\ga$, applying $E_\ga$ yields
\begin{equation}\label{Eq:fix perp}
\ol{\pi}(a)\ol{q}_F\in\ol{\spn}\{\ol{t}_\un{n}(X_\un{n})\ol{\pi}(\L_D)\ol{q}_D\ol{t}_\un{n}(X_\un{n})^*\mid \un{n}\perp F, F\subseteq D\subseteq[d]\}.
\end{equation}
Fix $\varepsilon>0$. 
By Lemma \ref{L:lim perp}, it is sufficient to find $\un{m}\perp F$ and $k_\un{m} \in \K(X_\un{m}\L_F)$ such that $\|\phi_\un{m}(a)+k_\un{m}\|<\varepsilon$.
By (\ref{Eq:fix perp}), for each $F\subseteq D\subseteq[d]$ there exists 
\[
c_D
\in
\spn\{\ol{t}_\un{n}(X_\un{n})\ol{\pi}(\L_D)\ol{q}_D\ol{t}_\un{n}(X_\un{n})^*\mid \un{n} \perp F\}
\]
such that
\[
\|\ol{\pi}(a)\ol{q}_F-\sum\{c_D\mid F\subseteq D\subseteq[d]\}\|<\varepsilon.
\]
Moreover, we may assume that
\[
c_D=\sum\{\sum_{j=1}^N \ol{t}_\un{n}(\xi_{\un{n},D}^j)\ol{\pi}(b_{\un{n}, D}^j)\ol{q}_D\ol{t}_\un{n}(\eta_{\un{n},D}^j)^* \mid \un{0}\leq\un{n}\leq m\cdot\un{1}_{F^c}\},
\]
for some $m,N\in\bN$, $\xi_{\un{n},D}^j,\eta_{\un{n},D}^j\in X_\un{n}$ and $b_{\un{n}, D}^j\in\L_D$, for all $j\in [N]$ and $\un{0}\leq\un{n}\leq m\cdot\un{1}_{F^c}$.
We may use the same $m$ and $N$ for every $c_D$ by padding with zeros, if necessary.
Fix the projection
\[
P \colon \F X \to X_{(m+1)\cdot\un{1}_{F^c}}.
\]

First we claim that $c_DP=0$ whenever $F\subsetneq D\subseteq[d]$, so that
\begin{align*}
\|\ol{\pi}(a)\ol{q}_FP-c_FP\| 
& =
\|(\ol{\pi}(a)\ol{q}_F-\sum\{c_D\mid F\subseteq D\subseteq[d]\})P\| \\
& \leq
\|\ol{\pi}(a)\ol{q}_F-\sum\{c_D\mid F\subseteq D\subseteq[d]\}\|<\varepsilon.
\end{align*}
Indeed, for each $\un{0}\leq\un{n}\leq m\cdot\un{1}_{F^c}$ and each $j\in[N]$, we have that $\ol{t}_\un{n}(\eta_{\un{n},D}^j)^*P$ has image in $X_\un{r}$ for some $\un{r} \neq \un{0}$ with $\supp\un{r}=F^c$, since $m<m+1$.
Note also that $D\cap F^c\neq\mt$ since $F \subsetneq D$, and thus
\[
\ol{q}_D\ol{t}_\un{n}(\eta_{\un{n},D}^j)^*P (\F X) \subseteq \ol{q}_D \sum \{ X_\un{r} \mid \supp \un{r} = F^c \}= \{0\}.
\]
Consequently 
\[
\ol{t}_\un{n}(\xi_{\un{n},D}^j)\ol{\pi}(b_{\un{n}, D}^j)\ol{q}_D\ol{t}_\un{n}(\eta_{\un{n},D}^j)^*P
=
0,
\foral \un{0}\leq\un{n}\leq m\cdot\un{1}_{F^c}, j\in[N],
\]
and hence $c_DP=0$.

Next we examine the terms of $c_FP$.
To this end, fix $\un{0}\neq\un{n}\leq m\cdot\un{1}_{F^c}$ and $j\in[N]$.
Take $\ze_{(m+1)\cdot\un{1}_{F^c}} \in X_{(m+1)\cdot\un{1}_{F^c}}$ such that 
\[
\ze_{(m+1)\cdot\un{1}_{F^c}}=\ze_\un{n} \ze_{(m+1)\cdot\un{1}_{F^c}-\un{n}}
\textup{ for some }
\ze_\un{n}\in X_\un{n},
\ze_{(m+1)\cdot\un{1}_{F^c}-\un{n}}\in X_{(m+1)\cdot\un{1}_{F^c}-\un{n}}.
\]
Note that $\supp((m+1)\cdot\un{1}_{F^c} - \un{n}) = F^c$, and so we have
\begin{align*}
\ol{t}_\un{n}(\xi_{\un{n},F}^j)\ol{\pi}(b_{\un{n},F}^j)\ol{q}_F\ol{t}_\un{n}(\eta_{\un{n},F}^j)^*P\ze_{(m+1)\cdot\un{1}_{F^c}} 
& =
\ol{t}_\un{n}(\xi_{\un{n},F}^j)\ol{\pi}(b_{\un{n},F}^j)\ol{q}_F\ol{t}_\un{n}(\eta_{\un{n},F}^j)^*\ze_{(m+1)\cdot\un{1}_{F^c}} \\
& =
\xi_{\un{n},F}^j \big(\phi_{(m+1)\cdot\un{1}_{F^c}-\un{n}}(b_{\un{n},F}^j\langle \eta_{\un{n},F}^j,\ze_\un{n} \rangle)\ze_{(m+1)\cdot\un{1}_{F^c}-\un{n}} \big) \\
& =
\big(\Theta_{\xi_{\un{n},F}^j b_{\un{n},F}^j, \eta_{\un{n},F}^j} (\ze_\un{n}) \big)\ze_{(m+1)\cdot\un{1}_{F^c}-\un{n}} \\
& =
\iota_\un{n}^{(m+1)\cdot\un{1}_{F^c}}(\Theta_{\xi_{\un{n},F}^j b_{\un{n},F}^j,\eta_{\un{n},F}^j})\ze_{(m+1)\cdot\un{1}_{F^c}}.
\end{align*}
We deduce that
\[
\ol{t}_\un{n}(\xi_{\un{n},F}^j)\ol{\pi}(b_{\un{n},F}^j)\ol{q}_F\ol{t}_\un{n}(\eta_{\un{n},F}^j)^*P
=
\iota_\un{n}^{(m+1)\cdot\un{1}_{F^c}}(\Theta_{\xi_{\un{n},F}^j b_{\un{n},F}^j,\eta_{\un{n},F}^j}),
\]
where we view the right hand side as an element of $\L(\F X)$ in the way described in Subsection \ref{Ss:prod sys}.
Arguing in a similar (and simpler) way when $\un{n}=\un{0}$, we deduce that
\[
c_FP=\phi_{(m+1)\cdot\un{1}_{F^c}}(b)+\sum\{\sum_{j=1}^N\iota_\un{n}^{(m+1)\cdot\un{1}_{F^c}}(\Theta_{\xi_{\un{n},F}^j b_{\un{n},F}^j,\eta_{\un{n},F}^j})\mid \un{0}\neq\un{n}\leq m\cdot\un{1}_{F^c}\},
\]
for some $b\in\L_F$.
By Lemma \ref{L:inv comp}, we have
\[
\phi_{(m+1)\cdot\un{1}_{F^c}}(b)\in\K(X_{(m+1)\cdot\un{1}_{F^c}}\L_F)
\]
and 
\[
\phi_{(m+1)\cdot\un{1}_{F^c}-\un{n}}(b_{\un{n},F}^j)\in\K(X_{(m+1)\cdot\un{1}_{F^c}-\un{n}}\L_F) \foral \un{0}\neq\un{n}\leq m\cdot\un{1}_{F^c}, j\in[N].
\]
An application of Lemma \ref{L:com id} gives that all of the summands of $c_FP$ belong to $\K(X_{(m+1)\cdot\un{1}_{F^c}}\L_F)$.
Therefore we have $c_FP \in \K(X_{(m+1)\cdot\un{1}_{F^c}}\L_F) \hookrightarrow \L(\F X)$, which satisfies
\[
\|\phi_{(m+1)\cdot\un{1}_{F^c}}(a)-c_FP\|=\|\ol{\pi}(a)\ol{q}_FP-c_FP\|<\varepsilon,
\]
where we use that 
\[
\phi_{(m+1)\cdot\un{1}_{F^c}}(a) = \ol{\pi}(a)P  = \ol{\pi}(a)\ol{q}_FP
\]
in the identification $\L(X_{(m+1)\cdot\un{1}_{F^c}}) \hookrightarrow \L(\F X)$.
Hence we have found $\un{m}:=(m+1)\cdot\un{1}_{F^c}\perp F$ and $k_\un{m}:=-c_FP \in \K(X_{\un{m}}\L_F)$ such that $\|\phi_{\un{m}}(a)+k_\un{m}\|<\varepsilon$, as required.
\end{proof}

\begin{remark}
Proposition \ref{P:JL lim} holds for $F=[d]$ as well, though the second claim requires a different argument.
Specifically, we compress $\ol{\pi}(a)\ol{q}_{[d]}\in\fJ_{\L,[d]}$ at the $\un{0}$-th entry to deduce that $a\in\L_{[d]}$.
Then $\phi_\un{0}(a)\in\K(A\L_{[d]})$, and so $\lim_{\un{m}\perp[d]}\|\phi_\un{m}(a)+\K(X_\un{m}\L_{[d]})\|=0$, as required.
\end{remark}

\subsection{Constructing (M)-$2^d$-tuples}

In this subsection we show how the maximal $2^d$-tuple inducing $\fJ_{\L}^{(\ol{\pi}_X,\ol{t}_X)}$, for $\L$ an (E)-$2^d$-tuple, is constructed.
The following definition is motivated in part by Proposition \ref{P:JL lim}.

\begin{definition}\label{D:L1}
Let $X$ be a strong compactly aligned product system with coefficients in a C*-algebra $A$ and let $\L$ be a $2^d$-tuple of $X$ that consists of ideals.
Fixing $\mt \neq F \subsetneq [d]$, we define
\[
\L_{\inv, F} := \bigcap_{\un{m}\perp F}X_\un{m}^{-1}(\cap_{F\subsetneq D}\L_D)
\qand
\L_{\lim, F} := \{a\in A \mid \lim_{\un{m}\perp F}\|\phi_\un{m}(a)+\K(X_\un{m}\L_F)\|=0\}.
\]
If $\L$ is in addition an (E)-$2^d$-tuple of $X$, then we define the $2^d$-tuple $\L^{(1)}$ of $X$ by
\[
\L_F^{(1)} 
:=
\begin{cases}
\{0\} & \text{ if } F = \mt, \\
\I_F \cap \L_{\inv, F} \cap \L_{\lim, F} & \text{ if } \mt\neq F \subsetneq [d], \\
\L_{[d]} & \text{ if } F = [d].
\end{cases}
\]
\end{definition}

Observe that each $\L_F^{(1)}$ is non-empty, as $0\in\L_F^{(1)}$.
Note also that $\L^{(1)}$ is an (E)-$2^d$-tuple by construction.
We will show that $\L^{(1)}$ consists of ideals, and so we can write $\L^{(k+1)} := (\L^{(k)})^{(1)}$ for $k \in \bN$ (by convention we set $\L^{(0)} := \L$).
To this end, first we note that $\L_{\inv, F}$ is an intersection of ideals and is thus an ideal itself.
To address $\L_{\lim, F}$, we have the following proposition.

\begin{proposition}\label{P:lim ideal}
Let $X$ be a strong compactly aligned product system with coefficients in a C*-algebra $A$. 
Let $\L$ be a $2^d$-tuple of $X$ that consists of ideals and fix $\mt \neq F \subsetneq [d]$.
Then we have that $\L_{\lim, F}$ is an ideal.
\end{proposition}

\begin{proof}
Firstly, we define
\[
B_F :=\bigoplus_{\un{m}\perp F}\L(X_\un{m})/\K(X_\un{m}\L_F)
\qand
c_0(B_F):=\{(S_{\un{m}})_{\un{m} \perp F}\in B_F\mid \lim_{\un{m} \perp F}\|S_{\un{m}}\|=0\},
\]
and observe that $c_0(B_F)$ is an ideal in $B_F$.
Consider the $*$-homomorphism $\psi_F$ defined by
\begin{gather*}
\psi_F \colon A\to\bigoplus_{\un{m}\perp F}\L(X_\un{m})\to B_F\to B_F/c_0(B_F); \\
\psi_F \colon a\mapsto (\phi_\un{m}(a))_{\un{m}\perp F}\mapsto(\phi_\un{m}(a)+\K(X_\un{m}\L_F))_{\un{m}\perp F}\mapsto(\phi_\un{m}(a)+\K(X_\un{m}\L_F))_{\un{m}\perp F}+c_0(B_F).
\end{gather*}
It is now a standard C*-result that
\[
\ker \psi_F = \{a\in A\mid \lim_{\un{m}\perp F}\|\phi_\un{m}(a)+\K(X_\un{m}\L_F)\|=0\},
\]
and thus $\L_{\lim, F} = \ker \psi_F$ is an ideal, as required.
\end{proof}

\begin{lemma}\label{L:L1 ideal inv}
Let $X$ be a strong compactly aligned product system with coefficients in a C*-algebra $A$. 
Let $\L$ be an $(E)$-$2^d$-tuple of $X$ that consists of ideals.
Then $\L^{(1)}$ is an (E)-$2^d$-tuple of $X$ that consists of ideals.
If $\L$ is invariant, then so is $\L^{(1)}$.
\end{lemma}

\begin{proof}
We have already remarked that $\L^{(1)}$ is an (E)-$2^d$-tuple.
The fact that $\L^{(1)}$ consists of ideals follows from Proposition \ref{P:lim ideal} and the discussion preceding it.

Now assume that $\L$ is invariant.
The invariance condition for $\L^{(1)}$ clearly holds when $F=\mt$ or $F=[d]$, so assume that $\mt\neq F\subsetneq[d]$.
Fix $\un{n} \perp F$, $\xi_\un{n},\eta_\un{n}\in X_\un{n}$, and $a \in \L_F^{(1)}$.
It suffices to show that 
\[
\sca{\xi_\un{n},\phi_\un{n}(a)\eta_\un{n}} \in \L_F^{(1)} = \I_F \cap \L_{\inv, F} \cap \L_{\lim, F}.
\]

If $\un{n}=\un{0}$ there is nothing to show, so assume that $\un{n}\neq\un{0}$. 
First we have $\sca{\xi_\un{n},\phi_\un{n}(a)\eta_\un{n}}\in\I_F$ by \cite[Proposition 2.7]{DK18}, since $a \in \I_F$. 
Now fix $\un{m}\perp F$ and $\xi_\un{m},\eta_\un{m}\in X_\un{m}$. 
Then 
\begin{align*}
\sca{\xi_\un{m},\phi_\un{m}(\sca{\xi_\un{n},\phi_\un{n}(a)\eta_\un{n}})\eta_\un{m}} 
& =
\sca{\xi_\un{n}\xi_\un{m},(\phi_\un{n}(a)\eta_\un{n})\eta_\un{m}} 
=\sca{\xi_\un{n}\xi_\un{m},\phi_{\un{n}+\un{m}}(a)(\eta_\un{n}\eta_\un{m})}
\in\cap_{F\subsetneq D}\L_D,
\end{align*} 
since $a\in\L_F^{(1)}$ and thus in particular $a\in \L_{\inv, F}$, noting that $\un{n}+\un{m}\perp F$. 
This proves that $\sca{\xi_\un{n},\phi_\un{n}(a)\eta_\un{n}}\in \L_{\inv, F}$. 

It remains to check that $\text{lim}_{\un{m}\perp F}\|\phi_\un{m}(\sca{\xi_\un{n},\phi_\un{n}(a)\eta_\un{n}})+\K(X_\un{m}\L_F)\|=0$. 
To this end, fix $\varepsilon>0$. 
Since $a\in \L_F^{(1)} \subseteq \L_{\lim, F}$, there exist $\un{m}\perp F$ and $k_{\un{m}}\in\K(X_\un{m}\L_F)$ such that 
\[
\|\phi_\un{m}(a)+k_{\un{m}}\|<\frac{\varepsilon}{\|\xi_\un{n}\|\cdot\|\eta_\un{n}\|}.
\] 
Without loss of generality, we may assume that $\un{m}\neq\un{0}$ (if $\un{m}=\un{0}$, then argue as in the proof of Lemma \ref{L:lim perp} to replace $\un{m}$ by $\un{i}$ for some $i\in F^c$). 
We have 
\[
\iota_\un{m}^{\un{m}+\un{n}}(k_{\un{m}})=u_{\un{m},\un{n}}(k_{\un{m}} \otimes\text{id}_{X_\un{n}})u_{\un{m},\un{n}}^*,
\]
and $k_{\un{m}} \otimes\text{id}_{X_\un{n}}\in\K((X_\un{m}\otimes_A X_\un{n})\L_F)$ by Corollary \ref{C:comp to comp}, noting that $\phi_\un{n}(\L_F)\subseteq\K(X_\un{n}\L_F)$ by Lemma \ref{L:inv comp}. 
It follows that 
\[
\iota_\un{m}^{\un{m}+\un{n}}(k_{\un{m}})\in\K(X_{\un{m}+\un{n}}\L_F).
\]
Next, define $\tau(\xi_\un{n}) \in\L(X_\un{m},X_{\un{m}+\un{n}})$ and $\tau(\eta_\un{n}) \in\L(X_\un{m},X_{\un{m}+\un{n}})$ by
\[
\tau(\xi_\un{n}) \xi_{\un{m}} = \xi_{\un{n}} \xi_{\un{m}}
\qand
\tau(\eta_{\un{n}}) \xi_{\un{m}} = \eta_{\un{n}} \xi_{\un{m}}
\]
for all $\xi_\un{m} \in X_{\un{m}}$, and observe that $\|\tau(\xi_\un{n})\|\leq\|\xi_\un{n}\|$ and $\|\tau(\eta_\un{n})\|\leq\|\eta_\un{n}\|$.
We then have 
\begin{equation} \label{eq:k claim}
\tau(\xi_\un{n})^*\iota_{\un{m}}^{\un{m}+\un{n}}(k_{\un{m}})\tau(\eta_\un{n})\in\K(X_\un{m}\L_F).
\end{equation}
Indeed, taking $\ze_\un{n},\ze_\un{n}'\in X_\un{n}, \ze_\un{m},\ze_\un{m}'\in X_\un{m}$ and $b\in\L_F$, we compute
\begin{align*}
\tau(\xi_{\un{n}})^* \Theta_{\ze_{\un{n}} \ze_{\un{m}} b, \ze'_{\un{n}} \ze'_{\un{m}}}^{X_{\un{m}+\un{n}}} \tau(\eta_{\un{n}})
& =
\Theta_{\phi_{\un{m}}(\sca{\xi_{\un{n}}, \ze_{\un{n}}}) \ze_{\un{m}} b, \phi_{\un{m}}(\sca{\eta_{\un{n}}, \ze'_{\un{n}}}) \ze'_{\un{m}}}^{X_\un{m}}
\in
\K(X_{\un{m}} \L_F).
\end{align*}
By taking finite linear combinations and their norm-limits, we conclude that 
\[
\tau(\xi_{\un{n}})^* \K(X_{\un{m} + \un{n}} \L_F) \tau(\eta_{\un{n}})
\subseteq
\K(X_{\un{m}} \L_F),
\]
which implies (\ref{eq:k claim}).
Moreover, a direct computation on the elements of $X_{\un{m}}$ yields
\[
\tau(\xi_\un{n})^*\phi_{\un{m}+\un{n}}(a)\tau(\eta_\un{n})=\phi_\un{m}(\sca{\xi_\un{n},\phi_\un{n}(a)\eta_\un{n}}).
\]
We then have 
\begin{align*}
\|\phi_\un{m}(\sca{\xi_\un{n},\phi_\un{n}(a)\eta_\un{n}})+\tau(\xi_\un{n})^*\iota_{\un{m}}^{\un{m}+\un{n}}(k_{\un{m}})\tau(\eta_\un{n})\|
& =
\|\tau(\xi_\un{n})^*(\phi_{\un{m}+\un{n}}(a)+\iota_\un{m}^{\un{m}+\un{n}}(k_{\un{m}}))\tau(\eta_\un{n})\| \\
&\leq
\|\xi_\un{n}\|\cdot\|\iota_\un{m}^{\un{m}+\un{n}}(\phi_\un{m}(a))+\iota_\un{m}^{\un{m}+\un{n}}(k_{\un{m}})\|\cdot\|\eta_\un{n}\| \\
& \leq
\|\xi_\un{n}\|\cdot\|\phi_\un{m}(a)+k_{\un{m}}\|\cdot\|\eta_\un{n}\| < \varepsilon.
\end{align*}
This shows that $\sca{\xi_\un{n},\phi_\un{n}(a)\eta_\un{n}}\in \L_{\lim, F}$ by (\ref{eq:k claim}) and Lemma \ref{L:lim perp} (since $\L$ is assumed to be invariant), and the proof is complete.
\end{proof}

Next we explore the interaction of the $\L^{(1)}$ construction with the partial ordering property.
To this end, we have the following auxiliary proposition.

\begin{proposition}\label{P:L1 inc}
Let $X$ be a strong compactly aligned product system with coefficients in a C*-algebra $A$.
Let $\L$ be an (E)-$2^d$-tuple of $X$ that is invariant and consists of ideals.
Then for each $F\subseteq[d]$, we have 
\[
\ol{\pi}_X(\L_F^{(1)})\ol{q}_{X,F}
\subseteq
\sum\{\mathfrak{J}_{\L,D}^{(\ol{\pi}_X,\ol{t}_X)}\mid F\subseteq D\subseteq[d]\}
\subseteq 
\mathfrak{J}_\L^{(\ol{\pi}_X,\ol{t}_X)}.
\]
\end{proposition}

\begin{proof}
It suffices to show the first inclusion of the statement, as the second holds by definition.
Without loss of generality, we may replace $(\ol{\pi}_X,\ol{t}_X)$ by the Fock representation $(\ol{\pi},\ol{t})$, and write $\mathfrak{J}_\L \equiv \fJ_\L^{(\ol{\pi},\ol{t})}$ and $\mathfrak{J}_{\L,F} \equiv \fJ_{\L,F}^{(\ol{\pi},\ol{t})}$ for all $F \subseteq [d]$.
The claim holds trivially when $F=\mt$ or $F=[d]$, so fix $\mt\neq F\subsetneq[d]$ and $a\in\L_F^{(1)}$.
We must show that $\ol{\pi}(a)\ol{q}_F\in\sum\{\mathfrak{J}_{\L,D}\mid F\subseteq D\subseteq[d]\}$.

To this end, fix $\varepsilon>0$. 
Since $a \in \L_F^{(1)} \subseteq \L_{\lim, F}$, there exists $\un{m} \perp F$ such that 
\[
\|\phi_\un{n}(a)+ \K(X_{\un{n}} \L_F)\|<\varepsilon \textup{ for every $\un{n} \geq \un{m}$ with $\un{n} \perp F$}.
\]  
Consider a suitable $m\in\bN$ such that
\[
\un{n} := (m+1) \cdot \un{1}_{F^c}\geq\un{m},
\]
and take $k_\un{n} \in \K(X_{\un{n}} \L_F)$ such that $\|\phi_\un{n}(a)+ k_{\un{n}}\|<\varepsilon$.
Next, for each $i\in F^c$ we define the projection
\[
W_i \colon \F X \to \sum \{X_\un{r}\mid \un{r}=(r_1,\dots,r_d)\in\bZ_+^d,r_i\leq m\},
\]
and consider the operator
\[
\ol{\pi}(a)\ol{q}_F\prod_{i\in F^c}(I-W_i)=\ol{\pi}(a)\ol{q}_F+\sum_{\mt\neq D\subseteq F^c}(-1)^{|D|}\ol{\pi}(a)\ol{q}_F\prod_{i\in D}W_i.
\]
The products can be taken in any order, since the projections $W_i$ commute. 
For each element $\un{r}=(r_1,\dots,r_d)\in\bZ_+^d$ and $\ze_\un{r}\in X_\un{r}$, we have 
\[
\ol{\pi}(a)\ol{q}_F\prod_{i\in F^c}(I-W_i)\ze_\un{r}=\begin{cases} \phi_\un{r}(a)\ze_\un{r} & \text{if} \; r_i\geq m+1 \; \text{for all} \; i\in F^c \; \text{and} \; \un{r}\perp F, \\ 0 & \text{if} \; r_i\leq m \; \text{for some} \; i\in F^c \; \text{or} \; \un{r}\not\perp F.\end{cases}
\]
It follows that 
\[
\ol{\pi}(a)\ol{q}_F\prod_{i\in F^c}(I-W_i)=\sum_{\un{r}\geq\un{n}, \un{r}\perp F}\phi_\un{r}(a).
\]
Then we have 
\begin{align*}
\|\sum_{\un{r}\geq\un{n}, \un{r}\perp F}\phi_\un{r}(a)+\sum_{\un{r}\geq\un{n},\un{r}\perp F}\iota_\un{n}^\un{r}(k_{\un{n}})\| 
& =\|\sum_{\un{r}\geq\un{n},\un{r}\perp F}\iota_\un{n}^\un{r}(\phi_\un{n}(a)+k_{\un{n}})\| \\
& =\text{sup}\{\|\iota_\un{n}^\un{r}(\phi_\un{n}(a)+k_{\un{n}})\|\mid \un{r}\geq\un{n}, \un{r}\perp F\} \\
& \leq \|\phi_\un{n}(a)+k_{\un{n}}\|
< \varepsilon,
\end{align*}
using that each $\iota_\un{n}^\un{r}$ is contractive.
Therefore, we deduce that
\begin{equation}\label{Eq: eps close}
\|\ol{\pi}(a)\ol{q}_F\prod_{i\in F^c}(I-W_i)+\sum_{\un{r}\geq\un{n},\un{r}\perp F}\iota_\un{n}^\un{r}(k_{\un{n}})\| < \varepsilon.
\end{equation}
We will prove two claims before proceeding to the completion of the proof.

\bigskip

\noindent 
\emph{Claim 1. The element $x_F:=\sum_{\un{r}\geq\un{n},\un{r}\perp F}\iota_\un{n}^\un{r}(k_{\un{n}})$ belongs to $\mathfrak{J}_{\L,F}$.}

\smallskip

\noindent 
\emph{Proof of Claim 1.}
First suppose $k_{\un{n}}=\nor{\cdot}\text{-}\lim_{n}k_{\un{n}, n}$, where each $k_{\un{n}, n}$ is a finite sum of rank-one operators in $\K(X_{\un{n}} \L_F)$. 
Then $x_F$ is the norm-limit of the elements $\sum_{\un{r}\geq\un{n},\un{r}\perp F}\iota_\un{n}^\un{r}(k_{\un{n}, n})$. 
Thus it suffices to show that each operator $\sum_{\un{r}\geq\un{n},\un{r}\perp F}\iota_\un{n}^\un{r}(k_{\un{n}, n})$ belongs to $\mathfrak{J}_{\L,F}$.
In turn, it suffices to show that
\begin{equation}\label{eq: xf}
\sum_{\un{r}\geq\un{n},\un{r}\perp F}\iota_\un{n}^\un{r}( \Theta_{\xi_\un{n}b,\eta_\un{n}} )=\ol{t}_\un{n}(\xi_\un{n})\ol{\pi}(b)\ol{q}_F\ol{t}_\un{n}(\eta_\un{n})^*,
\end{equation}
where $\xi_\un{n},\eta_\un{n}\in X_\un{n}$ and $b\in\L_F$.
First note that if $\un{s}\not\geq\un{n}$ or $\un{s}\not\perp F$, then
\[
\ol{t}_\un{n}(\xi_\un{n})\ol{\pi}(b)\ol{q}_F\ol{t}_\un{n}(\eta_\un{n})^*\ze_\un{s}=0 \foral \ze_\un{s}\in X_\un{s}.
\]
Indeed, if $\un{s}\not\geq\un{n}$, then $\ol{t}_\un{n}(\eta_\un{n})^*\ze_\un{s}=0$ by definition.
If $\un{s}\geq\un{n}$ but $\un{s}\not\perp F$, then $\supp(\un{s}-\un{n})\cap F\neq\mt$, and hence $\ol{q}_F\ol{t}_\un{n}(\eta_\un{n})^*\ze_\un{s}=0$.
Thus both sides of (\ref{eq: xf}) map $X_\un{s}$ to 0 when $\un{s}\not\geq\un{n}$ or $\un{s}\not\perp F$. 
Next suppose that $\un{s}\geq\un{n}$ and $\un{s}\perp F$, and let $\ze_\un{s}=\zeta_\un{n}\zeta_{\un{s}-\un{n}}$ for some $\zeta_\un{n}\in X_\un{n}$ and $\zeta_{\un{s}-\un{n}}\in X_{\un{s}-\un{n}}$.
Then we have 
\begin{align*}
\ol{t}_\un{n}(\xi_\un{n})\ol{\pi}(b)\ol{q}_F\ol{t}_\un{n}(\eta_\un{n})^*\ze_\un{s}
& =
\xi_\un{n}(\phi_{\un{s}-\un{n}}(b\sca{\eta_\un{n},\zeta_\un{n}})\zeta_{\un{s}-\un{n}}) 
=
\big(\Theta_{\xi_\un{n}b,\eta_\un{n}} (\zeta_\un{n}) \big) \zeta_{\un{s}-\un{n}} \\
& =
\iota_\un{n}^\un{s}(\Theta_{\xi_\un{n}b,\eta_\un{n}})(\zeta_\un{n}\zeta_{\un{s}-\un{n}}) 
= 
\sum_{\un{r}\geq\un{n},\un{r}\perp F}\iota_\un{n}^{\un{r}}( \Theta_{\xi_\un{n}b,\eta_\un{n}} ) \ze_\un{s},
\end{align*}
noting that $\un{s} - \un{n} \perp F$.
It follows that (\ref{eq: xf}) holds, finishing the proof of Claim 1.
\hfill{$\Box$}

\bigskip

\noindent 
\emph{Claim 2. The element $x_{F\cup D}:=\ol{\pi}(a)\ol{q}_F\prod_{i\in D}W_i$ belongs to $\mathfrak{J}_{\L,F\cup D}$ for each $\mt\neq D\subseteq F^c$.}

\smallskip

\noindent 
\emph{Proof of Claim 2.}
Let us first consider the case where $D=F^c$. 
For $\un{r}=(r_1,\dots,r_d)\in\bZ_+^d$ and $\ze_\un{r}\in X_\un{r}$, we have 
\[
\ol{\pi}(a)\ol{q}_F\prod_{i\in F^c}W_i\ze_\un{r}=\begin{cases} \phi_\un{r}(a)\ze_\un{r} & \text{if} \; \un{0} \leq \un{r}\leq m\cdot\un{1}_{F^c}, \\ 0 & \text{if} \; r_i\geq m+1 \; \text{for some} \; i\in F^c \; \text{or} \; \un{r}\not\perp F, \end{cases}
\]
from which it follows that
\[
\ol{\pi}(a)\ol{q}_F\prod_{i\in F^c}W_i=\sum_{\un{0}\leq\un{r}\leq m\cdot\un{1}_{F^c}}\phi_\un{r}(a),
\]
noting that the sum is finite.

Fix $\un{0}\leq\un{r}\leq m\cdot\un{1}_{F^c}$ and note that $\un{r}\perp F$. 
Since $a\in\L_F^{(1)}$ and $\L^{(1)}$ is an (E)-$2^d$-tuple that is invariant and consists of ideals by Lemma \ref{L:L1 ideal inv}, an application of Lemma \ref{L:inv comp} gives that $\phi_\un{r}(a)\in\K(X_\un{r}\L_F^{(1)})$. 
Additionally, $\L_F^{(1)}\subseteq \L_{\inv, F}$ by definition, and so in particular $\L_F^{(1)}\subseteq\L_{[d]}$. 
Consequently $\phi_\un{r}(a)\in\K(X_\un{r}\L_{[d]})$. 

We claim that $\phi_\un{r}(a)\in \mathfrak{J}_{\L,[d]}$ when viewed as an operator in $\L(\F X)$. 
First note that 
\[
\phi_\un{0}(a)=\ol{\pi}(a)\ol{q}_{[d]}\in\mathfrak{J}_{\L,[d]},
\]
as $a \in \L_F^{(1)} \subseteq \L_{[d]}$, so we may assume that $\un{r}\neq\un{0}$.
We see that 
\begin{equation}\label{eq:com r} 
\K(X_\un{r}\L_{[d]})
\subseteq 
\mathfrak{J}_{\L,[d]},
\end{equation} 
when identifying $\K(X_\un{r}\L_{[d]})$ within $\L(\F X)$.
Indeed, it suffices to show this for rank-one operators.
To this end, take $k_{\un{r}}=\Theta_{\xi_\un{r}b,\eta_\un{r}}$ for some $\xi_\un{r},\eta_\un{r}\in X_\un{r}$ and $b\in\L_{[d]}$.
We claim that 
\begin{equation}\label{eq:ffc}
\Theta_{\xi_\un{r}b,\eta_\un{r}}=\ol{t}_\un{r}(\xi_\un{r})\ol{\pi}(b)\ol{q}_{[d]}\ol{t}_\un{r}(\eta_\un{r})^*,
\end{equation}
when we view $\Theta_{\xi_\un{r}b,\eta_\un{r}}$ as an operator in $\L(\F X)$.
Notice that both sides of (\ref{eq:ffc}) map every summand $X_\un{s}$ for which $\un{s}\neq\un{r}$ to $0$.
On the other hand, arguing as we did with $x_F$, we deduce that the right hand side of (\ref{eq:ffc}) coincides with $\Theta_{\xi_\un{r}b,\eta_\un{r}}$ on the $X_\un{r}$ summand.
Hence (\ref{eq:ffc}) and in turn (\ref{eq:com r}) hold, and applying for $\phi_\un{r}(a)\in\K(X_\un{r}\L_{[d]})$ gives that $\phi_\un{r}(a)\in\mathfrak{J}_{\L,[d]}$.
We conclude that 
\[
x_{[d]}=\ol{\pi}(a)\ol{q}_F\prod_{i\in F^c}W_i\in \mathfrak{J}_{\L,[d]},
\]
being a finite sum of elements of $\mathfrak{J}_{\L,[d]}$. 

Now take $\mt\neq D\subsetneq F^c$. 
For $\un{r}=(r_1,\dots,r_d)\in\bZ_+^d$ and $\ze_\un{r}\in X_\un{r}$, we have 
\[
\ol{\pi}(a)\ol{q}_F\prod_{i\in D}W_i\ze_\un{r}=\begin{cases} \phi_\un{r}(a)\ze_\un{r} & \text{if} \; r_i\leq m \; \foral i\in D \; \text{and} \; \un{r}\perp F, \\ 0 & \text{if} \; r_i\geq m+1 \; \text{for some} \; i\in D \; \text{or} \; \un{r}\not\perp F. \end{cases}
\]
It follows that
\[
\ol{\pi}(a)\ol{q}_F\prod_{i\in D}W_i=\sum_{\un{0}\leq\un{r}\leq m\cdot\un{1}_D} \bigg(\sum_{\un{s}\perp F\cup D}\phi_{\un{r}+\un{s}}(a) \bigg),
\]
noting that the sum over $\un{r}$ is finite.

Fix $\un{0}\leq\un{r}\leq m\cdot\un{1}_D$ and consider $\sum_{\un{s}\perp F\cup D}\phi_{\un{r}+\un{s}}(a)$. 
Since $a\in\L_F^{(1)}$ and $\L^{(1)}$ is an (E)-$2^d$-tuple that is invariant and consists of ideals by Lemma \ref{L:L1 ideal inv}, an application of Lemma \ref{L:inv comp} gives that $\phi_\un{r}(a)\in\K(X_\un{r}\L_F^{(1)})$, noting that $\un{r}\perp F$.
Additionally, by definition we have 
\[
\L_F^{(1)} \subseteq \L_{\inv, F} \subseteq \cap_{F\subsetneq G}\L_G.
\]
In particular, notice that $D\cap F^c\neq\mt$ and hence $F\subsetneq F\cup D$.
Thus $\L_F^{(1)}\subseteq\L_{F\cup D}$, and consequently $\phi_\un{r}(a)\in\K(X_\un{r}\L_{F\cup D})$.
Notice that
\[
\sum_{\un{s}\perp F\cup D}\phi_{\un{r}+\un{s}}(a)=\sum_{\un{s}\perp F\cup D}\iota_{\un{r}}^{\un{r}+\un{s}}(\phi_\un{r}(a)).
\]
We claim that $\sum_{\un{s}\perp F\cup D}\iota_{\un{r}}^{\un{r}+\un{s}}(\phi_\un{r}(a))\in \mathfrak{J}_{\L,F\cup D}$.
First note that if $\un{r}=\un{0}$, then 
\[
\sum_{\un{s}\perp F\cup D}\iota_\un{0}^\un{s}(\phi_\un{0}(a))=\sum_{\un{s}\perp F\cup D}\phi_\un{s}(a)=\ol{\pi}(a)\ol{q}_{F\cup D}\in\mathfrak{J}_{\L,F\cup D},
\]
using that $a\in\L_F^{(1)}\subseteq\L_{F\cup D}$ in the final membership.
For $\un{r}\neq\un{0}$, we see that 
\begin{equation}\label{eq:com s}
\sum_{\un{s}\perp F\cup D}\iota_{\un{r}}^{\un{r}+\un{s}}(k_{\un{r}})\in\mathfrak{J}_{\L,F\cup D}
\foral k_{\un{r}}\in\K(X_\un{r}\L_{F\cup D}).
\end{equation}
Indeed, it suffices to show this for $k_{\un{r}}=\Theta_{\xi_\un{r}b,\eta_\un{r}}$ for some $\xi_\un{r},\eta_\un{r}\in X_\un{r}$ and $b\in\L_{F\cup D}$.
Arguing as in the proof of Claim 1 for (\ref{eq: xf}), we have 
\[
\sum_{\un{s}\perp F\cup D}\iota_{\un{r}}^{\un{r}+\un{s}}(\Theta_{\xi_\un{r}b,\eta_\un{r}})
=
\ol{t}_\un{r}(\xi_\un{r})\ol{\pi}(b)\ol{q}_{F\cup D}\ol{t}_\un{r}(\eta_\un{r})^*\in \mathfrak{J}_{\L,F\cup D}.
\]
Hence (\ref{eq:com s}) holds, and applying for $k_{\un{r}}=\phi_\un{r}(a)\in\K(X_\un{r}\L_{F\cup D})$ yields
\[
\sum_{\un{s}\perp F\cup D}\iota_{\un{r}}^{\un{r}+\un{s}}(\phi_\un{r}(a))\in\mathfrak{J}_{\L,F\cup D}.
\]
Therefore, we conclude that 
\[
x_{F\cup D}=\ol{\pi}(a)\ol{q}_F\prod_{i\in D}W_i\in \mathfrak{J}_{\L,F\cup D},
\]
being a finite sum of elements of $\mathfrak{J}_{\L,F\cup D}$. 
This finishes the proof of Claim 2. 
\hfill{$\Box$}

\bigskip

We now conclude the proof of the proposition.
By Claim 2, we have
\[
x:=\sum_{\mt\neq D\subseteq F^c}(-1)^{|D|}\ol{\pi}(a)\ol{q}_F\prod_{i\in D}W_i=\sum_{\mt\neq D\subseteq F^c}(-1)^{|D|}x_{F\cup D}\in\sum_{F\subsetneq D\subseteq[d]}\mathfrak{J}_{\L,D}.
\]
In total, we have
\[
\|\ol{\pi}(a)\ol{q}_F+(x+x_F)\|=\|\ol{\pi}(a)\ol{q}_F\prod_{i\in F^c}(I-W_i)+\sum_{\un{r}\geq\un{n},\un{r}\perp F}\iota_\un{n}^\un{r}(k_{\un{n}})\| < \varepsilon,
\]
using (\ref{Eq: eps close}) in the final inequality.
So, employing Claims 1 and 2 in tandem, we have shown that for every $\varepsilon >0$ the element $\ol{\pi}(a)\ol{q}_F$ is $\varepsilon$-close to an element of $\sum_{F\subseteq D\subseteq[d]}\mathfrak{J}_{\L,D}$. 
Hence $\ol{\pi}(a)\ol{q}_F\in\sum_{F\subseteq D\subseteq[d]}\mathfrak{J}_{\L,D}$, as required.
\end{proof}

We are now ready to prove that when $\L$ is in addition partially ordered, the (E)-$2^d$-tuple $\L^{(1)}$ is partially ordered, contains $\L$ and induces the same gauge-invariant ideal of $\N\T_X$.

\begin{proposition}\label{P:L1 po}
Let $X$ be a strong compactly aligned product system with coefficients in a C*-algebra $A$.
Let $\L$ be an (E)-$2^d$-tuple of $X$ that is invariant, partially ordered and consists of ideals.
Then $\L^{(1)}$ is an (E)-$2^d$-tuple of $X$ that is invariant, partially ordered, consists of ideals, and is such that $\L\subseteq\L^{(1)}$ and $\mathfrak{J}_\L^{(\ol{\pi}_X,\ol{t}_X)}=\mathfrak{J}_{\L^{(1)}}^{(\ol{\pi}_X,\ol{t}_X)}$.
\end{proposition}

\begin{proof}
By Lemma \ref{L:L1 ideal inv}, the family $\L^{(1)}$ is an (E)-$2^d$-tuple that is invariant and consists of ideals.
For the partial ordering property, it is immediate that $\{0\}=\L_\mt^{(1)}\subseteq\L_F^{(1)}$ for all $F\subseteq[d]$.
Likewise, we have 
\[
\L_F^{(1)}\subseteq X_\un{0}^{-1}(\cap_{F\subsetneq D}\L_D)=\cap_{F\subsetneq D}\L_D\subseteq\L_{[d]}=\L_{[d]}^{(1)}
\]
for all $\mt\neq F\subsetneq[d]$.
Next fix $\mt\neq F \subseteq F' \subsetneq[d]$, and let $a\in\L_F^{(1)}$.
Notice that $a\in \L_F^{(1)} \subseteq \I_F \subseteq \I_{F'}$ because $\I$ is partially ordered.
Next fix $\un{m}\perp F'$.
Since $F\subseteq F'$, we have $\un{m}\perp F$ and thus
\[
\sca{X_\un{m},aX_\un{m}}
\subseteq 
\cap_{F\subsetneq D}\L_D\subseteq\cap_{F'\subsetneq D}\L_D,
\]
using that $a \in \L_{\inv, F}$.
It follows that $a \in \L_{\inv, F'}$.
In order to prove that $a \in \L_{\lim, F'}$, we resort to Proposition \ref{P:JL lim}.
More specifically, it suffices to show that $\ol{\pi}_X(a)\ol{q}_{X,F'}\in\mathfrak{J}_\L^{(\ol{\pi}_X,\ol{t}_X)}$.
To this end, note that $a\in\bigcap\{\phi_\un{i}^{-1}(\K(X_\un{i}))\mid i\in F'\}$, since $\L^{(1)}$ is an (E)-$2^d$-tuple, and $\ol{\pi}_X(a)\ol{q}_{X,F}\in\mathfrak{J}_\L^{(\ol{\pi}_X,\ol{t}_X)}$ by Proposition \ref{P:L1 inc}.
An application of Lemma \ref{L:JL rising} then gives that $\ol{\pi}_X(a)\ol{q}_{X,F'}\in\mathfrak{J}_\L^{(\ol{\pi}_X,\ol{t}_X)}$, as required.
Hence $a\in\L_{F'}^{(1)}$ and we conclude that $\L^{(1)}$ is partially ordered.

Next we prove that $\L\subseteq\L^{(1)}$.
Notice that $\L_\mt\subseteq\L_\mt^{(1)}$ and $\L_{[d]}\subseteq\L_{[d]}^{(1)}$ trivially, so fix $\mt\neq F\subsetneq [d]$ and $a\in\L_F$.
Since $\L$ is an (E)-$2^d$-tuple, we have $a\in\I_F$.
Since $\L$ is invariant and partially ordered, we have 
\[
\sca{X_\un{m},aX_\un{m}} \subseteq \L_F \subseteq\cap_{F\subsetneq D}\L_D
\]
for all $\un{m} \perp F$, and thus $a \in \L_{\inv, F}$.
Note that $\phi_\un{m}(a) \in \K(X_\un{m}\L_F)$ for all $\un{m} \perp F$, by Lemma \ref{L:inv comp}.
Therefore, by choosing any $\un{m} \perp F$ and $k_{\un{m}} = -\phi_\un{m}(a) \in \K(X_\un{m}\L_F)$, we obtain $\|\phi_\un{m}(a)+k_{\un{m}}\| = 0$, and thus $a \in \L_{\lim, F}$ by Lemma \ref{L:lim perp}.
Hence $\L_F\subseteq\L_F^{(1)}$, as required.

Finally, we show that $\mathfrak{J}_\L^{(\ol{\pi}_X,\ol{t}_X)}=\mathfrak{J}_{\L^{(1)}}^{(\ol{\pi}_X,\ol{t}_X)}$.
The forward inclusion is immediate since $\L\subseteq\L^{(1)}$.
On the other hand, we have $\ol{\pi}_X(\L^{(1)}_F)\ol{q}_{X,F} \subseteq \mathfrak{J}_\L^{(\ol{\pi}_X,\ol{t}_X)}$ for all $F\subseteq[d]$ by Proposition \ref{P:L1 inc}, giving the reverse inclusion and completing the proof.
\end{proof}

We are now ready to provide a full characterisation of (M)-$2^d$-tuples.

\begin{theorem}\label{T:m fam v2}
Let $X$ be a strong compactly aligned product system with coefficients in a C*-algebra $A$ and suppose that $\L$ is a $2^d$-tuple of $X$.
Then $\L$ is an (M)-$2^d$-tuple of $X$ if and only if $\L$ satisfies the following four conditions:
\begin{enumerate}
\item $\L$ consists of ideals and $\L \subseteq \J$,
\item $\L$ is invariant,
\item $\L$ is partially ordered,
\item $\L^{(1)} \subseteq \L$.
\end{enumerate}
\end{theorem}

\begin{proof}
Assume that $\L$ is an (M)-$2^d$-tuple.
First we use Lemma \ref{L:e inv} and Lemma \ref{L:e po} together with maximality of $\L$ to deduce that $\L$ consists of ideals, that $\L \subseteq \I \subseteq \J$, and that $\L$ is invariant and partially ordered.
Proposition \ref{P:L1 po} together with maximality of $\L$ gives that $\L^{(1)}\subseteq\L$, proving the forward implication.

Now suppose that $\L$ is a $2^d$-tuple that satisfies conditions (i)-(iv).
Conditions (i) and (ii) imply that $\L$ is an (E)-$2^d$-tuple by Lemma \ref{L:J+inv}, and so we can consider $\L^{(1)}$.
It remains to check that $\L$ is maximal.
To this end, let $\M$ be the (M)-$2^d$-tuple such that $\mathfrak{J}_\L^{(\ol{\pi}_X,\ol{t}_X)} = \mathfrak{J}_\M^{(\ol{\pi}_X,\ol{t}_X)}$, as guaranteed by Proposition \ref{P:inj on J}.
Note also that $\L \subseteq \M$ by Proposition \ref{P:maximal}.
It suffices to show that $\M\subseteq\L$.

On one hand, we have $\M_\mt=\L_\mt=\{0\}$, since both $\M$ and $\L$ are (E)-$2^d$-tuples.
In order to show that $\M_F\subseteq\L_F$ for all $\mt\neq F\subseteq[d]$, we apply strong (downward) induction on $|F|$.

For the base case, take $a\in\M_{[d]}$.
Then $\ol{\pi}_X(a)\ol{q}_{X,[d]}\in\mathfrak{J}_\M^{(\ol{\pi}_X,\ol{t}_X)}=\mathfrak{J}_\L^{(\ol{\pi}_X,\ol{t}_X)}$.
Using the fact that $\N\T_X\cong\ca(\ol{\pi},\ol{t})$ canonically for the Fock representation $(\ol{\pi},\ol{t})$, we have $\ol{\pi}(a)\ol{q}_{[d]}\in\mathfrak{J}_\L^{(\ol{\pi},\ol{t})}$.
Conditions (i)-(iii) together with Propositions \ref{P: JLF} and \ref{P: cond ex} give that
\begin{align*}
\phi_{\un{0}}(a)
& =
\ol{q}_{[d]} \left( \ol{\pi}(a)\ol{q}_{[d]} \right) \ol{q}_{[d]} 
\in
\ol{q}_{[d]}^2 \left( \mathfrak{J}_\L^{(\ol{\pi},\ol{t})} \right) \ol{q}_{[d]}^2 
\subseteq
\ol{q}_{[d]} \left( \mathfrak{J}_{\L,[d]}^{(\ol{\pi},\ol{t})} \right) \ol{q}_{[d]} 
= 
\ol{q}_{[d]} \ol{\pi}(\L_{[d]}) \ol{q}_{[d]}
= 
\phi_{\un{0}}(\L_{[d]}),
\end{align*}
and thus $a\in\L_{[d]}$.
This shows that $\M_{[d]} \subseteq \L_{[d]}$.

Next, suppose that $\M_F\subseteq\L_F$ for all $\mt\neq F\subseteq[d]$ with $|F|\geq n+1$, where $1\leq n\leq d-1$.
Fix $\mt\neq F\subsetneq[d]$ with $|F|=n$.
Since $\M$ is an (E)-$2^d$-tuple, we have $ \M_F \subseteq \I_F$.
Note that $\M$, being an (M)-$2^d$-tuple, satisfies conditions (i)-(iv) by the forward implication.
In particular, we have that $\M$ is invariant, so 
\[
\sca{X_\un{m}, \M_F X_\un{m}} \subseteq \M_F \foral \un{m} \perp F.
\]
Likewise, we have that $\M$ is partially ordered, so $\sca{X_\un{m},\M_F X_\un{m}} \subseteq \cap_{F\subsetneq D}\M_D$ for all $\un{m} \perp F$.
By the inductive hypothesis, we have $\M_D=\L_D$ whenever $F\subsetneq D$, as $|D|\geq n+1$.
Hence 
\[
\sca{X_\un{m}, \M_F X_\un{m}} \subseteq \cap_{F\subsetneq D}\L_D \foral \un{m} \perp F.
\]
Thus $\M_F \subseteq \L_{\inv, F}$.
Moreover, by definition we have
\[
\ol{\pi}_X(\M_F)\ol{q}_{X,F} \subseteq \mathfrak{J}_\M^{(\ol{\pi}_X,\ol{t}_X)}=\mathfrak{J}_\L^{(\ol{\pi}_X,\ol{t}_X)}.
\]
Hence, by applying Proposition \ref{P:JL lim}, we deduce that $\M_F \subseteq \L_{\lim, F}$.
In total, we have
\[
\M_F \subseteq \I_F \cap \L_{\inv, F} \cap \L_{\lim, F} = \L^{(1)}_F,
\]
and thus $\M_F \subseteq \L_F$ by condition (iv), as required.
Induction then completes the proof.
\end{proof}

Iteration of the $\L^{(1)}$ construction constitutes the final ingredient for attaining maximality.
For if we start with an (E)-$2^d$-tuple $\L$ that is invariant, partially ordered and consists of ideals, then iterative applications of Proposition \ref{P:L1 po} produce a sequence of (E)-$2^d$-tuples that are invariant, partially ordered, consist of ideals and satisfy
\[
\L \subseteq\L^{(1)} \subseteq \dots \subseteq \L^{(k)}\subseteq\dots
\qand
\mathfrak{J}_{\L}^{(\ol{\pi}_X,\ol{t}_X)}=\mathfrak{J}_{\L^{(k)}}^{(\ol{\pi}_X,\ol{t}_X)}
\foral k \in \bN.
\]
Since each $\L^{(k)}$ induces the same gauge-invariant ideal of $\N\T_X$, if the sequence eventually stabilises then it must stabilise to an (M)-$2^d$-tuple by Theorem \ref{T:m fam v2}.

\begin{theorem}\label{T:d-1 m fam}
Let $X$ be a strong compactly aligned product system with coefficients in a C*-algebra $A$.
Let $\L$ be an (E)-$2^d$-tuple of $X$ that is invariant, partially ordered, and consists of ideals.
Fix $0\leq k \leq d$.
Then whenever $F\subseteq[d]$ satisfies $|F|=d-k$, we have $\L_F^{(m)} = \L_F^{(k)}$ for all $m\geq k$.
Consequently, $\L^{(d-1)}$ is the (M)-$2^d$-tuple that induces $\mathfrak{J}_\L^{(\ol{\pi}_X,\ol{t}_X)}$.
\end{theorem}

\begin{proof}
We proceed by applying strong induction on $k$.
For the base case, if $F = [d]$ then $\L_{[d]}^{(m)}=\L_{[d]}=\L_{[d]}^{(0)}$ for all $m\geq 0$ by construction.
Next, fix $0\leq N\leq d-1$ and suppose that whenever $0\leq n\leq N$ and $D\subseteq[d]$ satisfies $|D|=d-n$, we have $\L_D^{(m)}=\L_D^{(n)}$ for all $m\geq n$.

Now we fix $F\subseteq[d]$ such that $|F|=d-(N+1)$.
We must show that $\L_F^{(m)}=\L_F^{(N+1)}$ for all $m\geq N+1$.
First note that this is clear if $F=\mt$ (i.e., if $N=d-1$), as each $\L_\mt^{(m)}=\{0\}$.
So, without loss of generality, we may exclude the case where $N=d-1$ and consider $\mt\neq F\subsetneq[d]$.

Fix $m\geq N+1$.
We have already argued that $\L_F^{(N+1)}\subseteq\L_F^{(m)}$, so it remains to verify the reverse inclusion.
Take $a\in\L_F^{(m)}$.
In particular, we have $a\in\I_F$. 
Fix $\un{m}\perp F$ and note that $\sca{X_\un{m},aX_\un{m}} \subseteq \cap_{F\subsetneq D}\L_D^{(m-1)}$ by definition.
Notice that whenever $F\subsetneq D$, we must have $|D|=d-k$ for some $0\leq k\leq N$.
By the inductive hypothesis, we have $\L_D^{(k)}=\L_D^{(N)}=\L_D^{(m-1)}$ and hence $\sca{X_\un{m},aX_\un{m}} \subseteq \cap_{F\subsetneq D}\L_D^{(N)}$.
In turn, we have $a \in \L^{(N)}_{\inv, F}$.

Notice also that $\ol{\pi}_X(a)\ol{q}_{X,F}\in\mathfrak{J}_{\L^{(m)}}^{(\ol{\pi}_X,\ol{t}_X)}=\mathfrak{J}_{\L^{(N)}}^{(\ol{\pi}_X,\ol{t}_X)}$, and hence an application of Proposition \ref{P:JL lim} yields $a \in \L_{\lim, F}^{(N)}$.
In total, we have $a \in \I_F \cap \L_{\inv, F}^{(N)} \cap \L_{\lim, F}^{(N)} = \L_{F}^{(N+1)}$.
This proves that $\L_F^{(m)}=\L_F^{(N+1)}$ for all $m\geq N+1$.
By induction we obtain $\L_F^{(m)} = \L_F^{(k)}$ for all $m\geq k$, as required.

Finally, we have $\L^{(d-1)} = \L^{(d)} = (\L^{(d-1)})^{(1)}$ by the first claim.
Hence, by Proposition \ref{P:L1 po} and Theorem \ref{T:m fam v2}, we conclude that $\L^{(d-1)}$ is the (M)-$2^d$-tuple inducing $\mathfrak{J}_\L^{(\ol{\pi}_X,\ol{t}_X)}$, finishing the proof.
\end{proof}

\begin{remark}\label{R:m alg}
Given any (E)-$2^d$-tuple $\L$, we now have an algorithm for computing the (M)-$2^d$-tuple that induces $\mathfrak{J}_\L^{(\ol{\pi}_X,\ol{t}_X)}$.
Specifically, we apply Lemma \ref{L:e inv} and Lemma \ref{L:e po} to pass from $\L$ to the (E)-$2^d$-tuple $\PO(\Inv(\L))$, which is invariant, partially ordered, consists of ideals and satisfies \
\[
\L\subseteq\PO(\Inv(\L))
\qand
\mathfrak{J}_\L^{(\ol{\pi}_X,\ol{t}_X)}=\mathfrak{J}_{\PO(\Inv(\L))}^{(\ol{\pi}_X,\ol{t}_X)}.
\]
We then apply the $(\PO(\Inv(\L)))^{(1)}$ construction iteratively, and use Theorem \ref{T:d-1 m fam} to deduce that $(\PO(\Inv(\L)))^{(d-1)}$ is the (M)-$2^d$-tuple that induces $\mathfrak{J}_\L^{(\ol{\pi}_X,\ol{t}_X)}$.
Accordingly, we can modify Theorem \ref{T:d GIUT M} to account for all (E)-$2^d$-tuples.
\end{remark}

\begin{theorem}[$\bZ_+^d$-GIUT for (E)-$2^d$-tuples]\label{T:d GIUT E}
Let $X$ be a strong compactly aligned product system with coefficients in a C*-algebra $A$. 
Let $\L$ be an (E)-$2^d$-tuple of $X$ and $(\pi,t)$ be a Nica-covariant representation of $X$. 
Then $\N\O(\L,X)\cong\ca(\pi,t)$ via a canonical $*$-isomorphism if and only if $(\pi,t)$ admits a gauge action and 
\[
\L^{(\pi,t)}=\bigg(\PO \big(\Inv(\L) \big) \bigg)^{(d-1)}.
\]
\end{theorem}

\begin{proof}
We have that $(\PO(\Inv(\L)))^{(d-1)}$ is the (M)-$2^d$-tuple that induces $\fJ_{\L}^{(\ol{\pi}_X,\ol{t}_X)}$ by Remark \ref{R:m alg}.
In particular, we have $\N\O((\PO(\Inv(\L)))^{(d-1)},X)=\N\O(\L,X)$, and Theorem \ref{T:d GIUT M} finishes the proof.
\end{proof}

\section{NT-$2^d$-tuples and gauge-invariant ideals of $\N\T_X$}\label{S:g inv struc}

The (M)-$2^d$-tuples of $X$ parametrise the equivariant quotients that lie in-between $\N\T_X$ and $\N\O_X$.
We now pass to the parametrisation of the quotients that may not be injective on $X$.
We will circumvent this by ``deleting the kernel", i.e., by utilising the quotient product system construction explored in Section \ref{S:prod sys}. 

\subsection{NT-$2^d$-tuples}\label{Ss:NT tuple}

We begin by defining some auxiliary objects.

\begin{definition}\label{D:JF}
Let $X$ be a strong compactly aligned product system with coefficients in a C*-algebra $A$.
Fix $\mt\neq F\subseteq[d]$ and let $I\subseteq A$ be an ideal. 
We define the following subsets of $A$:
\begin{enumerate}
\item $X_F^{-1}(I):=\bigcap\{X_\un{n}^{-1}(I)\mid \un{0}\neq\un{n}\leq\un{1}_F\} = \{a \in A \mid \sca{X_{\un{n}}, a X_{\un{n}}} \subseteq I \foral \un{0} \neq \un{n} \leq \un{1}_F\}$,
\item $J_F(I,X):=\{a\in A\mid [\phi_\un{i}(a)]_I\in\K([X_\un{i}]_I)\foral i\in[d], aX_F^{-1}(I)\subseteq I\}$.
\end{enumerate}
\end{definition}

Notice that both $X_F^{-1}(I)$ and $J_F(I,X)$ are ideals of $A$.
These objects will play similar roles to the ideals $X^{-1}(I)$ and $J(I,X)$ for a C*-correspondence $X$ over $A$.
Let us collect some properties of $X_F^{-1}(I)$ and $J_F(I,X)$.

\begin{lemma}\label{L:quo ker}
Let $X$ be a strong compactly aligned product system with coefficients in a C*-algebra $A$.
Let $I\subseteq A$ be an ideal that is positively invariant for $X$.
Fix $\mt\neq F\subseteq[d]$ and $a\in A$.
Then the following are equivalent:
\begin{enumerate}
\item $[a]_I\in\bigcap\{\ker[\phi_\un{i}]_I\mid i\in F\}$;
\item $\sca{X_\un{i},aX_\un{i}}\subseteq I\foral i\in F$;
\item $\sca{X_\un{n},aX_\un{n}}\subseteq I\foral \un{n}\neq\un{0}$ satisfying $\supp\un{n}\subseteq F$.
\end{enumerate}
Consequently, we have $X_F^{-1}(I) = \bigcap \{ X_{\un{i}}^{-1}(I) \mid i \in F\}$.
\end{lemma}

\begin{proof}
We will prove [(i)$\Rightarrow$(ii)$\Rightarrow$(iii)$\Rightarrow$(ii)$\Rightarrow$(i)].
First assume that $[a]_I\in\bigcap\{\ker[\phi_\un{i}]_I\mid i\in F\}$.
Fixing $i\in F$, we have
\[
[\phi_\un{i}(a)\xi_\un{i}]_I=[\phi_\un{i}]_I ([a]_I) [\xi_\un{i}]_I=0 \foral\xi_\un{i}\in X_\un{i},
\]
which in turn yields $a X_\un{i} \subseteq X_\un{i}I$.
An application of \cite[Proposition 1.3]{Kat07} then gives that $\sca{X_\un{i},aX_\un{i}}\subseteq I$, as required.

Next, assume that $\sca{X_\un{i},aX_\un{i}}\subseteq I$ for all $i\in F$.
We must prove that $\sca{X_\un{n},aX_\un{n}}\subseteq I$ for all $\un{n}\neq\un{0}$ satisfying $\supp\un{n}\subseteq F$.
We proceed by induction on $|\un{n}|$.
If $|\un{n}|=1$, then $\un{n}=\un{i}$ for some $i\in F$, in which case $\sca{X_\un{i},aX_\un{i}}\subseteq I$ by assumption.
Now suppose the claim holds for all $\un{n}\neq\un{0}$ satisfying $\supp\un{n}\subseteq F$ and $|\un{n}|=N$ for some $N\in\bN$.
Fix $\un{m}\neq\un{0}$ satisfying $\supp\un{m}\subseteq F$ and $|\un{m}|=N+1$.
We may write $\un{m}$ in the form $\un{m}=\un{n}+\un{i}$ for some $i\in F$ and some $\un{n}\neq\un{0}$ satisfying $\supp\un{n}\subseteq F$ and $|\un{n}|=N$.
We then have
\begin{align*}
\sca{X_\un{m}, aX_\un{m}} 
& =
\sca{X_{\un{n}} \otimes_A X_{\un{i}}, a X_{\un{n}} \otimes_A X_{\un{i}}}
\subseteq 
[\sca{X_\un{i}, \phi_\un{i}(\sca{X_\un{n}, a X_\un{n}}) X_\un{i}}]
\subseteq 
[\sca{X_{\un{i}}, I X_{\un{i}}}] 
\subseteq I,
\end{align*}
using the inductive hypothesis for $\un{n}$ and positive invariance of $I$. 
Hence $\sca{X_\un{m},aX_\un{m}}\subseteq I$, and by induction we are done.

Finally, assume that $\sca{X_\un{n},aX_\un{n}}\subseteq I$ for all $\un{n}\neq\un{0}$ satisfying $\supp\un{n}\subseteq F$.
In particular, we have $\sca{X_\un{i},aX_\un{i}}\subseteq I$ for all $i\in F$.
Fixing $i \in F$, we have $a X_{\un{i}} \subseteq X_{\un{i}} I$ by \cite[Proposition 1.3]{Kat07}, and therefore $[\phi_{\un{i}}]_I ([a]_I) [\xi_{\un{i}}]_I = [\phi_{\un{i}}(a) \xi_{\un{i}}]_I = 0$ for all $\xi_{\un{i}} \in X_{\un{i}}$.
Hence $[a]_I \in \ker [\phi_{\un{i}}]_I$, from which it follows that $[a]_I\in\bigcap\{\ker[\phi_\un{i}]_I\mid i\in F\}$, finishing the proof of the equivalences.

The last claim follows from the equivalence of items (ii) and (iii).
\end{proof}

The next proposition relates the ideals $X_F^{-1}(I)$ and $J_F(I,X)$ to ideals of $[A]_I$, when $I$ is positively invariant.
This result is the higher-rank analogue of \cite[Lemma 5.2]{Kat07}.

\begin{lemma}\label{L:quo ideal}
Let $X$ be a strong compactly aligned product system with coefficients in a C*-algebra $A$.
Let $I\subseteq A$ be an ideal that is positively invariant for $X$.
Then the following hold for all $\mt\neq F\subseteq[d]$:
\begin{enumerate}
\item $X_F^{-1}(I)=[ \hspace{1pt} \cdot \hspace{1pt} ]_I^{-1}(\bigcap\{\ker[\phi_\un{i}]_I\mid i\in F\})$.
\item $J_F(I,X)=[ \hspace{1pt} \cdot \hspace{1pt} ]_I^{-1}(\J_F([X]_I))$.
\item $X_F^{-1}(I)\cap J_F(I,X)=I$.
\end{enumerate}
\end{lemma}

\begin{proof}
Fix $\mt\neq F\subseteq[d]$.
Firstly, note that $a\in[ \hspace{1pt} \cdot \hspace{1pt} ]_I^{-1}(\bigcap\{\ker[\phi_\un{i}]_I\mid i\in F\})$ if and only if $[a]_I\in\bigcap\{\ker[\phi_\un{i}]_I\mid i\in F\}$, which in turn is equivalent to having that $a\in X_F^{-1}(I)$ by Lemma \ref{L:quo ker}.
Thus $X_F^{-1}(I)=[ \hspace{1pt} \cdot \hspace{1pt} ]_I^{-1}(\bigcap\{\ker[\phi_\un{i}]_I\mid i\in F\})$, proving item (i).

Next, assume that $a\in J_F(I,X)$.
We must show that
\[
[a]_I\in\J_F([X]_I)=(\bigcap_{i\in F}\ker[\phi_\un{i}]_I)^\perp\cap(\bigcap_{i\in[d]}[\phi_\un{i}]_I^{-1}(\K([X_\un{i}]_I))).
\]
By definition, we have $[\phi_\un{i}]_I[a]_I=[\phi_\un{i}(a)]_I\in\K([X_\un{i}]_I)$ for all $i\in[d]$ and $aX_F^{-1}(I)\subseteq I$.
In particular, we have $[a]_I\in\bigcap_{i\in[d]}[\phi_\un{i}]_I^{-1}(\K([X_\un{i}]_I))$.
Now take $[b]_I\in\bigcap_{i\in F}\ker[\phi_\un{i}]_I$.
By item (i), we have 
$
b\in [ \hspace{1pt} \cdot \hspace{1pt} ]_I^{-1}(\bigcap_{i\in F}\ker[\phi_\un{i}]_I)=X_F^{-1}(I),
$
and so $ab\in I$.
In particular, we have $[a]_I[b]_I=[ab]_I=0$, which in turn implies that $[a]_I\in(\bigcap_{i\in F}\ker[\phi_\un{i}]_I)^\perp$, as required.

Now assume that $a\in[ \hspace{1pt} \cdot \hspace{1pt} ]_I^{-1}(\J_F([X]_I))$.
Then $[\phi_\un{i}]_I[a]_I=[\phi_\un{i}(a)]_I\in\K([X_\un{i}]_I)$ for all $i\in[d]$ by definition.
Take $b\in X_F^{-1}(I)$.
By item (i), we have $[b]_I\in\bigcap_{i\in F}\ker[\phi_\un{i}]_I$.
By definition, we have $[a]_I\in(\bigcap_{i\in F}\ker[\phi_\un{i}]_I)^\perp$, so $[ab]_I=[a]_I[b]_I=0$ and hence $ab\in I$.
Thus $aX_F^{-1}(I)\subseteq I$ and hence in total $a\in J_F(I,X)$, proving item (ii).

Using items (i) and (ii), and that $\J_F([X]_I) \subseteq (\bigcap_{i\in F}\ker[\phi_\un{i}]_I)^\perp$, we obtain
\begin{align*}
X_F^{-1}(I)\cap J_F(I,X) 
& =
[ \hspace{1pt} \cdot \hspace{1pt} ]_I^{-1} \bigg((\bigcap_{i\in F}\ker[\phi_\un{i}]_I)\cap\J_F([X]_I) \bigg) 
=
[ \hspace{1pt} \cdot \hspace{1pt} ]_I^{-1}(\{0\}) 
=I,
\end{align*}
proving item (iii).
\end{proof}

We are now ready to introduce the objects that will implement the parametrisation of the gauge-invariant ideals of $\N\T_X$.

\begin{definition}\label{D:NT tuple}
Let $X$ be a strong compactly aligned product system with coefficients in a C*-algebra $A$ and let $\L$ be a $2^d$-tuple of $X$.
We say that $\L$ is an \emph{NT-$2^d$-tuple of $X$} if the following four conditions hold:
\begin{enumerate}
\item $\L$ consists of ideals and $\L_F\subseteq J_F(\L_\mt,X)$ for all $\mt\neq F\subseteq[d]$,
\item $\L$ is $X$-invariant,
\item $\L$ is partially ordered,
\item $ [ \hspace{1pt} \cdot \hspace{1pt} ]_{\L_\mt}^{-1} \big( [\L_F]_{\L_\mt}^{(1)} \big) \subseteq \L_F$ for all $F\subseteq[d]$, where $[\L_F]_{\L_\mt}=\L_F/\L_\mt\subseteq [A]_{\L_\mt}$.
\end{enumerate}
\end{definition}

To make sense of condition (iv), first note that conditions (i) and (ii) imply that $\L_\mt$ is an ideal of $A$ that is positively invariant for $X$.
Hence we can make sense of $[X]_{\L_\mt}$ as a strong compactly aligned product system with coefficients in $[A]_{\L_\mt}$ by Proposition \ref{P:qnt sca}.
Condition (iii) implies in particular that $\L_\mt\subseteq\L_F$ for all $F\subseteq[d]$, and so by condition (i) we have $[\L]_{\L_\mt}:=\{[\L_F]_{\L_\mt}\}_{F\subseteq[d]}$ is a $2^d$-tuple of $[X]_{\L_\mt}$ that consists of ideals.
An application of Lemma \ref{L:quo ideal} gives that $[\L]_{\L_\mt}\subseteq\J([X]_{\L_\mt})$, while item (ii) implies that $[\L]_{\L_\mt}$ is $[X]_{\L_\mt}$-invariant.
Hence we have that $[\L]_{\L_\mt}$ is an (E)-$2^d$-tuple by Lemma \ref{L:J+inv}, and so we can consider the family $[\L]_{\L_\mt}^{(1)}$.
Note also that condition (iv) holds automatically for $F= \mt$ and $F=[d]$.

When the left action of each fibre of $X$ is by compacts, condition (iv) simplifies as follows.

\begin{proposition}\label{P:NTcom}
Let $X$ be a product system over $\bZ_+^d$ with coefficients in a C*-algebra $A$ and suppose that $\phi_\un{n}(A)\subseteq\K(X_\un{n})$ for all $\un{n}\in\bZ_+^d$.
Then a $2^d$-tuple $\L$ of $X$ is an NT-$2^d$-tuple of $X$ if and only if it satisfies conditions (i)-(iii) of Definition \ref{D:NT tuple}
and
\[
\bigg(\bigcap_{\un{n}\perp F}X_\un{n}^{-1}(J_F(\L_\mt,X))\bigg)\cap\L_{\inv,F}\cap\L_{\lim,F}\subseteq\L_F\foral\mt\neq F\subsetneq[d].
\]
\end{proposition}
\begin{proof}
Note that $X$ is automatically strong compactly aligned by Corollary \ref{C:autscaps}.
Without loss of generality, we may assume that $\L$ satisfies conditions (i)-(iii) of Definition \ref{D:NT tuple}.
Since condition (iv) of Definition \ref{D:NT tuple} holds automatically for $F=\mt$ and $F=[d]$, and intersections are preserved under pre-images, it suffices to show that the following items hold for fixed $\mt\neq F\subsetneq[d]$:
\begin{itemize}
\item[(i)] $[ \hspace{1pt} \cdot \hspace{1pt} ]_{\L_\mt}^{-1}(\I_F([X]_{\L_\mt}))=\bigcap_{\un{n}\perp F}X_\un{n}^{-1}(J_F(\L_\mt,X))$,
\item[(ii)] $[ \hspace{1pt} \cdot \hspace{1pt} ]_{\L_\mt}^{-1}([\L]_{\L_\mt,\inv,F})=\L_{\inv,F}$,
\item[(iii)] $[ \hspace{1pt} \cdot \hspace{1pt} ]_{\L_\mt}^{-1}([\L]_{\L_\mt,\lim,F})=\L_{\lim,F}$.
\end{itemize}
For the first item, recall that 
\[
\I_F([X]_{\L_\mt})=\bigcap_{\un{n}\perp F}[X_\un{n}]_{\L_\mt}^{-1}(\J_F([X]_{\L_\mt}))\qand \J_F([X]_{\L_\mt})=[J_F(\L_\mt,X)]_{\L_\mt},
\]
where the latter holds by Lemma \ref{L:quo ideal}.
Note that $\L_\mt\subseteq J_F(\L_\mt,X)$ by positive invariance of $\L_\mt$.
Item (i) now follows as a consequence of Lemma \ref{L:quotinv}, taking $I:=\L_\mt$ and $J:=J_F(\L_\mt,X)$.

For the second item, recall that
\[
[\L]_{\L_\mt,\inv,F}:=\bigcap_{\un{m}\perp F}[X_\un{m}]_{\L_\mt}^{-1}(\cap_{F\subsetneq D}[\L_D]_{\L_\mt}).
\]
It is routine to check that $\cap_{F\subsetneq D}[\L_D]_{\L_\mt}=[\cap_{F\subsetneq D}\L_D]_{\L_\mt}$.
By the partial ordering property of $\L$, we have $\L_\mt\subseteq\cap_{F\subsetneq D}\L_D$.
Consequently, item (ii) follows as a consequence of Lemma \ref{L:quotinv}, taking $I:=\L_\mt$ and $J:=\cap_{F\subsetneq D}\L_D$.

Finally, item (iii) follows by a direct application of Lemma \ref{L:newquot}, which applies since $\phi_\un{n}(A)\subseteq\K(X_\un{n})$ for all $\un{n}\in\bZ_+^d$.
This completes the proof.
\end{proof}

\begin{remark}\label{R:onlylim}
Let $X$ be as in Proposition \ref{P:NTcom}, and let $\L$ be a $2^d$-tuple of $X$ that satisfies conditions (i)-(iii) of Definition \ref{D:NT tuple}.
Note that the assumption that $\phi_\un{n}(A)\subseteq\K(X_\un{n})$ for all $\un{n}\in\bZ_+^d$ is only used to obtain $[ \hspace{1pt} \cdot \hspace{1pt} ]_{\L_\mt}^{-1}([\L]_{\L_\mt,\lim,F})=\L_{\lim,F}$.
In other words, it is true that 
\[
[ \hspace{1pt} \cdot \hspace{1pt} ]_{\L_\mt}^{-1}(\I_F([X]_{\L_\mt}))=\bigcap_{\un{n}\perp F}X_\un{n}^{-1}(J_F(\L_\mt,X))
\qand
[ \hspace{1pt} \cdot \hspace{1pt} ]_{\L_\mt}^{-1}([\L]_{\L_\mt,\inv,F})=\L_{\inv,F},
\]
even when $X$ is (just) strong compactly aligned.
\end{remark}

NT-$2^d$-tuples represent the higher-rank analogue of Katsura's T-pairs.

\begin{proposition}\label{P:NT+T}
Let $X=\{X_n\}_{n\in\bZ_+}$ be a product system with coefficients in a C*-algebra $A$.
Then the NT-$2$-tuples of $X$ are exactly the T-pairs of $X_1$.
\end{proposition}

\begin{proof}
First let $\L=\{\L_\mt,\L_{\{1\}}\}$ be an NT-$2$-tuple of $X$.
Then $\L$ consists of ideals.
Since $\L$ is partially ordered, we have $\L_\mt\subseteq\L_{\{1\}}$.
Invariance of $\L$ for $X$ gives in particular that $\left[\sca{X_1,\L_\mt X_1}\right]\subseteq\L_\mt$ and hence $\L_\mt$ is positively invariant for $X_1$.
Lastly, we have
\[
\L_{\{1\}}\subseteq J_{\{1\}}(\L_\mt,X) \equiv J(\L_\mt,X_1).
\]
We conclude that $\L$ is a T-pair of $X_1$, as required.

Now suppose that $\L=\{\L_\mt,\L_{\{1\}}\}$ is a T-pair of $X_1$.
By definition, this means that $\L$ consists of ideals, $\L_\mt$ is positively invariant for $X_1$ and $\L_\mt\subseteq\L_{\{1\}}\subseteq J(\L_\mt,X_1)$.
From this it is clear that $\L$ is partially ordered, so condition (iii) of Definition \ref{D:NT tuple} holds.
Likewise, we have
\[
\L_{\{1\}}\subseteq J(\L_\mt,X_1) \equiv J_{\{1\}}(\L_\mt,X),
\]
so condition (i) of Definition \ref{D:NT tuple} holds.
To show that $\L$ is $X$-invariant, it suffices to show that $\sca{X_n,\L_\mt X_n}\subseteq\L_\mt$ for all $n\in\bZ_+$.
This is immediate by inducting on $n$.
Lastly, note that condition (iv) of Definition \ref{D:NT tuple} holds trivially, since there are no proper, non-empty subsets of $\{1\}$.
Thus $\L$ is an NT-$2$-tuple, completing the proof.
\end{proof}

The following proposition links Definition \ref{D:NT tuple} and Theorem \ref{T:m fam v2}.
Moreover, it shows that (M)-$2^d$-tuples form a subclass of NT-$2^d$-tuples. 

\begin{proposition}\label{P:NT tuple + m fam}
Let $X$ be a strong compactly aligned product system with coefficients in a C*-algebra $A$.
Let $\L$ be a $2^d$-tuple of $X$ consisting of ideals satisfying $\L_\mt\subseteq\L_F$ for all $F\subseteq[d]$, and assume that $\L_\mt$ is positively invariant for $X$.
Then $\L$ is an NT-$2^d$-tuple of $X$ if and only if $[\L]_{\L_\mt}$ is an (M)-$2^d$-tuple of $[X]_{\L_\mt}$.
\end{proposition}

\begin{proof}
Since $\L_\mt$ is an ideal that is positively invariant for $X$, we can make sense of the strong compactly aligned product system $[X]_{\L_\mt}$.
Moreover, the fact that $\L$ consists of ideals satisfying $\L_\mt\subseteq\L_F$ for all $F\subseteq[d]$ gives that $[\L]_{\L_\mt}$ is a $2^d$-tuple of $[X]_{\L_\mt}$ that consists of ideals.

The forward implication follows by applying Lemma \ref{L:quo ideal} and resorting to the characterisation of (M)-$2^d$-tuples obtained in Theorem \ref{T:m fam v2}.
For the converse, assume that $[\L]_{\L_\mt}$ is an (M)-$2^d$-tuple of $[X]_{\L_\mt}$.
It suffices to show that $\L$ satisfies the four conditions of Definition \ref{D:NT tuple}.

Firstly, we have that $\L$ consists of ideals by assumption.
Fix $\mt\neq F\subseteq[d]$.
Since $\L_\mt$ is positively invariant for $X$, we have 
\[
J_F(\L_\mt,X)=[\hspace{1pt}\cdot\hspace{1pt}]_{\L_\mt}^{-1}(\J_F([X]_{\L_\mt}))
\]
by Lemma \ref{L:quo ideal}.
Since $[\L]_{\L_\mt}$ is an (M)-$2^d$-tuple of $[X]_{\L_\mt}$, an application of Theorem \ref{T:m fam v2} gives that $[\L_F]_{\L_\mt}\subseteq\J_F([X]_{\L_\mt})$, from which it follows that 
\[
\L_F\subseteq[\hspace{1pt}\cdot\hspace{1pt}]_{\L_\mt}^{-1}(\J_F([X]_{\L_\mt}))=J_F(\L_\mt,X),
\]
proving that condition (i) of Definition \ref{D:NT tuple} holds.

To see that $\L$ is $X$-invariant, fix $F\subseteq[d]$ and $\un{n}\perp F$.
We must show that $\sca{X_\un{n},\L_FX_\un{n}}\subseteq\L_F$.
By Theorem \ref{T:m fam v2}, we have that $[\L]_{\L_\mt}$ is $[X]_{\L_\mt}$-invariant, and therefore
\[
[\sca{X_\un{n},\L_FX_\un{n}}]_{\L_\mt}=\sca{[X_\un{n}]_{\L_\mt},[\L_F]_{\L_\mt}[X_\un{n}]_{\L_\mt}} \subseteq [\L_F]_{\L_\mt}.
\]
Hence $\sca{X_\un{n},\L_FX_\un{n}} \subseteq \L_F+\L_\mt$.
However, recall that $\L_F$ is an ideal that contains $\L_\mt$, so $\L_F+\L_\mt=\L_F$ and therefore $\sca{X_\un{n},\L_FX_\un{n}} \subseteq \L_F$.
Thus condition (ii) of Definition \ref{D:NT tuple} holds.

Next we check that $\L$ is partially ordered. 
To this end, fix $F\subseteq D\subseteq[d]$.
By Theorem \ref{T:m fam v2}, we have that $[\L]_{\L_\mt}$ is partially ordered and therefore $[\L_F]_{\L_\mt}\subseteq[\L_D]_{\L_\mt}$.
Since $\L_\mt \subseteq \L_D$, we obtain $\L_F \subseteq \L_D$, showing that condition (iii) of Definition \ref{D:NT tuple} holds.

Finally, condition (iv) of Definition \ref{D:NT tuple} follows from condition (iv) of Theorem \ref{T:m fam v2} applied to $[\L]_{\L_\mt}$, and the proof is complete.
\end{proof}

The interplay between NT-$2^d$-tuples and (M)-$2^d$-tuples allows the transferal of properties of (M)-$2^d$-tuples to the general setting, towards the complete parametrisation of the gauge-invariant ideals of $\N\T_X$.
To explore this further, we examine the interaction between NT-$2^d$-tuples and Nica-covariant representations.
The following lemma extends \cite[Lemma 5.10]{Kat07}.

\begin{lemma}\label{L:rep quo}
Let $X$ be a strong compactly aligned product system with coefficients in a C*-algebra $A$. 
Let $(\pi,t)$ be a Nica-covariant representation of $X$ and let $\L^{(\pi,t)}$ be the associated $2^d$-tuple of $X$ (see Definition \ref{D:rep fam}).
Then the following hold:
\begin{enumerate}
\item $\L_\mt^{(\pi,t)}$ is positively invariant for $X$.
\item $\ker t_{\un{n}}=X_{\un{n}}\L_\mt^{(\pi,t)}$ for all $\un{n}\in\bZ_+^d$.
\item There exists an injective Nica-covariant representation $(\dot{\pi},\dot{t})$ of $[X]_{\L_\mt^{(\pi,t)}}$ in $\ca(\pi,t)$ such that $\pi=\dot{\pi}\circ [\hspace{1pt}\cdot \hspace{1pt}]_{\L_\mt^{(\pi,t)}}, t_{\un{n}}=\dot{t}_{\un{n}}\circ[\hspace{1pt}\cdot\hspace{1pt}]_{\L_\mt^{(\pi,t)}}$ and $\psi_\un{n}=\dot{\psi}_\un{n}\circ [\hspace{1pt}\cdot\hspace{1pt}]_{\L_\mt^{(\pi,t)}}|_{\K(X_\un{n})}$ for all $\un{n} \in \bZ_+^d$, and therefore $\ca(\dot{\pi},\dot{t})=\ca(\pi,t)$.
If $(\pi,t)$ admits a gauge action, then so does $(\dot{\pi},\dot{t})$.
\item For each $\mt\neq F\subseteq [d]$, if $a\in\L_F^{(\pi,t)}$ then $[\phi_{\un{i}}(a)]_{\L_\mt^{(\pi,t)}}\in\K([X_{\un{i}}]_{\L_\mt^{(\pi,t)}})$ for all $i\in[d]$, and
\[
\dot{\pi}([a]_{\L_\mt^{(\pi,t)}})+\sum\{(-1)^{|\un{n}|}\dot{\psi}_{\un{n}}([\phi_{\un{n}}(a)]_{\L_\mt^{(\pi,t)}})\mid \un{0}\neq\un{n}\leq\un{1}_F\}=0.
\]
\item For each $\mt\neq F\subseteq[d]$, we have that $a\in\L_F^{(\pi,t)}$ if and only if for every $\un{0}\neq\un{n}\leq\un{1}_F$ there exists $k_\un{n}\in\K(X_\un{n})$ such that $[\phi_\un{n}(a)]_{\L_\mt^{(\pi,t)}}=[k_\un{n}]_{\L_\mt^{(\pi,t)}}$, satisfying
\[
\pi(a)+\sum\{(-1)^{|\un{n}|}\psi_\un{n}(k_\un{n})\mid \un{0}\neq\un{n}\leq\un{1}_F\}=0.
\]
\item For each $F\subseteq[d]$, we have $[\L_F^{(\pi,t)}]_{\L_\mt^{(\pi,t)}}=\L_F^{(\dot{\pi},\dot{t})}$.
\end{enumerate}
\end{lemma}

\begin{proof}
(i) We have that $\L^{(\pi,t)}$ is $X$-invariant by Proposition \ref{P:inj e fam}, and thus in particular $\L_\mt^{(\pi,t)}$ is positively invariant for $X$.

\smallskip

\noindent
(ii) This follows by \cite[Lemma 5.10 (ii)]{Kat07}, since $(\pi, t_{\un{n}})$ is a representation of $X_{\un{n}}$ for all $\un{n} \in \bZ_+^d$.

\smallskip

\noindent
(iii) Existence and injectivity of $(\dot{\pi}, \dot{t})$ follow by \cite[Lemma 5.10 (iii)]{Kat07}, applied to each $(\pi, t_\un{n})$.
By the form of $\dot{t}_\un{n}$ and Lemma \ref{L:Kat07}, it also follows that $\psi_\un{n}=\dot{\psi}_\un{n}\circ [\hspace{1pt}\cdot\hspace{1pt}]_{\L_\mt^{(\pi,t)}}|_{\K(X_\un{n})}$ for all $\un{n}\in\bZ_+^d$.
To see that $(\dot{\pi},\dot{t})$ is Nica-covariant, fix $\un{n},\un{m}\in\bZ_+^d \setminus \{\un{0}\}$.
By Lemma \ref{L:Kat07}, it suffices to check (\ref{Eq:nc}) for $[k_\un{n}]_{\L_\mt^{(\pi,t)}}$ and $[k_\un{m}]_{\L_\mt^{(\pi,t)}}$, where $k_\un{n} \in\K(X_\un{n})$ and $k_\un{m} \in \K(X_\un{m})$. 
Letting $\{\iota_{\un{n}}^{\un{n}+\un{m}}\}_{\un{n}, \un{m} \in \bZ_+^d}$ denote the connecting $*$-homomorphisms of $X$ and letting $\{j_{\un{n}}^{\un{n}+\un{m}}\}_{\un{n}, \un{m} \in \bZ_+^d}$ denote those of $[X]_{\L_\mt^{(\pi,t)}}$, by Proposition \ref{P:qnt ca} and Nica-covariance of $(\pi,t)$ we have
\begin{align*}
\dot{\psi}_\un{n}([k_\un{n}]_{\L_\mt^{(\pi,t)}})\dot{\psi}_\un{m}([k_\un{m}]_{\L_\mt^{(\pi,t)}}) 
& =
\psi_\un{n}(k_\un{n})\psi_\un{m}(k_\un{m}) \\
& =
\psi_{\un{n}\vee\un{m}}(\iota_\un{n}^{\un{n}\vee\un{m}}(k_\un{n})\iota_\un{m}^{\un{n}\vee\un{m}}(k_\un{m})) \\
& =
\dot{\psi}_{\un{n}\vee\un{m}}([\iota_\un{n}^{\un{n}\vee\un{m}}(k_\un{n})\iota_\un{m}^{\un{n}\vee\un{m}}(k_\un{m})]_{\L_\mt^{(\pi,t)}}) \\
& =
\dot{\psi}_{\un{n}\vee\un{m}}(j_\un{n}^{\un{n}\vee\un{m}}([k_\un{n}]_{\L_\mt^{(\pi,t)}})j_\un{m}^{\un{n}\vee\un{m}}([k_\un{m}]_{\L_\mt^{(\pi,t)}})),
\end{align*}
as required.

If $(\pi,t)$ admits a gauge action, then this is inherited by $(\dot{\pi},\dot{t})$ since $\ca(\dot{\pi},\dot{t})=\ca(\pi,t)$, finishing the proof of item (iii).

\smallskip

\noindent
(iv) Fix $\mt\neq F\subseteq[d]$ and $a\in\L_F^{(\pi,t)}$. 
Then $\pi(a) \in B_{(\un{0}, \un{1}_F]}^{(\pi,t)}$ by definition, and thus $\dot{\pi}([a]_{\L_\mt^{(\pi,t)}}) \in B_{(\un{0},\un{1}_F]}^{(\dot{\pi},\dot{t})}$ by item (iii).
In turn, we have $\dot{\pi}([a]_{\L_\mt^{(\pi,t)}})\dot{q}_F=0$ by (\ref{eq:out}).
An application of (\ref{eq:comp}) gives that 
\[
[\phi_\un{i}(a)]_{\L_\mt^{(\pi,t)}}
=
[\phi_\un{i}]_{\L_\mt^{(\pi,t)}}([a]_{\L_\mt^{(\pi,t)}})
\in
\K([X_\un{i}]_{\L_\mt^{(\pi,t)}}) \foral i\in[d],
\]
using the fact that $(\dot{\pi},\dot{t})$ is injective and Nica-covariant by item (iii).
Applying (\ref{eq:expphi}) for $(\dot{\pi},\dot{t})$, we obtain
\[
\dot{\pi}([a]_{\L_\mt^{(\pi,t)}})+\sum\{(-1)^{|\un{n}|}\dot{\psi}_{\un{n}}([\phi_{\un{n}}(a)]_{\L_\mt^{(\pi,t)}})\mid\un{0}\neq\un{n}\leq\un{1}_F\}=0,
\]
showing that item (iv) holds.

\smallskip

\noindent
(v) Fix $\mt\neq F\subseteq[d]$.
The reverse implication is immediate, so assume that $a\in\L_F^{(\pi,t)}$.
By item (iv), we have 
\[
\dot{\pi}([a]_{\L_\mt^{(\pi,t)}})+\sum\{(-1)^{|\un{n}|}\dot{\psi}_{\un{n}}([\phi_{\un{n}}(a)]_{\L_\mt^{(\pi,t)}})\mid \un{0}\neq\un{n}\leq\un{1}_F\}=0.
\]
An application of Lemma \ref{L:Kat07} gives that $[\phi_\un{n}(a)]_{\L_\mt^{(\pi,t)}}=[k_\un{n}]_{\L_\mt^{(\pi,t)}}$ for some $k_\un{n}\in\K(X_\un{n})$, for all $\un{0}\neq\un{n}\leq\un{1}_F$.
By item (iii), we have
\begin{align*}
0
&=
\dot{\pi}([a]_{\L_\mt^{(\pi,t)}}) +\sum\{(-1)^{|\un{n}|}\dot{\psi}_{\un{n}}([\phi_{\un{n}}(a)]_{\L_\mt^{(\pi,t)}})\mid\un{0}\neq\un{n}\leq\un{1}_F\} \\
& =
\dot{\pi}([a]_{\L_\mt^{(\pi,t)}})+\sum\{(-1)^{|\un{n}|}\dot{\psi}_{\un{n}}([k_\un{n}]_{\L_\mt^{(\pi,t)}})\mid\un{0}\neq\un{n}\leq\un{1}_F\} \\
&=
\pi(a)+\sum\{(-1)^{|\un{n}|}\psi_\un{n}(k_\un{n})\mid \un{0}\neq\un{n}\leq\un{1}_F\},
\end{align*}
showing that item (v) holds.

\smallskip

\noindent
(vi) First recall that $\L_F^{(\pi,t)}$ is an ideal satisfying $\L_\mt^{(\pi,t)}\subseteq\L_F^{(\pi,t)}$ for all $F\subseteq[d]$ by Proposition \ref{P:inj e fam}.
Thus we can make sense of the $2^d$-tuple $[\L^{(\pi,t)}]_{\L_\mt^{(\pi,t)}}$ of $[X]_{\L_\mt^{(\pi,t)}}$.
The claim holds trivially when $F=\mt$ (because $(\dot{\pi},\dot{t})$ is injective), so fix $\mt\neq F\subseteq [d]$ and take $a\in\L_F^{(\pi,t)}$.
Then item (iv) yields $[a]_{\L_\mt^{(\pi,t)}}\in\L_F^{(\dot{\pi},\dot{t})}$.
This shows that $[\L_F^{(\pi,t)}]_{\L_\mt^{(\pi,t)}}\subseteq\L_F^{(\dot{\pi},\dot{t})}$.

Now take $[a]_{\L_\mt^{(\pi,t)}}\in\L_F^{(\dot{\pi},\dot{t})}$.
Then by definition and Lemma \ref{L:Kat07}, for every $\un{0}\neq\un{n}\leq\un{1}_F$ there exists $k_\un{n}\in\K(X_{\un{n}})$ such that
\[
\dot{\pi}([a]_{\L_\mt^{(\pi,t)}})=\sum\{\dot{\psi}_\un{n}([k_\un{n}]_{\L_\mt^{(\pi,t)}})\mid \un{0}\neq\un{n}\leq\un{1}_F\}.
\]
Using item (iii), we obtain
\begin{align*}
\pi(a)
=
\dot{\pi}([a]_{\L_\mt^{(\pi,t)}})
=
\sum\{\dot{\psi}_\un{n}([k_\un{n}]_{\L_\mt^{(\pi,t)}})\mid \un{0}\neq\un{n}\leq\un{1}_F\}
=
\sum\{\psi_\un{n}(k_\un{n})\mid\un{0}\neq\un{n}\leq\un{1}_F\}.
\end{align*}
This shows that $a\in\L_F^{(\pi,t)}$, and hence $[a]_{\L_\mt^{(\pi,t)}}\in[\L_F^{(\pi,t)}]_{\L_\mt^{(\pi,t)}}$.
Consequently $\L_F^{(\dot{\pi},\dot{t})}\subseteq[\L_F^{(\pi,t)}]_{\L_\mt^{(\pi,t)}}$ and hence $[\L_F^{(\pi,t)}]_{\L_\mt^{(\pi,t)}}=\L_F^{(\dot{\pi},\dot{t})}$ for all $F\subseteq[d]$, finishing the proof.
\end{proof}

\begin{remark}\label{R:rep quo}
Let $\L^{(\pi,t)}$ be the $2^d$-tuple of $X$ associated with a Nica-covariant representation $(\pi,t)$ of $X$, and let $(\dot{\pi},\dot{t})$ be the injective Nica-covariant representation of $[X]_{\L_\mt^{(\pi,t)}}$ defined in item (iii) of Lemma \ref{L:rep quo}.
By Proposition \ref{P:inj rel}, we have that $(\dot{\pi},\dot{t})$ is an $\L^{(\dot{\pi},\dot{t})}$-relative CNP-representation of $[X]_{\L_\mt^{(\pi,t)}}$, giving the following commutative diagram 
\[
\xymatrix{
\N\T_{[X]_{\L_\mt^{(\pi,t)}}} \ar[dr] \ar[rr]^{\dot{\pi} \times \dot{t}} & & \ca(\dot{\pi},\dot{t}) \\
 & \N\O(\L^{(\dot{\pi},\dot{t})},[X]_{\L_\mt^{(\pi,t)}}) \ar[ur] &
}
\]
of canonical $*$-epimorphisms.
On the other hand, since $\L_\mt^{(\pi,t)}$ is positively invariant we obtain a canonical $*$-epimorphism $\N\T_X \to \N\T_{[X]_{\L_\mt^{(\pi,t)}}}$ that lifts the quotient map $X \to [X]_{\L_\mt^{(\pi,t)}}$.
We have $[\L^{(\pi,t)}]_{\L_\mt^{(\pi,t)}}=\L^{(\dot{\pi},\dot{t})}$ by item (vi) of Lemma \ref{L:rep quo}, and we have $\ca(\pi,t) = \ca(\dot{\pi}, \dot{t})$ by item (iii) of Lemma \ref{L:rep quo}.
Hence we obtain the following commutative diagram
\[
\xymatrix{
\N\T_X \ar[d] \ar[rr]^{\pi \times t} & & \ca(\pi,t) \\
\N\T_{[X]_{\L_\mt^{(\pi,t)}}} \ar[rr] \ar[urr]^{\dot{\pi}\times\dot{t}}& & \N\O([\L^{(\pi,t)}]_{\L_\mt^{(\pi,t)}},[X]_{\L_\mt^{(\pi,t)}}) \ar[u]
}
\]
of canonical $*$-epimorphisms.
\end{remark}

By Proposition \ref{P:rep m fam} and Theorem \ref{T:d GIUT M}, (M)-$2^d$-tuples are exactly of the form $\L^{(\pi,t)}$ for some injective Nica-covariant representation $(\pi,t)$ that admits a gauge action.
This is extended to NT-$2^d$-tuples by allowing $(\pi,t)$ to be non-injective.
We introduce the following notation.

\begin{definition}\label{D:piLtL}
Let $X$ be a strong compactly aligned product system with coefficients in a C*-algebra $A$.
For an NT-$2^d$-tuple $\L$ of $X$, we define the maps
\begin{align*}
\pi^\L & \colon A\to\N\O([\L]_{\L_\mt},[X]_{\L_\mt}); \pi^\L(a)=\pi_{[X]_{\L_\mt}}^{[\L]_{\L_\mt}}([a]_{\L_\mt})\foral a\in A, \\
t^\L_\un{n} & \colon X_\un{n}\to\N\O([\L]_{\L_\mt},[X]_{\L_\mt}); t^\L_\un{n}(\xi_\un{n})=t_{[X]_{\L_\mt},\un{n}}^{[\L]_{\L_\mt}}([\xi_\un{n}]_{\L_\mt})\foral\xi_\un{n}\in X_\un{n},\un{n}\in\bZ_+^d\setminus\{\un{0}\},
\end{align*}
where $(\pi_{[X]_{\L_\mt}}^{[\L]_{\L_\mt}},t_{[X]_{\L_\mt}}^{[\L]_{\L_\mt}})$ denotes the universal $[\L]_{\L_\mt}$-relative CNP-representation of $[X]_{\L_\mt}$.
\end{definition}

Checking that $(\pi^\L,t^\L)$ is a Nica-covariant representation is routine, as it is obtained from the canonical $*$-epimorphism
\[
\N\T_X \to \N\T_{[X]_{\L_\mt}} \to \N\O([\L]_{\L_\mt},[X]_{\L_\mt}),
\]
where we use that $\L_\mt$ is positively invariant for the existence of the first map, and that $[\L]_{\L_\mt}$ is an (M)-$2^d$-tuple (and thus in particular relative) of $[X]_{\L_\mt}$ by Proposition \ref{P:NT tuple + m fam} for the existence of the second map.
Additionally, notice that $(\pi^\L,t^\L)$ admits a  gauge action and that
\[
\ca(\pi^\L,t^\L) = \ca(\pi_{[X]_{\L_\mt}}^{[\L]_{\L_\mt}},t_{[X]_{\L_\mt}}^{[\L]_{\L_\mt}}) = \N\O([\L]_{\L_\mt},[X]_{\L_\mt}).
\]

\begin{proposition}\label{P:NT tuple all rep}
Let $X$ be a strong compactly aligned product system with coefficients in a C*-algebra $A$.
If $\L$ is an NT-$2^d$-tuple of $X$, then $\L^{(\pi^\L,t^\L)}=\L$.

Consequently, a $2^d$-tuple $\L$ of $X$ is an NT-$2^d$-tuple if and only if $\L = \L^{(\pi,t)}$ for some Nica-covariant representation $(\pi,t)$ of $X$ that admits a gauge action.
\end{proposition}

\begin{proof}
For the first claim, we denote the universal $[\L]_{\L_\mt}$-relative CNP-representation of $[X]_{\L_\mt}$ by $(\tilde{\pi},\tilde{t})$.
First we show that $\L_\mt^{(\pi^\L,t^\L)}=\L_\mt$. 
We have that
\[
a\in\L_\mt^{(\pi^\L,t^\L)}\iff\pi^\L(a)=0\iff\tilde{\pi}([a]_{\L_\mt})=0\iff[a]_{\L_\mt}=0\iff a\in\L_\mt,
\]
using that $(\tilde{\pi},\tilde{t})$ is injective by Proposition \ref{P:NO inj}, since $[\L]_{\L_\mt}$ is an (M)-$2^d$-tuple of $[X]_{\L_\mt}$ by Proposition \ref{P:NT tuple + m fam}.
Hence $\L_\mt^{(\pi^\L,t^\L)}=\L_\mt$.

Recall that $\ca(\tilde{\pi},\tilde{t})=\N\O([\L]_{\L_\mt},[X]_{\L_\mt})$ by definition, and thus by applying Theorem \ref{T:d GIUT M} for the (M)-$2^d$-tuple $[\L]_{\L_\mt}$ of $[X]_{\L_\mt}$ and the Nica-covariant representation $(\tilde{\pi},\tilde{t})$ of $[X]_{\L_\mt}$, we obtain $\L^{(\tilde{\pi},\tilde{t})}=[\L]_{\L_\mt}$.
Hence, for $\mt \neq F \subseteq [d]$, we have
\begin{align*}
a\in\L_F^{(\pi^\L,t^\L)}
& \iff
\pi^\L(a) = \sum\{\psi^\L_\un{n}(k_\un{n})\mid\un{0}\neq\un{n}\leq\un{1}_F\} \text{ for some $k_{\un{n}} \in \K(X_{\un{n}})$} \\
& \iff
\tilde{\pi}([a]_{\L_\mt}) = \sum\{\tilde{\psi}_\un{n}([k_\un{n}]_{\L_\mt})\mid \un{0}\neq\un{n}\leq\un{1}_F\} \text{ for some $k_{\un{n}} \in \K(X_{\un{n}})$} \\
& \iff
[a]_{\L_\mt} \in \L_F^{(\tilde{\pi},\tilde{t})} = [\L_F]_{\L_\mt} \\
& \iff
a \in \L_F + \L_\mt = \L_F,
\end{align*}
and so $\L_F^{(\pi^\L,t^\L)}=\L_F$, completing the proof of the first part.

For the second part, if $\L$ is an NT-$2^d$-tuple of $X$ then $\L = \L^{(\pi,t)}$ for $(\pi,t) := (\pi^\L, t^\L)$.
Conversely, if $\L = \L^{(\pi,t)}$ for some Nica-covariant representation $(\pi,t)$ of $X$ that admits a gauge action, let $(\dot{\pi},\dot{t})$ be the injective Nica-covariant representation of $[X]_{\L_\mt}$ guaranteed by item (iii) of Lemma \ref{L:rep quo}.
Then $[\L]_{\L_\mt} = [\L^{(\pi,t)}]_{\L_\mt^{(\pi,t)}} = \L^{(\dot{\pi}, \dot{t})}$ by item (vi) of Lemma \ref{L:rep quo}.
However, we have that $\L^{(\dot{\pi}, \dot{t})}$ is an (M)-$2^d$-tuple of $[X]_{\L_\mt}$ by Proposition \ref{P:rep m fam}, and thus $\L$ is an NT-$2^d$-tuple of $X$ by Proposition \ref{P:NT tuple + m fam}.
\end{proof}

Consequently, we have an extension of Proposition \ref{P:inj rel} for describing the kernels of induced $*$-representations that may not be injective on $X$.

\begin{proposition}\label{P:int ker gen}
Let $X$ be a strong compactly aligned product system with coefficients in a C*-algebra $A$.
Let $(\pi,t)$ be a Nica-covariant representation of $X$ that admits a gauge action, and let $\L^{(\pi,t)}$ be the associated NT-$2^d$-tuple of $X$.
Then
\begin{align*}
\ker\pi\times t
=
\langle \ol{\pi}_X(a) +\sum_{\un{0}\neq\un{n}\leq\un{1}_F}(-1)^{|\un{n}|}\ol{\psi}_{X,\un{n}}(k_\un{n}) \mid
& F\subseteq[d], a\in\L_F^{(\pi,t)}, k_\un{n}\in\K(X_\un{n}), \\ 
& [\phi_\un{n}(a)]_{\L_\mt^{(\pi,t)}}=[k_\un{n}]_{\L_\mt^{(\pi,t)}} \; \textup{for all} \; \un{0}\neq\un{n}\leq\un{1}_F \rangle.
\end{align*}

\begin{proof}
We denote the ideal on the right hand side by $\mathfrak{J}$.
For notational convenience, we drop the superscript $(\pi,t)$ and write $\L := \L^{(\pi,t)}$.

We begin by proving that $\mathfrak{J}\subseteq\ker\pi\times t$.
To this end, it suffices to show that $\ker\pi\times t$ contains the generators of $\mathfrak{J}$.
The generators of $\mathfrak{J}$ that are indexed by $F = \mt$ have the form $\ol{\pi}_X(a)$ for some $a\in\L_\mt\equiv\ker\pi$.
In this case we have $(\pi\times t)(\ol{\pi}_X(a))=\pi(a)=0$ trivially, so $\ol{\pi}_X(a)\in\ker \pi\times t$, as required.

Next, fix $\mt\neq F\subseteq[d], a\in\L_F$ and $k_\un{n}\in\K(X_\un{n})$ such that $[\phi_\un{n}(a)]_{\L_\mt}=[k_\un{n}]_{\L_\mt}$ for all $\un{0}\neq\un{n}\leq\un{1}_F$.
Let $(\dot{\pi},\dot{t})$ be the injective Nica-covariant representation of $[X]_{\L_\mt}$ admitting a gauge action, as guaranteed by item (iii) of Lemma \ref{L:rep quo}.
We then have
\begin{align*}
(\pi\times t)(\ol{\pi}_X(a)+\sum\{(-1)^{|\un{n}|}\ol{\psi}_{X,\un{n}}(k_\un{n})\mid\un{0}\neq\un{n}\leq\un{1}_F\})
& = \\
& \hspace{-4cm} =
\pi(a)+\sum\{(-1)^{|\un{n}|}\psi_\un{n}(k_\un{n})\mid \un{0}\neq\un{n}\leq\un{1}_F\} \\
& \hspace{-4cm}  =
\dot{\pi}([a]_{\L_\mt})+\sum\{(-1)^{|\un{n}|}\dot{\psi}_\un{n}([k_\un{n}]_{\L_\mt})\mid \un{0}\neq\un{n}\leq\un{1}_F\} \\
& \hspace{-4cm}  =
\dot{\pi}([a]_{\L_\mt})+\sum\{(-1)^{|\un{n}|}\dot{\psi}_\un{n}([\phi_\un{n}(a)]_{\L_\mt})\mid \un{0}\neq\un{n}\leq\un{1}_F\} = 0,
\end{align*}
where we have used item (iv) of Lemma \ref{L:rep quo} to obtain the final equality.
This shows that $\mathfrak{J}\subseteq\ker\pi\times t$.

We have $[\L]_{\L_\mt}=\L^{(\dot{\pi},\dot{t})}$ by item (vi) of Lemma \ref{L:rep quo}, and thus by applying Proposition \ref{P:inj rel} for $(\dot{\pi}, \dot{t})$ we obtain a canonical $*$-isomorphism
\[
\Phi \colon \N\O([\L]_{\L_\mt},[X]_{\L_\mt}) \to \ca(\dot{\pi},\dot{t}) = \ca(\pi,t).
\] 
By considering the representation $(\pi^\L,t^\L)$ of Definition \ref{D:piLtL}, and the canonical quotient map $Q_\mathfrak{J} \colon \N\T_X \to \N\T_X/\fJ$, we obtain the following commutative diagram
\begin{center}
\begin{tikzcd}
\N\T_X
	\arrow[d, "Q_\mathfrak{J}" swap]
	 \arrow[dr, "\pi^\L\times t^\L"]
	 \arrow[r,"\pi\times t"] &
\ca(\pi,t) \\
\N\T_X/\mathfrak{J}
	\arrow[r, "\exists!\Psi" swap, dashed] &
\N\O([\L]_{\L_\mt},[X]_{\L_\mt})
	\arrow[u, "\Phi" swap, "\cong"]
\end{tikzcd}
\end{center}
of $*$-epimorphisms.
Note that $\Psi$ exists because
\[
\mathfrak{J}\subseteq\ker\pi\times t=\ker \Phi \circ (\pi^\L\times t^\L) = \ker \pi^\L\times t^\L.
\]
Therefore, to prove that $\ker\pi\times t\subseteq\fJ$, it suffices to show that $\Psi$ is injective.

To this end, we define maps $\tilde{\pi}$ and $\tilde{t}_\un{n}$ by
\begin{align*}
\tilde{\pi} & \colon[A]_{\L_\mt}\to\N\T_X/\mathfrak{J}; \tilde{\pi}([a]_{\L_\mt})=Q_\mathfrak{J}(\ol{\pi}_X(a)), \\
\tilde{t}_\un{n} & \colon[X_\un{n}]_{\L_\mt}\to\N\T_X/\mathfrak{J}; \tilde{t}_\un{n}([\xi_\un{n}]_{\L_\mt})=Q_\mathfrak{J}(\ol{t}_{X,\un{n}}(\xi_\un{n})),
\end{align*}
for all $a\in A, \xi_\un{n}\in X_\un{n}$ and $\un{n}\in\bZ_+^d\setminus\{\un{0}\}$.
These maps are well-defined because $\ol{\pi}_X(\L_\mt) \subseteq\mathfrak{J}$ and $\ol{t}_{X, \un{n}}(X_\un{n} \L_\mt) \subseteq\mathfrak{J}$.
It is routine to check that $(\tilde{\pi},\tilde{t})$ is a Nica-covariant representation of $[X]_{\L_\mt}$, since  $\tilde{\psi}_\un{n}\circ[\hspace{1pt}\cdot\hspace{1pt}]_{\L_\mt}|_{\K(X_\un{n})}=Q_\mathfrak{J}\circ\ol{\psi}_{X,\un{n}}$ for all $\un{n}\in\bZ_+^d$ by definition of $(\tilde{\pi}, \tilde{t})$.
Additionally, we have $\ca(\tilde{\pi},\tilde{t})=\N\T_X/\mathfrak{J}$.

We claim that $(\tilde{\pi},\tilde{t})$ is an $[\L]_{\L_\mt}$-relative CNP representation.
To see this, fix $\mt\neq F\subseteq[d]$ and $a\in\L_F$.
For each $\un{0}\neq\un{n}\leq\un{1}_F$, there exists $k_\un{n}\in\K(X_\un{n})$ such that $[\phi_\un{n}(a)]_{\L_\mt}=[k_\un{n}]_{\L_\mt}$ by item (iv) of Lemma \ref{L:rep quo} and Lemma \ref{L:Kat07}.
Hence we obtain
\begin{align*}
\tilde{\pi}([a]_{\L_\mt})\tilde{q}_F
& =\tilde{\pi}([a]_{\L_\mt})+\sum\{(-1)^{|\un{n}|}\tilde{\psi}_\un{n}([\phi_\un{n}(a)]_{\L_\mt})\mid \un{0}\neq\un{n}\leq\un{1}_F\} \\
& =
Q_\mathfrak{J}(\ol{\pi}_X(a))+\sum\{(-1)^{|\un{n}|}Q_\mathfrak{J}(\ol{\psi}_{X,\un{n}}(k_\un{n}))\mid\un{0}\neq\un{n}\leq\un{1}_F\} \\
& =
Q_\mathfrak{J}(\ol{\pi}_X(a)+\sum\{(-1)^{|\un{n}|}\ol{\psi}_{X,\un{n}}(k_\un{n})\mid \un{0}\neq\un{n}\leq\un{1}_F\}) 
= 0,
\end{align*}
using that $\ol{\pi}_X(a)+\sum\{(-1)^{|\un{n}|}\ol{\psi}_{X,\un{n}}(k_\un{n})\mid \un{0}\neq\un{n}\leq\un{1}_F\}\in\mathfrak{J}$ in the final equality.
This shows that $(\tilde{\pi},\tilde{t})$ is an $[\L]_{\L_\mt}$-relative CNP-representation, and hence universality of $\N\O([\L]_{\L_\mt},[X]_{\L_\mt})$ guarantees a canonical $*$-epimorphism 
\[
\tilde{\Phi} \colon \N\O([\L]_{\L_\mt},[X]_{\L_\mt})\to\N\T_X/\mathfrak{J} = \ca(\tilde{\pi},\tilde{t}).
\]
It is routine to check that $\tilde{\Phi} \circ \Psi=\id_{\N\T_X/\mathfrak{J}}$ and thus $\Psi$ is injective, as required.
\end{proof}
\end{proposition}

\subsection{Parametrising the gauge-invariant ideals of $\N\T_X$}\label{Ss:g inv param}

We can now pass to the parametrisation of gauge-invariant ideals by NT-$2^d$-tuples.
For an NT-$2^d$-tuple $\L$ of $X$, we write
\[
\mathfrak{J}^\L:=\ker\pi^\L\times t^\L,
\; \text{for} \;
\pi^\L\times t^\L\colon\N\T_X\to\N\O([\L]_{\L_\mt},[X]_{\L_\mt}).
\]
It is clear that $\mathfrak{J}^\L$ is a gauge-invariant ideal of $\N\T_X$, since $(\pi^\L, t^\L)$ admits a gauge action and $\pi^\L\times t^\L$ is equivariant.
On the other hand, for a gauge-invariant ideal $\mathfrak{J}$ of $\N\T_X$, we write
\[
\L^\mathfrak{J}:=\L^{(Q_\mathfrak{J}\circ\ol{\pi}_X, Q_\mathfrak{J}\circ\ol{t}_X)},
\; \text{for} \;
Q_\mathfrak{J}\colon\N\T_X\to\N\T_X/\mathfrak{J}.
\]
The $2^d$-tuple $\L^\fJ$ is an NT-$2^d$-tuple of $X$ by Proposition \ref{P:NT tuple all rep}.
The following proposition shows that these correspondences are mutually inverse.

\begin{proposition}\label{P:tuple ideal}
Let $X$ be a strong compactly aligned product system with coefficients in a C*-algebra $A$.
Then the following hold:
\begin{enumerate}
\item If $\L$ is an NT-$2^d$-tuple of $X$, then 
\[
\N\T_X/\mathfrak{J^\L}\cong\N\O([\L]_{\L_\mt},[X]_{\L_\mt})
\qand
\L^{\mathfrak{J}^\L}=\L.
\]

\item If $\mathfrak{J}\subseteq\N\T_X$ is a gauge-invariant ideal, then
\[
\N\T_X/\mathfrak{J}\cong\N\O([\L^\mathfrak{J}]_{\L_\mt^{\mathfrak{J}}},[X]_{\L_\mt^{\mathfrak{J}}})
\qand
\mathfrak{J}^{\L^\mathfrak{J}}=\mathfrak{J}.
\]
\end{enumerate}
\end{proposition}

\begin{proof}
(i) Since $\fJ^\L\equiv\ker \pi^\L \times t^\L$, we have $\N\T_X/\fJ^\L \cong \N\O([\L]_{\L_\mt},[X]_{\L_\mt})$ by a canonical $*$-isomorphism $\Phi$.
We then have
\[
\L^{\fJ^\L} \equiv
\L^{(Q_{\mathfrak{J}^\L} \circ\ol{\pi}_X, Q_{\mathfrak{J}^\L} \circ \ol{t}_X)} =
\L^{(\Phi^{-1} \circ \pi^\L, \Phi^{-1} \circ t^\L)} =
\L^{(\pi^\L, t^\L)} = \L ,
\]
using Proposition \ref{P:NT tuple all rep} in the last equality.

\smallskip

\noindent
(ii) Note that $(Q_\mathfrak{J}\circ\ol{\pi}_X,Q_\mathfrak{J}\circ\ol{t}_X)$ is a Nica-covariant representation of $X$ that admits a gauge action and satisfies $\ca(Q_\mathfrak{J}\circ\ol{\pi}_X,Q_\mathfrak{J}\circ\ol{t}_X)=\N\T_X/\mathfrak{J}$.
Hence, applying Remark \ref{R:rep quo} for the representation $(\pi,t):=(Q_\mathfrak{J}\circ\ol{\pi}_X,Q_\mathfrak{J}\circ\ol{t}_X)$, we obtain the following commutative diagram
\[
\xymatrix@C=2cm{
\N\T_X \ar[drr]^{\pi^{\L^\mathfrak{J}}\times t^{\L^\mathfrak{J}}} \ar[rr]^{Q_\mathfrak{J}} \ar[d] & & \N\T_X/\fJ \\
\N\T_{[X]_{\L_\mt^\fJ}} \ar[rr]_Q & & \N\O([\L^{\fJ}]_{\L_\mt^{\fJ}}, [X]_{\L_\mt^{\fJ}}) \ar[u]_\Psi
}
\]
of canonical $*$-epimorphisms.
We have $\dot{\pi} \times \dot{t} = \Psi \circ Q$, where $(\dot{\pi},\dot{t})$ is defined as in item (iii) of Lemma \ref{L:rep quo}, again taking $(\pi,t):=(Q_\fJ\circ\ol{\pi}_X,Q_\fJ\circ\ol{t}_X)$.
Since $\L^{\fJ}$ is an NT-$2^d$-tuple of $X$, we have that $[\L^{\fJ}]_{\L_\mt^{\fJ}}$ is an (M)-$2^d$-tuple of $[X]_{\L_\mt^\fJ}$ by Proposition \ref{P:NT tuple + m fam}.
We have $[\L^{\fJ}]_{\L_\mt^{\fJ}}=\L^{(\dot{\pi},\dot{t})}$ by item (vi) of Lemma \ref{L:rep quo}, and hence an application of Theorem \ref{T:d GIUT M} yields that $\Psi$ is a $*$-isomorphism. 

By definition we have $\mathfrak{J}^{\L^\mathfrak{J}}\equiv\ker\pi^{\L^\mathfrak{J}}\times t^{\L^\mathfrak{J}}$.
Therefore, we obtain
\[
\mathfrak{J} 
= \ker Q_{\fJ} 
= \ker \Psi \circ (\pi^{\L^\mathfrak{J}} \times t^{\L^\mathfrak{J}})
= \ker \pi^{\L^\mathfrak{J}} \times t^{\L^\mathfrak{J}}
\equiv \mathfrak{J}^{\L^\mathfrak{J}},
\]
and the proof is complete.
\end{proof}

Using Propositions \ref{P:NT tuple all rep} and \ref{P:int ker gen}, we arrive at a concrete description of $\mathfrak{J}^\L$.

\begin{proposition}\label{P:JL conc}
Let $X$ be a strong compactly aligned product system with coefficients in a C*-algebra $A$.
If $\L$ is an NT-$2^d$-tuple of $X$, then we have
\begin{align*}
\mathfrak{J}^\L
=
\langle\ol{\pi}_X(a) +\sum_{\un{0}\neq\un{n}\leq\un{1}_F}(-1)^{|\un{n}|}\ol{\psi}_{X,\un{n}}(k_\un{n}) \mid 
& F\subseteq[d], a\in\L_F,  k_\un{n}\in\K(X_\un{n}), \\
& [\phi_\un{n}(a)]_{\L_\mt}=[k_\un{n}]_{\L_\mt} \; \textup{for all} \; \un{0}\neq\un{n}\leq\un{1}_F \rangle.
\end{align*}
If in addition $\L$ is a relative $2^d$-tuple of $X$, then $\mathfrak{J}^\L=\mathfrak{J}_\L^{(\ol{\pi}_X,\ol{t}_X)}$.
\end{proposition}

\begin{proof}
For the first part we apply Proposition \ref{P:int ker gen} for $(\pi,t):=(\pi^\L,t^\L)$, noting that $\L^{(\pi^\L, t^\L)}=\L$ by Proposition \ref{P:NT tuple all rep}.

Now assume that $\L$ is in addition a relative $2^d$-tuple of $X$, and recall that
\[
\mathfrak{J}_\L^{(\ol{\pi}_X,\ol{t}_X)}=\langle\ol{\pi}_X(a)+\sum_{\un{0}\neq\un{n}\leq\un{1}_F}(-1)^{|\un{n}|}\ol{\psi}_{X,\un{n}}(\phi_\un{n}(a))\mid a\in\L_F, F\subseteq[d]\rangle.
\]
It is clear that $\mathfrak{J}_\L^{(\ol{\pi}_X,\ol{t}_X)}\subseteq\mathfrak{J}^\L$ by the first part.
For the reverse inclusion, first note that 
\[
\ol{\pi}_X(\L_\mt)\subseteq\mathfrak{J}_\L^{(\ol{\pi}_X,\ol{t}_X)}
\]
by definition.
Next fix $\mt\neq F\subseteq[d], a\in\L_F$ and $k_\un{n}\in\K(X_\un{n})$ satisfying $[\phi_\un{n}(a)]_{\L_\mt}=[k_\un{n}]_{\L_\mt}$ for all $\un{0}\neq\un{n}\leq\un{1}_F$.
By the first part, it suffices to show that
\[
\ol{\pi}_X(a)+\sum_{\un{0}\neq\un{n}\leq\un{1}_F}(-1)^{|\un{n}|}\ol{\psi}_{X,\un{n}}(k_\un{n})\in\mathfrak{J}_\L^{(\ol{\pi}_X,\ol{t}_X)}.
\]
Fixing $\un{0}\neq\un{n}\leq\un{1}_F$, we have $[\phi_\un{n}(a)-k_\un{n}]_{\L_\mt}=0$ by assumption.
Since $\L$ is a relative $2^d$-tuple of $X$, we also have $\phi_\un{n}(a)\in\K(X_\un{n})$ and hence $\phi_\un{n}(a)-k_\un{n}\in\K(X_\un{n})$.
By Lemma \ref{L:Kat07} we have
\[
\ker \{ [\hspace{1pt} \cdot \hspace{1pt}]_{\L_\mt} \colon \K(X_{\un{n}})\to\K([X_{\un{n}}]_{\L_\mt}) \} =  \K(X_{\un{n}} \L_\mt),
\]
and therefore $\phi_\un{n}(a)-k_\un{n} = k_\un{n}'$ for some $k_\un{n}' \in \K(X_\un{n}\L_\mt)$.
Notice that 
\[
\ol{\psi}_{X,\un{n}}(\K(X_\un{n}\L_\mt))\subseteq\sca{\ol{\pi}_X(\L_\mt)}\subseteq\mathfrak{J}_\L^{(\ol{\pi}_X,\ol{t}_X)}
\]
for all $\un{0}\neq\un{n}\leq\un{1}_F$.
In total, we have
\begin{align*}
\ol{\pi}_X(a)+\sum_{\un{0}\neq\un{n}\leq\un{1}_F}(-1)^{|\un{n}|} \ol{\psi}_{X,\un{n}}(k_\un{n})
& = \\
& \hspace{-3cm} =
\bigg( \ol{\pi}_X(a)+\sum_{\un{0}\neq\un{n}\leq\un{1}_F}(-1)^{|\un{n}|}\ol{\psi}_{X,\un{n}}(\phi_\un{n}(a)) \bigg)
- \sum_{\un{0}\neq\un{n}\leq\un{1}_F}(-1)^{|\un{n}|}\ol{\psi}_{X,\un{n}}(k_\un{n}') \\
& \hspace{-3cm} \in
\mathfrak{J}_\L^{(\ol{\pi}_X,\ol{t}_X)}+\mathfrak{J}_\L^{(\ol{\pi}_X,\ol{t}_X)} 
=
\mathfrak{J}_\L^{(\ol{\pi}_X,\ol{t}_X)},
\end{align*}
as required.
\end{proof}

We now present the main theorem.

\begin{theorem}\label{T:NT param}
Let $X$ be a strong compactly aligned product system with coefficients in a C*-algebra $A$.
Then the set of NT-$2^d$-tuples of $X$ corresponds bijectively to the set of gauge-invariant ideals of $\N\T_X$ by the mutually inverse maps $\L\mapsto\mathfrak{J}^\L$ and $\mathfrak{J}\mapsto\L^\mathfrak{J}$, for all NT-$2^d$-tuples $\L$ of $X$ and all gauge-invariant ideals $\mathfrak{J}$ of $\N\T_X$.
Moreover, these maps respect inclusions.
\end{theorem}

\begin{proof}
The fact that the maps are well-defined follows from the discussion preceding Proposition \ref{P:tuple ideal}.
The fact that the maps are mutual inverses is guaranteed by Proposition \ref{P:tuple ideal}.
It remains to see that the maps preserve inclusions.

To this end, first take NT-$2^d$-tuples $\L_1$ and $\L_2$ of $X$ and suppose that $\L_1\subseteq\L_2$.
We must show that $\fJ^{\L_1}\subseteq \fJ^{\L_2}$.
It suffices to show that $\fJ^{\L_2}$ contains the generators of $\fJ^{\L_1}$, recalling their form from Proposition \ref{P:JL conc}.
Firstly, we have 
\[
\ol{\pi}_X(\L_{1,\mt})\subseteq\ol{\pi}_X(\L_{2,\mt})\subseteq\fJ^{\L_2}
\]
by definition.
Next, fix $\mt\neq F\subseteq[d], a\in\L_{1,F}$ and $k_\un{n}\in\K(X_\un{n})$ such that $[k_\un{n}]_{\L_{1,\mt}}=[\phi_\un{n}(a)]_{\L_{1,\mt}}$ for all $\un{0}\neq\un{n}\leq\un{1}_F$.
For each $\un{0}\neq\un{n}\leq\un{1}_F$, we make the identification 
\[
\L([[X_\un{n}]_{\L_{1,\mt}}]_{\L_{2,\mt}/\L_{1,\mt}}) \cong \L([X_\un{n}]_{\L_{2,\mt}}),
\]
so that 
\[
[\hspace{1pt}\cdot\hspace{1pt}]_{\L_{2,\mt}}=[\hspace{1pt}\cdot\hspace{1pt}]_{\L_{2,\mt}/\L_{1,\mt}}\circ[\hspace{1pt}\cdot\hspace{1pt}]_{\L_{1,\mt}},
\]
e.g., \cite[p. 112]{Kat07}.
Under this identification, we obtain
\[
[k_\un{n}]_{\L_{2,\mt}}=[\phi_\un{n}(a)]_{\L_{2,\mt}}  \foral \un{0}\neq\un{n}\leq\un{1}_F.
\]
Since $a\in\L_{1,F}\subseteq\L_{2,F}$, it then follows that
\[
\ol{\pi}_X(a)+\sum_{\un{0}\neq\un{n}\leq\un{1}_F}(-1)^{|\un{n}|}\ol{\psi}_{X,\un{n}}(k_\un{n})\in\fJ^{\L_2},
\]
as required.

Finally, fix gauge-invariant ideals $\fJ_1$ and $\fJ_2$ of $\N\T_X$ such that $\fJ_1\subseteq\fJ_2$.
It then follows that
\[
\L^{\fJ_1}_F 
\equiv 
\L_F^{(Q_{\fJ_1} \circ \ol{\pi}_X, Q_{\fJ_1} \circ \ol{t}_X)}
\subseteq
\L_F^{(Q_{\fJ_2} \circ \ol{\pi}_X, Q_{\fJ_2} \circ \ol{t}_X)}
\equiv
\L^{\fJ_2}_F
\foral F \subseteq [d],
\]
since $Q_{\fJ_1}$ and $Q_{\fJ_2}$ preserve the indices of the cores and $Q_{\fJ_2}$ factors through $Q_{\fJ_1}$.
\end{proof}

\begin{remark}\label{R:param}
To summarise, we have that the mappings
\begin{align*}
\L\mapsto\ker\pi^\L\times t^\L & \; \text{for all NT-$2^d$-tuples $\L$ of $X$}, \\
\mathfrak{J}\mapsto\L^{(Q_\mathfrak{J}\circ\ol{\pi}_X,Q_\mathfrak{J}\circ\ol{t}_X)} & \; \text{for all gauge-invariant ideals $\mathfrak{J}\subseteq\N\T_X$},
\end{align*}
are mutual inverses and respect inclusions, where $Q_\mathfrak{J}\colon\N\T_X\to\N\T_X/\mathfrak{J}$ is the quotient map.
The gauge-invariant ideal $\ker\pi^\L\times t^\L$ of $\N\T_X$ is given by
\begin{align*}
\ker\pi^\L\times t^\L
=
\langle\ol{\pi}_X(a) +\sum_{\un{0}\neq\un{n}\leq\un{1}_F}(-1)^{|\un{n}|}\ol{\psi}_{X,\un{n}}(k_\un{n})\mid
& F\subseteq[d], a\in\L_F, k_\un{n}\in\K(X_\un{n}), \\
& [\phi_\un{n}(a)]_{\L_\mt}=[k_\un{n}]_{\L_\mt} \; \textup{for all} \; \un{0}\neq\un{n}\leq\un{1}_F\rangle,
\end{align*}
by Proposition \ref{P:JL conc}, where $\L$ is an NT-$2^d$-tuple of $X$.
The NT-$2^d$-tuple $\L^{(Q_\mathfrak{J}\circ\ol{\pi}_X,Q_\mathfrak{J}\circ\ol{t}_X)}$ of $X$ is given by \[
\L_F^{(Q_\mathfrak{J}\circ\ol{\pi}_X,Q_\mathfrak{J}\circ\ol{t}_X)}
=
\begin{cases}
\ker Q_\mathfrak{J}\circ\ol{\pi}_X & \text{ if } F = \mt, \\
(Q_\mathfrak{J}\circ\ol{\pi}_X)^{-1}(B_{(\un{0},\un{1}_F]}^{(Q_\mathfrak{J}\circ\ol{\pi}_X,Q_\mathfrak{J}\circ\ol{t}_X)}) & \text{ if } \mt \neq F\subseteq[d],
\end{cases}
\]
where $\mathfrak{J}$ is a gauge-invariant ideal of $\N\T_X$.
\end{remark}

Note that the set of gauge-invariant ideals of $\N\T_X$ carries a canonical lattice structure, determined by the operations
\begin{align*}
\mathfrak{J}_1\vee\mathfrak{J}_2 :=\fJ_1+\fJ_2 \qand
\fJ_1\wedge\fJ_2 :=\fJ_1\cap\fJ_2,
\end{align*}
for all gauge-invariant ideals $\fJ_1,\fJ_2\subseteq\N\T_X$.
This, in tandem with Theorem \ref{T:NT param}, allows us to impose a canonical lattice structure on the set of NT-$2^d$-tuples of $X$, promoting the bijection to a lattice isomorphism.

\begin{definition}\label{D:NTlattice}
Let $X$ be a strong compactly aligned product system with coefficients in a C*-algebra $A$.
We equip the set of NT-$2^d$-tuples of $X$ with the lattice structure determined by the operations
\begin{align*}
\L_1\vee\L_2  :=\L^{\fJ^{\L_1}+\fJ^{\L_2}} \qand
\L_1\wedge\L_2  :=\L^{\fJ^{\L_1}\cap\fJ^{\L_2}},
\end{align*}
for all NT-$2^d$-tuples $\L_1$ and $\L_2$ of $X$.
\end{definition}

Next we describe the operations $\wedge$ and $\vee$ on the set of NT-$2^d$-tuples of $X$.
The operation $\wedge$ is intersection, in accordance with \cite[Proposition 5.8]{Kat07}.

\begin{proposition}\label{P:NTwedge}
Let $X$ be a strong compactly aligned product system with coefficients in a C*-algebra $A$ and let $\L_1$ and $\L_2$ be NT-$2^d$-tuples of $X$.
Then
\[
(\L_1\wedge\L_2)_F=\L_{1,F}\cap\L_{2,F}\foral F\subseteq[d].
\]
\end{proposition}

\begin{proof}
For notational convenience we set $\fJ:=\fJ^{\L_1}\cap\fJ^{\L_2}$, so that $\L^\fJ \equiv \L_1\wedge\L_2$.
For $F = \mt$ and $a \in A$, we have that
\[
a\in\L_\mt^\fJ\iff\ol{\pi}_X(a)\in\fJ\equiv\fJ^{\L_1}\cap\fJ^{\L_2}\iff a\in\L_\mt^{\fJ^{\L_1}}\cap\L_\mt^{\fJ^{\L_2}}=\L_{1,\mt}\cap\L_{2,\mt},
\]
using item (i) of Proposition \ref{P:tuple ideal} in the final equality.
Hence $(\L_1\wedge\L_2)_\mt=\L_{1,\mt}\cap\L_{2,\mt}$.

Next, fix $\mt\neq F\subseteq[d]$.
Since the parametrisation of Theorem \ref{T:NT param} preserves inclusions, we have $(\L_1\wedge\L_2)_F\subseteq\L_{1,F}\cap\L_{2,F}$.
For the reverse inclusion, take $a\in\L_{1,F}\cap\L_{2,F}$.
Since $(\L_1\wedge\L_2)_F$ is an ideal, it suffices to show that $aa^* \in (\L_1\wedge\L_2)_F$.
Since $a\in\L_{1,F}\cap\L_{2,F}$, there exist $f, g \in B_{(\un{0},\un{1}_F]}^{(\ol{\pi}_X,\ol{t}_X)}$ such that 
\[
\ol{\pi}_X(a) + f \in\fJ^{\L_1}
\qand
\ol{\pi}_X(a) + g \in\fJ^{\L_2}.
\]
Consider the element
\[
h := \ol{\pi}_X(a)g^* + f \ol{\pi}_X(a)^* + fg^*.
\]
Note that 
\[
\ol{\pi}_X(A)B_{(\un{0},\un{1}_F]}^{(\ol{\pi}_X,\ol{t}_X)}\subseteq B_{(\un{0},\un{1}_F]}^{(\ol{\pi}_X,\ol{t}_X)}
\qand
B_{(\un{0},\un{1}_F]}^{(\ol{\pi}_X,\ol{t}_X)}\ol{\pi}_X(A) \subseteq B_{(\un{0},\un{1}_F]}^{(\ol{\pi}_X,\ol{t}_X)},
\]
and recall that $B_{(\un{0},\un{1}_F]}^{(\ol{\pi}_X,\ol{t}_X)}$ is a C*-algebra. 
Hence $h \in B_{(\un{0},\un{1}_F]}^{(\ol{\pi}_X,\ol{t}_X)}$, and we obtain
\[
\ol{\pi}_X(aa^*) + h = (\ol{\pi}_X(a) + f) (\ol{\pi}_X(a) + g)^* \in \fJ^{\L_1} \cap \fJ^{\L_2} \equiv \fJ.
\]
By definition this means that $aa^* \in\L_F^\fJ \equiv (\L_1\wedge\L_2)_F$, as required.
\end{proof}

We have the following characterisation of the operation $\vee$.

\begin{proposition}\label{P:NTvee}
Let $X$ be a strong compactly aligned product system with coefficients in a C*-algebra $A$ and let $\L_1$ and $\L_2$ be NT-$2^d$-tuples of $X$.
Then
\begin{align*}
(\L_1\vee\L_2)_F
=
\begin{cases}
\ol{\pi}_X^{-1}(\fJ^{\L_1}+\fJ^{\L_2}) & \text{ if } F= \mt, \\
[\hspace{1pt}\cdot\hspace{1pt}]_{(\L_1\vee\L_2)_\mt}^{-1} \left[ \big( \left( \L_{1,F}+\L_{2,F}+(\L_1\vee\L_2)_\mt \right) / (\L_1\vee\L_2)_\mt \big)^{(d-1)} \right] & \text{ if } \mt\neq F\subseteq[d].
\end{cases}
\end{align*}
\end{proposition}

\begin{proof}
For notational convenience, we set $\L := \L_1 \vee \L_2$ and $\fJ:=\fJ^{\L_1}+\fJ^{\L_2}$, so that $\L \equiv \L^\fJ$.
For $F = \mt$ and $a\in A$, we have that
\[
a\in\L_\mt\iff Q_\fJ(\ol{\pi}_X(a))=0\iff \ol{\pi}_X(a)\in\fJ\iff a\in\ol{\pi}_X^{-1}(\fJ)=\ol{\pi}_X^{-1}(\fJ^{\L_1}+\fJ^{\L_2}).
\]
Consequently, we obtain $\L_\mt=\ol{\pi}_X^{-1}(\fJ^{\L_1}+\fJ^{\L_2})$, as required.

Next, in a slight abuse of notation, we denote by $\L_{1}+\L_{2}+\L_\mt$ the $2^d$-tuple of $X$ defined by
\[
(\L_{1}+\L_{2}+\L_\mt)_F := \L_{1,F}+\L_{2,F}+\L_\mt \foral F\subseteq[d].
\]
Note that $\L_{1}+\L_{2}+\L_\mt$ consists of ideals and that $\L_\mt\subseteq(\L_1+\L_2+\L_\mt)_F$ for all $F\subseteq[d]$, so we can make sense of the $2^d$-tuple $[\L_1+\L_2+\L_\mt]_{\L_\mt}$ of $[X]_{\L_\mt}$.
First we check that the family of ideals $[\L_{1}+\L_{2}+\L_\mt]_{\L_\mt}$ is an (E)-$2^d$-tuple of $[X]_{\L_\mt}$ that is invariant and partially ordered.

The fact that $[\L_{1}+\L_{2}+\L_\mt]_{\L_\mt}$ is invariant and partially ordered follows from the corresponding properties of the NT-$2^d$-tuples $\L_1$ and $\L_2$, as well as the fact that $\L_\mt$ is positively invariant for $X$.
Next recall that $[\L]_{\L_\mt}$ is an (M)-$2^d$-tuple, and thus it is contained in $\I([X]_{\L_\mt})$.
Hence, in order to show that $[\L_{1}+\L_{2}+\L_\mt]_{\L_\mt}$ is an (E)-$2^d$-tuple, it suffices to show that
\[
[\L_{1}+\L_{2}+\L_\mt]_{\L_\mt}\subseteq [\L]_{\L_\mt}.
\]
In turn, it suffices to show that
\begin{equation}\label{eq:inc l}
\L_{1,F}+\L_{2,F}+\L_\mt\subseteq\L_F \foral F\subseteq[d].
\end{equation}
This is immediate, since $\fJ^{\L_1}, \fJ^{\L_2} \subseteq \fJ$ and the parametrisation of Theorem \ref{T:NT param} respects inclusions.
Thus $\L_{1}, \L_{2} \subseteq \L$, and by definition $\L_\mt \subseteq \L_F$, showing that (\ref{eq:inc l}) holds.

Hence we may consider the $(d-1)$-iteration $[\L_{1}+\L_{2}+\L_\mt]_{\L_\mt}^{(d-1)}$.
For notational convenience, let $\L'$ be the $2^d$-tuple of $X$ defined by
\[
\L_F'
:=
\begin{cases} 
\L_\mt & \text{if } F=\mt, \\
[\hspace{1pt}\cdot\hspace{1pt}]_{\L_\mt}^{-1} \big([\L_{1,F}+\L_{2,F}+\L_\mt]_{\L_\mt}^{(d-1)}\big) & \text{if } \mt \neq F\subseteq[d],
\end{cases}
\]
as per the statement of the proposition.
Note that $\L'$ consists of ideals and satisfies $\L_\mt'\subseteq\L_F'$ for all $F\subseteq[d]$.
Moreover, we have $[\L']_{\L_\mt}=[\L_1+\L_2+\L_\mt]_{\L_\mt}^{(d-1)}$ is the (M)-$2^d$-tuple of $[X]_{\L_\mt}$ that induces $\fJ_{[\L_{1}+\L_{2}+\L_\mt]_{\L_\mt}}^{(\ol{\pi}_{[X]_{\L_\mt}},\ol{t}_{[X]_{\L_\mt}})}$ by Theorem \ref{T:d-1 m fam}.
In particular, $\L'$ is an NT-$2^d$-tuple of $X$ by an application of Proposition \ref{P:NT tuple + m fam}.

It now suffices to show that $\fJ=\fJ^{\L'}$, as the parametrisation of Theorem \ref{T:NT param} then yields
\[
\L \equiv \L^\fJ = \L^{\fJ^{\L'}}=\L',
\]
as required.
To this end, we construct the following commutative diagram
\begin{center}
\begin{tikzcd}
& & & & \N\T_X/\fJ \\
\N\T_X
	\arrow[urrrr,"Q_\fJ"] 
	\arrow[rrr,"\phantom{ooo} Q"]
	\arrow[drrrr,"Q_{\fJ^{\L'}}", swap]
& & & \N\T_{[X]_{\L_\mt}}
	\arrow[ur,"\Phi_\fJ", swap]
	\arrow[dr,"\Phi_{\fJ^{\L'}}"] & \\
& & & & \N\T_X/\fJ^{\L'}
\end{tikzcd}
\end{center}
of $*$-epimorphisms, where $Q$ is the usual lift of $X \to [X]_{\L_\mt}$.
The maps 
\[
Q_{\fJ} \equiv (Q_\fJ\circ\ol{\pi}_X)\times(Q_\fJ\circ\ol{t}_X)
\qand
Q_{\fJ^{\L'}} \equiv (Q_{\fJ^{\L'}}\circ\ol{\pi}_X)\times(Q_{\fJ^{\L'}}\circ\ol{t}_X)
\]
are the canonical quotient maps.
The map $\Phi_\fJ$ is induced by the injective Nica-covariant representation of $[X]_{\L_\mt}$ obtained by applying item (iii) of Lemma \ref{L:rep quo} for $(\pi,t):=(Q_\fJ\circ\ol{\pi}_X,Q_\fJ\circ\ol{t}_X)$.
The map $\Phi_{\fJ^{\L'}}$ is obtained analogously, using $(Q_{\fJ^{\L'}}\circ\ol{\pi}_X,Q_{\fJ^{\L'}}\circ\ol{t}_X)$ in place of $(Q_\fJ\circ\ol{\pi}_X,Q_\fJ\circ\ol{t}_X)$ and the fact that $\L_\mt' = \L_\mt$.
Since each map is canonical, it suffices to show that the kernel $\ker \Phi_{\fJ} = Q(\fJ)$ coincides with the kernel $\ker \Phi_{\fJ^{\L'}} = Q(\fJ^{\L'})$.
Then $\N\T_X/\fJ \cong \N\T_X/\fJ^{\L'}$ by the map $f + \fJ \mapsto f + \fJ^{\L'}$ for all $f \in \N\T_X$.
To this end, we have the following three claims.

\smallskip

\noindent
\textit{Claim 1. With the aforementioned notation, we have
\[
Q(\fJ^{\L'}) = \fJ_{[\L']_{\L_\mt}}^{(\ol{\pi}_{[X]_{\L_\mt}},\ol{t}_{[X]_{\L_\mt}})}.
\]
}

\noindent
\textit{Proof of Claim 1.}
By an application of item (i) of Proposition \ref{P:tuple ideal}, we have a canonical $*$-isomorphism
\[
\N\T_X/\mathfrak{J^{\L'}}
\cong
\N\O([\L']_{\L_\mt'},[X]_{\L_\mt'})
=
\N\T_{[X]_{\L_\mt}}/ \fJ_{[\L']_{\L_\mt}}^{(\ol{\pi}_{[X]_{\L_\mt}},\ol{t}_{[X]_{\L_\mt}})},
\]
using that $\L'$ is an NT-$2^d$-tuple and that $\L_\mt' = \L_\mt$.
This $*$-isomorphism is induced by $Q$  in the sense that it maps $f + \fJ^{\L'}$ to $Q(f) +  \fJ_{[\L']_{\L_\mt}}^{(\ol{\pi}_{[X]_{\L_\mt}},\ol{t}_{[X]_{\L_\mt}})}$ for all $f \in \N\T_X$.
It follows that $Q(\fJ^{\L'})=\fJ_{[\L']_{\L_\mt}}^{(\ol{\pi}_{[X]_{\L_\mt}},\ol{t}_{[X]_{\L_\mt}})}$, as required.
\hfill{$\Box$}

\smallskip

\noindent
\textit{Claim 2. Let $\L_1'$ and $\L_2'$ be the $2^d$-tuples of $X$ defined by $\L_{1,F}':= \L_{1,F}+\L_\mt$ and $\L_{2,F}':=\L_{2,F}+\L_\mt$ for all $F\subseteq[d]$.
Then $[\L_1']_{\L_\mt}$ and $[\L_2']_{\L_\mt}$ are (E)-$2^d$-tuples of $[X]_{\L_\mt}$ that consist of ideals, and
\[
\fJ_{[\L']_{\L_\mt}}^{(\ol{\pi}_{[X]_{\L_\mt}},\ol{t}_{[X]_{\L_\mt}})}
=
\fJ_{[\L_1']_{\L_{\mt}}}^{(\ol{\pi}_{[X]_{\L_\mt}},\ol{t}_{[X]_{\L_\mt}})}+\fJ_{[\L_2']_{\L_{\mt}}}^{(\ol{\pi}_{[X]_{\L_\mt}},\ol{t}_{[X]_{\L_\mt}})}.
\]
}

\noindent
\textit{Proof of Claim 2.}
Both $[\L_1']_{\L_{\mt}}$ and $[\L_2']_{\L_{\mt}}$ consist of ideals of $[A]_{\L_\mt}$ and are (E)-$2^d$-tuples of $[X]_{\L_\mt}$, since both are contained in the (E)-$2^d$-tuple $[\L_1+\L_2+\L_\mt]_{\L_\mt}$ of $[X]_{\L_\mt}$.
Moreover, it is routine to check that $[\L_1']_{\L_{\mt}} + [\L_2']_{\L_{\mt}} = [\L_1+\L_2+\L_\mt]_{\L_\mt}$.
Hence we obtain
\[
\fJ_{[\L']_{\L_\mt}}^{(\ol{\pi}_{[X]_{\L_\mt}},\ol{t}_{[X]_{\L_\mt}})}
=
\fJ_{[\L_1+\L_2+\L_\mt]_{\L_\mt}}^{(\ol{\pi}_{[X]_{\L_\mt}},\ol{t}_{[X]_{\L_\mt}})}
=
\fJ_{[\L_1']_{\L_{\mt}} + [\L_2']_{\L_{\mt}}}^{(\ol{\pi}_{[X]_{\L_\mt}},\ol{t}_{[X]_{\L_\mt}})}
=
\fJ_{[\L_1']_{\L_{\mt}}}^{(\ol{\pi}_{[X]_{\L_\mt}},\ol{t}_{[X]_{\L_\mt}})} + \fJ_{[\L_2']_{\L_{\mt}}}^{(\ol{\pi}_{[X]_{\L_\mt}},\ol{t}_{[X]_{\L_\mt}})},
\]
using Lemma \ref{L:sumofe} in the final equality.
\hfill{$\Box$}

\smallskip

\noindent
\textit{Claim 3. With the aforementioned notation, we have
\[
Q(\fJ)
=
\fJ_{[\L_1']_{\L_{\mt}}}^{(\ol{\pi}_{[X]_{\L_\mt}},\ol{t}_{[X]_{\L_\mt}})}+\fJ_{[\L_2']_{\L_{\mt}}}^{(\ol{\pi}_{[X]_{\L_\mt}},\ol{t}_{[X]_{\L_\mt}})}.
\]
}

\noindent
\textit{Proof of Claim 3.}
For notational convenience, we will denote the right hand side by $\fJ'$.
For the forward inclusion, we show that $Q(f)\in\fJ'$ for all generators $f$ of $\fJ^{\L_1}$.
The same holds for $\fJ^{\L_2}$ by symmetry, giving that $Q(\fJ) \subseteq \fJ'$.
To this end, we resort to Proposition \ref{P:JL conc}.
Note that
\[
Q(\ol{\pi}_X(\L_{1,\mt}))=\ol{\pi}_{[X]_{\L_\mt}}(\{0\})\subseteq\fJ'
\]
since $\L_{1,\mt}\subseteq\L_\mt$.
Now fix $\mt\neq F\subseteq[d], a\in\L_{1,F}$ and $k_\un{n}\in\K(X_\un{n})$ such that $[\phi_\un{n}(a)]_{\L_{1,\mt}}=[k_\un{n}]_{\L_{1,\mt}}$ for all $\un{0}\neq\un{n}\leq\un{1}_F$.
Identifying $\L([[X_\un{n}]_{\L_{1,\mt}}]_{\L_\mt/\L_{1,\mt}})$ with $\L([X_\un{n}]_{\L_\mt})$ for all $\un{0}\neq\un{n}\leq\un{1}_F$, we can write $[\hspace{1pt}\cdot\hspace{1pt}]_{\L_\mt}=[\hspace{1pt}\cdot\hspace{1pt}]_{\L_\mt/\L_{1,\mt}}\circ[\hspace{1pt}\cdot\hspace{1pt}]_{\L_{1,\mt}}$ and deduce that
\[
[\phi_\un{n}(a)]_{\L_\mt}
=
[k_\un{n}]_{\L_\mt}
\foral\un{0}\neq\un{n}\leq\un{1}_F.
\]
Consequently, we obtain
\begin{align*}
Q(\ol{\pi}_X(a)+\sum_{\un{0}\neq\un{n}\leq\un{1}_F}(-1)^{|\un{n}|}\ol{\psi}_{X,\un{n}}(k_\un{n})) 
& =
\ol{\pi}_{[X]_{\L_\mt}}([a]_{\L_\mt})+\sum_{\un{0}\neq\un{n}\leq\un{1}_F}(-1)^{|\un{n}|}\ol{\psi}_{[X]_{\L_\mt},\un{n}}([k_\un{n}]_{\L_\mt}) \\
& =
\ol{\pi}_{[X]_{\L_\mt}}([a]_{\L_\mt})+\sum_{\un{0}\neq\un{n}\leq\un{1}_F}(-1)^{|\un{n}|}\ol{\psi}_{[X]_{\L_\mt},\un{n}}([\phi_\un{n}(a)]_{\L_\mt}) \\
& =
\ol{\pi}_{[X]_{\L_\mt}}([a]_{\L_\mt})\ol{q}_{[X]_{\L_\mt},F} 
\in\fJ_{[\L_1']_{\L_{\mt}}}^{(\ol{\pi}_{[X]_{\L_\mt}},\ol{t}_{[X]_{\L_\mt}})}\subseteq\fJ',
\end{align*}
using that $a\in\L_{1,F} \subseteq \L_{1,F}'$.

For the reverse inclusion, it suffices to show that $Q(\fJ)$ contains the generators of $\fJ_{[\L_1']_{\L_\mt}}^{(\ol{\pi}_{[X]_{\L_\mt}},\ol{t}_{[X]_{\L_\mt}})}$.
The same holds for $\fJ_{[\L_2']_{\L_\mt}}^{(\ol{\pi}_{[X]_{\L_\mt}},\ol{t}_{[X]_{\L_\mt}})}$ by symmetry, concluding the proof of the claim.
Fix $\mt\neq F\subseteq[d]$ and $a\in \L_{1,F}' \equiv \L_{1,F}+\L_\mt$.
Then in particular $a \in \L_F$ by (\ref{eq:inc l}).
By item (v) of Lemma \ref{L:rep quo}, for each $\un{0}\neq\un{n}\leq\un{1}_F$ there exists $k_\un{n}\in\K(X_\un{n})$ such that $[\phi_\un{n}(a)]_{\L_\mt}=[k_\un{n}]_{\L_\mt}$ and
\[
\ol{\pi}_X(a)+\sum_{\un{0}\neq\un{n}\leq\un{1}_F}(-1)^{|\un{n}|}\ol{\psi}_{X,\un{n}}(k_\un{n})
\in \fJ \equiv \fJ^\L.
\]
Thus we have
\begin{align*}
\ol{\pi}_{[X]_{\L_\mt}}([a]_{\L_\mt})\ol{q}_{[X]_{\L_\mt},F} 
& =
\ol{\pi}_{[X]_{\L_\mt}}([a]_{\L_\mt})+\sum_{\un{0}\neq\un{n}\leq\un{1}_F}(-1)^{|\un{n}|}\ol{\psi}_{[X]_{\L_\mt},\un{n}}([\phi_\un{n}(a)]_{\L_\mt}) \\
& =
Q\big( \ol{\pi}_X(a)+\sum_{\un{0}\neq\un{n}\leq\un{1}_F}(-1)^{|\un{n}|}\ol{\psi}_{X,\un{n}}(k_\un{n}) \big)
\in
Q(\fJ),
\end{align*}
as required.
\hfill{$\Box$}

\smallskip

Using Claims 1, 2 and 3 (and adopting the nomenclature therein), we conclude that
\begin{align*}
\ker \Phi_{\fJ^{\L'}}
=
Q(\fJ^{\L'}) 
= 
\fJ_{[\L']_{\L_\mt}}^{(\ol{\pi}_{[X]_{\L_\mt}},\ol{t}_{[X]_{\L_\mt}})} 
=
\fJ_{[\L_1']_{\L_{\mt}}}^{(\ol{\pi}_{[X]_{\L_\mt}},\ol{t}_{[X]_{\L_\mt}})}+\fJ_{[\L_2']_{\L_{\mt}}}^{(\ol{\pi}_{[X]_{\L_\mt}},\ol{t}_{[X]_{\L_\mt}})}
=
Q(\fJ)
=
\ker \Phi_{\fJ},
\end{align*}
and the proof is complete.
\end{proof}

By making minor changes to Theorem \ref{T:NT param}, we can parametrise the gauge-invariant ideals of $\N\O(\K,X)$ for any relative $2^d$-tuple $\K$ of $X$.
In particular, we can parametrise the gauge-invariant ideals of $\N\O_X$.
We begin with a definition.

\begin{definition}\label{D:rel NO tuple}
Let $X$ be a strong compactly aligned product system with coefficients in a C*-algebra $A$.
Let $\K$ be a relative $2^d$-tuple of $X$ and let $\L$ be a  $2^d$-tuple of $X$.
We say that $\L$ is a \emph{$\K$-relative NO-$2^d$-tuple of $X$} if $\L$ is an NT-$2^d$-tuple of $X$ and $\K \subseteq \L$.
We refer to the $\I$-relative NO-$2^d$-tuples of $X$ simply as \emph{NO-$2^d$-tuples of $X$}.
\end{definition}

It will follow from Theorem \ref{T:NO param} that the set of $\K$-relative NO-$2^d$-tuples of $X$ is non-empty.
As we have seen in Proposition \ref{P:NT+T}, NT-$2^d$-tuples constitute the higher-rank analogue of Katsura's T-pairs, and the same holds for NO-$2^d$-tuples and Katsura's O-pairs.

\begin{proposition}\label{P:NO+O}
Let $X=\{X_n\}_{n\in\bZ_+}$ be a product system with coefficients in a C*-algebra $A$.
Then the NO-$2$-tuples of $X$ are exactly the O-pairs of $X_1$.
\end{proposition}

\begin{proof}
This is immediate by Proposition \ref{P:NT+T}, since $\I_{\{1\}}=\J_{\{1\}}=J_{X_1}$.
\end{proof}

The lattice operations of Definition \ref{D:NTlattice} restrict to the set of $\K$-relative NO-$2^d$-tuples of $X$.

\begin{proposition}\label{P:Krellattice}
Let $X$ be a strong compactly aligned product system with coefficients in a C*-algebra $A$.
Let $\K$ be a relative $2^d$-tuple of $X$ and let $\L_1$ and $\L_2$ be $\K$-relative NO-$2^d$-tuples of $X$.
Then $\L_1\vee\L_2$ and $\L_1\wedge\L_2$ are $\K$-relative NO-$2^d$-tuples of $X$.
\end{proposition}

\begin{proof}
For $F\subseteq[d]$, we have $\K_F\subseteq\L_{1,F}$ and $\K_F\subseteq\L_{2,F}$ by definition.
Hence 
\[
\K_F\subseteq\L_{1,F}\cap\L_{2,F}=(\L_1\wedge\L_2)_F,
\] 
using Proposition \ref{P:NTwedge} in the equality.
Hence $\L_1\wedge\L_2$ is a $\K$-relative NO-$2^d$-tuple of $X$.

Next, note that 
\[
\K_\mt\subseteq\L_{1,\mt}\subseteq\ol{\pi}_X^{-1}(\fJ^{\L_1}+\fJ^{\L_2})=(\L_1\vee\L_2)_\mt
\]
by Proposition \ref{P:NTvee}.
Now fix $\mt \neq F \subseteq [d]$.
Since $\K_F \subseteq \L_{1,F}+\L_{2,F}+(\L_1\vee\L_2)_\mt$, we obtain
\begin{align*}
[\K_F]_{(\L_1\vee\L_2)_\mt}
& \subseteq
[\L_{1,F}+\L_{2,F}+(\L_1\vee\L_2)_\mt]_{(\L_1\vee\L_2)_\mt} 
\subseteq
[\L_{1,F}+\L_{2,F}+(\L_1\vee\L_2)_\mt]_{(\L_1\vee\L_2)_\mt}^{(d-1)},
\end{align*}
and by Proposition \ref{P:NTvee} we have $\K_F\subseteq(\L_1\vee\L_2)_F$, completing the proof.
\end{proof}

For a relative $2^d$-tuple $\K$ of $X$, we write $Q_\K\colon\N\T_X\to\N\O(\K,X)$ for the canonical quotient map.
Equivariance of $Q_\K$ gives that $\fJ$ is a gauge-invariant ideal of $\N\O(\K,X)$ if and only if $\fJ = Q_\K(\fJ')$ for the gauge-invariant ideal $\fJ' : = Q_\K^{-1}(\fJ)$ of $\N\T_X$.
With this we can adapt the parametrisation of Theorem \ref{T:NT param} to account for $\N\O(\K,X)$.

\begin{theorem}\label{T:NO param}
Let $X$ be a strong compactly aligned product system with coefficients in a C*-algebra $A$ and let $\K$ be a relative $2^d$-tuple of $X$.
Equip the set of $\K$-relative NO-$2^d$-tuples of $X$ with the lattice structure of Definition \ref{D:NTlattice} (suitably restricted) and equip the set of gauge-invariant ideals of $\N\O(\K,X)$ with the usual lattice structure.
Then these sets are isomorphic as lattices via the map
\begin{equation}\label{Eq:NO param}
\L\mapsto Q_\K(\mathfrak{J}^\L),
\text{ for the canonical quotient map }
Q_\K \colon \N\T_X \to \N\O(\K,X),
\end{equation}
for all $\K$-relative NO-$2^d$-tuples $\L$ of $X$.
Moreover, this map preserves inclusions.
\end{theorem}

\begin{proof}
First note that the proposed lattice structure on the set of $\K$-relative NO-$2^d$-tuples of $X$ is well-defined by Proposition \ref{P:Krellattice}.
The comments preceding the statement of the theorem show that (\ref{Eq:NO param}) constitutes a well-defined map.

Next we check that the mapping is injective and surjective.
To this end, first we show that $\mathfrak{J}_\K^{(\ol{\pi}_X,\ol{t}_X)}\subseteq\mathfrak{J}^\L$ whenever $\L$ is a $\K$-relative NO-$2^d$-tuple of $X$.
Fix $F\subseteq[d]$ and $a\in\K_F$.
It suffices to show that $\ol{\pi}_X(a)\ol{q}_{X,F}\in\mathfrak{J}^\L$.
Since $\L$ is a $\K$-relative NO-$2^d$-tuple of $X$, we have $a\in\K_F\subseteq\L_F$.
Likewise, because $\phi_\un{n}(a)\in\K(X_\un{n})$ for all $\un{0}\neq\un{n}\leq\un{1}_F$, by Proposition \ref{P:JL conc} we obtain
\[
\ol{\pi}_X(a)\ol{q}_{X,F}=\ol{\pi}_X(a)+\sum\{(-1)^{|\un{n}|}\ol{\psi}_{X,\un{n}}(\phi_\un{n}(a))\mid \un{0}\neq\un{n}\leq\un{1}_F\}\in\mathfrak{J}^\L,
\]
as required.
Consequently, we have the following commutative diagram
\begin{center}
\begin{tikzcd}
\N\T_X
	\arrow[r, "Q_{\mathfrak{J}^\L}"]
	\arrow[d, "Q_\K" swap]
& \N\T_X/\mathfrak{J}^\L \\
\N\O(\K,X)
	\arrow[ur, "\exists!\Phi", swap, dashed]
\end{tikzcd}
\end{center}
of canonical $*$-epimorphisms, so that $\ker\Phi=Q_\K(\mathfrak{J}^\L)$.
Thus we obtain a $*$-isomorphism
\begin{equation}\label{Eq:isoL}
\tilde{\Phi}\colon \N\O(\K,X)/Q_\K(\mathfrak{J}^\L)\to\N\T_X/\mathfrak{J}^\L; \tilde{\Phi}(f+Q_\K(\mathfrak{J}^\L))=\Phi(f)\foral f\in\N\O(\K,X).
\end{equation}

For injectivity of the map (\ref{Eq:NO param}), suppose we have $\K$-relative NO-$2^d$-tuples $\L$ and $\L'$ of $X$ such that $Q_\K(\mathfrak{J}^\L)=Q_\K(\mathfrak{J}^{\L'})$.
Applying (\ref{Eq:isoL}) for $\L$ and $\L'$, we obtain a $*$-isomorphism
\[
\N\T_X/\mathfrak{J}^\L\to\N\T_X/\mathfrak{J}^{\L'}; Q_{\mathfrak{J}^\L}(f) \mapsto Q_{\mathfrak{J}^{\L'}}(f)\foral f\in\N\T_X.
\]
In turn, it follows that $\mathfrak{J}^\L=\mathfrak{J}^{\L'}$ and hence $\L=\L'$ by Theorem \ref{T:NT param}.

For surjectivity of the map (\ref{Eq:NO param}), let $\fJ$ be a gauge-invariant ideal of $\N\O(\K,X)$.
Then $Q_\K^{-1}(\fJ)$ is a gauge-invariant ideal of $\N\T_X$ and thus $Q_\K^{-1}(\fJ) = \fJ^\L$ for a unique NT-$2^d$-tuple $\L$ of $X$ by Theorem \ref{T:NT param}.
It suffices to show that $\K \subseteq \L$.
To this end, let $\L'$ be the NT-$2^d$-tuple so that
\[
\fJ^{\L'} = \mathfrak{J}_\K^{(\ol{\pi}_X,\ol{t}_X)}.
\]
We have $\K \subseteq \L'$ by definition of $\L'$.
Moreover, we have
\[
\fJ^{\L'} = \mathfrak{J}_\K^{(\ol{\pi}_X,\ol{t}_X)} \subseteq Q_\K^{-1}(\fJ) = \fJ^\L,
\]
and so $\L' \subseteq \L$ since the parametrisation of Theorem \ref{T:NT param} respects inclusions.
Thus $\K \subseteq \L' \subseteq \L$, as required.

The map (\ref{Eq:NO param}) respects inclusions and the lattice structure since it is a restriction of the first parametrisation map of Theorem \ref{T:NT param} (followed by the $*$-homomorphism $Q_\K$), which satisfies these properties.
\end{proof}

A direct consequence of Theorem \ref{T:NO param} is that, if $\fJ$ is a gauge-invariant ideal of $\N\T_X$, then $\L^\fJ$ is a $\K$-relative NO-$2^d$-tuple if and only if the quotient map $Q_\fJ \colon \N\T_X \to \N\T_X/\fJ$ factors through the quotient map $Q_\K \colon \N\T_X \to \N\O(\K,X)$.

Applying Theorem \ref{T:NO param} for $\K=\I$ provides the parametrisation of the gauge-invariant ideals of $\N\O_X$.

\begin{corollary}\label{C:CNP param}
Let $X$ be a strong compactly aligned product system with coefficients in a C*-algebra $A$.
Equip the set of NO-$2^d$-tuples of $X$ with the lattice structure of Definition \ref{D:NTlattice} (suitably restricted) and equip the set of gauge-invariant ideals of $\N\O_X$ with the usual lattice structure.
Then these sets are isomorphic as lattices via the map
\begin{equation}\label{Eq:NO param 2}
\L\mapsto Q_\I(\mathfrak{J}^\L),
\text{ for the canonical quotient map }
Q_\I \colon \N\T_X \to \N\O_X,
\end{equation}
for all NO-$2^d$-tuples $\L$ of $X$.
Moreover, this map respects inclusions.
\end{corollary}

Theorem \ref{T:NT param} recaptures the parametrisation of Katsura \cite{Kat07}, as presented in Theorem \ref{T:Kat par}.
More generally, Theorem \ref{T:NO param} recaptures \cite[Proposition 11.9]{Kat07}.

\section{Applications}\label{S:app}

The applications that we consider in this section pertain to product systems over $\bZ_+^d$ that are regular, or arise from C*-dynamical systems, or from strong finitely aligned higher-rank graphs, or whose fibres (apart from the coefficient algebra) admit finite frames.
We begin by exploring the situation where an ideal can be placed as the $\L_\mt$-member of an NO-$2^d$-tuple $\L$, which will be helpful for further examples.

\subsection{Participation in an NO-$2^d$-tuple}\label{Ss:part}

If $I\subseteq A$ is an ideal that is positively invariant for $X$, then the quotient map $X \to [X]_I$ induces a canonical $*$-epimorphism
\[
[\hspace{1pt}\cdot\hspace{1pt}]_I \colon \N\T_X \to \N\T_{[X]_I},
\]
due to the universality of $\N\T_X$.
It is well known that this map does not in general descend to a canonical $*$-epimorphism $\N\O_X\to\N\O_{[X]_I}$, even for $d=1$ (see Example \ref{E:surjcounter} for a counterexample).
However, using the NO-$2^d$-tuple machinery, we can determine precisely when this occurs.
To this end, we introduce the following definition, modelled after \cite[Definition 4.8]{Kat07}.

\begin{definition}\label{D:neginv}
Let $X$ be a strong compactly aligned product system with coefficients in a C*-algebra $A$.
We say that an ideal $I\subseteq A$ is \emph{negatively invariant for $X$} if
\[
\I_F\cap X_F^{-1}(I)\subseteq I\foral \mt\neq F\subseteq[d].
\]
\end{definition}

Definition \ref{D:neginv} leads to the following natural extension of \cite[Proposition 5.3]{Kat07}.

\begin{lemma}\label{L:neginvchar}
Let $X$ be a strong compactly aligned product system with coefficients in a C*-algebra $A$ and let $I\subseteq A$ be an ideal.
Then $I$ is negatively invariant for $X$ if and only if $\I_F\subseteq J_F(I,X)$ for all $\mt\neq F\subseteq[d]$.
\end{lemma}

\begin{proof}
Assume that $I$ is negatively invariant for $X$.
Fix $\mt \neq F\subseteq[d]$ and take $a\in\I_F$.
Then $\phi_\un{i}(a)\in\K(X_\un{i})$ for all $i\in[d]$ and so $[\phi_\un{i}(a)]_I\in\K([X_\un{i}]_I)$ for all $i\in[d]$ by Lemma \ref{L:Kat07}.
Moreover, we have
\[
aX_F^{-1}(I)\subseteq \I_F\cap X_F^{-1}(I)\subseteq I.
\]
Hence $a \in J_F(I, X)$ and we conclude that $\I_F \subseteq J_F(I,X)$ for all $\mt\neq F\subseteq[d]$, as required.

Now assume that $\I_F\subseteq J_F(I,X)$ for all $\mt\neq F\subseteq[d]$.
Fix $\mt\neq F\subseteq[d]$ and take an element $a\in\I_F\cap X_F^{-1}(I)$.
We have $a\in J_F(I,X)$ by assumption, so $aX_F^{-1}(I)\subseteq I$.
Since $a \in X_F^{-1}(I)$, we obtain $a \in I$.
It follows that $I$ is negatively invariant, completing the proof.
\end{proof}

In order to motivate further the importance of the $\L_\mt$-member of an NO-$2^d$-tuple $\L$, we give the following proposition.

\begin{proposition}\label{P:inj phi}
Let $X$ be a strong compactly aligned product system with coefficients in a C*-algebra $A$ and let $\L$ be an NO-$2^d$-tuple of $X$.
If $\L$ satisfies 
\[
\I_F([X]_{\L_\mt})\subseteq(\I_F(X)+\L_\mt)/\L_\mt
\foral F\subseteq[d],
\]
then $[\L]_{\L_\mt} = \I([X]_{\L_\mt})$ and thus $\N\O([\L]_{\L_\mt},[X]_{\L_\mt}) = \N\O_{[X]_{\L_\mt}}$.
\end{proposition}

\begin{proof}
By assumption, for all $F \subseteq [d]$ we have
\begin{align*}
\I_F([X]_{\L_\mt})
\subseteq
(\I_F(X)+\L_\mt)/\L_\mt
\subseteq 
(\L_F + \L_\mt)/\L_\mt
=
[\L_F]_{\L_\mt}
\subseteq
\I_F([X]_{\L_\mt}),
\end{align*}
using that $\L$ is an NO-$2^d$-tuple in the second inclusion, and that $[\L]_{\L_{\mt}}$ is an (M)-$2^d$-tuple of $[X]_{\L_\mt}$ by Proposition \ref{P:NT tuple + m fam} in the final inclusion.
Hence 
\[
[\L]_{\L_\mt} = \I([X]_{\L_\mt}),
\]
and thus by definition 
\[
\N\O([\L]_{\L_\mt},[X]_{\L_\mt}) = \N\O_{[X]_{\L_\mt}},
\] 
as required.
\end{proof}

\begin{definition}\label{D:NO part}
Let $X$ be a strong compactly aligned product system with coefficients in a C*-algebra $A$ and let $I\subseteq A$ be an ideal.
We say that $I$ \emph{participates in an NO-$2^d$-tuple of $X$} if there exists an NO-$2^d$-tuple $\L$ of $X$ such that $\L_\mt = I$.
\end{definition}

If $I$ is a positively invariant ideal for $X$, then negative invariance of $I$ is necessary and sufficient for the quotient map $X \to [X]_I$ to induce a canonical $*$-epimorphism between the corresponding Cuntz-Nica-Pimsner algebras.

\begin{proposition}\label{P:NO can surj}
Let $X$ be a strong compactly aligned product system with coefficients in a C*-algebra $A$ and let $I\subseteq A$ be an ideal.
Then the following are equivalent:
\begin{enumerate}
\item $I$ participates in an NO-$2^d$-tuple of $X$;
\item $I$ is positively and negatively invariant for $X$;
\item $I$ is positively invariant for $X$ and the quotient map $X \to [X]_I$ lifts to a canonical $*$-epimorphism $\N\O_X\to\N\O_{[X]_I}$.
\end{enumerate}
\end{proposition}

\begin{proof}
\noindent
[(i)$\Rightarrow$(ii)]: Assume that $I$ participates in an NO-$2^d$-tuple $\L$ of $X$.
Then in particular $I = \L_\mt$ is positively invariant for $X$.
On the other hand, since $\I(X) \subseteq \L$ we have
\[
\I_F(X) \subseteq\L_F\subseteq J_F(\L_\mt,X)=J_F(I,X) \foral \mt \neq F \subseteq [d].
\]
Lemma \ref{L:neginvchar} then gives that $I$ is negatively invariant for $X$.

\smallskip

\noindent
[(ii)$\Rightarrow$(iii)]: Assume that $I$ is positively and negatively invariant for $X$ and let $\L$ be the $2^d$-tuple of $X$ (consisting of ideals) defined by 
\[
\L_F := \I_F(X) + I
\foral 
F \subseteq [d].
\]
Since $\I(X)$ is invariant and $I$ is positively invariant for $X$, we have that $\L$ is invariant.
Since $\I(X)$ is partially ordered, we also have that $\L$ is partially ordered.
Moreover, by Lemma \ref{L:neginvchar} we have 
\[
I \subseteq \L_F \equiv \I_F(X) + I \subseteq J_F(I,X) \foral \mt \neq F \subseteq [d],
\]
where we also use that $I$ is positively invariant and so $I \subseteq J_F(I,X)$.
Therefore, the family $[\L]_I$ is an invariant and partially ordered $2^d$-tuple of $[X]_I$ that consists of ideals and satisfies $[\L]_I \subseteq \J([X]_I)$ by Lemma \ref{L:quo ideal}.
Hence $[\L]_I$ is an (E)-$2^d$-tuple of $[X]_I$ by Lemma \ref{L:J+inv}.
Consequently we have the canonical $*$-epimorphisms
\[
\N\T_X \stackrel{{[\hspace{1pt}\cdot\hspace{1pt}]_I}}{\longrightarrow} \N\T_{[X]_I} \stackrel{Q}{\longrightarrow} \N\O([\L]_I, [X]_I) \longrightarrow \N\O_{[X]_I},
\]
where the final $*$-epimorphism follows from the co-universal property of $\N\O_{[X]_I}$.

In order to deduce the required $*$-epimorphism, it suffices to close the following diagram
\[
\xymatrix{
\N\T_X \ar[rr]^{[\hspace{1pt}\cdot\hspace{1pt}]_I} \ar[d]_{Q_{\I(X)}} & & \N\T_{[X]_I} \ar[d]^Q & & \\
\N\O_X \ar@{-->}[rr] & & \N\O([\L]_I, [X]_I) \ar[rr] & & \N\O_{[X]_I}
}
\]
by a canonical $*$-epimorphism.
Hence it suffices to show that the kernel $\ker Q_{\I(X)} = \fJ_{\I(X)}^{(\ol{\pi}_X, \ol{t}_X)}$ is contained in the kernel $\ker Q \circ [\hspace{1pt}\cdot\hspace{1pt}]_I$.

To this end, recall that the generators of $\ker Q_{\I(X)} = \fJ_{\I(X)}^{(\ol{\pi}_X, \ol{t}_X)}$ are of the form $\ol{\pi}_X(a) \ol{q}_{X,F}$, where $a \in \I_F(X)$ and $F \subseteq [d]$.
Fix $F \subseteq [d]$ and $a \in \I_F(X)\subseteq\L_F$, so that $[a]_I \in [\L_F]_I$.
Then by definition we have
\[
[\ol{\pi}_X(a) \ol{q}_{X,F}]_I = \ol{\pi}_{[X]_I}([a]_I) \ol{q}_{[X]_I, F} \in \fJ_{[\L]_I}^{(\ol{\pi}_{[X]_I}, \ol{t}_{[X]_I})} = \ker Q,
\]
and thus $\ol{\pi}_X(a) \ol{q}_{X,F} \in \ker Q \circ [\hspace{1pt}\cdot\hspace{1pt}]_I$.
It follows that $\ker Q_{\I(X)} \subseteq \ker Q \circ [\hspace{1pt}\cdot\hspace{1pt}]_I$, as required.

\smallskip

\noindent
[(iii)$\Rightarrow$(i)]: Assume that $I$ is positively invariant for $X$ and that the quotient map $X \to [X]_I$ lifts to a canonical $*$-epimorphism $\Phi \colon \N\O_X\to\N\O_{[X]_I}$.
Then $\ker\Phi\subseteq\N\O_X$ is a gauge-invariant ideal, and we can consider the gauge-invariant ideal
\[
\mathfrak{J}:=Q_{\I(X)}^{-1}(\ker\Phi)\subseteq\N\T_X.
\]
In particular we have $\mathfrak{J}^{\I(X)} \subseteq\mathfrak{J}$, and so $\L^\fJ$ is an NO-$2^d$-tuple of $X$ since the parametrisation of Theorem \ref{T:NT param} respects inclusions.
We obtain a sequence of canonical $*$-isomorphisms
\[
\N\O([\L^\mathfrak{J}]_{\L_\mt^\mathfrak{J}},[X]_{\L_\mt^\mathfrak{J}})\cong\N\T_X/\mathfrak{J}\cong\N\O_X/Q_{\I(X)}(\mathfrak{J}) \equiv \N\O_X/ \ker \Phi \cong\N\O_{[X]_I},
\]
where the first $*$-isomorphism is given by item (ii) of Proposition \ref{P:tuple ideal}.
By restricting to the coefficient algebras, we have $[A]_{\L_{\mt}^\fJ} \cong [A]_I$ by the map $a + \L_\mt^\fJ \mapsto a + I$ for all $a \in A$.
Thus $\L_\mt^\mathfrak{J}=I$, as required.
\end{proof}

We now collect in one place the main results of this subsection and their consequences.

\begin{corollary}\label{C:NO source}
Let $X$ be a strong compactly aligned product system with coefficients in a C*-algebra $A$.
Let $I\subseteq A$ be an ideal that is positively invariant for $X$ and satisfies 
\[
\I_F([X]_I) = (\I_F(X)+I)/I \foral F\subseteq[d].
\]
Let the $2^d$-tuple $\L$ of $X$ be defined by $\L_F:=\I_F(X)+I$ for all $F\subseteq[d]$.
Then the following hold:
\begin{enumerate}
\item The $2^d$-tuple $\L$ is an NO-$2^d$-tuple of $X$.
\item $\N\O_X/Q_\I(\mathfrak{J}^\L)\cong\N\O_{[X]_I}$ canonically, where $Q_\I\colon\N\T_X\to\N\O_X$ is the quotient map.
\item If $I\subseteq\bigcap\{\phi_\un{i}^{-1}(\K(X_\un{i}))\mid i\in[d]\}$, then 
\[
Q_\I(\mathfrak{J}^\L) = \sca{Q_\I(\ol{\pi}_X(I))}=\ol{\spn}\{t_{X,\un{n}}^\I(X_\un{n})\pi_X^\I(I)t_{X,\un{m}}^\I(X_\un{m})^*\mid \un{n},\un{m}\in\bZ_+^d\},
\]
and thus $\N\O_{IXI}$ is a hereditary, full C*-subalgebra of $Q_\I(\fJ^\L)$.
\end{enumerate}
\end{corollary}

\begin{proof}
(i) By assumption we have $[\L]_I=\I([X]_I)$ and $\L_\mt= I$.
Thus $[\L]_I$ is an (M)-$2^d$-tuple of $[X]_I$, and in turn $\L$ is an NT-$2^d$-tuple of $X$ by Proposition \ref{P:NT tuple + m fam}.
Moreover, $\I(X) \subseteq \L$ by definition and thus $\L$ is an NO-$2^d$-tuple, proving item (i).

\smallskip

\noindent
(ii) By item (i) of Proposition \ref{P:tuple ideal}, we have the canonical $*$-isomorphisms
\[
\N\O_X/Q_\I(\mathfrak{J}^\L)\cong\N\T_X/\mathfrak{J}^\L\cong\N\O([\L]_I,[X]_I) = \N\O_{[X]_I},
\]
proving item (ii).

\smallskip

\noindent
(iii) Assume that $I\subseteq\bigcap\{\phi_\un{i}^{-1}(\K(X_\un{i}))\mid i\in[d]\}$.
This ensures that $\L$ is both an NT-$2^d$-tuple and a relative $2^d$-tuple of $X$.
Hence $\mathfrak{J}^\L=\mathfrak{J}_\L^{(\ol{\pi}_X,\ol{t}_X)}$ by Proposition \ref{P:JL conc}.
Since 
\[
\pi_X^\I(\I_F(X))q_{X,F}^\I = \{0\} \foral F\subseteq[d],
\]
we have
\begin{align*}
Q_\I(\mathfrak{J}^\L) 
& =
Q_\I(\sca{\ol{\pi}_X(\L_F)\ol{q}_{X,F}\mid F\subseteq[d]}) \\
& =
\sca{\pi_X^\I(\L_F)q_{X,F}^\I \mid F\subseteq[d]} \\
& =
\sca{\pi_X^\I(I),\pi_X^\I(I)q_{X,F}^\I\mid \mt\neq F\subseteq[d]}.
\end{align*}

To see that 
\[
\sca{\pi_X^\I(I),\pi_X^\I(I)q_{X,F}^\I\mid \mt\neq F\subseteq[d]}=\sca{Q_\I(\ol{\pi}_X(I))}=\sca{\pi_X^\I(I)},
\]
it suffices to show that the right hand side contains the generators of the left hand side.
To this end, fix $a\in I$ and $\mt\neq F\subseteq[d]$.
An application of Proposition \ref{P:prod cai} gives that
\[
\pi_X^\I(a)q_{X,F}^\I=\pi_X^\I(a)+\sum\{(-1)^{|\un{n}|}\psi_{X,\un{n}}^\I(\phi_\un{n}(a))\mid \un{0}\neq\un{n}\leq\un{1}_F\}.
\]
Fixing $\un{0}\neq\un{n}\leq\un{1}_F$, we have $\un{n}=\un{1}_D$ for some $\mt\neq D\subseteq F$.
Hence we obtain
\[
\psi_{X,\un{n}}^\I(\phi_\un{n}(a))
=
\nor{\cdot}\text{-}\lim_\la \psi_{X,\un{n}}^\I(\phi_\un{n}(a)e_{D,\la})
=
\nor{\cdot}\text{-}\lim_\la \pi_X^\I(a)\psi_{X,\un{n}}^\I(e_{D,\la})
\in
\sca{Q_\I(\ol{\pi}_X(I))},
\]
where we use the nomenclature of Proposition \ref{P:sca ai} in conjunction with \cite[(2.4)]{DK18}.
Hence $\psi_{X,\un{n}}^\I(\phi_\un{n}(a))\in\sca{Q_\I(\ol{\pi}_X(I))}$ for all $\un{0}\neq\un{n}\leq\un{1}_F$, and therefore 
\[
\pi_X^\I(a)q_{X,F}^\I \in \sca{Q_\I(\ol{\pi}_X(I))},
\]
as required.
In total, we conclude that $Q_\I(\mathfrak{J}^\L)=\sca{Q_\I(\ol{\pi}_X(I))}$.

For the final equality, we have
\begin{align*}
\sca{Q_\I(\ol{\pi}_X(I))}
=
\sca{\pi_X^\I(I)}
=
\ol{\spn}\{t_{X,\un{n}}^\I(X_\un{n})\pi_X^\I(I)t_{X,\un{m}}^\I(X_\un{m})^*\mid \un{n},\un{m}\in\bZ_+^d\},
\end{align*}
using item (i) of Proposition \ref{P:IXIherfull} in the second equality.
The last assertion of item (iii) follows from Proposition \ref{P:NOIXIembed}, finishing the proof.
\end{proof}

\subsection{Regular product systems}

Let $X$ be a regular product system over $\bZ_+^d$ with coefficients in a C*-algebra $A$.
Recall that $X$ is automatically strong compactly aligned by Corollary \ref{C:autscaps}.
Also, an ideal $I\subseteq A$ is negatively invariant for $X$ if and only if $X_F^{-1}(I)\subseteq I$ for all $\mt\neq F\subseteq[d]$, since $\I_F=A$ for all $\mt\neq F\subseteq[d]$.
The gauge-invariant ideals of $\N\O_X$ can be parametrised particularly succinctly.
We start with some auxiliary results.

\begin{lemma}\label{L:regNO}
Let $X$ be a regular product system over $\bZ_+^d$ with coefficients in a C*-algebra $A$ and let $I\subseteq A$ be an ideal.
Then the following hold:
\begin{enumerate}
\item If $I$ is negatively invariant for $X$, then $J_F(I,X)=A$ for all $\mt\neq F\subseteq[d]$.
\item If $I$ is positively and negatively invariant for $X$, then $[X]_I$ is regular.
\end{enumerate}
\end{lemma}

\begin{proof}
For item (i), fix $\mt \neq F \subseteq [d]$.
We have $\I_F(X) = A$ by regularity.
Since $I$ is negatively invariant, Lemma \ref{L:neginvchar} yields $A = \I_F(X) \subseteq J_F(I,X)$, as required.

For item (ii), Lemma \ref{L:Kat07} yields $[\phi_\un{n}]_I([A]_I)\subseteq\K([X_\un{n}]_I)$ for all $\un{n}\in\bZ_+^d$, using regularity of $X$.
For injectivity, it suffices to show that $[\phi_\un{i}]_I$ is injective for all $i\in[d]$ by Proposition \ref{P:inj+comp}.
Accordingly, fix $i\in[d]$ and $[a]_I \in \ker [\phi_\un{i}]_I$.
Then we have $\sca{X_\un{i},aX_\un{i}}\subseteq I$ by Lemma \ref{L:quo ker}, and so $a\in X_{\{i\}}^{-1}(I) \subseteq I$, since $I$ is negatively invariant and $X$ is regular.
Consequently $[a]_I=0$, showing that $[\phi_\un{i}]_I$ is injective.
In total, we have that $[X]_I$ is regular, as required.
\end{proof}

\begin{proposition}\label{P:regNOposinv}
Let $X$ be a regular product system over $\bZ_+^d$ with coefficients in a C*-algebra $A$ and let $\L$ be a $2^d$-tuple of $X$.
Then $\L$ is an NO-$2^d$-tuple of $X$ if and only if $\L_\mt$ is a positively and negatively invariant ideal for $X$ and $\L_F=A$ for all $\mt\neq F\subseteq[d]$.
\end{proposition}

\begin{proof}
Assume that $\L$ is an NO-$2^d$-tuple of $X$, so that $\I(X)\subseteq\L$.
We have that $\L_\mt$ is a positively and negatively invariant ideal for $X$ by Proposition \ref{P:NO can surj}.
Next fix $\mt\neq F\subseteq[d]$.
By regularity of $X$, we have $A=\I_F(X)\subseteq\L_F$.
Thus $\L_F=A$ for all $\mt\neq F\subseteq[d]$, proving the forward implication.

Conversely, assume that $\L_\mt$ is a positively and negatively invariant ideal for $X$ and $\L_F=A$ for all $\mt\neq F\subseteq[d]$.
We start by showing that $\L$ is an NT-$2^d$-tuple of $X$.
Note that $\L$ consists of ideals satisfying $\L_\mt\subseteq\L_F$ for all $F\subseteq[d]$, and $\L_\mt$ is in particular positively invariant for $X$.
Hence it suffices to show that $[\L]_{\L_\mt}$ is an (M)-$2^d$-tuple of $[X]_{\L_\mt}$ by Proposition \ref{P:NT tuple + m fam}.
By item (ii) of Lemma \ref{L:regNO}, we have that $[X]_{\L_\mt}$ is regular and thus 
\[
\I_\mt([X]_{\L_\mt})=\{0\}=[\L_\mt]_{\L_\mt}\qand \I_F([X]_{\L_\mt})=[A]_{\L_\mt}=[\L_F]_{\L_\mt} \foral \mt\neq F\subseteq[d].
\]
Thus $\I([X]_{\L_\mt})=[\L]_{\L_\mt}$ and we conclude that $[\L]_{\L_\mt}$ is an (M)-$2^d$-tuple of $[X]_{\L_\mt}$ by Remark \ref{R:0+I}.
Note also that
\[
\I_\mt(X)=\{0\}\subseteq\L_\mt\qand\I_F(X)=A\subseteq\L_F \foral \mt\neq F\subseteq[d]
\]
by regularity of $X$.
Thus $\L$ is an NO-$2^d$-tuple, finishing the proof.
\end{proof}

\begin{corollary}\label{C:regNObij}
Let $X$ be a regular product system over $\bZ_+^d$ with coefficients in a C*-algebra $A$.
Then the association
\begin{equation}\label{Eq:NO to posneg}
I\mapsto\L_I, \; \text{where} \; \L_{I,\mt}:=I \; \text{and} \; \L_{I,F}:=A\foral \mt\neq F\subseteq[d],
\end{equation}
defines a bijection between the set of ideals of $A$ that are positively and negatively invariant for $X$ and the set of NO-$2^d$-tuples of $X$.
Hence the set of gauge-invariant ideals of $\N\O_X$ corresponds bijectively to the set of ideals of $A$ that are positively and negatively invariant for $X$.
\end{corollary}

\begin{proof}
By Proposition \ref{P:regNOposinv}, the map (\ref{Eq:NO to posneg}) is well-defined and a bijection.
The last assertion follows from Corollary \ref{C:CNP param}.
\end{proof}

We obtain the following consequence of Corollary \ref{C:regNObij} when $A$ is simple.

\begin{corollary}\label{C:regNOsimp}
Let $X$ be a regular product system over $\bZ_+^d$ with coefficients in a non-zero simple C*-algebra $A$.
Then $\N\O_X$ contains no non-trivial gauge-invariant ideals.
\end{corollary}

\begin{proof}
By simplicity, the only ideals of $A$ are $\{0\}$ and $A$, both of which are positively and negatively invariant for $X$.
Note that negative invariance of $\{0\}$ follows from the fact that $\I_F \cap X_F^{-1}(\{0\})=\{0\}$ for all $\mt\neq F\subseteq[d]$.
An application of Corollary \ref{C:regNObij} finishes the proof.
\end{proof}

Accounting for the gauge-invariant ideals of $\N\T_X$ when $X$ is regular and $A$ is simple reduces to a combinatorial argument.
To this end, we consider $\P([d])$ with the usual partial order of inclusion.
Fixing notation, for a family $\S \subseteq \P([d])$ we write
\[
\min(\S) := \{F \in \S \mid \textup{ $F$ is minimal in $\S$}\},
\]
and
\[
\sca{\S} := \{F \subseteq [d] \mid \textup{ there exists $F' \in \S$ such that $F' \subseteq F$}\}.
\]
It follows that $\min(\S)$ consists of pairwise incomparable elements of $\S$ and that $\sca{\min(\S)} = \sca{\S}$.
Moreover, a family $\S \subseteq \P([d])$ consists of pairwise incomparable elements of $\P([d])$ if and only if $\S = \min(\S)$.
Note that if $\S\subseteq\P([d])$ is empty or a singleton, then it is a family of pairwise incomparable subsets vacuously.

\begin{proposition}\label{P:regNTbij}
Let $X$ be a regular product system over $\bZ_+^d$ with coefficients in a non-zero simple C*-algebra $A$.
The following are equivalent for a $2^d$-tuple $\L$ of $X$:
\begin{enumerate}
\item $\L$ is an NT-$2^d$-tuple of $X$;
\item $\L$ is an (M)-$2^d$-tuple of $X$ or $\L_F = A$ for all $F \subseteq [d]$;
\item $\L$ is partially ordered and consists of ideals, each of which is either $A$ or $\{0\}$.
\end{enumerate}
If any (and thus all) of the items (i)-(iii) hold, then $\N\O(\L,X)$ either factors through the quotient map $\N\T_X \to \N\O_X$ or $\N\O(\L,X) = \{0\}$.
\end{proposition}

\begin{proof}
To see that items (i)-(iii) are equivalent, first suppose that $\L$ is an NT-$2^d$-tuple of $X$.
If $\L_\mt = \{0\}$, then $\L$ is an (M)-$2^d$-tuple of $X$ by Proposition \ref{P:NT tuple + m fam}.
If $\L_\mt \neq \{0\}$, then by simplicity of $A$ we have $\L_\mt = A$.
It follows that $\L_F=A$ for all $F\subseteq[d]$ by the partial ordering property of $\L$.
This shows that item (i) implies item (ii).

Conversely, if $\L$ is an (M)-$2^d$-tuple of $X$ then it is an NT-$2^d$-tuple of $X$ by Proposition \ref{P:NT tuple + m fam}.
On the other hand, if $\L_F = A$ for all $F \subseteq [d]$, then $[\L_F]_{\L_\mt}=\{0\}$ for all $F\subseteq[d]$.
Thus $[\L]_{\L_\mt}$ is an (M)-$2^d$-tuple of $[X]_{\L_\mt}$ by Remark \ref{R:0+I}, and an application of Proposition \ref{P:NT tuple + m fam} then yields that $\L$ is an NT-$2^d$-tuple of $X$.
This shows that item (ii) implies item (i).

Clearly item (ii) implies item (iii) by Theorem \ref{T:m fam v2} (and simplicity of $A$).
Finally, suppose that $\L$ is partially ordered and consists of ideals, each of which is either $A$ or $\{0\}$.
If $\L_\mt = A$, then $\L_F = A$ for all $F \subseteq [d]$ by the partial ordering property of $\L$.
On the other hand, if $\L_\mt = \{0\}$ then $\L \subseteq \J(X)$ since $\J_F(X) = A$ for all $\mt \neq F \subseteq [d]$ by regularity of $X$, and thus $\L$ satisfies condition (i) of Theorem \ref{T:m fam v2}.
Condition (ii) of the latter is satisfied trivially, and condition (iii) is satisfied by assumption. 
Moreover, by definition we have $\L_F^{(1)} = \L_F$ when $F = \mt$ or $F=[d]$.
For $\mt \neq F \subsetneq [d]$, we have that either $\L_F = A$ or $\L_F = \{0\}$.
If $\L_F = A$, then $\L_F^{(1)} \subseteq \L_F$ trivially.
If $\L_F = \{0\}$, then $\L_{\lim, F} = \{0\}$ due to injectivity of $X$, and thus $\L_F^{(1)} = \{0\} = \L_F$.
Therefore $\L$ satisfies condition (iv) of Theorem \ref{T:m fam v2} and so is an (M)-$2^d$-tuple of $X$ by the latter, finishing the proof of the equivalences.

For the second claim, if $\L$ satisfies item (ii) then we have that either $\N\O(\L,X)$ factors through $\N\T_X \to \N\O_X$ (if $\L$ is an (M)-$2^d$-tuple of $X$) or $\N\O(\L,X) = \{0\}$ (if $\L_F=A$ for all $F\subseteq[d]$), and the proof is complete.
\end{proof}

\begin{corollary}\label{C:regNTbij}
Let $X$ be a regular product system over $\bZ_+^d$ with coefficients in a non-zero simple C*-algebra $A$.
Then the association
\begin{equation}\label{Eq:NTsimpbij}
\S\mapsto\L_\S, \; \text{where} \; \L_{\S,F}:=\begin{cases} A & \text{if $F \in \sca{\S}$}, \\ \{0\} & \text{otherwise}, \end{cases} \foral F\subseteq[d],
\end{equation}
defines a bijection between the set of families of pairwise incomparable subsets of $[d]$ and the set of NT-$2^d$-tuples of $X$.
Hence the set of gauge-invariant ideals of $\N\T_X$ corresponds bijectively to the set of families of pairwise incomparable subsets of $[d]$.
\end{corollary}

\begin{proof}
Item (iii) of Proposition \ref{P:regNTbij} implies that the map (\ref{Eq:NTsimpbij}) is well-defined.
For bijectivity, it suffices to construct an inverse map.
To this end, fixing an NT-$2^d$-tuple $\L$ of $X$, we set
\[
\S_\L:=\{F\subseteq[d]\mid \L_F=A\}\subseteq\P([d]).
\]
Note that $\S_\L=\sca{\S_\L}$ by the partial ordering property of $\L$.
It follows that the map
\[
\L\mapsto\min(\S_\L),
\]
where $\L$ is an NT-$2^d$-tuple of $X$, is the inverse of the map (\ref{Eq:NTsimpbij}), as required.
The last assertion follows from Theorem \ref{T:NT param}, and the proof is complete.
\end{proof}

\subsection{C*-dynamical systems}\label{Ss:dynsys}

In this subsection we interpret the parametrisation of Theorem \ref{T:NT param} in the case of C*-dynamical systems.
As a corollary we recover the well-known parametrisation for crossed products.
We use this class of product systems to produce an example for which the quotient map $X \to [X]_I$ does not lift to a canonical $*$-epimorphism $\N\O_X \to \N\O_{[X]_I}$ for a positively invariant ideal $I \subseteq A$.

The structure of the Nica-Pimsner and Cuntz-Nica-Pimsner algebras of a dynamical system were studied in \cite{DFK17, Kak17}.
In particular, the form of the CNP-relations that arose in this setting was one point of motivation for looking further into strong compactly aligned product systems in \cite{DK18}.
We start by establishing notation and terminology from \cite{DFK17}.

\begin{definition}\label{D:c*-dynsys}
A \emph{C*-dynamical system} $(A,\al,\bZ_+^d)$ consists of a C*-algebra $A$ and a unital semigroup homomorphism $\al\colon\bZ_+^d\to\End(A)$.
The system will be called \emph{injective} (resp. \emph{automorphic}) if $\al_\un{n}$ is injective (resp. a $*$-automorphism) for all $\un{n}\in\bZ_+^d$.
The system will be called \emph{unital} if $A$ is unital and $\al_\un{n}(1_A) = 1_A$ for all $\un{n}\in\bZ_+^d$.
\end{definition}

Let $(A,\al,\bZ_+^d)$ be a C*-dynamical system. We associate a product system $X_\al := \{X_{\al, \un{n}} \}_{\un{n}\in\bZ_+^d}$ with $(A,\al,\bZ_+^d)$ by setting
\[
X_{\al, \un{n}}:=[\al_\un{n}(A)A]
\foral
\un{n}\in\bZ_+^d,
\]
where each $X_{\al, \un{n}}$ becomes a C*-correspondence over $A$ via the left action
\[
\phi_\un{n}\colon A\to\L(X_{\al, \un{n}}); \phi_\un{n}(a)\xi_\un{n}=\al_\un{n}(a)\xi_\un{n}\foral a\in A, \xi_\un{n}\in X_{\al, \un{n}}.
\]
Notice that $X_{\al, \un{0}}=A$ and that $\al_{\un{n}}(A) \subseteq X_{\al, \un{n}}$ for all $\un{n} \in \bZ_+^d$. 
The product system structure is implemented by the multiplication maps 
\begin{equation}\label{Eq:dynsysmult}
\xi_\un{n} \otimes \xi_\un{m} \mapsto\al_\un{m}(\xi_\un{n})\xi_\un{m}\foral \xi_\un{n}\in X_{\al, \un{n}},\xi_\un{m}\in X_{\al, \un{m}},\un{n},\un{m}\in\bZ_+^d.
\end{equation}
All left actions are by compacts since $\phi_{\un{n}}(ab^*) = \Theta_{\al_{\un{n}}(a), \al_{\un{n}}(b)}$ for all $a, b \in A$ and $\un{n} \in \bZ_+^d$.
Thus $X_\al$ is strong compactly aligned by Corollary \ref{C:autscaps}.
For each $\un{n}\in\bZ_+^d$, we have $\ker\al_\un{n}=\ker\phi_\un{n}$; thus $(A, \al, \bZ_+^d)$ is injective if and only if $X_\al$ is injective.
For an ideal $I\subseteq A$, we have
\[
X_{\al, \un{n}}^{-1}(I)=\al_\un{n}^{-1}(I) \foral \un{n}\in\bZ_+^d.
\]

\begin{definition}\label{D:covpair}
Let $(A,\al,\bZ_+^d)$ be a C*-dynamical system. 
A \emph{covariant pair} for $(A,\al,\bZ_+^d)$ is a pair $(\pi,V)$ acting on a Hilbert space $H$ such that $\pi\colon A\to \B(H)$ is a $*$-homomorphism, $V\colon\bZ_+^d\to \B(H)$ is a semigroup homomorphism, and
\begin{equation} \label{Eq:cov dyn}
\pi(a)V_\un{n}=V_\un{n}\pi(\al_\un{n}(a))\foral \un{n}\in\bZ_+^d, a\in A.
\end{equation}
We say that $V$ is \emph{contractive/isometric} if $V_{\un{n}}$ is contractive/isometric for all $\un{n} \in \bZ_+^d$.
We say that $V$ is \emph{Nica-covariant} if it is contractive and $V_\un{n}^*V_\un{m}=V_\un{m}V_\un{n}^*$ for all $\un{n},\un{m}\in\bZ_+^d$ satisfying $\un{n}\wedge\un{m}=\un{0}$.

We say that $(\pi,V)$ is \emph{contractive/isometric/Nica-covariant} if $V$ is contractive/isometric/Nica-covariant.
An isometric Nica-covariant pair $(\pi,V)$ for $(A,\al,\bZ_+^d)$ is called \emph{Cuntz-Nica-Pimsner covariant} (abbrev. \emph{CNP-covariant}) if
\begin{equation}\label{Eq:cnp dyn}
\pi(a)\prod_{i\in F}(I-V_\un{i}V_\un{i}^*)=0\foral a\in\bigcap_{\un{n}\perp F}\al_\un{n}^{-1}((\bigcap_{i\in F}\ker\al_\un{i})^\perp)\;\text{and} \; \mt\neq F\subseteq[d].
\end{equation}
We write $\ca(\pi,V)$ for the C*-subalgebra of $\B(H)$ generated by the images of $\pi$ and $V$.
We write $\N\T(A,\al)$ for the C*-algebra that is universal with respect to the isometric Nica-covariant pairs for $(A,\al,\bZ_+^d)$, and $\N\O(A,\al)$ for the C*-algebra that is universal with respect to the CNP-covariant pairs for $(A,\al,\bZ_+^d)$, both of which are generated by a copy of $A$ and $\bZ_+^d$.
\end{definition}

The connection of (\ref{Eq:cnp dyn}) with the CNP-representations of $X_\al$ was established in \cite{DFK17}, where it is shown that
\[
\J_F=(\bigcap_{i\in F}\ker\al_\un{i})^\perp 
\qand 
\I_F=\bigcap_{\un{n}\perp F}\al_\un{n}^{-1}((\bigcap_{i\in F}\ker\al_\un{i})^\perp)
\]
for all $\mt \neq F \subseteq [d]$.
The isometric Nica-covariant (resp. CNP-covariant) pairs of $(A, \al, \bZ_+^d)$ induce Nica-covariant representations (resp. CNP-representations) of $X_\al$ by the association $(\pi,V) \mapsto (\pi, t)$, where $t_{\un{n}}(\al_{\un{n}}(a)b) = V_{\un{n}} \pi(\al_{\un{n}}(a) b)$ for all $a,b\in A$ and $\un{n} \in \bZ_+^d\setminus\{\un{0}\}$.
In general, we have
\[
\N\T_{X_\al} \hookrightarrow \N\T(A,\al) \qand \N\O_{X_\al} \hookrightarrow \N\O(A,\al).
\]
The embeddings are surjective when each $\al_\un{n}$ is non-degenerate, so that $X_{\al, \un{n}} \equiv [\al_\un{n}(A)A] = A$ for all $\un{n}\in\bZ_+^d$.

\begin{proposition}\label{P:repnbij} \cite{DFK17}
Let $(A,\al,\bZ_+^d)$ be a unital C*-dynamical system.
Then there exists a bijective correspondence between the set of Nica-covariant (resp. CNP-) representations $(\pi,t)$ of $X_\al$ with $\pi$ unital, and the set of isometric Nica-covariant (resp. CNP-covariant) pairs $(\pi,V)$ for $(A,\al,\bZ_+^d)$ with $\pi$ unital.
Consequently,
\[
\N\T_{X_\al} \cong \N\T(A,\al)\qand \N\O_{X_\al} \cong \N\O(A,\al)
\]
by canonical $*$-isomorphisms.
\end{proposition}

\begin{proof}
The association $(\pi,t) \mapsto (\pi,V)$ is defined by $V_\un{n}:=t_\un{n}(1_\un{n})$ for all $\un{n}\in\bZ_+^d$.
Notice that given any Nica-covariant representation $(\pi,t)$ of $X_\al$ or any isometric Nica-covariant pair $(\pi,V)$ of $(A,\al,\bZ_+^d)$, we have that $\pi(1_A)$ is a unit for $\ca(\pi,t)$ and $\ca(\pi,V)$.
Hence, by suitably restricting the codomains of $\pi, V$ and each $t_\un{n}$, we deduce the canonical $*$-isomorphisms.
\end{proof}

Injectivity of a unital system has an equivalent reformulation in terms of the semigroup representation in the Cuntz-Nica-Pimsner algebra.
This is a standard result in the theory of semicrossed products, but we include a proof here.

\begin{proposition}\label{P:injunitary}
Let $(A,\al,\bZ_+^d)$ be a unital C*-dynamical system.
Let $\N\O_{X_\al}=\ca(\pi,V)$, where $(\pi,V)$ is a CNP-covariant pair of $(A,\al,\bZ_+^d)$ such that $\pi$ is unital and injective.
Then $\al_\un{i}$ is injective for all $i\in[d]$ if and only if $V_\un{i}$ is a unitary for all $i\in[d]$.
\end{proposition}

\begin{proof}
First assume that $\al_{\un{i}}$ is injective for all $i \in [d]$, and so $X_\al$ is regular by Proposition \ref{P:inj+comp}.
Hence we have $\I_F=A$ for all $\mt\neq F\subseteq[d]$, and the CNP-covariance of $(\pi, V)$ gives that
\[
\pi(a)(I-V_\un{i}V_\un{i}^*)=0 \foral a\in\I_{\{i\}}=A, i\in[d].
\]
Applying for $a=1_A$ gives that $I-V_\un{i}V_\un{i}^*=0$, and so $V_\un{i}$ is a unitary for all $i\in[d]$.

Now assume that $V_\un{i}$ is a unitary for all $i\in[d]$.
Fixing $a \in A$ and $i \in [d]$, we have
\[
\pi(a)=\pi(a)V_\un{i}V_\un{i}^*=V_\un{i}\pi(\al_\un{i}(a))V_\un{i}^*,
\]
using (\ref{Eq:cov dyn}).
Thus it follows that if $a\in\ker\al_\un{i}$, then $a\in\ker\pi=\{0\}$ and hence $\al_\un{i}$ is injective, finishing the proof.
\end{proof}

Next we translate the key concepts of Sections \ref{S:relcnpalg} and \ref{S:g inv struc} into the language of C*-dynamical systems. 
Firstly, note that all $2^d$-tuples of $X_\al$ are relative since all left actions are by compacts.
Let $\L$ be a $2^d$-tuple of $X_\al$ consisting of ideals.
By construction of $X_\al$, we see that $\L$ is $X_\al$-invariant if and only if
\[
\L_F\subseteq\bigcap_{\un{n}\perp F}\al_\un{n}^{-1}(\L_F)\foral F\subseteq[d].
\]
In particular, an ideal $I\subseteq A$ is positively invariant for $X_\al$ if and only if
\[
I\subseteq\bigcap_{\un{n}\in\bZ_+^d}\al_\un{n}^{-1}(I).
\]
Also, $I$ is negatively invariant for $X_\al$ if and only if
\[
\I_F\cap(\bigcap\{\al_\un{n}^{-1}(I)\mid \un{0}\neq\un{n}\leq\un{1}_F\})\subseteq I\foral \mt\neq F\subseteq[d].
\]
Moreover, we have a $*$-isomorphism
\[
\K(X_{\al, \un{n}}) \to [\al_{\un{n}}(A) A \al_{\un{n}}(A)]; \Theta_{\xi_{\un{n}}, \eta_{\un{n}}} \mapsto \xi_{\un{n}} \eta_{\un{n}}^* \foral \xi_{\un{n}}, \eta_{\un{n}} \in X_{\al, \un{n}}, \un{n}\in\bZ_+^d.
\]
This follows by a double application of \cite[Lemma 4.6.1]{BO08} in the same way one obtains that $A\cong\K(A)$.
By restriction we obtain a $*$-isomorphism
\[
\K(X_{\al, \un{n}} I) \to [\al_{\un{n}}(A) I  \al_{\un{n}}(A)]; \Theta_{\xi_{\un{n}}, \eta_{\un{n}}} \mapsto \xi_{\un{n}} \eta_{\un{n}}^* \foral \xi_{\un{n}}, \eta_{\un{n}} \in X_{\al, \un{n}} I, \un{n}\in\bZ_+^d,
\]
for an ideal $I\subseteq A$.
Therefore, for each $F \subseteq [d]$ and $\un{m} \in \bZ_+^d$, we have
\[
\|\phi_\un{m}(a)+\K(X_{\al, \un{m}}I)\|=\|\al_\un{m}(a) + [\al_\un{m}(A) I \al_\un{m}(A)]\| \foral a \in A.
\]
Additionally, for each $\mt\neq F\subseteq[d]$, we have
\[
J_F(I,X_\al)=\{a\in A\mid a(\bigcap\{\al_\un{n}^{-1}(I)\mid \un{0}\neq\un{n}\leq\un{1}_F\})\subseteq I\}.
\]
When $I$ is positively invariant for $X_\al$, we define the C*-dynamical system $([A]_I,[\al]_I,\bZ_+^d)$ by
\[
[\al_\un{n}]_I[a]_I=[\al_\un{n}(a)]_I \foral a \in A, \un{n}\in\bZ_+^d.
\]
Hence we obtain the product system $X_{[\al]_I}$ over $\bZ_+^d$ with coefficients in $[A]_I$.
Note that
\[
\ker [\phi_{\un{n}}]_I = \ker [\al_{\un{n}}]_I = [\al^{-1}_{\un{n}}(I)]_I
\foral
\un{n} \in \bZ_+^d.
\]
Moreover, we have $[X_\al]_I \cong X_{[\al]_I}$ when each $X_{\al,\un{n}}$ is non-degenerate.
If $(A,\al,\bZ_+^d)$ is unital, then so is $([A]_I,[\al]_I,\bZ_+^d)$.

\begin{corollary}\label{C:dynsysjargNT}
Let $(A,\al,\bZ_+^d)$ be a C*-dynamical system. 
Let $\K$ and $\L$ be $2^d$-tuples of $X_\al$.
Then $\L$ is a $\K$-relative NO-$2^d$-tuple of $X_\al$ if and only if $\K\subseteq\L$ and the following hold:
\vspace{2pt}
\begin{enumerate} \setlength\itemsep{.3em}
\item $\L$ consists of ideals and $\L_F \cap (\bigcap_{i\in F}\al_\un{i}^{-1}(\L_\mt))\subseteq\L_\mt$ for all $\mt\neq F\subseteq[d]$,
\item $\L_F\subseteq\bigcap_{\un{n}\perp F}\al_\un{n}^{-1}(\L_F)\foral F\subseteq[d]$,
\item $\L$ is partially ordered,
\item $I_{1,F}\cap I_{2,F}\cap I_{3,F}\subseteq\L_F$ for all $\mt\neq F\subsetneq[d]$, where
\vspace{.3em}
\begin{itemize}  \setlength\itemsep{.3em}
\item $I_{1,F}:=\bigcap_{\un{n}\perp F}\al_\un{n}^{-1}(\{a\in A\mid a(\bigcap_{i\in F}\al_\un{i}^{-1}(\L_\mt))\subseteq\L_\mt\})$,
\item $I_{2,F}:=\bigcap_{\un{m}\perp F}\al_\un{m}^{-1}(\cap_{F\subsetneq D}\L_D)$,
\item $I_{3,F}:=\{a\in A\mid \lim_{\un{m}\perp F}\|\al_\un{m}(a)+[\al_\un{m}(A)\L_F\al_\un{m}(A)]\|=0\}$. 
\end{itemize}
\end{enumerate}
\end{corollary}

\begin{proof}
The result follows immediately by the remarks preceding the statement, applied in conjunction with Definitions \ref{D:NT tuple} and \ref{D:rel NO tuple}, as well as Proposition \ref{P:NTcom}.
Note that the latter applies since the left action of each fibre of $X_\al$ is by compacts.
For item (i) and the definition of $I_{1,F}$, we also use that 
\[
\bigcap\{\al_\un{n}^{-1}(\L_\mt)\mid \un{0}\neq\un{n}\leq\un{1}_F\}=\bigcap_{i\in F}\al_\un{i}^{-1}(\L_\mt)
\foral
\mt\neq F\subseteq[d],
\]
which follows from Lemma \ref{L:quo ker} since $\L_\mt\subseteq A$ is  positively invariant.
\end{proof}

Next we turn our attention to showing that, in general, we do not have a canonical $*$-epimorphism $\N\O_X\to\N\O_{[X]_I}$ for a positively invariant ideal $I\subseteq A$.
In particular, injectivity of $(A,\al,\bZ_+^d)$ need not imply injectivity of $([A]_I,[\al]_I,\bZ_+^d)$.

\begin{example}\label{E:surjcounter}
Let $(A,\al,\bZ_+)$ be a unital and injective C*-dynamical system.
Suppose $I\subseteq A$ is an ideal that is positively invariant for $X_\al$ (i.e., $I\subseteq\al^{-1}(I)$), and so $[X_\al]_I \cong X_{[\al]_I}$.

We claim that we do not have a canonical $*$-epimorphism $\Phi \colon \N\O_{X_\al}\to\N\O_{X_{[\al]_I}}$, in general.
To reach contradiction, assume that such a map $\Phi$ exists.
Write $\N\O_{X_\al}=\ca(\pi,V)$ and $\N\O_{X_{[\al]_I}}=\ca(\si,W)$ for CNP-covariant pairs $(\pi,V)$ and $(\si,W)$ of $(A,\al,\bZ_+)$ and $([A]_I,[\al]_I,\bZ_+)$ (respectively), where $\pi$ and $\si$ are unital and injective.
Since $\al$ is injective, we have that $V_1$ is a unitary by Proposition \ref{P:injunitary}, and thus $W_1 = \Phi(V_1)$ is also a unitary.
Another application of Proposition \ref{P:injunitary} gives that $[\al]_I$ is injective.
However, this leads to contradiction when we choose $A,\al$ and $I$ so that $[\al]_I$ is not injective.

For such an example, let $B$ be a unital C*-algebra and $\be\in\End(B)$ be unital and non-injective.
Set $A=\bigoplus_{n\in\bN}B$ and define $\al \in \End(A)$ by $\al((b_n)_{n \in \bN})=(\be(b_1),b_1,b_2,\dots)$ for all $(b_n)_{n\in\bN}\in A$, noting that $\al$ is unital and injective.
Set
\[
I:=\{(b_n)_{n\in \bN} \in A\mid b_1=0\},
\]
which is an ideal that is positively invariant for $X_\al$.
Then $[\al]_I \in \End([A]_I)$ is conjugate to $\be \in \End(B)$, and thus $[\al]_I$ is not injective by the choice of $\be$.

Notice that this does \emph{not} contradict Proposition \ref{P:NO can surj} because $I$ does not participate in any NO-$2$-tuple of $X_\al$.
Indeed, suppose that $I$ participates in an NO-$2$-tuple $\L$ of $X_\al$, i.e., $\{I, \L_{\{1\}}\}$ is an O-pair of $X_{\al,1}$ by Proposition \ref{P:NO+O}.
Since $\al$ is injective, we have $J_{X_{\al,1}}=A$ and thus $\L_{\{1\}} = A$.
Hence $\al^{-1}(I) = A \cap \al^{-1}(I)\subseteq I$ by item (i) of Corollary \ref{C:dynsysjargNT}.
However, $\be$ is not injective and so we can choose $0 \neq b\in\ker\be$.
Then $\al(b,0,0,\dots)=(0,b,0,\dots)\in I$ but $(b,0,0,\dots)\notin I$, giving the contradiction that $\al^{-1}(I)\not\subseteq I$.
\end{example}

Next we turn to injective C*-dynamical systems $(A,\al,\bZ_+^d)$, so that $X_\al$ is regular.
Hence the gauge-invariant ideal structure of $\N\O_{X_\al}$ falls under the purview of Corollary \ref{C:regNObij}.

\begin{corollary}\label{C:injdynbij}
Let $(A,\al,\bZ_+^d)$ be an injective C*-dynamical system. 
Then the association
\begin{equation}
I\mapsto\L_I, \; \text{where} \; \L_{I,\mt}:=I \; \text{and} \; \L_{I,F}:=A\foral \mt\neq F\subseteq[d],
\end{equation}
defines a bijection between the set of ideals $I \subseteq A$ satisfying $\al_\un{n}(I)\subseteq I$ and $\al_\un{n}^{-1}(I)\subseteq I$ for all $\un{n}\in\bZ_+^d$ and the set of NO-$2^d$-tuples of $X_\al$, which in turn induces a bijection with the set of gauge-invariant ideals of $\N\O_{X_\al}$.
\end{corollary}

\begin{proof}
Note that the set of ideals of $A$ that are positively and negatively invariant for $X_\al$ is exactly the set of ideals $I$ of $A$ which satisfy
\begin{equation}\label{Eq:invrel}
I\subseteq\bigcap_{\un{n}\in\bZ_+^d}\al_\un{n}^{-1}(I) \qand \bigcup_{i\in[d]}\al_\un{i}^{-1}(I)\subseteq I.
\end{equation}
Notice that $\bigcup_{i\in[d]}\al_\un{i}^{-1}(I)\subseteq I$ holds if and only if $\al_\un{n}^{-1}(I)\subseteq I$ for all $\un{n}\in\bZ_+^d$.
Consequently, $I$ satisfies (\ref{Eq:invrel}) if and only if $\al_\un{n}(I)\subseteq I$ and $\al_\un{n}^{-1}(I)\subseteq I$ for all $\un{n}\in\bZ_+^d$.
Since $X_\al$ is regular, the claim now follows from Corollary \ref{C:regNObij}.
\end{proof}

Now let $(A,\al,\bZ_+^d)$ be an automorphic C*-dynamical system.
Then we can uniquely extend $\al$ to a group action $\al \colon \bZ^d \to \Aut(A)$ since $\bZ_+^d\subseteq\bZ^d$ is a spanning cone, and consider the crossed product C*-algebra $A\rtimes_\al\bZ^d$.
We have $\N\O_{X_\al}\cong A\rtimes_\al\bZ^d$ by an equivariant $*$-isomorphism, which thus preserves the lattice structure of the gauge-invariant ideals.
Accordingly, Corollary \ref{C:injdynbij} recovers the following well-known result from the theory of crossed products.

\begin{corollary}\label{C:autdynlat}
Let $(A,\al,\bZ_+^d)$ be an automorphic C*-dynamical system.
Then the bijection of Corollary \ref{C:injdynbij} induces a lattice isomorphism between the set of $\al$-invariant ideals of $A$ and the set of gauge-invariant ideals of $A \rtimes_{\al} \bZ^d$.
\end{corollary}

\begin{proof}
Notice that in the automorphic case, the ideals $I\subseteq A$ satisfying $\al_\un{n}(I)\subseteq I$ and $\al_\un{n}^{-1}(I)\subseteq I$ for all $\un{n}\in\bZ_+^d$ are exactly the ideals $I\subseteq A$ satisfying $\al_\un{n}(I)=I$ for all $\un{n}\in\bZ_+^d$.
The set of all such ideals carries the usual lattice structure.
Consider the mapping
\begin{equation*}
I\mapsto\L_I, \; \text{where} \; \L_{I,\mt}:=I \; \text{and} \; \L_{I,F}:=A\foral \mt\neq F\subseteq[d],
\end{equation*}
from Corollary \ref{C:injdynbij}.
In view of Corollary \ref{C:CNP param} and the comments preceding the statement, it suffices to show that if $I,J\subseteq A$ are $\al$-invariant ideals, then
\[
\L_{I \cap J}=\L_I\wedge\L_J
\qand
\L_{I +J} = \L_I \vee \L_J.
\]

For the operation $\wedge$, this follows by Proposition \ref{P:NTwedge}.
For the operation $\vee$, this amounts to showing that $\L_{I+J,F}=(\L_I\vee\L_J)_F$ for all $F\subseteq[d]$.
To this end, first fix $\mt\neq F\subseteq[d]$.
Since $\L_I\vee\L_J$ is an NO-$2^d$-tuple of the regular product system $X_\al$, we have
\[
\I_F=A\subseteq(\L_I\vee\L_J)_F\subseteq A.
\]
By definition of $\L_{I+J}$, we obtain
\[
\L_{I+J,F}=A=(\L_I\vee\L_J)_F,
\]
as required.
It remains to check that $\L_{I+J,\mt}=(\L_I\vee\L_J)_\mt$.
By Proposition \ref{P:NTvee}, this amounts to showing that
\[
I+J=\ol{\pi}_{X_\al}^{-1}(\fJ^{\L_I}+\fJ^{\L_J}).
\]
The forward inclusion is immediate.
For the reverse inclusion, we identify $\N\O_{X_\al}$ with $\ca(\pi, U)$ for a CNP-covariant pair $(\pi,U)$ where $\pi$ is injective and $U$ is a unitary representation, and let $Q \colon \N\T_{X_\al} \to \N\O_{X_\al}$ be the quotient map.
Since $[X_{\al}]_I$ is regular by item (ii) of Lemma \ref{L:regNO}, an application of item (iii) of Corollary \ref{C:NO source} gives that
\[
Q(\fJ^{\L_I}) = \ol{\spn} \{ U_{\un{n}} \pi(I) U_{\un{m}}^* \mid \un{n}, \un{m} \in \bZ_+^d \},
\]
and likewise for $Q(\fJ^{\L_J})$.
Applying the faithful conditional expectation $E_\be\colon\N\O_{X_\al}\to\N\O_{X_\al}^\be$ then gives that
\begin{align*}
Q(\fJ^{\L_I} + \fJ^{\L_J})^\be 
= 
\ol{\spn} \{ U_{\un{n}} \pi(I + J) U_{\un{n}}^* \mid \un{n} \in \bZ_+^d \}
=
\ol{\spn} \{ \pi \circ \al_{\un{n}}^{-1}(I + J) \mid \un{n} \in \bZ_+^d \} 
=
\pi(I + J),
\end{align*}
using $\al$-invariance of $I$ and $J$ together with (\ref{Eq:cov dyn}).
Therefore, if $a \in \ol{\pi}_{X_\al}^{-1}(\fJ^{\L_I} + \fJ^{\L_J})$ then 
\[
\pi(a) = Q(\ol{\pi}_{X_\al}(a)) \in Q(\fJ^{\L_I} + \fJ^{\L_J})^\be = \pi(I +J).
\]
Injectivity of $\pi$ then gives that $a \in I+J$, as required.
\end{proof}

In the automorphic case, we also have
\[
(A\rtimes_\al\bZ^d)/(I\rtimes_{\al|_I}\bZ^d)\cong A/I\rtimes_{[\al]_I}\bZ^d
\]
whenever $I\subseteq A$ is an ideal satisfying $\al_\un{n}(I)=I$ for all $\un{n}\in\bZ_+^d$.
This is in accordance with Corollary \ref{C:NO source}, i.e.,
\[
\N\O_{X_\al}/\sca{Q_\I(\ol{\pi}_{X_\al}(I))}\cong\N\O_{X_{[\al]_{I}}}.
\]

\subsection{Higher-rank graphs}\label{Ss:hrgraph}

In this subsection we interpret the parametrisation in the case of strong finitely aligned higher-rank graphs, which also accounts for row-finite higher-rank graphs.
The parametrisation we offer is in terms of vertex sets.
As a corollary, we recover the parametrisation of Raeburn, Sims and Yeend \cite[Theorem 5.2]{RSY03} for locally convex row-finite higher-rank graphs.
The concepts from the theory of higher-rank graphs that we utilise are taken from \cite{RS05, RSY03, RSY04, Sim06}.
For this subsection, we will reserve $d$ for the degree map of a graph $(\La,d)$ of rank $k$.

Fix $k\in\bN$.
A \emph{$k$-graph} $(\La,d)$ consists of a countable small category $\La=(\text{Obj}(\La),\text{Mor}(\La),r,s)$ together with a functor $d\colon\La\to\bZ_+^k$, called the \emph{degree map}, satisfying the \emph{factorisation property}: 
\begin{quote}
For all $\la\in\text{Mor}(\La)$ and $\un{m},\un{n}\in\bZ_+^k$ such that $d(\la)=\un{m}+\un{n}$, there exist unique $\mu,\nu\in\text{Mor}(\La)$ such that $d(\mu)=\un{m}, d(\nu)=\un{n}$ and $\la=\mu\nu$.
\end{quote}
Here we view $\bZ_+^k$ as a category consisting of a single object, and whose morphisms are exactly the elements of $\bZ_+^k$ (when viewed as a set).
Composition in this category is given by entrywise addition, and the identity morphism is $\un{0}$.
Therefore, $d$ being a functor means that 
\[
d(\la\mu)=d(\la)+d(\mu) \text{ when } \la,\mu\in\text{Mor}(\La) \text{ satisfy } r(\mu)=s(\la),
\]
and 
\[
d(\id_v)=\un{0} \foral v\in\text{Obj}(\La).
\]
We view $k$-graphs as generalised graphs, and therefore refer to the elements of $\text{Obj}(\La)$ as \emph{vertices} and the elements of $\text{Mor}(\La)$ as \emph{paths}.
Fixing $\la\in\text{Mor}(\La)$, the factorisation property guarantees that 
\[
d(\la)=\un{0} \text{ if and only if } \la=\id_{s(\la)}.
\]
Hence we may identify $\text{Obj}(\La)$ with $\{\la\in\text{Mor}(\La)\mid d(\la)=\un{0}\}$, and consequently we may write $\la\in\La$ instead of $\la\in\text{Mor}(\La)$ without any ambiguity.

Fix a $k$-graph $(\La,d)$.
Given $\la\in\La$ and $E\subseteq\La$, we define
\[
\la E:=\{\la\mu\in\La\mid \mu\in E, r(\mu)=s(\la)\}\qand E\la:=\{\mu\la\in\La\mid \mu\in E, r(\la)=s(\mu)\}.
\]
In particular, we may replace $\la$ by a vertex $v\in\La$ and write
\[
vE:=\{\la\in E\mid r(\la)=v\}\qand Ev:=\{\la\in E\mid s(\la)=v\}.
\]
Analogously, given $E,F\subseteq\La$, we define
\[
EF:=\{\la\mu\in\La\mid \la\in E,\mu\in F, r(\mu)=s(\la)\}.
\]
Fixing $\un{n}\in\bZ_+^k$, we set
\begin{align*}
\La^\un{n} :=\{\la\in\La\mid d(\la)=\un{n}\} \; \text{ and } \;
\La^{\leq\un{n}} :=\{\la\in\La\mid d(\la)\leq\un{n} \; \text{and} \; s(\la)\La^\un{i}=\mt\;\text{if}\; d(\la)+\un{i}\leq\un{n}\}.
\end{align*}
We will refer to the elements of $\La^\un{i}$ as \emph{edges} for all $i\in[k]$.
Given any $v\in\La^\un{0}$, the set $v\La^{\leq\un{n}}$ is non-empty (indeed, an element can be constructed inductively).

Suppose that we have $\un{\ell},\un{m},\un{n}\in\bZ_+^k$ satisfying $\un{\ell}\leq\un{m}\leq\un{n}$, and $\la\in\La$ satisfying $d(\la)=\un{n}$.
By the factorisation property, there exist unique paths $\la(\un{0},\un{\ell}), \la(\un{\ell},\un{m}), \la(\un{m},\un{n}) \in \La$ such that 
\[
d(\la(\un{0},\un{\ell}))=\un{\ell}, d(\la(\un{\ell},\un{m}))=\un{m}-\un{\ell}, d(\la(\un{m},\un{n}))=\un{n}-\un{m}, \text{ and }\la=\la(\un{0},\un{\ell})\la(\un{\ell},\un{m})\la(\un{m},\un{n}).
\]
Fixing $\mu,\nu\in\La$, we define the set of \emph{minimal common extensions} of $\mu$ and $\nu$ by
\[
\text{MCE}(\mu,\nu):=\{\la\in\La\mid d(\la)=d(\mu)\vee d(\nu), \la(\un{0},d(\mu))=\mu, \la(\un{0},d(\nu))=\nu\}.
\]
Notice that $\text{MCE}(\mu,\nu)$ may be empty, e.g., when $r(\mu)\neq r(\nu)$.
Intrinsically related to $\text{MCE}(\mu,\nu)$ is the set
\begin{align*}
\La^{\text{min}}(\mu,\nu) & :=\{(\al,\be)\in\La\times\La\mid \mu\al=\nu\be\in\text{MCE}(\mu,\nu)\} \\
						  & =\{(\al,\be)\in\La\times\La\mid \mu\al=\nu\be, d(\mu\al)=d(\mu)\vee d(\nu)=d(\nu\be)\}.
\end{align*}
Note that if $\mu,\nu\in\La$ satisfy $d(\mu) = d(\nu)$, then $\La^{\text{min}}(\mu,\nu)$ being non-empty implies that $\mu=\nu$.
Given a vertex $v\in\La^\un{0}$, a subset $E\subseteq v\La$ is called \emph{exhaustive} if for every $\la\in v\La$ there exists $\mu\in E$ such that $\La^{\text{min}}(\la,\mu)\neq\mt$.

A $k$-graph $(\La,d)$ is said to be \emph{finitely aligned} if 
\[
|\text{MCE}(\mu,\nu)|<\infty \foral \mu,\nu\in\La;
\]
equivalently, if 
\[
|\La^{\text{min}}(\mu,\nu)|<\infty \foral \mu,\nu\in\La.
\]
Following \cite[Definition 7.2]{DK18}, we say that $(\La,d)$ is \emph{strong finitely aligned} if $(\La,d)$ is finitely aligned and for every $\la\in\La$ and $i\in[k]$ satisfying $d(\la)\perp\un{i}$, there are at most finitely many edges $e\in\La^\un{i}$ such that $\La^{\text{min}}(\la,e)\neq\mt$.
A $k$-graph $(\La,d)$ is said to be \emph{row-finite} if $|v\La^\un{n}|<\infty$ for all $v\in\La^\un{0}$ and $\un{n}\in\bZ_+^k$.
Row-finite $k$-graphs are in particular strong finitely aligned. 

We say that $(\La,d)$ is \emph{locally convex} if, for all $v\in\La^\un{0}$, $\la\in v\La^\un{i}$ and $\mu\in v\La^\un{j}$ for $i,j\in[k]$ with $i\neq j$, we have that $s(\la)\La^\un{j}$ and $s(\mu)\La^\un{i}$ are non-empty.
Finally, $(\La,d)$ is said to be \emph{sourceless} if $v\La^\un{n}\neq\mt$ for all $v\in\La^\un{0}$ and $\un{n}\in\bZ_+^k$.
Any sourceless $k$-graph is automatically locally convex.

Let $(\La,d)$ be a finitely aligned $k$-graph.
A set of partial isometries $\{T_{\lambda}\}_{\lambda \in \La}$ in a C*-algebra is called a \emph{Toeplitz-Cuntz-Krieger $\La$-family} if the following hold:
\begin{enumerate}
\item[(TCK1)]
$\{T_v\}_{v\in \La^{\un{0}}}$ is a collection of pairwise orthogonal projections;
\item[(TCK2)]
$T_{\la} T_{\mu} = \de_{s(\la), r(\mu)} T_{\la \mu}$ for all $\la, \mu \in \La$;
\item[(TCK3)]
$T_{\la}^* T_{\mu} = \sum_{(\alpha,\beta)\in \La^{\min}(\la, \mu)} T_{\alpha} T_{\beta}^*$ for all $\la, \mu \in \La$.
\end{enumerate}
A Toeplitz-Cuntz-Krieger $\La$-family $\{T_{\lambda}\}_{\lambda \in \La}$ is called a \emph{Cuntz-Krieger $\La$-family} if it satisfies:
\begin{enumerate}
\item[(CK)]
$\prod_{\lambda \in E}(T_v - T_{\lambda} T_{\lambda}^*) = 0$ for every $v\in \La^{\un{0}}$ and all non-empty finite exhaustive sets $E \subseteq v\La$.
\end{enumerate}
Note that $T_\la^*T_\mu=0$ whenever $\la,\mu\in\La$ satisfy $\La^{\text{min}}(\la,\mu)=\mt$ by (TCK3).
Likewise, (TCK3) implies that $T_\la^*T_\la=T_{s(\la)}$ for all $\la\in\La$.
The C*-algebra $\ca(\La)$ is the universal one with respect to the Cuntz-Krieger $\La$-families, and satisfies a Gauge-Invariant Uniqueness Theorem \cite[Theorem 4.2]{RSY04}.

Every $k$-graph $(\La,d)$ is canonically associated with a product system $X(\La) := \{X_\un{n}(\La)\}_{\un{n}\in\bZ_+^k}$ with coefficients in the C*-algebra $c_0(\La^\un{0})$, where we view $\La^\un{0}$ as a discrete space.
Firstly, set $X_\un{0}(\La):=c_0(\La^\un{0})$, which we view as a C*-correspondence over itself in the usual way.
For each $v\in\La^\un{0}$, we write $\delta_v\in c_0(\La^\un{0})$ for the projection on $\{v\}$.
For every $\un{0} \neq \un{n} \in \bZ_+^k$, we consider the linear space $c_{00}(\La^\un{n})$ and write $\de_\la$ for its generators.
A right pre-Hilbert $c_0(\La^\un{0})$-module structure on $c_{00}(\La^\un{n})$ is given by
\[
\sca{\xi_\un{n}, \eta_\un{n}}(v) := \sum_{s(\lambda) = v}\overline{\xi_\un{n}(\lambda)}\eta_\un{n}(\lambda) \qand (\xi_\un{n}a)(\lambda) :=\xi_\un{n}(\lambda)a(s(\lambda)),
\]
for all $\xi_\un{n},\eta_\un{n}\in c_{00}(\La^\un{n}), a\in c_0(\La^\un{0}), v\in\La^\un{0}$ and $\la\in\La^\un{n}$.
We write $X_{\un{n}}(\La)$ for the right Hilbert C*-module completion of $c_{00}(\La^{\un{n}})$.
A left action $\phi_\un{n}$ of $c_0(\La^\un{0})$ on $X_\un{n}(\La)$ is induced by
\[
\phi_\un{n}(a)\colon c_{00}(\La^\un{n})\to c_{00}(\La^\un{n}); (\phi_\un{n}(a)\xi_\un{n})(\la)=a(r(\la))\xi_\un{n}(\la)\foral a\in c_0(\La^\un{0}), \xi_\un{n}\in c_{00}(\La^\un{n}), \la\in \La^\un{n},
\]
thereby imposing a C*-correspondence structure on $X_\un{n}(\La)$.
Fixing $\un{n},\un{m}\in\bZ_+^k$, we define a multiplication map $u_{\un{n},\un{m}}$ of $X(\La)$ by
\[
u_{\un{n},\un{m}}\colon X_\un{n}(\La)\otimes_{c_0(\La^\un{0})}X_\un{m}(\La)\to X_{\un{n}+\un{m}}(\La); u_{\un{n},\un{m}}(\delta_\la\otimes\delta_\mu)=\begin{cases} \delta_{\la\mu} & \text{if} \; r(\mu)=s(\la), \\ 0 & \text{otherwise,} \end{cases}
\]
for all $\la\in\La^\un{n}$ and $\mu\in\La^\un{m}$, rendering $X(\La)$ a product system over $\bZ_+^k$ with coefficients in $c_0(\La^\un{0})$.

For $\un{n}\in\bZ_+^k$ and $\la,\mu,\nu\in\La^\un{n}$, we have
\[
\Theta_{\delta_\la,\delta_\mu}^{X_\un{n}(\La)}(\delta_\nu)=\delta_\la\sca{\delta_\mu,\delta_\nu}=\begin{cases} \delta_\la\delta_{s(\mu)} & \text{if} \; \nu=\mu, \\ 0 & \text{otherwise}.\end{cases}
\]
In turn, $\Theta_{\delta_\la,\delta_\mu}^{X_\un{n}(\La)}$ is non-zero if and only if $s(\la)=s(\mu)$, and we obtain
\begin{equation*}
\K(X_\un{n}(\La))=\ol{\spn}\{\Theta_{\delta_\la,\delta_\mu}^{X_\un{n}(\La)}\mid \la,\mu\in\La^\un{n}, s(\la)=s(\mu)\}.
\end{equation*}
A $k$-graph $(\La,d)$ is finitely aligned if and only if $X(\La)$ is compactly aligned \cite[Theorem 5.4]{RS05}.
In this case, if $\la_1,\la_2\in\La^\un{n}$ and $\mu_1,\mu_2\in\La^\un{m}$ are non-trivial paths that satisfy $s(\la_1)=s(\la_2)$ and $s(\mu_1)=s(\mu_2)$, then
\begin{equation*}
\iota_\un{n}^{\un{n}\vee\un{m}}(\Theta_{\delta_{\la_1},\delta_{\la_2}}^{X_\un{n}(\La)})\iota_\un{m}^{\un{n}\vee\un{m}}(\Theta_{\delta_{\mu_1},\delta_{\mu_2}}^{X_\un{m}(\La)})=\sum_{(\al,\be)\in\La^{\text{min}}(\la_2,\mu_1)}\Theta_{\delta_{\la_1\al},\delta_{\mu_2\be}}^{X_{\un{n}\vee\un{m}}(\La)}.
\end{equation*}
Similarly, $(\La,d)$ is strong finitely aligned if and only if $X(\La)$ is strong compactly aligned \cite[Proposition 7.1]{DK18}.
We record the following well-known results for future reference.

\begin{lemma}\label{L:hrcomp}
Let $(\La,d)$ be a $k$-graph. 
Fix $\un{n}\in\bZ_+^k$ and $v\in\La^\un{0}$. 
Then the following hold:
\begin{enumerate}
\item $\phi_\un{n}(\delta_v)=0$ if and only if $v\La^\un{n}=\mt$.
\item $\phi_\un{n}(\delta_v)\in\K(X_\un{n}(\La))$ if and only if $|v\La^\un{n}|<\infty$.
\end{enumerate}
\end{lemma}

\begin{lemma}\label{L:hrrfsprodsys}
Let $(\La,d)$ be a $k$-graph.
Then the following hold:
\begin{enumerate}
\item $(\La,d)$ is row-finite if and only if $\phi_\un{n}(c_0(\La^\un{0}))\subseteq\K(X_\un{n}(\La))$ for all $\un{n}\in\bZ_+^k$.
\item $(\La,d)$ is sourceless if and only if $X(\La)$ is injective.
\item $(\La,d)$ is row-finite and sourceless if and only if $X(\La)$ is regular.
\end{enumerate}
\end{lemma}

The following proposition implies that $\ca(\La)\cong\N\O_{X(\La)}$ canonically, for finitely aligned $\La$.

\begin{proposition}\label{P:hrrepnbij}
Let $(\La,d)$ be a finitely aligned $k$-graph. 
Then there exists a bijective correspondence between the set of Nica-covariant (resp. CNP-) representations of $X(\La)$ and the set of Toeplitz-Cuntz-Krieger (resp. Cuntz-Krieger) $\La$-families.
\end{proposition}

We will use the duality between ideals of $c_0(\La^{\un{0}})$ and subsets of $\La^{\un{0}}$ given by the mutually inverse mappings
\begin{align*}
I & \mapsto H_I:=\{v\in\La^\un{0}\mid \delta_v\in I\}, \; \text{for all ideals} \; I\subseteq c_0(\La^\un{0}); \\
H & \mapsto I_H:=\ol{\spn}\{\delta_v\mid v\in H\}, \foral H \subseteq \La^\un{0}.
\end{align*}
Note that this duality implements a lattice isomorphism, and that $I_\mt=\{0\}$ and $I_{\La^\un{0}}=c_0(\La^\un{0})$.

If $(\La,d)$ is strong finitely aligned, then we can simplify (CK) using the machinery of $F$-tracing vertices. 
Fix $\mt\neq F\subseteq[k]$.
We say that $v\in\La^\un{0}$ is an \emph{$F$-source} if $v\La^\un{i}=\mt$ for all $i\in F$.
We say that $v\in\La^\un{0}$ is \emph{$F$-tracing} if:
\begin{quote}
for every $\la\in\La$ with $d(\la)\perp F$ and $r(\la)=v$, we have that $s(\la)$ is not an $F$-source, and $|s(\la)\La^\un{i}|<\infty$ for all $i\in[k]$.
\end{quote}
If $\la\in\La$ satisfies $d(\la)\perp F$ and $r(\la)$ is $F$-tracing, then $s(\la)$ is also $F$-tracing.
Indeed, by \cite[p. 589]{DK18}, we have
\[
H_{\J_F}=\{v\in\La^\un{0}\mid v \; \text{is not an $F$-source and} \; |v\La^\un{i}|<\infty\foral i\in[k]\},
\]
and 
\[
H_{\I_F}=\{v\in\La^\un{0}\mid v \; \text{is $F$-tracing}\}.
\]
In \cite[Theorem 7.6]{DK18}, it is shown that a Toeplitz-Cuntz-Krieger $\La$-family $\{T_\la\}_{\la\in\La}$ is a Cuntz-Krieger $\La$-family if and only if it satisfies 
\begin{enumerate}
\item[(CK$^\prime$)] $\prod\{T_v-T_\la T_\la^*\mid \la\in v\La^\un{i},i\in F\}=0$, for every $F$-tracing vertex $v$ and every $\mt\neq F\subseteq[k]$.
\end{enumerate}

Let $\L$ be a $2^k$-tuple of $X(\La)$ that consists of ideals.
For notational convenience, we set 
\[
H_{\L,F}:=H_{\L_F} \foral F\subseteq[k] \qand H_\L:=\{H_{\L,F}\}_{F\subseteq[k]}.
\]
For an ideal $I \subseteq c_0(\La^\un{0})$, we have
\begin{equation} \label{eq:n la}
H_{X_{\un{n}}(\La)^{-1}(I)}=\{v\in\La^\un{0}\mid s(v\La^\un{n})\subseteq H_I \}
\foral \un{n} \in \bZ_+^k.
\end{equation}
From this we obtain
\begin{equation}\label{eq:hrinv}
H_{\L_{\inv,F}}=\bigcap_{\un{m}\perp F}\{v\in\La^\un{0}\mid s(v\La^\un{m})\subseteq\cap_{F\subsetneq D}H_{\L,D}\}\foral\mt\neq F\subsetneq[k].
\end{equation}
To address each $\L_{\lim,F}$, we have the following proposition.

\begin{proposition}\label{P:hrlim}
Let $(\La,d)$ be a strong finitely aligned $k$-graph.
Let $\L$ be a $2^k$-tuple of $X(\La)$ that consists of ideals and let $H_\L$ be the corresponding family of sets of vertices of $\La$.
Then, fixing $\mt\neq F\subsetneq[k]$, a vertex $v\in\La^\un{0}$ belongs to $H_{\L_{\lim,F}}$ if and only if there exists $\un{m}\perp F$ such that whenever $\un{n}\perp F$ and $\un{n}\geq\un{m}$, we have $s(v\La^\un{n})\subseteq H_{\L,F}$ and $|v\La^\un{n}|<\infty$.
\end{proposition}
\begin{proof}
For the forward implication, take $v\in H_{\L_{\lim,F}}$.
Then we have
\[
\lim_{\un{m}\perp F}\|\phi_\un{m}(\delta_v)+\K(X_\un{m}(\La)\L_F)\|=0.
\] 
By definition there exists $\un{m}\perp F$ such that whenever $\un{n}\perp F$ and $\un{n}\geq\un{m}$, we have
\[
\|\phi_\un{n}(\delta_{v})+\K(X_\un{n}(\La)\L_F)\|<1/2.
\]
Fixing $\un{n}\perp F$ such that $\un{n}\geq\un{m}$, the fact that $\phi_\un{n}(\delta_{v})+\K(X_\un{n}(\La)\L_F)$ is a projection then implies that 
\[
\phi_\un{n}(\delta_{v})\in\K(X_\un{n}(\La)\L_F).
\]
By (\ref{Eq: comp}), we have $\phi_\un{n}(\delta_{v})\in\K(X_\un{n}(\La))$ and 
\[
\sca{X_\un{n}(\La),\phi_\un{n}(\delta_{v})X_\un{n}(\La)}\subseteq\L_F.
\]
The former implies that $|v\La^\un{n}|<\infty$ by item (ii) of Lemma \ref{L:hrcomp}, and the latter implies that $s(v\La^\un{n})\subseteq H_{\L,F}$ by (\ref{eq:n la}).
Hence the forward implication holds.

Now assume that there exists $\un{m}\perp F$ such that whenever $\un{n}\perp F$ and $\un{n}\geq\un{m}$, we have $s(v\La^\un{n})\subseteq H_{\L,F}$ and $|v\La^\un{n}|<\infty$.
Fix $\varepsilon>0$ and $\un{n}\perp F$ satisfying $\un{n}\geq\un{m}$.
By assumption we have $|v\La^\un{n}|<\infty$, and so an application of item (ii) of Lemma \ref{L:hrcomp} gives that 
\[
\phi_\un{n}(\delta_{v})\in\K(X_\un{n}(\La)).
\]
Likewise, the assumption that $s(v\La^\un{n})\subseteq H_{\L, F}$ and (\ref{eq:n la}) imply
\[
\sca{X_\un{n}(\La),\phi_\un{n}(\delta_{v})X_\un{n}(\La)}\subseteq\L_F.
\]
An application of (\ref{Eq: comp}) then gives
\[
\phi_\un{n}(\delta_{v})\in\K(X_\un{n}(\La)\L_F),
\]
and hence
\[
\|\phi_\un{n}(\delta_{v})+\K(X_\un{n}(\La)\L_F)\|=0<\varepsilon.
\]
Therefore $\lim_{\un{m}\perp F}\|\phi_\un{m}(\delta_{v})+\K(X_\un{m}(\La)\L_F)\|=0$, and thus $v\in H_{\L_{\lim,F}}$, finishing the proof.
\end{proof}

Next we present a construction of \cite{RSY03} for an arbitrary $k$-graph $(\La, d)$.
A subset $H$ of $\La^\un{0}$ is called \emph{hereditary in $\La$} if whenever $v\in H$ and $v\La w\neq\mt$ (where $w\in\La^\un{0}$), we have $w\in H$.
Due to duality, hereditarity is captured by positive invariance.

\begin{lemma}\label{L:herposinv}
Let $(\La,d)$ be a $k$-graph and let $H$ be a subset of $\La^\un{0}$.
Then $H$ is hereditary if and only if $I_H$ is positively invariant for $X(\La)$.
\end{lemma}

\begin{proof}
We note that for every $\la, \mu \in \La$ and a vertex $v$ we have
\[
\sca{\de_\la, \de_v \de_\mu} = \de_{\la, \mu} \de_{v, r(\mu)} \de_{s(\la)}.
\]
Therefore, if $H$ is hereditary and $\la, \mu \in \La^{\un{n}}$ for $\un{n}\in \bZ_+^d$ and $v \in H$, we obtain $\sca{\de_\la, \de_v \de_\mu} \in I_H$.
As the $\de_\la$ form the generators for $X_{\un{n}}(\La)$ and $\un{n}$ is arbitrary, we have that $I_H$ is positively invariant.

Conversely, if $I_H$ is positively invariant and we choose $v \in H$ and $\la \in \La^{\un{n}}$ with $r(\la) = v$, then we have $\de_{s(\la)} = \sca{\de_\la, \de_v \de_\la} \in I_H$, and thus $s(\la) \in H$ by duality.
This shows that $H$ is hereditary, and the proof is complete.
\end{proof}

For a hereditary subset $H$ of $\La^\un{0}$, we form a new countable small category $\Ga(\La\setminus H)$ by
\[
\Ga(\La\setminus H):=(\La^\un{0}\setminus H, \{\la\in\La\mid s(\la)\not\in H\}, r, s),
\]
i.e., the objects of $\Ga(\La\setminus H)$ are the vertices of $\La$ that do not belong to $H$, and the morphisms are those paths in $\La$ whose source does not belong to $H$.
Notice that the range and source maps are inherited from $\La$.
As $\Ga(\La\setminus H)$ is a subcategory of $\La$ and $H$ is hereditary, $\Ga(\La\setminus H)$ inherits a $k$-graph structure from $\La$ (e.g., see the proof of \cite[Theorem 5.2]{RSY03}).
Hence we can form the product systems $X(\Ga(\La\setminus H))$ and $[X(\La)]_{I_H}$, using Proposition \ref{P:qnt} and Lemma \ref{L:herposinv} for the latter.
To avoid confusion, we will denote the corresponding generators by $\delta_{\la,\La}$, for $\la\in\La$, and by $\delta_{\mu,\Ga}$, for $\mu\in\Ga(\La\setminus H)$.

\begin{proposition}\label{P:hrqnt}
Let $(\La,d)$ be a $k$-graph and let $H\subseteq\La^\un{0}$ be hereditary.
Then $X(\Ga(\La\setminus H))$ and $[X(\La)]_{I_H}$ are unitarily equivalent by the family of maps $\{W_\un{n}\}_{\un{n}\in\bZ_+^k}$ defined by
\[
W_\un{n}\colon X_\un{n}(\Ga(\La\setminus H))\to [X_\un{n}(\La)]_{I_H}; W_\un{n}(\delta_{\la,\Ga})=[\delta_{\la,\La}]_{I_H}\foral \la\in\Ga(\La\setminus H)^\un{n}, \un{n}\in\bZ_+^k.
\]
\end{proposition}

\begin{proof}
Since $\Ga(\La \setminus H)$ is a sub-$k$-graph of $\La$, each map $W_{\un{n}}$ is realised as the composition of canonical maps so that
\[
W_\un{n}\colon X_{\un{n}}(\Ga(\La \setminus H)) \hookrightarrow X_\un{n}(\La) \to [X_\un{n}(\La)]_{I_H},
\]
where the first map is isometric and the second is the quotient map.
This map is an isometry because of the duality $H \mapsto I_H$.
Moreover, $\de_{\la, \La} \in X_{\un{n}}(\La) I_H$ if and only if $s(\la) \in H$, giving that $W_{\un{n}}$ is invertible.
It is straightforward to check that the family $\{W_{\un{n}}\}_{\un{n} \in \bZ_+^k}$ satisfies the conditions of unitary equivalence.
\end{proof}

\begin{corollary}\label{C:Gafa}
Let $(\La,d)$ be a $k$-graph and let $H\subseteq\La^\un{0}$ be hereditary.
If $(\La,d)$ is (strong) finitely aligned, then so is $\Ga(\La\setminus H)$.
\end{corollary}

\begin{proof}
Since $X(\Ga(\La\setminus H)) \cong [X(\La)]_{I_H}$ by Proposition \ref{P:hrqnt}, the result follows by combining Proposition \ref{P:qnt ca} (resp. Proposition \ref{P:qnt sca}) with Remark \ref{R:unitca} (resp. Remark \ref{R:unitsca}).
\end{proof}

We now turn our attention to relative NO-$2^k$-tuples in the case of a strong finitely aligned $k$-graph $(\La,d)$.
Our first aim is to describe the NT-$2^k$-tuples of $X(\La)$ from a graph theoretic perspective, by translating items (i)-(iv) of Definition \ref{D:NT tuple} into properties on vertices.
We give the following lemma towards the characterisation of item (i) of Definition \ref{D:NT tuple}.

\begin{lemma}\label{L:hrNT1}
Let $(\La,d)$ be a strong finitely aligned $k$-graph and let $H \subseteq \La^{\un{0}}$ be hereditary in $\La$.
Then, fixing $\mt\neq F\subseteq[k]$, the vertex set associated with $J_F(I_H, X(\La))$ is given by
\[
H_{J_F(I_H,X(\La))}
=
H
\cup 
\{v \notin H \mid |v\Ga(\La\setminus H)^\un{i}|<\infty \; \forall i\in[k] \; \textup{and} \; v \; \textup{is not an $F$-source in $\Ga(\La\setminus H)$}\}.
\]
\end{lemma}

\begin{proof}
Fix $v\in\La^\un{0}$ and recall that $v\in H_{J_F(I_H,X(\La))}$ if and only if $\delta_{v,\La}\in J_F(I_H,X(\La))$.
Combining Lemma \ref{L:quo ideal} and Lemma \ref{L:herposinv}, we have 
\[
J_F(I_H,X(\La))=[ \hspace{1pt} \cdot \hspace{1pt} ]_{I_H}^{-1}(\J_F([X(\La)]_{I_H})).
\]
Thus $\delta_{v,\La}\in J_F(I_H,X(\La))$ if and only if $[\delta_{v,\La}]_{I_H}\in\J_F([X(\La)]_{I_H})$.
Consider the unitary equivalence 
\[
\{W_\un{n}\colon X_\un{n}(\Ga(\La\setminus H))\to [X_\un{n}(\La)]_{I_H}\}_{\un{n}\in\bZ_+^k}
\]
of Proposition \ref{P:hrqnt}.
Then we have 
\[
[\delta_{v,\La}]_{I_H}\in\J_F([X(\La)]_{I_H})
\]
if and only if 
\[
W_\un{0}^{-1}([\delta_{v,\La}]_{I_H})\in W_\un{0}^{-1}(\J_F([X(\La)]_{I_H}))=\J_F(X(\Ga(\La\setminus H))),
\]
using Remark \ref{R:JIunit} in the final equality.
Notice that if $v\not\in H$, then 
\[
W_\un{0}^{-1}([\delta_{v,\La}]_{I_H})=\delta_{v,\Ga}
\]
by definition.
In this case, we have $\delta_{v,\Ga}\in\J_F(X(\Ga(\La\setminus H)))$ if and only if $|v\Ga(\La\setminus H)^\un{i}|<\infty$ for all $i\in[k]$ and $v$ is not an $F$-source in $\Ga(\La\setminus H)$, by applying the comments succeeding Proposition \ref{P:hrrepnbij} to $\Ga(\La\setminus H)$.
Thus $v\in H_{J_F(I_H,X(\La))}$ if and only if either $v\in H$ or $v\not\in H$ but $|v\Ga(\La\setminus H)^\un{i}|<\infty$ for all $i\in[k]$ and $v$ is not an $F$-source in $\Ga(\La\setminus H)$, proving the result.
\end{proof}

Next we translate items (ii) and (iii) of Definition \ref{D:NT tuple}.

\begin{definition}\label{D:Fperpher}
Let $(\La,d)$ be a $k$-graph.
\begin{enumerate}
\item Given $F \subseteq [k]$, we say that $H\subseteq\La^\un{0}$ is \emph{$F^\perp$-hereditary in $\La$} if $s(H\La^\un{n}) \subseteq H$ for all $\un{n} \perp F$.

\item We say that a family $H:=\{H_F\}_{F\subseteq[k]}$ of subsets of $\La^\un{0}$ is \emph{hereditary in $\La$} if $H_F$ is $F^\perp$-hereditary in $\La$ for all $F\subseteq[k]$.
\end{enumerate}
\end{definition}

Notice that $H\subseteq\La^\un{0}$ is $\mt^\perp$-hereditary if and only if $H$ is hereditary in the usual sense.
For a strong finitely aligned $k$-graph $\La$, a $2^k$-tuple $\L$ of $X(\La)$ that consists of ideals is $X(\La)$-invariant if and only if the associated family $H_\L$ of sets of vertices is hereditary.

\begin{definition}\label{D:hrpo}
Let $(\La,d)$ be a $k$-graph and let $H:=\{H_F\}_{F\subseteq[k]}$ be a family of subsets of $\La^\un{0}$.
We say that $H$ is \emph{partially ordered} if $H_{F_1}\subseteq H_{F_2}$ whenever $F_1\subseteq F_2\subseteq[k]$.
\end{definition}

If $\La$ is a strong finitely aligned $k$-graph, then a $2^k$-tuple $\L$ of $X(\La)$ that consists of ideals is partially ordered if and only if the associated family $H_\L$ of sets of vertices is partially ordered.

Finally, to translate item (iv) of Definition \ref{D:NT tuple}, we will need the following definition.

\begin{definition}\label{D:satfam}
Let $(\La,d)$ be a strong finitely aligned $k$-graph.
Let $H:=\{H_F\}_{F\subseteq[k]}$ be a family of subsets of $\La^\un{0}$.
We say that $H$ is \emph{absorbent in $\La$} if the following holds for every $\mt\neq F\subsetneq[k]$: a vertex $v \in \La^{\un{0}}$ belongs to $H_F$ whenever it satisfies
\begin{enumerate}
\item $v$ is $F$-tracing,
\item $s(v\La^\un{m})\subseteq\cap_{F\subsetneq D}H_D\foral \un{m}\perp F$, and
\item there exists $\un{m} \perp F$ such that whenever $\un{n}\perp F$ and $\un{n}\geq\un{m}$, we have $s(v\La^\un{n})\subseteq H_F$ and $|v\La^\un{n}|<\infty$.
\end{enumerate}
\end{definition}

\begin{proposition}\label{P:hrNT}
Let $(\La,d)$ be a strong finitely aligned $k$-graph.
Let $\L$ be a $2^k$-tuple of $X(\La)$ that consists of ideals and let $H_\L$ be the corresponding family of sets of vertices of $\La$.
Then $\L$ is an NT-$2^k$-tuple of $X(\La)$ if and only if  the following four conditions hold:
\begin{enumerate}
\item for each $\mt\neq F\subseteq[k]$, we have
\[
H_{\L, F} 
\subseteq 
H_{\L, \mt} 
\cup 
\{v \notin H_{\L, \mt} \mid |v\Ga(\La\setminus H_{\L, \mt})^\un{i}|<\infty \; \forall i\in[k]\; \textup{and} \; v \; \textup{is not an $F$-source in $\Ga(\La\setminus H_{\L, \mt})$}\},
\]
\item $H_\L$ is hereditary in $\La$,
\item $H_\L$ is partially ordered,
\item $H_\L \setminus H_{\L, \mt} := \{H_{\L, F} \setminus H_{\L, \mt} \}_{F\subseteq[k]}$ is absorbent in $\Ga(\La\setminus H_{\L, \mt})$.
\end{enumerate}
\end{proposition}

\begin{proof}
We have already commented on the equivalence of invariance of $\L$ with item (ii), and of the partial ordering of $\L$ with item (iii).
By Lemma \ref{L:herposinv}, we have that $\L_\mt$ is positively invariant for $X(\La)$ if and only if $H_{\L, \mt}$ is hereditary.
In turn, we obtain the equivalence of item (i) with item (i) of Definition \ref{D:NT tuple} by Lemma \ref{L:hrNT1}.
To complete the proof, let $\L$ be a $2^k$-tuple of $X(\La)$ that satisfies items (i)-(iii) of Definition \ref{D:NT tuple}.
We will show that item (iv) of Definition \ref{D:NT tuple} is equivalent to item (iv) of the statement.

Since $H_{\L, \mt}$ is hereditary, we can form the product system $[X(\La)]_{\L_\mt}$ and the $k$-graph $\Ga:=\Ga(\La\setminus H_{\L, \mt})$.
We identify $X(\Ga)$ with $[X(\La)]_{\L_\mt}$ as in Proposition \ref{P:hrqnt}, and for $F \subseteq [k]$ we write
\[
\L_F^\Ga := \ol{\spn}\{\de_{v, \Ga} \mid v \in H_{\L, F} \setminus H_{\L, \mt} \} \subseteq c_0(\Ga^{\un{0}}),
\]
which corresponds to $[\L_F]_{\L_\mt}$ under the identification $X(\Ga)\cong [X(\La)]_{\L_\mt}$.
Now notice that
\[
[\hspace{1pt}\cdot\hspace{1pt}]_{\L_\mt}^{-1} ([\L_F]_{\L_\mt}^{(1)} ) \subseteq \L_F
\foral 
F \subseteq [k]
\]
(which holds automatically for $F = \mt$ and $F=[k]$) if and only if 
\[
(\L_F^\Ga)^{(1)} \subseteq \L_F^\Ga 
\foral 
\mt \neq F \subsetneq [k],
\]
which is in turn equivalent to having that
\begin{equation}\label{Eq:abscond}
\I_F(X(\Ga))\cap\L^\Ga_{\inv,F}\cap\L^\Ga_{\lim,F}\subseteq\L_F^\Ga\foral\mt\neq F\subsetneq[k],
\end{equation}
by definition.
We will show that (\ref{Eq:abscond}) holds exactly when $H_\L \setminus H_{\L, \mt}$ is absorbent in $\Ga$, thereby completing the proof.

To this end, fix $\mt \neq F \subsetneq [k]$ and $v\in\Ga^\un{0}$.
We have that $\delta_{v,\Ga}\in\I_F(X(\Ga))$ if and only if $v$ is $F$-tracing in $\Ga$.
Likewise, by (\ref{eq:hrinv}) we have that $\delta_{v,\Ga}\in\L^\Ga_{\inv,F}$ if and only if 
\[
s(v\Ga^\un{m})\subseteq\cap_{F\subsetneq D}(H_{\L,D} \setminus H_{\L, \mt})
\foral 
\un{m}\perp F.
\]
Finally, we have that $\delta_{v,\Ga}\in\L^\Ga_{\lim,F}$ if and only if there exists $\un{m}\perp F$ such that whenever $\un{n}\perp F$ and $\un{n}\geq\un{m}$, we have $s(v\Ga^\un{n})\subseteq H_{\L,F} \setminus H_{\L, \mt}$ and $|v\Ga^\un{n}|<\infty$ by Proposition \ref{P:hrlim}.
It follows that (\ref{Eq:abscond}) holds if and only if $H_\L\setminus H_{\L,\mt}$ is absorbent in $\Ga$, as required.
\end{proof}

In the row-finite case, the characterisation of Proposition \ref{P:hrNT} simplifies as follows.

\begin{proposition}\label{P:rfNT}
Let $(\La,d)$ be a row-finite $k$-graph.
Let $\L$ be a $2^k$-tuple of $X(\La)$ that consists of ideals and let $H_\L$ be the corresponding family of sets of vertices of $\La$.
Then $\L$ is an NT-$2^k$-tuple of $X(\La)$ if and only if the following four conditions hold:
\begin{enumerate}
\item for each $\mt\neq F\subseteq[k]$, we have
\[
H_{\L, F} 
\subseteq 
H_{\L, \mt} 
\cup 
\{v \notin H_{\L, \mt} \mid v \; \textup{is not an $F$-source in $\Ga:=\Ga(\La\setminus H_{\L, \mt})$}\},
\]
\item $H_\L$ is hereditary in $\La$,
\item $H_\L$ is partially ordered,
\item $H_{1,F}\cap H_{2,F}\cap H_{3,F}\subseteq H_{\L,F}$ for all $\mt\neq F\subsetneq[k]$, where
\vspace{.3em}
\begin{itemize}
\item $H_{1,F}:=\bigcap_{\un{n}\perp F}\{v\in\La^\un{0}\mid s(v\La^\un{n})\subseteq H_{\L, \mt}\cup\{v \notin H_{\L, \mt} \mid v \; \textup{is not an $F$-source in $\Ga$}\}\}$,
\item $H_{2,F}:=\bigcap_{\un{m}\perp F}\{v\in\La^\un{0}\mid s(v\La^\un{m})\subseteq\cap_{F\subsetneq D}H_{\L,D}\}$,
\item $H_{3,F}$ is the set of all $v\in\La^\un{0}$ for which there exists $\un{m}\perp F$ such that whenever $\un{n}\perp F$ and $\un{n}\geq\un{m}$, we have $s(v\La^\un{n})\subseteq H_{\L,F}$.
\end{itemize}
\end{enumerate}
\end{proposition}

\begin{proof}
Firstly, note that $\La$ is in particular strong finitely aligned, so we are free to use Proposition \ref{P:hrNT}.
Next, assuming that $H_{\L,\mt}$ is hereditary, we have that $\Ga$ inherits row-finiteness from $\La$ and therefore $|v\Ga^\un{i}|<\infty$ for all $v\in\Ga^\un{0}$ and $i\in[k]$ automatically. 
Consequently, items (i)-(iii) of the statement coincide with items (i)-(iii) of Proposition \ref{P:hrNT}, which are in turn equivalent to items (i)-(iii) of Definition \ref{D:NT tuple}.
Thus, without loss of generality, we may assume that $\L$ satisfies items (i)-(iii) of Definition \ref{D:NT tuple}.
Since $\phi_\un{n}(c_0(\La^\un{0}))\subseteq\K(X_\un{n}(\La))$ for all $\un{n}\in\bZ_+^k$ by item (i) of Lemma \ref{L:hrrfsprodsys}, by Proposition \ref{P:NTcom} it suffices to show that item (iv) of the statement is equivalent to the following condition:
\[
\bigg(\bigcap_{\un{n}\perp F}X_\un{n}(\La)^{-1}(J_F(\L_\mt,X(\La)))\bigg)\cap\L_{\inv,F}\cap\L_{\lim,F}\subseteq\L_F\foral\mt\neq F\subsetneq[k].
\]
To this end, fix $\mt\neq F\subsetneq[k]$.
The vertex set associated with $\bigcap_{\un{n}\perp F}X_\un{n}(\La)^{-1}(J_F(\L_\mt,X(\La)))$ is nothing but $H_{1,F}$, which can be seen by combining (\ref{eq:n la}) and Lemma \ref{L:hrNT1}.
Likewise, we have $H_{\L_{\inv,F}}=H_{2,F}$ by (\ref{eq:hrinv}).
Finally, we have $H_{\L_{\lim,F}}=H_{3,F}$ by Proposition \ref{P:hrlim}, noting that the stipulation that $|v\La^\un{n}|<\infty$ can be dropped by row-finiteness of $\La$.
The result now follows from the fact that the duality between ideals of $c_0(\La^\un{0})$ and subsets of $\La^\un{0}$ preserves inclusions and intersections.
\end{proof}

The characterisation of relative NO-$2^k$-tuples in the case of strong finite alignment (resp. row-finiteness) follows directly from Proposition \ref{P:hrNT} (resp. Proposition \ref{P:rfNT}), as inclusion of ideals corresponds to inclusion of their associated vertex sets.

\begin{corollary}\label{C:hrNO}
Let $(\La,d)$ be a strong finitely aligned (resp. row-finite) $k$-graph.
Let $\K$ be a relative $2^k$-tuple of $X(\La)$ that consists of ideals and let $H_\K$ be the corresponding family of sets of vertices of $\La$.
Let $\L$ be a $2^k$-tuple of $X(\La)$ that consists of ideals and let $H_\L$ be the corresponding family of sets of vertices of $\La$.
Then the following are equivalent: 
\begin{enumerate}
\item $\L$ is a $\K$-relative NO-$2^k$-tuple of $X(\La)$;
\item $H_\L$ satisfies \textup{(i)-(iv)} of Proposition \ref{P:hrNT} (resp. Proposition \ref{P:rfNT}) and $H_{\K,F}\subseteq H_{\L,F}$ for all $F\subseteq[k]$.
\end{enumerate}
In particular, the following are equivalent:
\begin{enumerate}
\item $\L$ is an NO-$2^k$-tuple of $X(\La)$;
\item $H_\L$ satisfies \textup{(i)-(iv)} of Proposition \ref{P:hrNT} (resp. Proposition \ref{P:rfNT}) and every $F$-tracing vertex of $\La$ belongs to $H_{\L,F}$ for all $\mt\neq F\subseteq[k]$.
\end{enumerate}
\end{corollary}

Finally, we turn our attention to the case of a locally convex row-finite $k$-graph $(\La,d)$.
In accordance with \cite{RSY03}, a subset $H$ of $\La^\un{0}$ is called \emph{saturated} if whenever a vertex $v\in\La^\un{0}$ satisfies $s(v\La^{\leq\un{i}})\subseteq H$ for some $i\in[k]$, we have $v\in H$.
Note that $\mt$ is vacuously saturated.
When $H\subseteq\La^\un{0}$ is both hereditary and saturated, the row-finite $k$-graph $\Ga(\La\setminus H)$ is also locally convex \cite[Theorem 5.2]{RSY03}.
The \emph{saturation} $\ol{H}^{\text{s}}$ of $H \subseteq \La^\un{0}$ is the smallest saturated subset of $\La^\un{0}$ that contains $H$.
The saturation of a hereditary set is also hereditary \cite[Lemma 5.1]{RSY03}.

\begin{proposition}\label{P:hersatlat} \cite{BPRS00, RSY03}
Let $(\La,d)$ be a locally convex row-finite $k$-graph.
Then the operations
\begin{align*}
H_1\vee H_2 & :=\ol{H_1\cup H_2}^{\text{s}} \qand
H_1\wedge H_2 :=H_1\cap H_2
\end{align*}
for hereditary saturated subsets $H_1, H_2 \subseteq \La^{\un{0}}$ define a lattice structure on the set of hereditary saturated subsets of $\La^\un{0}$.
\end{proposition}

Local convexity implies that $\J_F(X(\La))$ and $\I_F(X(\La))$ coincide for all $F\subseteq[k]$.

\begin{proposition}\label{P:lcideals}
Let $(\La,d)$ be a locally convex row-finite $k$-graph.
Then $\J_F(X(\La))=\I_F(X(\La))$ for all $F\subseteq[k]$.
\end{proposition}

\begin{proof}
The claim holds trivially when $F=\mt$, so fix $\mt\neq F\subseteq[k]$.
It suffices to show that $\J_F(X(\La))\subseteq\I_F(X(\La))$.
To this end, since $\La$ is row-finite this amounts to showing that if $v\in\La^\un{0}$ is not an $F$-source, then it is $F$-tracing.
Recalling the definition of $F$-tracing vertices, we proceed by induction on the length of the degree of the paths $\la \in v \La$ with $d(\la) \perp F$.

For $|d(\la)| = 0$, there is nothing to show (this accounts for $F=[k]$).
For $|d(\la)|=1$, we have $d(\la)=\un{i}$ for some $i\in F^c$.
Since $v$ is not an $F$-source, we can find $\mu \in v\La^\un{j}$ for some $j\in F$.
Since $i\neq j, \la\in v\La^\un{i}$ and $\mu\in v\La^\un{j}$, local convexity of $(\La,d)$ gives in particular that $s(\la)\La^\un{j} \neq \mt$, and thus $s(\la)$ is not an $F$-source, as required.

Now assume that $s(\la)$ is not an $F$-source for all $\la\in v\La$ satisfying $d(\la)\perp F$ and $|d(\la)|=N$ for some $N\in\bN$.
Fix $\la \in v\La$ such that $d(\la)\perp F$ and $|d(\la)|=N+1$.
Then $d(\la)=\un{n}+\un{i}$ for some $\un{n}\perp F$ satisfying $|\un{n}|=N$ and some $i\in F^c$.
The factorisation property gives unique paths $\mu,\nu\in\La$ such that 
\[
d(\mu)=\un{n}, d(\nu)=\un{i} \text{ and } \la=\mu\nu.
\]
Note that 
\[
v=r(\la)=r(\mu\nu)=r(\mu),
\]
so the inductive hypothesis implies that $s(\mu)\La^\un{j}\neq\mt$ for some $j\in F$.
We also have $\nu\in s(\mu)\La^\un{i}$ and $i\neq j$, so local convexity of $(\La,d)$ gives that $s(\nu)\La^\un{j}\neq\mt$.
In other words, $s(\nu)=s(\la)$ is not an $F$-source, as required.
By induction, the proof is complete.
\end{proof}

\begin{proposition}\label{P:lcrfCNP}
Let $(\La,d)$ be a locally convex row-finite $k$-graph and let $H\subseteq\La^\un{0}$ be hereditary and saturated.
Then
\[
\I_F([X(\La)]_{I_H}) = (\I_F(X(\La))+I_H)/I_H \foral F\subseteq[k],
\]
and consequently
\[
J_F(I_H,X(\La))=\I_F(X(\La))+I_H\foral \mt\neq F\subseteq[k].
\]
\end{proposition}

\begin{proof}
The first claim holds trivially when $F=\mt$, so fix $\mt\neq F\subseteq[k]$.
Recall that $\Ga:=\Ga(\La\setminus H)$ is a locally convex row-finite $k$-graph, and that $X(\Ga)$ and $[X(\La)]_{I_H}$ are unitarily equivalent via the family $\{W_\un{n}\colon X_\un{n}(\Ga(\La\setminus H))\to [X_\un{n}(\La)]_{I_H}\}_{\un{n}\in\bZ_+^k}$ of Proposition \ref{P:hrqnt}.
We have 
\[
W_\un{0}(\I_F(X(\Ga)))=\I_F([X(\La)]_{I_H})
\]
by Remark \ref{R:JIunit}.
Moreover, since $\La$ and $\Ga$ are locally convex and row-finite, Proposition \ref{P:lcideals} gives that 
\[
\I_F(X(\La))=\J_F(X(\La))
\qand
\I_F(X(\Ga))=\J_F(X(\Ga)).
\]
Hence it suffices to show that
\begin{equation}\label{Eq:hrJ}
(\J_F(X(\La))+I_H)/I_H=W_\un{0}(\J_F(X(\Ga))).
\end{equation}
Note that
\begin{equation}\label{Eq:hrJqn}
(\J_F(X(\La))+I_H)/I_H=\ol{\spn}\{[\delta_{v,\La}]_{I_H}\mid v\not\in H, v \; \text{is not an $F$-source in $\La$}\}.
\end{equation}
Thus, to prove the forward inclusion of (\ref{Eq:hrJ}), it suffices to show that $[\delta_{v,\La}]_{I_H}\in W_\un{0}(\J_F(X(\Ga)))$ whenever $v\not\in H$ and $v$ is not an $F$-source in $\La$.

Fix such a $v\in\La^\un{0}$ and note that $W_\un{0}^{-1}([\delta_{v,\La}]_{I_H})=\delta_{v,\Ga}$.
We claim that $v$ is not an $F$-source in $\Ga$.
Towards contradiction, suppose that $v\Ga^\un{i}=\mt$ for all $i\in F$.
Since $v$ is not an $F$-source in $\La$, there exists $i\in F$ such that $v\La^\un{i}\neq\mt$.
For each $\la\in v\La^\un{i}$, we must have $s(\la)\in H$, as otherwise we would obtain $v\Ga^\un{i}\neq\mt$.
Thus $s(v\La^{\leq \un{i}}) = s(v\La^\un{i}) \subseteq H$.
Since $H$ is saturated, we obtain the contradiction that $v\in H$, establishing the forward inclusion of (\ref{Eq:hrJ}).

For the reverse inclusion of (\ref{Eq:hrJ}), take $v\in\Ga^\un{0}$, i.e., $v\not\in H$, such that $v$ is not an $F$-source in $\Ga$.
In particular, $v\in\La^\un{0}$ is not an $F$-source in $\La$.
Hence 
\[
W_\un{0}(\delta_{v,\Ga}) = [\delta_{v,\La}]_{I_H} \in (\J_F(X(\La))+I_H)/I_H
\]
by (\ref{Eq:hrJqn}), giving (\ref{Eq:hrJ}).

The last claim follows by item (ii) of Lemma \ref{L:quo ideal} and the fact that $I_H \subseteq J_F(I_H,X(\La))$ for all $\mt \neq F \subseteq [k]$, as $I_H$ is positively invariant.
\end{proof}

\begin{proposition}\label{P:hrposneginv}
Let $(\La,d)$ be a locally convex row-finite $k$-graph and let $H$ be a subset of $\La^\un{0}$.
Then $H$ is hereditary and saturated if and only if $I_H$ is positively and negatively invariant for $X(\La)$.
\end{proposition}

\begin{proof}
Assume that $H$ is hereditary and saturated.
Then $I_H$ is positively invariant for $X(\La)$ by Lemma \ref{L:herposinv}.
Fix $\mt\neq F\subseteq[k]$.
By Proposition \ref{P:lcrfCNP}, we obtain
\[
\I_F(X(\La))\subseteq\I_F(X(\La))+I_H=J_F(I_H, X(\La)).
\]
Hence $I_H$ is negatively invariant for $X(\La)$ by Lemma \ref{L:neginvchar}, as required.

Now assume that $I_H$ is positively and negatively invariant for $X(\La)$.
We have that $H$ is hereditary by Lemma \ref{L:herposinv}, so it remains to check that $H$ is saturated.
Accordingly, fix $v\in\La^\un{0}$ and suppose that $s(v\La^{\leq\un{i}})\subseteq H$ for some $i\in[k]$.
We must show that $v\in H$.
This is clear when $v\La^\un{i}=\mt$, as in this case $v\La^{\leq\un{i}}=\{v\}$; so assume that $v\La^\un{i}\neq\mt$.
In this case we have $s(v\La^{\leq \un{i}}) = s(v\La^\un{i}) \subseteq H$, and thus $\delta_{v,\La}\in X(\La)_{\{i\}}^{-1}(I_H)$ by (\ref{eq:n la}).
Since $v\La^\un{i}\neq\mt$, we have that $v$ is not an $\{i\}$-source, and thus $v$ is $\{i\}$-tracing by Proposition \ref{P:lcideals}.
Hence
\[
\delta_{v,\La}\in\I_{\{i\}}(X(\La))\cap X(\La)_{\{i\}}^{-1}(I_H)\subseteq I_H,
\]
using negative invariance of $I_H$ in the final inclusion.
Consequently, we obtain $v\in H$ and hence $H$ is saturated, finishing the proof.
\end{proof}

\begin{proposition}\label{P:hrNObij}
Let $(\La,d)$ be a locally convex row-finite $k$-graph. 
Then the association
\[
I\mapsto \L_I, \; \text{where} \; \L_{I,F} := \I_F(X(\La))+I \foral F\subseteq[k],
\]
defines a bijection between the set of ideals of $c_0(\La^\un{0})$ that are positively and negatively invariant for $X(\La)$ and the set of NO-$2^k$-tuples of $X(\La)$, which in turn induces a bijection with the set of gauge-invariant ideals of $\N\O_{X(\La)}$.
\end{proposition}

\begin{proof}
By Proposition \ref{P:lcrfCNP} and Corollary \ref{C:NO source}, the map is well-defined and clearly injective.
For surjectivity, we must show that if $\L$ is an NO-$2^k$-tuple then necessarily $\L_F= \I_F(X(\La)) + \L_{\mt}$ for every $\mt\neq F \subseteq [k]$ (as the equality clearly holds when $F=\mt$).
We have that $\L_\mt$ is positively and negatively invariant for $X(\La)$ by Proposition \ref{P:NO can surj}.
In turn, we obtain
\[
(\I_F(X(\La))+\L_\mt)/\L_\mt\subseteq\L_F/\L_\mt\subseteq J_F(\L_\mt,X(\La))/\L_\mt=(\I_F(X(\La))+\L_\mt)/\L_\mt \foral\mt\neq F\subseteq[k],
\]
using Propositions \ref{P:lcrfCNP} and \ref{P:hrposneginv} in the final equality.
Thus $\L_F = \I_F(X(\La)) + \L_\mt$ for all $\mt\neq F\subseteq[k]$, showing that the map of the statement is bijective.
The last claim follows from Corollary \ref{C:CNP param}, completing the proof.
\end{proof}

To recover \cite[Theorem 5.2]{RSY03}, we need the following identification.
Let $H\subseteq\La^\un{0}$ be a hereditary vertex set.
Then
\[
\La(H):=(H,\{\la\in\La\mid r(\la)\in H\},r,s)
\]
is a locally convex row-finite sub-$k$-graph of $\La$, and $X(\La(H))$ and $I_HX(\La)I_H$ are unitarily equivalent via the family of maps $\{W_\un{n}\}_{\un{n}\in\bZ_+^k}$ defined by
\[
W_\un{n}\colon X_\un{n}(\La(H))\to I_HX_\un{n}(\La)I_H; \delta_{\la, \La(H)}\mapsto\delta_{\la,\La}\foral\la\in\La(H)^\un{n}, \un{n}\in\bZ_+^k.
\]
Consequently, by Remark \ref{R:noiso} and Proposition \ref{P:NOIXIembed} we obtain
\[
\ca(\La(H)) \cong \N\O_{X(\La(H))} \cong \N\O_{I_HX(\La)I_H} \cong I_H\ca(\La)I_H,
\]
where we identify $I_H$ with its faithful image inside $\ca(\La)$ in the final $*$-isomorphism.

\begin{corollary}\label{C:RSYparam1} \cite[Theorem 5.2]{RSY03}
Let $(\La,d)$ be a locally convex row-finite $k$-graph. 
Equip the set of hereditary saturated subsets of $\La^\un{0}$ with the lattice structure of Proposition \ref{P:hersatlat}, and equip the set of NO-$2^k$-tuples of $X(\La)$ with the lattice structure of Definition \ref{D:NTlattice} (suitably restricted).
Let $\{T_\la\}_{\la \in \La}$ be the universal Cuntz-Krieger $\La$-family.
Then the following hold:

\begin{enumerate}
\item The set of hereditary saturated subsets of $\La^\un{0}$ and the set of NO-$2^k$-tuples of $X(\La)$ are isomorphic as lattices via the map
\begin{equation}\label{Eq:hrlatiso}
H\mapsto\L_{I_H}, \; \text{where} \; \L_{I_H, F} := \I_F(X(\La)) + I_H \; \text{for all} \; F \subseteq [k].
\end{equation}
Consequently, the set of hereditary saturated subsets of $\La^\un{0}$ and the set of gauge-invariant ideals of $\ca(\La)$ are isomorphic as lattices.

\item Let $H$ be a hereditary and saturated subset of $\La^{\un{0}}$ and let $Q \colon \N\T_{X(\La)} \to \ca(\La)$ be the canonical quotient map.
Then the quotient $\ca(\La)/ Q(\fJ^{\L_{I_H}})$ is canonically $*$-isomorphic to the graph C*-algebra $\ca(\Ga(\La \setminus H))$.

\item Let $H$ be a hereditary subset of $\La^{\un{0}}$.
Then $\ca(\La(H))$ is canonically $*$-isomorphic to the C*-subalgebra $\ca(T_\la \mid r(\la) \in H)$ of $\ca(\La)$, and this C*-subalgebra is a full corner of the ideal generated by $\{T_v \mid v \in H\}$ in $\ca(\La)$.
If $H$ is in addition saturated, then the ideal generated by $\{T_v \mid v \in H\}$ in $\ca(\La)$ coincides with  $Q(\fJ^{\L_{I_H}})$, for the canonical quotient map $Q \colon \N\T_{X(\La)} \to \ca(\La)$.
\end{enumerate}
\end{corollary}

\begin{proof}
(i) Firstly, note that (\ref{Eq:hrlatiso}) is the composition of the duality map $H\mapsto I_H$ and the map $I\mapsto \L_I$ of Proposition \ref{P:hrNObij}, where $H\subseteq\La^\un{0}$ is hereditary and saturated, and $I$ is an ideal of $c_0(\La^\un{0})$ that is positively and negatively invariant for $X(\La)$.
These maps are bijections by Propositions \ref{P:hrposneginv} and \ref{P:hrNObij}, respectively, and hence the map  (\ref{Eq:hrlatiso}) is also a bijection.
Therefore, to prove the first claim, it suffices to show that (\ref{Eq:hrlatiso}) preserves the lattice structure.
To this end, fix hereditary saturated subsets $H_1$ and $H_2$ of $\La^\un{0}$.
We must show that
\[
\L_{I_{H_1}}\wedge\L_{I_{H_2}} = \L_{I_{H_1\wedge H_2}}
\qand
\L_{I_{H_1}}\vee\L_{I_{H_2}} = \L_{I_{H_1\vee H_2}}.
\]

For the operation $\wedge$, recall that $H_1 \wedge H_2 \equiv H_1 \cap H_2$ and $I_{H_1} I_{H_2} = I_{H_1 \cap H_2}$.
Hence for each $F \subseteq [k]$ we obtain
\begin{align*}
(\L_{I_{H_1}}\wedge\L_{I_{H_2}})_{F}
& =
\L_{I_{H_1}, F} \L_{I_{H_2}, F} 
=
(\I_F(X(\La))+I_{H_1})(\I_F(X(\La))+I_{H_2}) \\
& =
\I_F(X(\La))+I_{H_1} I_{H_2} 
= 
\I_F(X(\La))+I_{H_1 \cap H_2}
=
\L_{I_{H_1\wedge H_2}, F},
\end{align*}
by Proposition \ref{P:NTwedge}.
For the operation $\vee$, we must show that 
\[
(\L_{I_{H_1}}\vee\L_{I_{H_2}})_F = \L_{I_{\ol{H_1\cup H_2}^\text{s},F}}
\foral
F\subseteq [k].
\]
For $F=\mt$, we must show that 
\[
(\L_{I_{H_1}}\vee\L_{I_{H_2}})_\mt
\equiv
\ol{\pi}_{X(\La)}^{-1}(\fJ^{\L_{I_{H_1}}}+\fJ^{\L_{I_{H_2}}})
=
I_{\ol{H_1\cup H_2}^\text{s}},
\]
using Proposition \ref{P:NTvee} in the first identity.
We have $H_1 \subseteq \ol{H_1\cup H_2}^\text{s}$ and hence we obtain $I_{H_1}\subseteq I_{\ol{H_1\cup H_2}^\text{s}}$.
Thus 
\[
\L_{I_{H_1}}\subseteq\L_{I_{\ol{H_1\cup H_2}^\text{s}}},
\]
and so 
\[
\fJ^{\L_{I_{H_1}}} \subseteq\fJ^{\L_{I_{\ol{H_1\cup H_2}^\text{s}}}}
\]
using that the parametrisation of Theorem \ref{T:NT param} respects inclusions.
Likewise for $H_2$.
In turn, we obtain 
\[
\ol{\pi}_{X(\La)}^{-1}(\fJ^{\L_{I_{H_1}}}+\fJ^{\L_{I_{H_2}}})\subseteq\ol{\pi}_{X(\La)}^{-1}(\fJ^{\L_{I_{\ol{H_1\cup H_2}^\text{s}}}})=\L_\mt^{\fJ^{\L_{I_{\ol{H_1\cup H_2}^\text{s}}}}}=I_{\ol{H_1\cup H_2}^\text{s}}.
\]
Conversely, note that $(\L_{I_{H_1}}\vee\L_{I_{H_2}})_\mt$ is positively and negatively invariant for $X(\La)$ by Proposition \ref{P:NO can surj}, since it participates in the NO-$2^k$-tuple $\L_{I_{H_1}}\vee\L_{I_{H_2}}$.
An application of Proposition \ref{P:hrposneginv} then gives that $H_{(\L_{I_{H_1}}\vee\L_{I_{H_2}})_\mt}$ is hereditary and saturated.
Additionally, we have 
\[
I_{H_1\cup H_2}\subseteq (\L_{I_{H_1}}\vee\L_{I_{H_2}})_\mt
\]
by definition.
From this we deduce that 
\[
H_1\cup H_2\subseteq H_{(\L_{I_{H_1}}\vee\L_{I_{H_2}})_\mt}.
\]
Minimality of the saturation then implies that 
\[
\ol{H_1\cup H_2}^\text{s}\subseteq H_{(\L_{I_{H_1}}\vee\L_{I_{H_2}})_\mt}.
\]
Consequently we have 
\[
I_{\ol{H_1\cup H_2}^\text{s}}\subseteq(\L_{I_{H_1}}\vee\L_{I_{H_2}})_\mt,
\]
as required.

Now fix $\mt\neq F\subseteq[k]$.
Recall from the preceding argument that $\fJ^{\L_{I_{H_1}}}+\fJ^{\L_{I_{H_2}}}\subseteq\fJ^{\L_{I_{\ol{H_1\cup H_2}^\text{s}}}}$ and hence 
\[
\L_{I_{H_1}}\vee\L_{I_{H_2}}
\equiv
\L^{\fJ^{\L_{I_{H_1}}}+\fJ^{\L_{I_{H_2}}}}
\subseteq
\L^{\fJ^{\L_{I_{\ol{H_1\cup H_2}^\text{s}}}}}=\L_{I_{\ol{H_1\cup H_2}^\text{s}}}
\equiv
\L_{I_{H_1\vee H_2}},
\]
since the parametrisation of Theorem \ref{T:NT param} respects inclusions.
In particular, we have 
\[
(\L_{I_{H_1}}\vee\L_{I_{H_2}})_F\subseteq \L_{I_{H_1\vee H_2},F}.
\]
For the reverse inclusion, first recall that
\[
(\L_{I_{H_1}}\vee\L_{I_{H_2}})_F
=
[\hspace{1pt}\cdot\hspace{1pt}]_{(\L_{I_{H_1}}\vee\L_{I_{H_2}})_\mt}^{-1}[((\L_{I_{H_1},F}+\L_{I_{H_2},F}+(\L_{I_{H_1}}\vee\L_{I_{H_2}})_\mt)/(\L_{I_{H_1}}\vee\L_{I_{H_2}})_\mt)^{(k-1)}],
\]
by Proposition \ref{P:NTvee}.
Observe that
\begin{align*}
\I_F(X(\La)) + (\L_{I_{H_1}}\vee\L_{I_{H_2}})_\mt
&\subseteq
\L_{I_{H_1},F} + (\L_{I_{H_1}}\vee\L_{I_{H_2}})_\mt \\
& =
\I_F(X(\La))+I_{H_1}+(\L_{I_{H_1}}\vee\L_{I_{H_2}})_\mt \\
& \subseteq
\I_F(X(\La))+(\L_{I_{H_1}}\vee\L_{I_{H_2}})_\mt,
\end{align*}
using that $I_{H_1}\subseteq(\L_{I_{H_1}}\vee\L_{I_{H_2}})_\mt$ in the final inclusion.
Likewise for $H_2$.
Therefore, we have
\[
(\L_{I_{H_1}}\vee\L_{I_{H_2}})_F=[\hspace{1pt}\cdot\hspace{1pt}]_{(\L_{I_{H_1}}\vee\L_{I_{H_2}})_\mt}^{-1}[((\I_F(X(\La))+(\L_{I_{H_1}}\vee\L_{I_{H_2}})_\mt)/(\L_{I_{H_1}}\vee\L_{I_{H_2}})_\mt)^{(k-1)}].
\]
We then obtain
\begin{align*}
[\L_{I_{H_1\vee H_2},F}]_{(\L_{I_{H_1}}\vee\L_{I_{H_2}})_\mt} 
& =
[\I_F(X(\La))+I_{\ol{H_1\cup H_2}^\text{s}}]_{(\L_{I_{H_1}}\vee\L_{I_{H_2}})_\mt} \\
& =
[\I_F(X(\La))+(\L_{I_{H_1}}\vee\L_{I_{H_2}})_\mt]_{(\L_{I_{H_1}}\vee\L_{I_{H_2}})_\mt} \\
& \subseteq
\left([\I_F(X(\La))+(\L_{I_{H_1}}\vee\L_{I_{H_2}})_\mt]_{(\L_{I_{H_1}}\vee\L_{I_{H_2}})_\mt}\right)^{(k-1)},
\end{align*}
using that $I_{\ol{H_1\cup H_2}^\text{s}} = (\L_{I_{H_1}}\vee\L_{I_{H_2}})_\mt$ in the second equality.
Thus 
\[
\L_{I_{H_1\vee H_2},F}\subseteq (\L_{I_{H_1}}\vee\L_{I_{H_2}})_F,
\]
as required.
The second claim now follows by an application of Corollary \ref{C:CNP param}.

\smallskip

\noindent 
(ii) Let $H\subseteq\La^\un{0}$ be hereditary and saturated. 
Then 
\[
\I_F([X(\La)]_{I_H})=(\I_F(X(\La))+I_H)/I_H \foral F\subseteq[k]
\]
by Proposition \ref{P:lcrfCNP}.
Hence, applying items (ii) and (iii) of Corollary \ref{C:NO source} and adopting the notation therein, we obtain
\begin{equation*}
\N\O_{X(\La)}/\sca{Q_{\I}(\ol{\pi}_{X(\La)}(I_H))}\cong\N\O_{[X(\La)]_{I_H}}\cong\N\O_{X(\Ga(\La\setminus H))},
\end{equation*}
using Remark \ref{R:noiso} and Proposition \ref{P:hrqnt} in the final $*$-isomorphism.
Note that item (iii) of Corollary \ref{C:NO source} applies due to row-finiteness of $(\La,d)$.
Item (ii) of the statement follows.

\smallskip

\noindent
(iii) The first statement follows by the comments preceding the corollary, using Proposition \ref{P:NOIXIembed} for fullness.
For the corner property, note that the generators of $c_0(\La^{\un{0}})$ form a countable approximate unit of projections.
The second statement is an immediate consequence of item (iii) of Corollary \ref{C:NO source}.
\end{proof}

\subsection{Finite frames}\label{Ss:finfram}

In this subsection we apply the NT-$2^d$-tuple machinery to the case of a product system $X$ over $\bZ_+^d$ in which $X_\un{i}$ admits a finite frame for all $i\in[d]$.
In \cite{Kak20}, the second-named author used the quotients of $\N\T_X$ by the ideals $\sca{\ol{\pi}_X(A) \ol{q}_{X, \un{i}} \mid i \in F}$ for $\mt \neq F \subseteq [d]$ in order to examine the structure of subinvariant KMS-states.
A note was made in \cite{Kak20} that such quotients can be realised as the Cuntz-Nica-Pimsner algebra of an induced product system.
We will see here how this is achieved from the parametrisation that we have established.

\begin{definition}\label{D:finfram}
Let $X$ be a right Hilbert module over a C*-algebra $A$.
A finite non-empty subset $\{\xi^{(j)}\}_{j\in [N]}$ of $X$ is said to be a \emph{finite frame of $X$} if $\sum_{j=1}^N\Theta_{\xi^{(j)},\xi^{(j)}}=\id_X$.
When such a subset exists, we say that $X$ \emph{admits a finite frame}.
\end{definition}

Observe that a right Hilbert C*-module $X$ admits a finite frame if and only if $\id_X\in\K(X)$, which in turn holds if and only if $\K(X)=\L(X)$.
Hence $X$ is projective and finitely generated.

\begin{lemma}\label{L:tensfinfram}
Let $X$ and $Y$ be C*-correspondences over a C*-algebra $A$.
Suppose that $Y$ admits a finite frame $\{y^{(j)}\}_{j\in[N]}$.
Then
\[
\Theta_{x_1,x_2}^X \otimes \id_Y=\sum_{j=1}^N\Theta_{x_1\otimes y^{(j)}, x_2\otimes y^{(j)}}^{X \otimes_A Y}, \foral x_1,x_2\in X.
\]
If in addition $X$ admits a finite frame $\{x^{(i)}\}_{i\in[M]}$, then $X \otimes_A Y$ admits the finite frame 
\[
\{x^{(i)} \otimes y^{(j)} \mid i \in [M], j \in [N]\}.
\]
\end{lemma}

\begin{proof}
This follows by a direct computation on elementary tensors.
\end{proof}

Admission of a finite frame is preserved when we pass to quotients, since $[\id_X]_I=\id_{[X]_I}$.

\begin{lemma}\label{L:finframquot}
Let $X$ be a right Hilbert module over a C*-algebra $A$ and let $I\subseteq A$ be an ideal.
If $X$ admits a finite frame $\{\xi^{(j)}\}_{j\in[N]}$, then $[X]_I$ admits the finite frame $\{[\xi^{(j)}]_I\}_{j\in[N]}$.
\end{lemma}

Passing to product systems, we have the following lemma.

\begin{lemma}\label{L:fibfinfram}
Let $X$ be a product system over $\bZ_+^d$ with coefficients in a C*-algebra $A$.
Then $X_\un{i}$ admits a finite frame for all $i\in [d]$ if and only if $X_\un{n}$ admits a finite frame for all $\un{n}\in\bZ_+^d\setminus\{\un{0}\}$.
\end{lemma}

\begin{proof}
The forward implication follows by inducting on $|\un{n}|$ and using Lemma \ref{L:tensfinfram}.
The converse is immediate.
\end{proof}

We do not assume that $X_\un{0}=A$ admits a finite frame, as this would force $A$ to be unital, and the results of this subsection hold even when $A$ is non-unital. 
Notice that $X$ is automatically strong compactly aligned when every $X_{\un{i}}$ admits a finite frame by Corollary \ref{C:autscaps}.

\begin{corollary}\label{C:iotafram}
Let $X$ be a product system over $\bZ_+^d$ with coefficients in a C*-algebra $A$, wherein $X_\un{i}$ admits a finite frame for all $i\in[d]$.
Fix $\un{n},\un{m}\in\bZ_+^d\setminus\{\un{0}\}$ and let $\{\xi_{\un{m}}^{(j)}\}_{j\in[N_\un{m}]}$ be a finite frame of $X_\un{m}$.
Then we have
\[
\iota_\un{n}^{\un{n}+\un{m}}(\Theta_{\xi_\un{n},\eta_\un{n}}^{X_{\un{n}}})=\sum_{j=1}^{N_\un{m}}\Theta_{\xi_\un{n}\xi_{\un{m}}^{(j)},\eta_\un{n}\xi_{\un{m}}^{(j)}}^{X_{\un{n} + \un{m}}},
\foral
\xi_\un{n},\eta_\un{n}\in X_\un{n}.
\]
In particular, if $(\pi,t)$ is a representation of $X$ then
\[
\psi_{\un{n}+\un{m}}(\iota_\un{n}^{\un{n}+\un{m}}(\Theta_{\xi_\un{n},\eta_\un{n}}^{X_{\un{n}}}))=\sum_{j=1}^{N_\un{m}}t_\un{n}(\xi_\un{n})t_\un{m}(\xi_{\un{m}}^{(j)})t_\un{m}(\xi_{\un{m}}^{(j)})^*t_\un{n}(\eta_\un{n})^*,
\foral
\xi_\un{n},\eta_\un{n}\in X_\un{n}.
\]
\end{corollary}

\begin{proof}
First note that the finite frame of $X_\un{m}$ exists by Lemma \ref{L:fibfinfram}.
The proof now follows by a direct computation using Lemma \ref{L:tensfinfram} and the properties of $(\pi,t$).
\end{proof}

\begin{remark}\label{R:finframq}
Let $X$ be a product system over $\bZ_+^d$ with coefficients in a C*-algebra $A$.
Suppose that every $X_{\un{i}}$ admits a finite frame $\{\xi_{\un{i}}^{(j)}\}_{j\in[N_\un{i}]}$.
If $(\pi,t)$ is a Nica-covariant representation of $X$, then we have
\[
p_{\un{i}}
=
\psi_{\un{i}}(\id_{X_\un{i}})
=
\psi_{\un{i}}(\sum_{j=1}^{N_\un{i}}\Theta_{\xi_{\un{i}}^{(j)},\xi_{\un{i}}^{(j)}})
=
\sum_{j=1}^{N_\un{i}} t_{\un{i}}(\xi_{\un{i}}^{(j)}) t_{\un{i}}(\xi_{\un{i}}^{(j)})^*
\in \ca(\pi,t),
\foral i \in [d].
\]
Since the left action of each fibre of $X$ is by compacts, we have $\pi(A)q_F\subseteq\ca(\pi,t)$ for all $F\subseteq[d]$ by Proposition \ref{P:prod cai}.
However, we may still have $q_F\notin\ca(\pi,t)$ for $F\subseteq[d]$ (unless $\ca(\pi,t)$ is unital).
\end{remark}

\begin{proposition}\label{P:itoF}
Let $X$ be a product system over $\bZ_+^d$ with coefficients in a C*-algebra $A$, wherein each left action is by compact operators.
Let $(\pi,t)$ be a Nica-covariant representation of $X$ and fix $\mt\neq F\subseteq[d]$.
Then for each $\mt\neq D\subseteq F$, we have
\[
\pi(A)q_D \subseteq \sca{\pi(A)q_\un{i}\mid i\in F}\subseteq\ca(\pi,t).
\]
\end{proposition}

\begin{proof}
Without loss of generality, assume that $D = [m]$ and $F = [n]$ for $m \leq n$.
By the Hewitt-Cohen Factorisation Theorem, for $a \in A$ there exist $b_1, \dots, b_m \in A$ such that $a = b_1 \cdots b_m$.
In turn, we obtain
\[
\pi(a) q_D = (\pi(b_1) \cdots \pi(b_m))(q_{\un{1}} \cdots q_{\un{m}}) = (\pi(b_1)q_{\un{1}}) \cdots (\pi(b_m) q_{\un{m}}) \in \sca{\pi(A)q_\un{i}\mid i\in F},
\]
using that $q_{\un{i}} \in \pi(A)'$ for all $i\in D$ by Proposition \ref{P:pf reducing}.
This finishes the proof.
\end{proof}

We now pass to the study of the quotient $\N\T_X/\sca{\ol{\pi}_X(A)\ol{q}_{X,\un{i}}\mid i\in F}$ for $\mt \neq F \subseteq [d]$.
We first turn our attention to the case of $F = [d]$, where the quotient corresponds to the finite part of the Wold decomposition of a KMS-state in the context of \cite{Kak20}.

\begin{proposition}\label{P:alldir}
Let $X$ be a product system over $\bZ_+^d$ with coefficients in a C*-algebra $A$, wherein $X_\un{i}$ admits a finite frame for all $i\in[d]$.
Consider the gauge-invariant ideal
\[
\fJ := \sca{\ol{\pi}_X(A)\ol{q}_{X,\un{i}}\mid i\in [d]}\subseteq\N\T_X.
\]
Then
\[
\L^{\fJ}_\mt = \{a \in A \mid \lim_{\un{n} \in \bZ_+^d} \|\phi_{\un{n}}(a)\| = 0\}
\qand
\L^{\fJ}_F = A \foral \mt \neq F \subseteq [d].
\]
Moreover, the product system $[X]_{\L^{\fJ}_\mt}$ is regular, and thus there is a canonical $*$-isomorphism
\[
\N\T_X / \fJ \cong \N\O_{[X]_{\L^{\fJ}_\mt}}.
\]
If $X$ is in addition injective, then $\L_\mt^\fJ = \{0\}$ and thus $\N\T_X/\fJ = \N\O_X$.
\end{proposition}

\begin{proof}
Let $Q_\fJ \colon \N\T_X \to \N\T_X/\fJ$ be the canonical quotient map.
The fact that 
\[
\L^{\fJ}_\mt \equiv \ker Q_\fJ \circ \ol{\pi}_X
=
\{a \in A \mid \lim_{\un{n} \in \bZ_+^d} \|\phi_{\un{n}}(a)\| = 0\}
\]
follows from \cite[Proposition 4.3]{Kak20}.
Next, fix $\mt \neq F \subseteq [d]$ and $a \in A$.
By Proposition \ref{P:itoF}, we have $\ol{\pi}_X(a) \ol{q}_{X,F} \in \fJ$.
Thus $A \subseteq \L^\fJ_F$, as required.

Next we show that $[X]_{\L_\mt^\fJ}$ is regular.
By Proposition \ref{P:inj+comp}, and since the left actions of the fibres of $X$ (and hence of $[X]_{\L_\mt^\fJ}$) are by compacts, it suffices to show that $[\phi_\un{i}]_{\L_\mt^\fJ}$ is injective for all $i\in[d]$.
To this end, fix $i \in [d]$, a finite frame $\{\xi_{\un{i}}^{(j)}\}_{j\in[N]}$ of $X_\un{i}$ and $[a]_{\L_\mt^\fJ} \in \ker [\phi_\un{i}]_{\L_\mt^\fJ}$.
We have $[\phi_\un{i}(a)\xi_{\un{i}}^{(j)}]_{\L_\mt^\fJ} = 0$ for every $j \in [N]$, and hence $\phi_\un{i}(a)\xi_{\un{i}}^{(j)}\in X_\un{i} \L_\mt^\fJ$ for all $j\in[N]$.
For notational convenience, we write 
\[
\pi_\fJ := Q_\fJ \circ \ol{\pi}_X \qand t_{\fJ,\un{n}} := Q_\fJ \circ \ol{t}_{X,\un{n}} \foral\un{n}\in\bZ_+^d\setminus\{0\}.
\] 
Note that $(\pi_\fJ,t_\fJ)$ is a Nica-covariant representation of $X$.
Thus we have
\[
\pi_\fJ(a)t_{\fJ,\un{i}}(\xi_{\un{i}}^{(j)})=t_{\fJ,\un{i}}(\phi_\un{i}(a)\xi_{\un{i}}^{(j)})\in t_{\fJ,\un{i}}(X_\un{i}\L_\mt^\fJ)=t_{\fJ,\un{i}}(X_\un{i})\pi_\fJ(\L_\mt^\fJ)=\{0\},
\foral j \in [N].
\]
In turn, we obtain $\pi_\fJ(a)t_{\fJ,\un{i}}(\xi_{\un{i}}^{(j)})t_{\fJ,\un{i}}(\xi_{\un{i}}^{(j)})^*=0$ for all $j\in[N]$, and therefore
\[
\pi_\fJ(a)p_{\fJ,\un{i}}=\pi_\fJ(a)\sum_{j=1}^{N}t_{\fJ,\un{i}}(\xi_{\un{i}}^{(j)})t_{\fJ,\un{i}}(\xi_{\un{i}}^{(j)})^*=0.
\]
We also have $\ol{\pi}_X(a)\ol{q}_{X,\un{i}}\in\mathfrak{J}$ by definition, and hence $\pi_\fJ(a)q_{\fJ,\un{i}} = Q_{\fJ}(\ol{\pi}_X(a)\ol{q}_{X,\un{i}}) = 0$.
Consequently, we obtain
\[
\pi_\fJ(a)=\pi_\fJ(a)q_{\fJ,\un{i}}+\pi_\fJ(a)p_{\fJ,\un{i}}=0.
\]
Hence $a\in\L_\mt^\fJ$ and thus $[a]_{\L_\mt^\fJ}=0$, proving that $[\phi_\un{i}]_{\L_\mt^\fJ}$ is injective, as required.

Regularity of $[X]_{\L_\mt^\fJ}$ implies that $\I([X]_{\L_\mt^\fJ}) = [\L^\fJ]_{\L_\mt^\fJ}$.
Therefore, we conclude that
\[
\N\T_X / \fJ \cong \N\O([\L^\fJ]_{\L_\mt^\fJ}, [X]_{\L_\mt^\fJ}) = \N\O_{[X]_{\L_\mt^\fJ}},
\]
using item (ii) of Proposition \ref{P:tuple ideal} to establish the canonical $*$-isomorphism.

If $X$ is injective, then every $\phi_{\un{n}}$ is isometric and hence $\L_\mt^\fJ = \{0\}$, finishing the proof.
\end{proof}

Next we move to the characterisation of the quotient of $\N\T_X$ by $\sca{\ol{\pi}_X(A)\ol{q}_{X,\un{i}}\mid i\in F}$, for some fixed $\mt\neq F\subsetneq[d]$, as the Cuntz-Nica-Pimsner algebra of an appropriate quotient of a new product system $Y^F$ over $\bZ_+^{|F|}$.
The key is that every non-trivial fibre of $Y^F$ inherits a finite frame, and $\N\T_X\cong\N\T_{Y^F}$ in a canonical way.

For ease of notation, we will assume that $F=[r]$ for some $r < d$. 
This is sufficient, as the ensuing arguments can be adapted to account for general $\mt\neq F\subsetneq[d]$ by relabelling the elements.
For each $\un{n}=(n_1,\dots,n_d)\in\bZ_+^d$, we define
\begin{align*}
\un{n}_F:= & (n_1,\dots,n_r)\in\bZ_+^r \qand
\un{n}_{F^\perp}:= (n_{r+1},\dots,n_d)\in\bZ_+^{d-r}.
\end{align*}
Therefore we can canonically express every $\un{n}\in\bZ_+^d$ as $\un{n}=(\un{n}_F,\un{n}_{F^\perp})$.
Conversely, given $\un{k}\in\bZ_+^r$ and $\un{\ell}\in\bZ_+^{d-r}$, we can form $\un{n} :=(\un{k},\un{\ell}) \in \bZ_+^d$ so that $\un{n}_F=\un{k}$ and $\un{n}_{F^\perp}=\un{\ell}$.

For the remainder of the subsection, we will take $X$ to be compactly aligned and identify $\N\T_X$ with $\ca(\ol{\pi},\ol{t})$ for the Fock representation $(\ol{\pi},\ol{t})$.
We use $F$ to split the gauge action $\be$ of $(\ol{\pi},\ol{t})$ into two parts.
More specifically, we define families $\be_F$ and $\be_{F^\perp}$ via
\[
\be_F:=\{\be_{(\un{x},\un{1}_{[d-r]})}\}_{\un{x}\in\bT^r}\qand \be_{F^\perp}:=\{\be_{(\un{1}_{[r]},\un{y})}\}_{\un{y}\in\bT^{d-r}}.
\]
Note that $\be_F$ and $\be_{F^\perp}$ are point-norm continuous families of $*$-automorphisms of $\N\T_X$, being restrictions of $\be$.
For $\un{n}\in\bZ_+^d$ and $\xi_\un{n}\in X_\un{n}$, we have
\[
\be_{F,\un{x}}(\ol{t}_\un{n}(\xi_\un{n}))=\be_{(\un{x},\un{1}_{[d-r]})}(\ol{t}_\un{n}(\xi_\un{n}))=\un{x}^{\un{n}_F}\ol{t}_\un{n}(\xi_\un{n})\foral\un{x}\in\bT^r,
\]
and
\[
\be_{F^\perp,\un{y}}(\ol{t}_\un{n}(\xi_\un{n}))=\be_{(\un{1}_{[r]},\un{y})}(\ol{t}_\un{n}(\xi_\un{n}))=\un{y}^{\un{n}_{F^\perp}}\ol{t}_\un{n}(\xi_\un{n}) \foral\un{y}\in\bT^{d-r}.
\]
In particular, we have $\be_{F,\un{x}}(\ol{\pi}(a)) = \ol{\pi}(a) = \be_{F^\perp,\un{y}}(\ol{\pi}(a))$ for all $\un{x}\in\bT^r,\un{y}\in\bT^{d-r}$ and $a\in A$.

\begin{definition}\label{D:YF}
Let $X$ be a compactly aligned product system over $\bZ_+^d$ with coefficients in a C*-algebra $A$ and let $F=[r]$ for some $r < d$.
We define the C*-algebra $B^{F^\perp}$ by
\[
B^{F^\perp}:=\ca(\ol{\pi}(A),\ol{t}_\un{i}(X_\un{i})\mid i\in F^c)\subseteq\N\T_X,
\]
and a collection of linear spaces $Y^{F}:=\{Y_\un{n}^{F}\}_{\un{n}\in\bZ_+^r}$ by
\begin{align*}
Y_\un{0}^{F} & :=B^{F^\perp}
\qand
Y_\un{n}^{F} :=[\ol{t}_{(\un{n},\un{0})}(X_{(\un{n},\un{0})})B^{F^\perp}]\subseteq\N\T_X\foral\un{n}\in\bZ_+^r\setminus\{\un{0}\}.
\end{align*}
\end{definition}

Since $\ol{\pi}(A) \subseteq B^{F^\perp}$, by using an approximate unit of $A$ we obtain $\ol{t}_{(\un{n},\un{0})}(X_{(\un{n}, \un{0})}) \subseteq Y_{\un{n}}^F$ for all $\un{n} \in \bZ_+^r$.
If $\un{n}, \un{m} \in \bZ_+^d$ have support contained in $F^c$, then the same is true of $\un{n} \vee \un{m}$.
Hence Nica covariance yields
\begin{equation}\label{Eq:Bchar}
B^{F^\perp}=\ca(\ol{\pi}(A),\ol{t}_\un{n}(X_\un{n})\mid \un{n}\perp F)=\ol{\spn}\{\ol{t}_\un{n}(X_\un{n})\ol{t}_\un{m}(X_\un{m})^*\mid \un{n},\un{m}\perp F\}.
\end{equation}
Note that $\be_{F}|_ {B^{F^\perp}} = \id$.
We can restrict $\be_{F^\perp}$ to $B^{F^\perp}$ to obtain a point-norm continuous family
\[
\be_{F^\perp}|_{B^{F^\perp}}:=\{\be_{F^\perp,\un{z}}|_{B^{F^\perp}}\}_{\un{z}\in\bT^{d-r}}
\]
of $*$-automorphisms of $B^{F^\perp}$.
We may characterise $B^{F^\perp}$ as the Toeplitz-Nica-Pimsner algebra of a compactly aligned product system over $\bZ_+^{d-r}$.

\begin{definition}\label{D:ZFperp}
Let $X$ be a product system over $\bZ_+^d$ with coefficients in a C*-algebra $A$ and let $F=[r]$ for some $r < d$.
We define the collection of C*-correspondences $Z^{F^\perp}:=\{Z_\un{n}^{F^\perp}\}_{\un{n}\in\bZ_+^{d-r}}$ over $A$ by
\begin{align*}
Z_\un{0}^{F^\perp} := A
\qand
Z_\un{n}^{F^\perp} := X_{(\un{0},\un{n})} \foral \un{n}\in\bZ_+^{d-r} \setminus \{\un{0}\}.
\end{align*}
\end{definition}

We can endow $Z^{F^\perp}$ with a canonical product system structure, inherited from $X$.

\begin{lemma}\label{L:Zprodsys}
Let $X$ be a product system over $\bZ_+^d$ with coefficients in a C*-algebra $A$ and let $F=[r]$ for some $r < d$.
Then $Z^{F^\perp}$ inherits a product system structure from $X$.
Moreover, if every non-trivial fibre of $X$ admits a finite frame (resp. $X$ is compactly aligned, strong compactly aligned), then every non-trivial fibre of $Z^{F^\perp}$ admits a finite frame (resp. $Z^{F^\perp}$ is compactly aligned, strong compactly aligned).
\end{lemma}

\begin{proof}
This is immediate when seeing $Z^{F^\perp}$ as a subfamily of $X$.
\end{proof}

\begin{proposition}\label{P:NTZcongB}
Let $X$ be a compactly aligned product system over $\bZ_+^d$ with coefficients in a C*-algebra $A$ and let $F=[r]$ for some $r < d$.
Then the maps
\begin{align*}
\pi & \colon A\to B^{F^\perp}; \pi(a)=\ol{\pi}(a) \foral a\in A, \\
t_\un{n} & \colon Z_\un{n}^{F^\perp} \to B^{F^\perp}; t_\un{n}(z_\un{n})=\ol{t}_{(\un{0},\un{n})}(z_\un{n}) \foral z_\un{n}\in Z_\un{n}^{F^\perp}, \un{n}\in\bZ_+^{d-r}\setminus\{\un{0}\},
\end{align*}
form an injective Nica-covariant representation $(\pi,t)$ of $Z^{F^\perp}$ which induces a canonical $*$-isomorphism $\N\T_{Z^{F^\perp}}\cong B^{F^\perp}$.
\end{proposition}

\begin{proof}
For notational convenience, we set $B:=B^{F^\perp}$ and $Z:=Z^{F^\perp}$.
The maps are well-defined by (\ref{Eq:Bchar}).
It is routine to check that $(\pi,t)$ constitutes an injective representation of $Z$ with $\ca(\pi,t) = B$.
Since $Z$ inherits the product system structure of $X$ and $(\ol{\pi},\ol{t})$ is Nica-covariant, it follows that $(\pi,t)$ is Nica-covariant.
By the universal property of $\N\T_{Z}$, there exists a canonical $*$-epimorphism $\pi \times t \colon \N\T_{Z} \to B$.
We will show that $\pi \times t$ is injective, thereby completing the proof.

Without loss of generality, we may take $(\ol{\pi}_Z,\ol{t}_Z)$ to be the Fock representation of $Z$.
Note that $\be_{F^\perp}|_B$ defines a gauge action of $(\pi,t)$.
Let $\ga$ denote the gauge action of $(\ol{\pi}_{Z},\ol{t}_{Z})$.
Then $\pi \times t$ is equivariant, and so it suffices to show that the restriction
\[
(\pi \times t)|_{\N\T_{Z}^\ga}\colon \N\T_{Z}^\ga\to B^{\be_{F^\perp}|_{B}}
\]
to the fixed point algebras is injective.
We have that $\sum_{\un{n}\perp F}X_\un{n}$ is reducing for $B^{\be_{F^\perp}|_{B}}$ by construction, and thus the map
\[
\Psi\colon B^{\be_{F^\perp}|_B}\to\L(\sum_{\un{n}\perp F}X_\un{n}); \Psi(b)=b|_{\sum_{\un{n}\perp F}X_\un{n}}\foral b\in B^{\be_{F^\perp}|_B},
\]
is a well-defined $*$-homomorphism.
By identifying $\F Z \cong \sum_{\un{n}\perp F}X_\un{n}$, we obtain that $\Psi$ is a left inverse of $(\pi \times t)|_{\N\T_{Z}^\ga}$, as required.
\end{proof}

We isolate the following corollary from the proof of Proposition \ref{P:NTZcongB}.

\begin{corollary}\label{C:psiinj}
Let $X$ be a compactly aligned product system over $\bZ_+^d$ with coefficients in a C*-algebra $A$ and let $F=[r]$ for some $r < d$.
Then the map
\[
\Psi\colon (B^{F^\perp})^{\be_{F^\perp}|_{B^{F^\perp}}}\to\L(\sum_{\un{n}\perp F}X_\un{n}); \Psi(b)=b|_{\sum_{\un{n}\perp F}X_\un{n}}\foral b\in (B^{F^\perp})^{\be_{F^\perp}|_{B^{F^\perp}}}
\]
is an injective $*$-homomorphism.
In particular, we have
\[
\|b|_{\sum_{\un{n}\perp F}X_\un{n}}\|=\|b\| \foral b\in (B^{F^\perp})^{\be_{F^\perp}|_{B^{F^\perp}}}.
\]
\end{corollary}

Next we focus on $Y^F$ and explore the connection with $X$.

\begin{lemma}\label{L:Yc*-cor}
Let $X$ be a compactly aligned product system over $\bZ_+^d$ with coefficients in a C*-algebra $A$ and let $F=[r]$ for some $r < d$.
Then $Y_\un{n}^F$ is a sub-C*-correspondence of $\N\T_X$ over $B^{F^\perp}$ for all $\un{n}\in\bZ_+^r$.
\end{lemma}

\begin{proof}
For notational convenience, we set $B:=B^{F^\perp}$ and $Y:=Y^F$.
We view $Y_\un{0}$ as a C*-correspondence over itself in the usual way.
For $\un{n}\in\bZ_+^r\setminus\{\un{0}\}$, the space $Y_\un{n}$ inherits the usual inner product from $\N\T_X$ (now taking values in $B$) as well as the usual right multiplication in $\N\T_X$ (now by elements of $B$).
A direct computation gives that the norm induced by the $B$-valued inner product coincides with the restriction norm on $\N\T_X$, and thus $Y_{\un{n}}$ is a right Hilbert $B$-module.

It remains to equip $Y_\un{n}$ with a left action $\phi_{Y_\un{n}}\colon B\to\L(Y_\un{n})$.
To this end, fix $b\in B$ and define
\[
\phi_{Y_\un{n}}(b)\colon Y_\un{n}\to Y_\un{n}; \phi_{Y_\un{n}}(b)y_\un{n}=by_\un{n} \foral y_\un{n}\in Y_\un{n},
\]
where the multiplication is performed in $\N\T_X$.
To see that this map is well-defined, fix $\un{k},\un{\ell}\in \bZ_+^d$ satisfying $\un{k}, \un{\ell}\perp F$.
Thus we may write $\un{k}=(\un{0},\un{k}_{F^\perp})$ and $\un{\ell}=(\un{0},\un{\ell}_{F^\perp})$.
Since $(\un{0},\un{\ell}_{F^\perp}) \perp (\un{n},\un{0})$, Nica-covariance yields
\begin{align*}
\ol{t}_{(\un{0},\un{k}_{F^\perp})}(X_{(\un{0},\un{k}_{F^\perp})}) 
\ol{t}_{(\un{0},\un{\ell}_{F^\perp})}(X_{(\un{0},\un{\ell}_{F^\perp})})^*
\ol{t}_{(\un{n},\un{0})}(X_{(\un{n},\un{0})})B 
& \subseteq \\
& \hspace{-4cm} \subseteq
[\ol{t}_{(\un{n},\un{0})}(X_{(\un{n},\un{0})})\ol{t}_{(\un{0},\un{k}_{F^\perp})}(X_{(\un{0},\un{k}_{F^\perp})})\ol{t}_{(\un{0},\un{\ell}_{F^\perp})}(X_{(\un{0},\un{\ell}_{F^\perp})})^*B] \\
& \hspace{-4cm} \subseteq
[\ol{t}_{(\un{n},\un{0})}(X_{(\un{n},\un{0})})BB] 
\subseteq
[\ol{t}_{(\un{n},\un{0})}(X_{(\un{n},\un{0})})B] \equiv Y_\un{n},
\end{align*}
using (\ref{Eq:Bchar}) to pass to the last line.
It then follows by another application of (\ref{Eq:Bchar}) that $BY_\un{n}\subseteq Y_\un{n}$, and hence $\phi_{Y_\un{n}}(b)$ is well-defined.
It is also clear that $\phi_{Y_\un{n}}(b)\in \L(Y_\un{n})$ with adjoint $\phi_{Y_\un{n}}(b^*)$, and so $\phi_{Y_\un{n}}$ is well-defined.
It is routine to check that $\phi_{Y_\un{n}}$ is a $*$-homomorphism, and the proof is complete.
\end{proof}

\begin{lemma}\label{L:Yprodsys}
Let $X$ be a compactly aligned product system over $\bZ_+^d$ with coefficients in a C*-algebra $A$ and let $F=[r]$ for some $r < d$.
Then $Y^F$ carries a canonical structure as a product system over $\bZ_+^r$ with coefficients in $B^{F^\perp}$, with the multiplication maps given by multiplication in $\N\T_X$.

Moreover, if $\{\xi_{(\un{n}, \un{0})}^{(j)}\}_{j\in[N]}$ is a finite frame of $X_{(\un{n}, \un{0})}$ for $\un{n} \in \bZ_+^r \setminus \{\un{0}\}$, then $\{\ol{t}_{(\un{n}, \un{0})}(\xi_{(\un{n}, \un{0})}^{(j)})\}_{j\in[N]}$ is a finite frame of $Y^F_{\un{n}}$.
\end{lemma}

\begin{proof}
For notational convenience, we set $B:=B^{F^\perp}$ and $Y:=Y^F$.
Lemma \ref{L:Yc*-cor} gives that the fibres of $Y$ are C*-correspondences over $B$, so it remains to construct multiplication maps 
\[
v_{\un{n},\un{m}}\colon Y_\un{n}\otimes_{B} Y_\un{m}\to Y_{\un{n}+\un{m}} \foral \un{n},\un{m}\in\bZ_+^r
\]
that are compatible with axioms (i)-(v) for product systems.

Axiom (i) is satisfied by construction, and axioms (ii) and (iii) determine the maps $v_{\un{0},\un{n}}$ and $v_{\un{n},\un{0}}$, respectively, for all $\un{n}\in\bZ_+^r$.
For $\un{n},\un{m}\in\bZ_+^r\setminus\{\un{0}\}$, we start by defining a map
\[
v_{\un{n},\un{m}}\colon Y_\un{n}\times Y_\un{m}\to Y_{\un{n}+\un{m}}; (y_\un{n},y_\un{m})\mapsto y_\un{n}y_\un{m}\foral y_\un{n}\in Y_\un{n} \; \text{and} \; y_\un{m}\in Y_\un{m},
\]
where the multiplication is performed in $\N\T_X$.
To see that $v_{\un{n},\un{m}}$ is well-defined, observe that
\begin{align*}
\ol{t}_{(\un{n},\un{0})}(X_{(\un{n},\un{0})})B\ol{t}_{(\un{m},\un{0})}(X_{(\un{m},\un{0})})B 
& \subseteq
[\ol{t}_{(\un{n},\un{0})}(X_{(\un{n},\un{0})}) \ol{t}_{(\un{m},\un{0})}(X_{(\un{m},\un{0})})B] \\
&\subseteq
[\ol{t}_{(\un{n}+\un{m},\un{0})}(X_{(\un{n}+\un{m},\un{0})})B] 
\equiv
Y_{\un{n}+\un{m}},
\end{align*}
using that $BY_\un{m}\subseteq Y_\un{m}$ by Lemma \ref{L:Yc*-cor}.
This shows that $Y_\un{n}Y_\un{m}\subseteq Y_{\un{n}+\un{m}}$ and hence $v_{\un{n},\un{m}}$ is well-defined.
It is routine to check that $v_{\un{n},\un{m}}$ is bilinear and $B$-balanced, and therefore it linearises to a map on $Y_\un{n}\odot_{B} Y_\un{m}$, denoted by the same symbol.
For all $y_\un{n}, y_\un{n}'\in Y_\un{n}$ and $y_\un{m},y_\un{m}'\in Y_\un{m}$, we have
\begin{align*}
\sca{v_{\un{n},\un{m}}(y_\un{n}\otimes y_\un{m}),v_{\un{n},\un{m}}(y_\un{n}'\otimes y_\un{m}')} 
& =
y_\un{m}^*y_\un{n}^*y_\un{n}'y_\un{m}' 
=
\sca{y_\un{n}\otimes y_\un{m},y_\un{n}'\otimes y_\un{m}'},
\end{align*}
from which it follows that $\sca{v_{\un{n},\un{m}}(\zeta),v_{\un{n},\un{m}}(\zeta')}=\sca{\zeta,\zeta'}$ for all $\zeta, \zeta' \in Y_\un{n}\odot_{B} Y_\un{m}$. 
In particular, we have that $v_{\un{n},\un{m}}$ is bounded and hence it extends to $Y_\un{n}\otimes_{B} Y_\un{m}$.
A direct computation gives that $v_{\un{n},\un{m}}$ is a $B$-bimodule map, and the preceding calculation shows that it is isometric.
To see that $v_{\un{n},\un{m}}$ is surjective, fix $\xi_{(\un{n},\un{0})}\in X_{(\un{n},\un{0})}, \xi_{(\un{m},\un{0})}\in X_{(\un{m},\un{0})}, b\in B$ and an approximate unit $(u_\la)_{\la\in\La}$ of $A$.
We have that
\begin{align*}
\ol{t}_{(\un{n}+\un{m},\un{0})}(\xi_{(\un{n},\un{0})}\xi_{(\un{m},\un{0})})b 
& =
\nor{\cdot}\text{-}\lim_\la \ol{t}_{(\un{n},\un{0})}(\xi_{(\un{n},\un{0})})\ol{\pi}(u_\la)\ol{t}_{(\un{m},\un{0})}(\xi_{(\un{m},\un{0})})b \\
& =
\nor{\cdot}\text{-}\lim_\la v_{\un{n},\un{m}}(\ol{t}_{(\un{n},\un{0})}(\xi_{(\un{n},\un{0})})\ol{\pi}(u_\la)\otimes\ol{t}_{(\un{m},\un{0})}(\xi_{(\un{m},\un{0})})b).
\end{align*}
Hence $\ol{t}_{(\un{n}+\un{m},\un{0})}(\xi_{(\un{n},\un{0})}\xi_{(\un{m},\un{0})})b$ can be realised as the norm-limit of a net that is contained in the closed linear space $v_{\un{n},\un{m}}(Y_\un{n}\otimes_{B} Y_\un{m})$.
It follows that
\[
\ol{t}_{(\un{n}+\un{m},\un{0})}(X_{(\un{n}+\un{m},\un{0})})B\subseteq v_{\un{n},\un{m}}(Y_\un{n}\otimes_{B} Y_\un{m}),
\]
and hence $v_{\un{n},\un{m}}$ is surjective.
Associativity of the multiplication maps follows by associativity of multiplication in $\N\T_X$.
In total, we have that $Y$ is a product system over $\bZ_+^r$ with coefficients in $B$, as required.

Finally, let $\un{n} \in \bZ_+^r \setminus \{\un{0}\}$ and let $\{\xi_{(\un{n}, \un{0})}^{(j)}\}_{j\in[N]}$ be a finite frame of $X_{(\un{n}, \un{0})}$.
For $y_\un{n}=\ol{t}_{(\un{n},\un{0})}(\xi_{(\un{n},\un{0})})b$, where $\xi_{(\un{n},\un{0})}\in X_{(\un{n},\un{0})}$ and $b\in B$, we have
\begin{align*}
\sum_{j=1}^{N}\Theta_{\ol{t}_{(\un{n},\un{0})}(\xi_{(\un{n},\un{0})}^{(j)}),\ol{t}_{(\un{n},\un{0})}(\xi_{(\un{n},\un{0})}^{(j)})}^{Y_{\un{n}}}   (y_\un{n}) 
& =
\sum_{j=1}^{N}\ol{t}_{(\un{n},\un{0})}(\xi_{(\un{n},\un{0})}^{(j)})\ol{t}_{(\un{n},\un{0})}(\xi_{(\un{n},\un{0})}^{(j)})^*\ol{t}_{(\un{n},\un{0})}(\xi_{(\un{n},\un{0})})b \\
& =
\ol{t}_{(\un{n},\un{0})}(\sum_{j=1}^{N}\Theta_{\xi_{(\un{n},\un{0})}^{(j)},\xi_{(\un{n},\un{0})}^{(j)}}^{X_{(\un{n},\un{0})}}   (\xi_{(\un{n},\un{0})}))b 														
=
\ol{t}_{(\un{n},\un{0})}(\xi_{(\un{n},\un{0})})b 														     =
y_\un{n},
\end{align*}
as $\{\xi_{(\un{n},\un{0})}^{(j)}\}_{j\in[N]}$ is a finite frame of $X_{(\un{n},\un{0})}$.
Since $\ol{t}_{(\un{n},\un{0})}(X_{(\un{n},\un{0})})B$ densely spans $Y_\un{n}$, it follows that
\[
\sum_{j=1}^{N}\Theta_{\ol{t}_{(\un{n},\un{0})}(\xi_{(\un{n},\un{0})}^{(j)}),\ol{t}_{(\un{n},\un{0})}(\xi_{(\un{n},\un{0})}^{(j)})}^{Y_{\un{n}}} = \id_{Y_\un{n}},
\]
and the proof is complete.
\end{proof}

Next we show that $\N\T_X\cong\N\T_{Y^F}$ when each $X_\un{i}$ admits a finite frame.
Note that each $Y_\un{i}^F$ admits a finite frame by Lemma \ref{L:Yprodsys}.
We start by representing $Y^F$ in $\N\T_X$.

\begin{proposition}\label{P:YonNTX}
Let $X$ be a product system over $\bZ_+^d$ with coefficients in a C*-algebra $A$, wherein $X_\un{i}$ admits a finite frame for all $i\in[d]$.
Let $F=[r]$ for some $r < d$.
Define the maps
\begin{align*}
\pi & \colon B^{F^\perp}\to\N\T_X; \pi(b)=b \foral b\in B^{F^\perp}, \\
t_\un{n} & \colon Y_\un{n}^F\to\N\T_X; t_\un{n}(y_\un{n})=y_\un{n} \foral y_\un{n}\in Y_\un{n}^F, \un{n}\in\bZ_+^r\setminus\{\un{0}\}.
\end{align*}
Then $(\pi,t)$ is an injective Nica-covariant representation of $Y^F$ satisfying $\ca(\pi,t)=\N\T_X$.
\end{proposition}

\begin{proof}
For notational convenience, we set $B:=B^{F^\perp}$ and $Y:=Y^F$.
For $\un{s}\in\bZ_+^d$, let $P_\un{s}$ denote the projection of $\F X$ onto $\sum_{\un{k}\geq\un{s}}X_\un{k}$.
Then $P_\un{s}\ol{t}_\un{s}(\zeta_\un{s})=\ol{t}_\un{s}(\zeta_\un{s})$ for all $\zeta_\un{s}\in X_\un{s}$.

It is routine to check that $(\pi,t)$ defines an injective representation of the product system $Y$.
To see that $\ca(\pi,t)=\N\T_X$, it suffices to show that $\ca(\pi,t)$ contains the generators of $\N\T_X$.
By (\ref{Eq:Bchar}) we have $\ol{t}_\un{n}(X_\un{n})\subseteq\ca(\pi,t)$ for all $\un{n} \perp F$.
It remains to see that 
\[
\ol{t}_{(\un{k},\un{\ell})}(X_{(\un{k},\un{\ell})})\subseteq\ca(\pi,t)
\foral
\un{k}\in\bZ_+^r\setminus\{\un{0}\}, \un{\ell}\in\bZ_+^{d-r}.
\]
Fixing $\xi_{(\un{k},\un{0})}\in X_{(\un{k},\un{0})}$ and $\xi_{(\un{0},\un{\ell})}\in X_{(\un{0},\un{\ell})}$, we have
\[
\ol{t}_{(\un{k},\un{\ell})}(\xi_{(\un{k},\un{0})}\xi_{(\un{0},\un{\ell})})=\ol{t}_{(\un{k},\un{0})}(\xi_{(\un{k},\un{0})})\ol{t}_{(\un{0},\un{\ell})}(\xi_{(\un{0},\un{\ell})})=t_\un{k}(\ol{t}_{(\un{k},\un{0})}(\xi_{(\un{k},\un{0})}))\pi(\ol{t}_{(\un{0},\un{\ell})}(\xi_{(\un{0},\un{\ell})}))\in\ca(\pi,t).
\]
Hence we have
\[
\ol{t}_{(\un{k},\un{\ell})}(X_{(\un{k},\un{\ell})})\subseteq[\ol{t}_{(\un{k},\un{0})}(X_{(\un{k},\un{0})})\ol{t}_{(\un{0},\un{\ell})}(X_{(\un{0},\un{\ell})})]\subseteq\ca(\pi,t),
\]
as required.

Next we show that $(\pi,t)$ is Nica-covariant.
Let $\{\iota_{\un{n}}^{\un{n} + \un{m}}\}_{\un{n}, \un{m} \in \bZ_+^r}$ denote the connecting $*$-homomorphisms of $Y$.
Fix $\un{n},\un{m}\in\bZ_+^r \setminus \{\un{0}\}, k_\un{n}\in\K(Y_\un{n})$ and $k_\un{m}\in\K(Y_\un{m})$.
We have
\[
\psi_{\un{n}\vee\un{m}}(\iota_\un{n}^{\un{n}\vee\un{m}}(k_\un{n})\iota_\un{m}^{\un{n}\vee\un{m}}(k_\un{m}))=\psi_{\un{n}\vee\un{m}}(\iota_\un{n}^{\un{n}\vee\un{m}}(k_\un{n}))\psi_{\un{n}\vee\un{m}}(\iota_\un{m}^{\un{n}\vee\un{m}}(k_\un{m}))
\]
by Proposition \ref{P:autcaps}.
Therefore, we must show that
\begin{equation}\label{Eq:YNC}
\psi_{\un{n}\vee\un{m}}(\iota_\un{n}^{\un{n}\vee\un{m}}(k_\un{n}))\psi_{\un{n}\vee\un{m}}(\iota_\un{m}^{\un{n}\vee\un{m}}(k_\un{m}))
=
\psi_\un{n}(k_\un{n})\psi_\un{m}(k_\un{m}).
\end{equation}
This holds trivially when $\un{n} = \un{n}\vee\un{m} = \un{m}$, so we consider the cases where $\un{n} \neq \un{n}\vee\un{m}$ or $\un{m} \neq \un{n}\vee\un{m}$.

Assume that $\un{n} \neq \un{n}\vee\un{m}$.
We proceed by showing that (\ref{Eq:YNC}) holds for $k_\un{n}=\Theta_{y_\un{n},y_\un{n}'}^{Y_{\un{n}}}$ and $k_\un{m}=\Theta_{y_\un{m},y_\un{m}'}^{Y_{\un{m}}}$, where
\begin{align*}
y_\un{n} & =\ol{t}_{(\un{n},\un{0})}(\xi_{(\un{n},\un{0})})b, 
y_\un{n}'  =\ol{t}_{(\un{n},\un{0})}(\eta_{(\un{n},\un{0})}), 
y_\un{m} =\ol{t}_{(\un{m},\un{0})}(\xi_{(\un{m},\un{0})})c,
y_\un{m}' =\ol{t}_{(\un{m},\un{0})}(\eta_{(\un{m},\un{0})}),
\end{align*}
for some $\xi_{(\un{n},\un{0})},\eta_{(\un{n},\un{0})}\in X_{(\un{n},\un{0})}, \xi_{(\un{m},\un{0})},\eta_{(\un{m},\un{0})}\in X_{(\un{m},\un{0})}$ and $b,c\in B$.
To this end, we compute $\psi_{\un{n}\vee\un{m}}(\iota_\un{n}^{\un{n}\vee\un{m}}(\Theta_{y_\un{n},y_\un{n}'}^{Y_{\un{n}}}))$ and $\psi_{\un{n}\vee\un{m}}(\iota_\un{m}^{\un{n}\vee\un{m}}(\Theta_{y_\un{m},y_\un{m}'}^{Y_{\un{m}}}))$.

For $\psi_{\un{n}\vee\un{m}}(\iota_\un{n}^{\un{n}\vee\un{m}}(\Theta_{y_\un{n},y_\un{n}'}^{Y_{\un{n}}}))$, let $\{\xi^{(j)}_{(\un{n}\vee\un{m}-\un{n},\un{0})}\}_{j\in[N]}$ be a finite frame of $X_{(\un{n}\vee\un{m}-\un{n},\un{0})}$.
By Lemma \ref{L:Yprodsys}, we have $\{\ol{t}_{(\un{n}\vee\un{m}-\un{n},\un{0})}(\xi^{(j)}_{(\un{n}\vee\un{m}-\un{n},\un{0})})\}_{j\in[N]}$ is a finite frame of $Y_{\un{n}\vee\un{m}-\un{n}}$.
Note that
\[
P_{({\un{n}\vee\un{m}-\un{n}},\un{0})}
=
\sum_{j=1}^{N} \ol{t}_{(\un{n}\vee\un{m}-\un{n},\un{0})}(\xi^{(j)}_{(\un{n}\vee\un{m}-\un{n},\un{0})})\ol{t}_{(\un{n}\vee\un{m}-\un{n},\un{0})}(\xi^{(j)}_{(\un{n}\vee\un{m}-\un{n},\un{0})})^*\in\N\T_X.
\]
Observe that $P_{({\un{n}\vee\un{m}-\un{n}},\un{0})}$ belongs to the commutant of $\ol{t}_{(\un{0},\un{\ell})}(X_{(\un{0},\un{\ell})})$ (and of $\ol{t}_{(\un{0},\un{\ell})}(X_{(\un{0},\un{\ell})})^*$ by taking adjoints) for all $\un{\ell}\in\bZ_+^{d-r}$, since for $\un{s}\in\bZ_+^d$ we have
\[
(\un{n}\vee\un{m}-\un{n},\un{0})\leq\un{s}\iff (\un{n}\vee\un{m}-\un{n},\un{0}) \leq \un{s} + (\un{0},\un{\ell}) \foral  \un{\ell}\in\bZ_+^{d-r}.
\]
In particular, we have that $P_{(\un{n}\vee\un{m}-\un{n},\un{0})}$ belongs to $B'$.
Additionally, we have
\[
\ol{t}_{(\un{n},\un{0})}(\zeta_{(\un{n},\un{0})})P_{(\un{n}\vee\un{m}-\un{n},\un{0})}=P_{(\un{n}\vee\un{m},\un{0})}\ol{t}_{(\un{n},\un{0})}(\zeta_{(\un{n},\un{0})}),\foral\zeta_{(\un{n},\un{0})}\in X_{(\un{n},\un{0})},
\]
and by taking adjoints we obtain
\[
P_{(\un{n}\vee\un{m}-\un{n},\un{0})} \ol{t}_{(\un{n},\un{0})}(\zeta_{(\un{n},\un{0})})^* 
= 
\ol{t}_{(\un{n},\un{0})}(\zeta_{(\un{n},\un{0})})^* P_{(\un{n}\vee\un{m},\un{0})}, \foral \zeta_{(\un{n},\un{0})}\in X_{(\un{n},\un{0})}.
\]
This follows from the observation that for $\un{s}\in\bZ_+^d$ we have
\[
(\un{n}\vee\un{m}-\un{n},\un{0})\leq\un{s}\iff (\un{n}\vee\un{m},\un{0})\leq(\un{n},\un{0}) + \un{s}.
\]
An application of Corollary \ref{C:iotafram} gives that
\begin{align*}
\psi_{\un{n}\vee\un{m}}(\iota_\un{n}^{\un{n}\vee\un{m}}(\Theta_{y_\un{n},y_\un{n}'}^{Y_{\un{n}}})) 
& =
\sum_{j=1}^{N} y_\un{n}  \ol{t}_{(\un{n}\vee\un{m}-\un{n},\un{0})}(\xi^{(j)}_{(\un{n}\vee\un{m}-\un{n},\un{0})})\ol{t}_{(\un{n}\vee\un{m}-\un{n},\un{0})}(\xi^{(j)}_{(\un{n}\vee\un{m}-\un{n},\un{0})})^* (y_\un{n}')^* \\
& =
\ol{t}_{(\un{n},\un{0})}(\xi_{(\un{n},\un{0})})bP_{({\un{n}\vee\un{m}-\un{n}},\un{0})}\ol{t}_{(\un{n},\un{0})}(\eta_{(\un{n},\un{0})})^* \\
& = 
\ol{t}_{(\un{n},\un{0})}(\xi_{(\un{n},\un{0})})b\ol{t}_{(\un{n},\un{0})}(\eta_{(\un{n},\un{0})})^*P_{(\un{n}\vee\un{m},\un{0})} \\
& = 
P_{(\un{n}\vee\un{m},\un{0})}\ol{t}_{(\un{n},\un{0})}(\xi_{(\un{n},\un{0})})b\ol{t}_{(\un{n},\un{0})}(\eta_{(\un{n},\un{0})})^*.
\end{align*}

For $\psi_{\un{n}\vee\un{m}}(\iota_\un{m}^{\un{n}\vee\un{m}}(\Theta_{y_\un{m},y_\un{m}'}^{Y_{\un{m}}}))$, we consider two cases.
If $\un{m} = \un{n}\vee\un{m}$, then we have
\[
\psi_{\un{n}\vee\un{m}}(\iota_\un{m}^{\un{n}\vee\un{m}}(\Theta_{y_\un{m},y_\un{m}'}^{Y_{\un{m}}}))
=
\psi_{\un{m}}(\Theta_{y_\un{m},y_\un{m}'}^{Y_{\un{m}}})
=
\ol{t}_{(\un{m},\un{0})}(\xi_{(\un{m},\un{0})})c\ol{t}_{(\un{m},\un{0})}(\eta_{(\un{m},\un{0})})^*,
\]
and if $\un{m} \neq \un{n}\vee\un{m}$ then we have
\begin{align*}
\psi_{\un{n}\vee\un{m}}(\iota_\un{m}^{\un{n}\vee\un{m}}(\Theta_{y_\un{m},y_\un{m}'}^{Y_{\un{m}}})) 
& =
P_{(\un{n}\vee\un{m},\un{0})}\ol{t}_{(\un{m},\un{0})}(\xi_{(\un{m},\un{0})})c \ol{t}_{(\un{m},\un{0})}(\eta_{(\un{m},\un{0})})^*
\end{align*}
by swapping $\un{n}$ and $\un{m}$, as well as $b$ and $c$, in the preceding arguments.

Since $P_{(\un{n}\vee\un{m},\un{0})}=P_{(\un{n},\un{0})}P_{(\un{m},\un{0})}$, we conclude that
\begin{align*}
\psi_{\un{n}\vee\un{m}}(\iota_\un{n}^{\un{n}\vee\un{m}}(\Theta_{y_\un{n},y_\un{n}'}^{Y_{\un{n}}})) 
& 
\psi_{\un{n}\vee\un{m}}(\iota_\un{m}^{\un{n}\vee\un{m}}(\Theta_{y_\un{m},y_\un{m}'}^{Y_{\un{m}}})) = \\
& =
\ol{t}_{(\un{n},\un{0})}(\xi_{(\un{n},\un{0})})b\ol{t}_{(\un{n},\un{0})}(\eta_{(\un{n},\un{0})})^*P_{(\un{n}\vee\un{m},\un{0})}\ol{t}_{(\un{m},\un{0})}(\xi_{(\un{m},\un{0})})c \ol{t}_{(\un{m},\un{0})}(\eta_{(\un{m},\un{0})})^* \\
& =
\ol{t}_{(\un{n},\un{0})}(\xi_{(\un{n},\un{0})}) b \bigg[\ol{t}_{(\un{n},\un{0})}(\eta_{(\un{n},\un{0})})^*P_{(\un{n},\un{0})} \bigg] \bigg[ P_{(\un{m},\un{0})}\ol{t}_{(\un{m},\un{0})}(\xi_{(\un{m},\un{0})}) \bigg] c \ol{t}_{(\un{m},\un{0})}(\eta_{(\un{m},\un{0})})^* \\
& =
\ol{t}_{(\un{n},\un{0})}(\xi_{(\un{n},\un{0})})b\ol{t}_{(\un{n},\un{0})}(\eta_{(\un{n},\un{0})})^*\ol{t}_{(\un{m},\un{0})}(\xi_{(\un{m},\un{0})})c \ol{t}_{(\un{m},\un{0})}(\eta_{(\un{m},\un{0})})^* \\
& =
\psi_\un{n}(\Theta_{y_\un{n},y_\un{n}'}^{Y_{\un{n}}})\psi_\un{m}(\Theta_{y_\un{m},y_\un{m}'}^{Y_{\un{m}}}),
\end{align*}
using that $P_{(\un{n}\vee\un{m},\un{0})}^2=P_{(\un{n}\vee\un{m},\un{0})}$ when $\un{m} \neq \un{n}\vee\un{m}$ in the first equality.
By taking finite linear combinations and their norm-limits, we conclude that (\ref{Eq:YNC}) holds when $\un{n} \neq \un{n}\vee\un{m}$.

Since (\ref{Eq:YNC}) is symmetric with respect to $\un{n}$ and $\un{m}$, taking adjoints (and relabelling) deals with the case where $\un{m} \neq \un{n} \vee \un{m}$, showing that $(\pi,t)$ is Nica-covariant.
This completes the proof.
\end{proof}

We now arrive at the next main result of this subsection, namely that the decomposition of $X$ along $\mt \neq F \subsetneq [d]$ induces a similar decomposition of the Toeplitz-Nica-Pimsner algebras.
The following is noted in \cite[Proof of Theorem 4.4 (i)]{Kak20}.

\begin{theorem}\label{T:NTXcongNTY}
Let $X$ be a product system over $\bZ_+^d$ with coefficients in a C*-algebra $A$, wherein $X_\un{i}$ admits a finite frame for all $i\in[d]$.
Let $F =[r]$ for some $r < d$.
Consider the product system $Y^F$ over $B^{F^\perp}$ related to $X$ and $F$, and define the maps
\begin{align*}
\pi & \colon B^{F^\perp}\to\N\T_X; \pi(b)=b \foral b\in B^{F^\perp}, \\
t_\un{n} & \colon Y_\un{n}^F\to\N\T_X; t_\un{n}(y_\un{n})=y_\un{n} \foral y_\un{n}\in Y_\un{n}^F, \un{n}\in\bZ_+^r\setminus\{\un{0}\}.
\end{align*}
Then the induced map $\pi \times t \colon \N\T_{Y^F} \to \N\T_X$ is a $*$-isomorphism.
\end{theorem}

\begin{proof}
For notational convenience, we set $B:=B^{F^\perp}$ and $Y:=Y^F$.
By Proposition \ref{P:YonNTX} we have that $(\pi,t)$ is Nica-covariant and thus $\pi \times t$ is well-defined.
Let $\ga$ be the gauge action of $(\ol{\pi}_Y,\ol{t}_Y)$.
It is routine to check that $\be_F$ defines a gauge action of $(\pi,t)$.
Then $\pi \times t$ is an equivariant $*$-epimorphism, and it suffices to show that $\pi \times t$ is injective on $\N\T_Y^\ga$.
By Proposition \ref{P:inj fp}, this amounts to showing that $\pi \times t$ is injective on the $[\un{0}, \un{1}_{[r]}]$-core.

Towards contradiction, suppose that $\ker\pi \times t \cap B_{[\un{0},\un{1}_{[r]}]}^{(\ol{\pi}_Y,\ol{t}_Y)} \neq \{0\}$.
We claim that we can find $0 \neq f \in \ker\pi \times t \cap B_{[\un{0},\un{1}_{[r]}]}^{(\ol{\pi}_Y,\ol{t}_Y)}$ of the form
\begin{equation} \label{eq:form f}
f=\ol{\pi}_Y(b)+\sum\{\ol{\psi}_{Y,\un{n}}(k_\un{n})\mid \un{0}\neq\un{n}\leq\un{1}_{[r]}\},
\; \text{where} \; 
0 \neq b \in B_+
\; \text{and each} \;
k_\un{n}\in\K(Y_\un{n}).
\end{equation}
Indeed, start by taking $0\neq g\in\ker\pi \times t\cap B_{[\un{0},\un{1}_{[r]}]}^{(\ol{\pi}_Y,\ol{t}_Y)}$, so that
\[
g=\ol{\pi}_Y(b')+\sum\{\ol{\psi}_{Y,\un{n}}(k_\un{n}')\mid \un{0}\neq\un{n}\leq\un{1}_{[r]}\},
\; \text{where} \; 
b' \in B
\; \text{and each} \;
k_\un{n}' \in \K(Y_\un{n}).
\]
If $b' \neq 0$, then choose $f = g^*g$.
If $b' =0 $, then choose $\un{0}\neq\un{m}\leq\un{1}_{[r]}$ minimal such that $k_\un{m}'\neq 0$.
We may assume that $\un{m}\neq\un{1}_{[r]}$, as otherwise we would have $g=\ol{\psi}_{Y,\un{1}_{[r]}}(k_{\un{1}_{[r]}}')$ and injectivity of $\psi_{\un{1}_{[r]}}$ would give that $g=0$, a contradiction.
Since $k_\un{m}'\neq 0$, we may find $0\neq y_\un{m}\in Y_\un{m}$ such that $k_\un{m}'y_\un{m}\neq 0$.
We set
\[
f:=\ol{t}_{Y,\un{m}}(k_\un{m}'y_\un{m})^*g\ol{t}_{Y,\un{m}}(y_\un{m}) \in \ker\pi \times t \cap B_{[\un{0},\un{1}_{[r]}]}^{(\ol{\pi}_Y,\ol{t}_Y)},
\]
and we note that
\[
f=\ol{\pi}_Y(\sca{k_\un{m}'y_\un{m},k_\un{m}'y_\un{m}})+\sum\{\ol{\psi}_{Y,\un{n}}(k_\un{n}'')\mid \un{0}\neq\un{n}\leq\un{1}_{[r]}-\un{m}\} 
\]
for suitably defined $k_\un{n}''\in\K(Y_\un{n})$, for all $\un{0}\neq\un{n}\leq\un{1}_{[r]}-\un{m}$.
Notice that $0 \neq \sca{k_\un{m}'y_\un{m},k_\un{m}'y_\un{m}}\in B_+$ by construction, and so by padding with zeroes we deduce that $f$ has the required form.

Hence, without loss of generality, we may assume that $f$ is of the form (\ref{eq:form f}).
We have
\[
\pi(b)+\sum\{\psi_\un{n}(k_\un{n})\mid \un{0}\neq\un{n}\leq\un{1}_{[r]}\} = (\pi \times t)(f) = 0,
\]
and hence $\pi(b)q_F=0$ by (\ref{eq:out}).
Fixing $i\in F$, let $\{\xi_{(\un{i},\un{0})}^{(j)}\}_{j\in[N]}$ be a finite frame of $X_{(\un{i},\un{0})}$.
Then $\{\ol{t}_{(\un{i},\un{0})}(\xi_{(\un{i},\un{0})}^{(j)})\}_{j\in[N]}$ is a finite frame of $Y_\un{i}$ by Lemma \ref{L:Yprodsys}.
Hence we have
\begin{align*}
p_\un{i} 
& =
\sum_{j=1}^{N}t_\un{i}(\ol{t}_{(\un{i},\un{0})}(\xi_{(\un{i},\un{0})}^{(j)}))t_\un{i}(\ol{t}_{(\un{i},\un{0})}(\xi_{(\un{i},\un{0})}^{(j)}))^* 
=
\sum_{j=1}^{N}\ol{t}_{(\un{i},\un{0})}(\xi_{(\un{i},\un{0})}^{(j)})\ol{t}_{(\un{i},\un{0})}(\xi_{(\un{i},\un{0})}^{(j)})^* 
=\ol{p}_{(\un{i},\un{0})}.
\end{align*}
In turn, we obtain
\begin{align*}
b\ol{q}_F
& =
b\prod_{i\in F}(I-\ol{p}_{(\un{i},\un{0})}) 
=
b\prod_{i\in F}(I-p_\un{i})
=
b q_F
=
\pi(b)q_F 
=0.
\end{align*}
Since $\ol{p}_{(\un{i},\un{0})}\in \N\T_X^{\be_{F^\perp}}$ for all $i\in F$, we have
\begin{align*}
E_{\be_{F^\perp}}(b) \ol{q}_F
& =
E_{\be_{F^\perp}}(b) + \sum\{(-1)^{|D|} E_{\be_{F^\perp}}(b) \prod_{i\in D}\ol{p}_{(\un{i},\un{0})}\mid \mt\neq D\subseteq F\} \\
& =
E_{\be_{F^\perp}}\bigg( b + \sum\{(-1)^{|D|} b \prod_{i\in D}\ol{p}_{(\un{i},\un{0})}\mid \mt\neq D\subseteq F\} \bigg)
=
E_{\be_{F^\perp}}(b\ol{q}_F)
=
0,
\end{align*}
where in the second equality we use that $E_{\be_{F^\perp}}$ is an $\N\T_X^{\be_{F^\perp}}$-bimodule map.
In particular, for every $\un{n} \in \bZ_+^d$ satisfying $\un{n} \perp F$, we have 
\[
E_{\be_{F^\perp}}(b)\xi_\un{n} = E_{\be_{F^\perp}}(b)\ol{q}_F\xi_\un{n}= 0, \foral \xi_{\un{n}} \in X_{\un{n}},
\] 
and thus $E_{\be_{F^\perp}}(b) = 0$ by Corollary \ref{C:psiinj}.
Since $b \in B_+$, faithfulness of $E_{\be_{F^\perp}}$ gives the contradiction that $b = 0$, and the proof is complete.
\end{proof}

We are now ready to capture the quotient of $\N\T_X$ by the ideal $\sca{\ol{\pi}(A)\ol{q}_{\un{i}}\mid i\in F}$ induced by $F$ as a Cuntz-Nica-Pimsner algebra.

\begin{theorem}\label{T:somedir}
Let $X$ be a product system over $\bZ_+^d$ with coefficients in a C*-algebra $A$, wherein $X_\un{i}$ admits a finite frame for all $i\in[d]$.
Let $F=[r]$ for some $r < d$.
Consider the product system $Y^F$ over $B^{F^\perp}$ related to $X$ and $F$, and define the ideal
\[
I_{Y^F} := \ker \{ B^{F^\perp} \to \N\T_{Y^F}/\sca{\ol{\pi}_{Y^F}(B^{F^\perp}) \ol{q}_{Y^F,\un{i}}\mid i\in F} \}.
\]
Then the canonical $*$-isomorphism $\N\T_{Y^F} \cong \N\T_X$ descends to a $*$-isomorphism
\[
\N\O_{[Y^F]_{I_{Y^F}}} \cong \N\T_X/\sca{\ol{\pi}(A)\ol{q}_\un{i}\mid i\in F}.
\]
\end{theorem}

\begin{proof}
For notational convenience, we set $B:=B^{F^\perp}$, $Y:=Y^F$ and $I := I_{Y^F}$.
We define
\[
\mathfrak{J}_Y:=\sca{\ol{\pi}_Y(B)\ol{q}_{Y,\un{i}}\mid i\in F}\subseteq\N\T_Y
\qand
\mathfrak{J}_X:=\sca{\ol{\pi}(A)\ol{q}_\un{i}\mid i\in F}\subseteq\N\T_X.
\]
By applying Proposition \ref{P:alldir} to $Y$, we have 
\[
\N\O_{[Y]_I} \cong \N\T_Y/ \fJ_Y.
\]
We note here that $[Y]_I$ admits a finite frame by Lemma \ref{L:finframquot} and Lemma \ref{L:Yprodsys}.
Let $\pi \times t \colon \N\T_{Y} \to \N\T_X$ be the canonical $*$-isomorphism of Theorem \ref{T:NTXcongNTY}.
We will show that $(\pi \times t)(\mathfrak{J}_Y)=\mathfrak{J}_X$.
This ensures that $\pi \times t$ descends to a $*$-isomorphism on the quotients, and thus
\[
\N\O_{[Y]_I} \cong \N\T_Y/ \fJ_Y \cong \N\T_X/ \fJ_X,
\]
as required.
To this end, fix $i\in F$.
We have $(\pi \times t)(\ol{p}_{Y,\un{i}})=\ol{p}_\un{i}$ by use of a finite frame expansion, and thus 
\[
(\pi \times t)(\ol{\pi}_Y(b)\ol{q}_{Y,\un{i}})=b\ol{q}_\un{i}, \foral b \in B.
\]
Next fix $\un{n} \in \bZ_+^d$ such that $\un{n} \perp F$, and observe that
\[
\ol{q}_\un{i} \ol{t}_\un{n}(X_\un{n}) = \ol{t}_\un{n}(X_\un{n})\ol{q}_\un{i} \subseteq [\ol{t}_\un{n}(X_\un{n})\ol{\pi}(A)\ol{q}_\un{i}] \subseteq \mathfrak{J}_X,
\]
using Proposition \ref{P:pf reducing} in the first equality.
By taking adjoints, we also deduce that 
\[
\ol{t}_\un{n}(X_\un{n})^*\ol{q}_\un{i} = \ol{q}_\un{i} \ol{t}_\un{n}(X_\un{n})^* \subseteq\mathfrak{J}_X.
\]
Hence we have
\[
\ol{t}_\un{n}(X_\un{n}) \ol{t}_\un{m}(X_\un{m})^*\ol{q}_\un{i} \subseteq \mathfrak{J}_X,
\foral \un{n}, \un{m} \in \bZ_+^d \; \text{satisfying} \; \un{n}, \un{m} \perp F.
\]
Thus, by taking finite linear combinations and their norm-limits and using (\ref{Eq:Bchar}), we derive that $B\ol{q}_\un{i}\subseteq\fJ_X$ and thus in particular
\[
(\pi \times t)(\ol{\pi}_Y(b)\ol{q}_{Y,\un{i}})=b\ol{q}_\un{i}\in\mathfrak{J}_X,
\foral 
b \in B.
\]
Therefore $\pi \times t$ maps the generators of $\mathfrak{J}_Y$ into $\mathfrak{J}_X$, and it follows that  $(\pi \times t)(\mathfrak{J}_Y)\subseteq\mathfrak{J}_X$.
For the reverse inclusion, fix $i\in F$.
Then we have
\[
\ol{\pi}(a)\ol{q}_\un{i} = (\pi \times t)(\ol{\pi}_Y(\ol{\pi}(a))\ol{q}_{Y,\un{i}}), 
\foral a \in A.
\]
Note that $\ol{\pi}_Y(\ol{\pi}(A))\ol{q}_{Y,\un{i}} \subseteq \mathfrak{J}_Y$, and so the generators of $\mathfrak{J}_X$ are contained in $(\pi \times t)(\mathfrak{J}_Y)$.
Thus $\mathfrak{J}_X \subseteq (\pi \times t)(\mathfrak{J}_Y)$, completing the proof.
\end{proof}

Having generalised the first part of Proposition \ref{P:alldir}, next we account for the injectivity clause.
We will need the following proposition.
Recall that the projections $\ol{p}_{\un{i}}$ can be defined for a (just) compactly aligned product system by Remark \ref{R:pi proj}.

\begin{proposition}\label{P:fpiinj}
Let $X$ be a compactly aligned product system over $\bZ_+^d$ with coefficients in a C*-algebra $A$.
Let $F=[r]$ for some $r < d$ and fix $i\in F$.
If $X_{\un{i}}$ is injective, then the map
\[
\Phi\colon (B^{F^\perp})^{\be_{F^\perp}|_{B^{F^\perp}}}\to\L(\F X); b\mapsto b\ol{p}_\un{i}\foral b\in (B^{F^\perp})^{\be_{F^\perp}|_{B^{F^\perp}}}
\]
is an injective $*$-homomorphism.
\end{proposition}

\begin{proof}
For notational convenience, we set $B:=B^{F^\perp}$.
Note that $\ol{p}_{\un{i}} \in \ol{t}_{\un{n}}(X_{\un{n}})'$ for all $\un{n} \perp F$, and therefore $\Phi$ is a well-defined $*$-homomorphism since
\begin{align*}
B^{\be_{F^\perp}|_B} & =\ol{\spn}\{\ol{t}_\un{n}(X_\un{n})\ol{t}_\un{n}(X_\un{n})^*\mid \un{n}\perp F\}. 
\end{align*}
For each $m \in \bN$, define the finite $\vee$-closed set $\S_m := \{ \un{n} \in \bZ_+^d \mid \un{n} \leq m \cdot\un{1}_{F^c}\}$,
and notice that $B^{\be_{F^\perp}|_B}$ is the direct limit of the C*-subalgebras
\[
B^{\be_{F^\perp}|_B}_m:=\spn\{\ol{\psi}_\un{n}(\K(X_\un{n})) \mid \un{n} \in \S_m\}\foral m\in\bN.
\]
Hence, to show that $\Phi$ is injective it suffices to show that $\Phi$ is injective on every $B_m^{\be_{F^\perp}|_B}$.

To this end, fix $m\in\bN$ and let 
\[
b:=\sum_{\un{n}\in \S_m}\ol{\psi}_\un{n}(k_\un{n}) \in \ker \Phi.
\]
Then, in particular, we have
\[
b|_{\sum_{\un{\ell}\geq\un{i}}X_\un{\ell}}
=
\Phi(b)|_{\sum_{\un{\ell}\geq\un{i}}X_\un{\ell}}
=
0.
\]
Notice that $\un{m}+\un{i}\geq\un{i}$ whenever $\un{m}\perp F$, and therefore
\[
\sup\{\|b|_{X_{\un{m}+\un{i}}}\|\mid \un{m}\perp F\}\leq
\|b|_{\sum_{\un{\ell}\geq\un{i}}X_\un{\ell}}\|
=0.
\]
For each $\un{n}\in \S_m$, we have $\ol{\psi}_\un{n}(k_\un{n})=\sum_{\un{\ell}\geq\un{n}}\iota_\un{n}^\un{\ell}(k_\un{n}) \in \L(\F X)$.
Hence for each $\un{m}\perp F$ we have
\[
\|b|_{X_{\un{m}+\un{i}}}\|=\|\sum_{\un{n}\in \S_m} \sum_{\un{\ell}\geq\un{n}}\iota_\un{n}^\un{\ell}(k_\un{n}) |_{X_{\un{m}+\un{i}}}\|=\|\sum_{\substack{\un{n}\in \S_m \\ \un{n}\leq\un{m}+\un{i}}}\iota_\un{n}^{\un{m}+\un{i}}(k_\un{n})\|.
\]
We then compute
\begin{align*}
0 
& = 
\sup\{\|b|_{X_{\un{m}+\un{i}}}\|\mid \un{m}\perp F\}
=
\sup\{\|\sum_{\substack{\un{n}\in \S_m \\ \un{n}\leq\un{m}+\un{i}}}\iota_\un{n}^{\un{m}+\un{i}}(k_\un{n})\|\mid \un{m}\perp F\} \\
& =
\sup\{\|\sum_{\substack{\un{n}\in \S_m \\ \un{n}\leq\un{m}}}\iota_\un{n}^{\un{m}+\un{i}}(k_\un{n})\|\mid \un{m}\perp F\}
=
\sup\{\|\iota_\un{m}^{\un{m}+\un{i}}(\sum_{\substack{\un{n}\in \S_m \\ \un{n}\leq\un{m}}}\iota_\un{n}^\un{m}(k_\un{n}))\|\mid \un{m}\perp F\} \\
& =
\sup\{\|\sum_{\substack{\un{n}\in \S_m \\ \un{n}\leq\un{m}}}\iota_\un{n}^\un{m}(k_\un{n})\|\mid \un{m}\perp F\} 
=
\|\sum_{\un{m}\perp F} \sum_{\substack{\un{n}\in \S_m \\ \un{n}\leq\un{m}}}\iota_\un{n}^\un{m}(k_\un{n}) \| \\
& =
\|\sum_{\un{n}\in \S_m} \sum_{\substack{\un{m}\perp F \\ \un{m}\geq\un{n}}}\iota_\un{n}^\un{m}(k_\un{n}) \| 
=
\|b|_{\sum_{\un{m}\perp F}X_\un{m}}\| 
=
\|b\|.
\end{align*}
In the second line we use that $\un{n}\leq\un{m}+\un{i}$ if and only if $\un{n}\leq\un{m}$, whenever $\un{n}\in\S_m$ and $\un{m} \perp F$.
In the third line we use injectivity of $X_{\un{i}}$ to deduce that $\iota_\un{m}^{\un{m}+\un{i}}$ is isometric for all $\un{m}\perp F$.
In the final line we apply Corollary \ref{C:psiinj}.
In total, we have $b=0$ and the proof is complete.
\end{proof}

\begin{proposition}\label{P:XinjYinj}
Let $X$ be a compactly aligned product system over $\bZ_+^d$ with coefficients in a C*-algebra $A$.
Let $F=[r]$ for some $r < d$ and let $i \in F$.
If $X_\un{i}$ admits a finite frame and is injective, then $Y_\un{i}^F$ admits a finite frame and is injective.
\end{proposition}

\begin{proof}
For notational convenience, we set $B:=B^{F^\perp}$ and $Y:=Y^F$.
Let $\{\xi_{\un{i}}^{(j)}\}_{j\in[N]}$ be a finite frame of $X_\un{i}$.
Then $\{\ol{t}_\un{i}(\xi_{\un{i}}^{(j)})\}_{j\in[N]}$ defines a finite frame of $Y_\un{i}$ by Lemma \ref{L:Yprodsys}.
Next let $b \in \ker \phi_{Y_{\un{i}}}$, so that
\[
b^*b\ol{p}_\un{i}
=
\sum_{j=1}^{N} b^*b\ol{t}_\un{i}(\xi_{\un{i}}^{(j)})\ol{t}_\un{i}(\xi_{\un{i}}^{(j)})^*
=
\sum_{j=1}^{N} b^* \left(\phi_{Y_{\un{i}}}(b) \ol{t}_\un{i}(\xi_{\un{i}}^{(j)}) \right) \ol{t}_\un{i}(\xi_{\un{i}}^{(j)})^*
=
0.
\]
Noting that $\ol{p}_\un{i}\in\N\T_X^{\be_{F^\perp}}$, we obtain $E_{\be_{F^\perp}}(b^*b)\ol{p}_\un{i}=0$ by using that $E_{\be_{F^\perp}}$ is a bimodule map over $\N\T_X^{\be_{F^\perp}}$.
An application of Proposition \ref{P:fpiinj} gives that $E_{\be_{F^\perp}}(b^*b)=0$, and since $E_{\be_{F^\perp}}$ is faithful we obtain $b^*b = 0$.
We conclude that $b=0$, and the proof is complete.
\end{proof}

\begin{corollary}\label{C:somedirreg}
Let $X$ be a product system over $\bZ_+^d$ with coefficients in a C*-algebra $A$, wherein $X_\un{i}$ admits a finite frame for all $i \in [d]$.
Let $F=[r]$ for some $r < d$.
If $X_{\un{i}}$ is injective for all $i \in F$, then $Y^F$ is regular, and the canonical $*$-isomorphism $\N\T_{Y^F} \cong \N\T_X$ descends to a $*$-isomorphism
\[
\N\O_{Y^F} \cong \N\T_X/\sca{\ol{\pi}(A)\ol{q}_\un{i}\mid i\in F}.
\]
\end{corollary}

\begin{proof}
By Propositions \ref{P:inj+comp} and \ref{P:XinjYinj}, we have that $Y^F$ is regular.
Next, recall the definition of the ideal $I_{Y^F}$ of Theorem \ref{T:somedir}.
An application of the last part of Proposition \ref{P:alldir} to $Y^F$ gives that $I_{Y^F} = \{0\}$.
Thus the final claim follows by a direct application of Theorem \ref{T:somedir}.
\end{proof}

The description in Theorem \ref{T:somedir} first applies the $Y^F$ construction to $X$ and then passes to a quotient.
We can have an alternative route by first considering a quotient of $X$ and then applying the $Y^F$ construction.
To avoid confusion, we will denote by $Y^F_X$ the product system induced by $F=[r]$ and $X$, as in Definition \ref{D:YF}.

\begin{theorem}\label{T:somedir2}
Let $X$ be a product system over $\bZ_+^d$ with coefficients in a C*-algebra $A$, wherein $X_\un{i}$ admits a finite frame for all $i\in[d]$.
Let $F=[r]$ for some $r < d$, and let 
\[
\fJ := \sca{\ol{\pi}_X(A) \ol{q}_{X, \un{i}} \mid i \in F } \subseteq \N\T_X
\qand
I^F_X := \L_\mt^{\fJ} \subseteq A.
\]
Then the following hold:
\begin{enumerate}
\item $I^F_X = \{a \in A \mid \lim_{\un{n} \in \bZ_+^r} \|\phi_{(\un{n}, \un{0})}(a)\| = 0\}$.
\item $[X_{\un{i}}]_{I^F_X}$ is injective for all $i \in F$.
\item The product system $Y^F_{[X]_{I^F_X}}$ is regular.
\item The association $X \to [X]_{I^F_X} \to Y^F_{[X]_{I^F_X}}$ induces a $*$-isomorphism $\N\T_X / \fJ \cong \N\O_{Y^F_{[X]_{I^F_X}}}$.
\end{enumerate}
\end{theorem}

\begin{proof}
To ease notation, we write $I:=I^F_X$.
We note that the product systems we consider here all admit a finite frame, by virtue of Lemma \ref{L:finframquot} and Lemma \ref{L:Yprodsys}.
We note that the quotient $[X]_I$ admits a finite frame by Lemma \ref{L:finframquot}.
Item (i) follows by \cite[Proposition 4.3]{Kak20}.
By following the same arguments as in the proof of Proposition \ref{P:alldir}, but for $i \in F$ instead of $i \in [d]$, we obtain that $[X_{\un{i}}]_{I}$ is injective for all $i \in F$.
An application of Corollary \ref{C:somedirreg} yields that $Y^F_{[X]_I}$ is regular, and that there is a canonical $*$-isomorphism
\[
\N\T_{[X]_I} / \sca{\ol{\pi}_{[X]_I}([A]_I)\ol{q}_{[X]_I, \un{i}} \mid i\in F} \cong \N\O_{Y_{[X]_I}^F}.
\]
It suffices to show that the canonical map $[\hspace{1pt}\cdot\hspace{1pt}]_I \colon \N\T_X \to \N\T_{[X]_I}$ descends to a $*$-isomorphism
\[
\Phi \colon \N\T_X/\fJ \to \N\T_{[X]_I} / \sca{\ol{\pi}_{[X]_I}([A]_I)\ol{q}_{[X]_I, \un{i}} \mid i\in F}.
\]

To this end, we have $[\ol{\pi}_{X}(a) \ol{q}_{X, \un{i}}]_I = \ol{\pi}_{[X]_I}([a]_I) \ol{q}_{[X]_I, \un{i}}$ for all $a \in A$ and $i \in F$.
Therefore
\[
[\fJ]_I
=
\sca{\ol{\pi}_{[X]_I}([A]_I)\ol{q}_{[X]_I, \un{i}} \mid i\in F},
\]
and thus $\Phi$ is a well-defined $*$-epimorphism.
On the other hand, by applying item (ii) of Proposition \ref{P:tuple ideal} to $\fJ$ and noting that $I \equiv \L^\fJ_\mt$, we obtain a canonical $*$-epimorphism
\[
\Psi \colon \N\T_{[X]_I} \to \N\O([\L^\fJ]_{I}, [X]_I) \cong \N\T_X/ \fJ,
\]  
such that
\begin{align*}
\Psi(\ol{\pi}_{[X]_I}([a]_I)) & = \ol{\pi}_X(a) + \fJ 
\qand
\Psi(\ol{t}_{[X]_I, \un{n}}([\xi_{\un{n}}]_I) = \ol{t}_{X,\un{n}}(\xi_{\un{n}}) + \fJ 
\end{align*}
for all $a \in A, \xi_\un{n} \in X_{\un{n}}$ and $\un{n} \in \bZ_+^d$.
In particular, we have $\Psi( \ol{p}_{[X]_I, \un{i}} ) = \ol{p}_{X, \un{i}} + \fJ$ for all $i \in F$ using a finite frame expansion, and thus
\[
\Psi( \ol{\pi}_{[X]_I}([a]_I) \ol{q}_{[X]_I, \un{i}} ) = \ol{\pi}_X(a) \ol{q}_{X, \un{i}} + \fJ = 0,
\foral
a \in A, i \in F,
\]
by definition of $\fJ$.
Hence $\Psi$ descends to a canonical $*$-epimorphism
\[
\tilde{\Psi} \colon \N\T_{[X]_I} / \sca{\ol{\pi}_{[X]_I}([A]_I)\ol{q}_{[X]_I, \un{i}} \mid i\in F} \to \N\T_X/\fJ.
\]
By definition $\tilde{\Psi}$ is a left inverse of $\Phi$ and thus $\Phi$ is a $*$-isomorphism, as required.
\end{proof}

Theorems \ref{T:somedir} and \ref{T:somedir2} show that there is no difference as when to consider the quotient product system.

\begin{corollary}\label{C:somedir}
Let $X$ be a product system over $\bZ_+^d$ with coefficients in a C*-algebra $A$, wherein $X_\un{i}$ admits a finite frame for all $i\in[d]$.
Let $F=[r]$ for some $r < d$.
On the one hand, define the positively invariant ideal
\[
I_{Y_X^F} := \ker \{ Y^F_{X, \un{0}} \to \N\T_{Y^F_X} / \sca{\ol{\pi}_{Y^F_X}(Y^F_{X, \un{0}})\ol{q}_{Y^F_X,\un{i}}\mid i\in F} \}
\]
for the product system $Y^F_X$ related to $X$ and $F$.
On the other hand, define the positively invariant ideal
\[
I_X^F := \ker\{ A \to \N\T_X/ \sca{\ol{\pi}_X(A) \ol{q}_{X, \un{i}} \mid i \in F } \}
\]
for $X$, and consider the product system $Y^F_{[X]_{I_X^F}}$ related to $[X]_{I_X^F}$ and $F$.
Then there are canonical $*$-isomorphisms
\[
\N\O_{[Y^F_X]_{I_{Y_X^F}}} 
\cong
\N\T_X / \sca{\ol{\pi}_X(A) \ol{q}_{X, \un{i}} \mid i \in F }
\cong
\N\O_{Y^F_{[X]_{I_X^F}}}.
\]
If, in addition, $X_{\un{i}}$ is injective for all $i \in F$, then $Y_X^F$ is regular, $I_{Y_X^F} = \{0\}$ and $I_X^F = \{0\}$.
\end{corollary}

\begin{proof}
It suffices to comment on the last part.
If $X_{\un{i}}$ is injective for all $i \in F$, then $Y_X^F$ is regular by Corollary \ref{C:somedirreg}.
Thus it follows that $I_{Y_X^F} = \{0\}$ since the coefficient algebra embeds in the Cuntz-Nica-Pimsner algebra.
Finally, every $\phi_{(\un{n}, \un{0})}$ with $\un{n} \in \bZ_+^r$ is isometric, and thus $I_X^F = \{0\}$ by item (i) of Theorem \ref{T:somedir2}.
\end{proof}


\end{document}